\documentclass[a4paper, twoside]{scrartcl}
\usepackage{amsfonts,amssymb,amsmath,amsthm,amstext,amssymb,mathtools,mathrsfs,cancel}
\usepackage{subcaption,float,color,xcolor,graphicx,enumitem,multicol}
\usepackage{dsfont,yfonts}  
\usepackage[normalem]{ulem} 
\usepackage[authoryear,sort,round]{natbib}  
\usepackage{fullpage} 
\usepackage{hyperref,nicefrac,rotating,lscape,colortbl}

\addtokomafont{disposition}{\sffamily} 

\usepackage{colonequals}

\usepackage{tabularx}
\usepackage{listliketab}
\usepackage{booktabs}
\usepackage{pifont}

\usepackage{adjustbox}
\usepackage{tikz}
\usetikzlibrary{bayesnet,er,trees,shapes.symbols,mindmap,arrows,decorations,shapes.misc,shapes.arrows,chains,matrix,positioning,scopes,decorations.pathmorphing,patterns,cd,decorations.pathreplacing,calligraphy,fit}
\usepackage{tikz-cd} 
\usepackage{pgfplots}

\usepackage{soul}

\usepackage{dashbox}

\usepackage{pdflscape}

\usepackage{longtable}

\newcommand*{\OMExhaustive}{Exhaustive}
\newcommand*{\OMexhaustive}{exhaustive}

\renewcommand*{\geq}{\geqslant}

\renewcommand*{\leq}{\leqslant}

\newcommand*{\catMetProb}{\mathsf{MetProb}}
\newcommand*{\catSepMetProb}{\mathsf{SepMetProb}}
\newcommand*{\catSet}{\mathsf{Set}}

\newcommand*{\fM}{\mathfrak{M}}

\newcommand*{\cQ}{\mathcal{Q}}

\renewcommand*{\emptyset}{\varnothing}

\newcommand*{\bigO}{\mathcal{O}}
\newcommand*{\bigTheta}{\Theta}
\newcommand*{\littleOmega}{\omega}
\newcommand*{\Reals}{\mathbb{R}}
\newcommand*{\Naturals}{\mathbb{N}}

\newcommand*{\Leb}[1]{\mathcal{L}^{#1}}
\newcommand*{\ball}[2]{B_{#2}(#1)}
\newcommand*{\cball}[2]{\overline{B}_{#2}(#1)}
\newcommand*{\oball}[2]{\ball{#1}{#2}}
\newcommand*{\rcdf}[2]{\mu(\ball{#1}{#2})}

\newcommand*{\orcdf}[2]{\mu(\oball{#1}{#2})}
\newcommand*{\prob}[1]{\mathcal{P}(#1)}
\newcommand*{\ballsig}[1]{\mathcal{A}(#1)}
\newcommand*{\Borel}[1]{\mathcal{B}(#1)}

\DeclareMathOperator{\Uniform}{\mathcal{U}}

\newcommand*{\dirac}[1]{\delta_{#1}}
\newcommand*{\closure}[1]{\,\overline{\!{#1}}}

\newcommand*{\interior}[1]{\mathring{#1}}

\newcounter{ntodo}
\newcommand{\todo}[1]{\stepcounter{ntodo}\bgroup\color{red}#1\egroup}
\newcommand{\reporttodo}{\bgroup\color{red}There are \arabic{ntodo} ``todo'' commands in the text.\egroup}
\newcounter{nHL}
\newcommand{\HLsays}[1]{\stepcounter{nHL}\bgroup\color{orange}Hefin:~#1\egroup}
\newcommand{\reportHL}{\bgroup\color{orange}There are \arabic{nHL} ``HL/Hefin'' commands in the text.\egroup}
\newcounter{nIK}
\newcommand{\IKsays}[1]{\stepcounter{nIK}\bgroup\color{cyan}Ilja:~#1\egroup}
\newcommand{\reportIK}{\bgroup\color{cyan}There are \arabic{nIK} ``IK/Ilja'' commands in the text.\egroup}
\newcounter{nTJS}
\newcommand{\TJSsays}[1]{\stepcounter{nTJS}\bgroup\color{olive}Tim:~#1\egroup}
\newcommand{\reportTJS}{\bgroup\color{olive}There are \arabic{nTJS} ``TJS/Tim'' commands in the text.\egroup}

\newcommand{\Ratio}[4]{\ensuremath{\,\mathfrak{R}_{#3}^{#4}(#1,#2)}}

\newcommand*{\ns}{\ensuremath{\mathsf{ns}}}
\newcommand*{\as}{\ensuremath{\mathsf{as}}}
\newcommand*{\cp}{\ensuremath{\mathsf{cp}}}
\newcommand*{\cpas}{\ensuremath{\mathsf{cs}}}

\newcommand*{\proofpartorcasestyle}[1]{\textsf{(#1)}}
\newcommand*{\proofpartorcase}[1]{\medskip\noindent\proofpartorcasestyle{#1}}

\newcommand*{\swgs}{\textswab{gs:}}
\newcommand*{\swpgs}{\textswab{pgs:}}
\newcommand*{\swgps}{\textswab{gps:}}
\newcommand*{\swps}{\textswab{ps:}}
\newcommand*{\sws}{\textswab{s:}}
\newcommand*{\swgw}{\textswab{gw}}
\newcommand*{\swpgw}{\textswab{pgw}}
\newcommand*{\swpw}{\textswab{pw}}
\newcommand*{\sww}{\textswab{w}}
\newcommand*{\swwp}{\textswab{wp}}
\newcommand*{\swwg}{\textswab{wg}}
\newcommand*{\swgwp}{\textswab{gwp}}
\newcommand*{\swgpw}{\textswab{gpw}}
\newcommand*{\swpwg}{\textswab{pwg}}
\newcommand*{\swwpg}{\textswab{wpg}}
\newcommand*{\swwgp}{\textswab{wgp}}
\newcommand*{\swp}{\textswab{p}}
\newcommand*{\swg}{\textswab{g}}
\newcommand*{\swa}{\textswab{a}}
\newcommand*{\swe}{\textswab{e}}

\newcommand*{\swwa}{\textswab{wa}}
\newcommand*{\swwap}{\textswab{wap}}
\newcommand*{\swpwa}{\textswab{pwa}}
\newcommand*{\swwpa}{\textswab{wpa}}

\newcommand*{\swgwa}{\textswab{gwa}}
\newcommand*{\swwga}{\textswab{wga}}
\newcommand*{\swwag}{\textswab{wag}}
\newcommand*{\swpgwa}{\textswab{pgwa}}
\newcommand*{\swgpwa}{\textswab{gpwa}}

\newcommand*{\swgwpa}{\textswab{gwpa}}
\newcommand*{\swpwga}{\textswab{pwga}}
\newcommand*{\swgwap}{\textswab{gwap}}
\newcommand*{\swwpga}{\textswab{wpga}}
\newcommand*{\swwgpa}{\textswab{wgpa}}

\newcommand*{\swpwag}{\textswab{pwag}}
\newcommand*{\swwgap}{\textswab{wgap}}
\newcommand*{\swwpag}{\textswab{wpag}}
\newcommand*{\swwapg}{\textswab{wapg}}
\newcommand*{\swwagp}{\textswab{wagp}}

\newcommand*{\rMP}{\hyperref[item:def_mode_axioms_rMP]{\text{MP}}}
\newcommand*{\CP}{\hyperref[item:mode_axioms_CP]{\text{CP}}}
\newcommand*{\LP}{\hyperref[item:mode_axioms_Lebesgue]{\text{LP}}}
\newcommand*{\AP}{\hyperref[item:mode_axioms_Dirac]{\text{AP}}}
\newcommand*{\SP}{\hyperref[item:def_mode_axioms_SP]{\text{SP}}}

\newcommand*{\even}{\textup{even}}
\newcommand*{\odd}{\textup{odd}}

\newcommand*{\dist}[2]{\textup{dist}(#1,#2)}
\newcommand*{\Ball}[2]{B_{#2}(#1)}

\newcommand*{\sMass}[2]{M_{#2}^{#1}}

\newcommand*{\bR}{\mathbb{R}}
\newcommand*{\bN}{\mathbb{N}}
\newcommand*{\bZ}{\mathbb{Z}}

\makeatletter
\newcommand*\quark{\mathpalette\quark@{.5}}
\newcommand*\quark@[2]{\mathbin{\vcenter{\hbox{\scalebox{#2}{$\; \m@th#1\bullet \;$}}}}}
\makeatother

\makeatletter
\newcommand{\mylabel}[2]{#2\def\@currentlabel{#2}\label{#1}}
\makeatother

\newcolumntype{Y}{>{\fussy}X}

\newcommand*{\rd}{\mathrm{d}}

\newcommand*{\defterm}{\textbf}

\DeclareMathOperator{\supp}{supp}

\DeclareMathOperator*{\argmax}{arg\,max}
\DeclareMathOperator*{\argmin}{arg\,min}

\usepackage{pifont}
\newcommand{\cmark}{\ding{51}}
\newcommand{\xmark}{\ding{55}}

\definecolor{mygray}{rgb}{0.65,0.65,0.65}
\definecolor{accent1}{HTML}{F6CF00}
\definecolor{accent2}{HTML}{56B4E9}
\definecolor{accent3}{HTML}{D55E00}
\definecolor{accent4}{HTML}{343434}

\definecolor{darkgreen}{rgb}{0,0.4,0}

\DeclarePairedDelimiter\abs{\lvert}{\rvert}
\DeclarePairedDelimiter\mynorm{\lVert}{\rVert}
\makeatletter
\let\oldabs\abs
\def\abs{\@ifstar{\oldabs}{\oldabs*}}
\let\oldnorm\mynorm
\def\mynorm{\@ifstar{\oldnorm}{\oldnorm*}}
\makeatother

\newcommand*{\defeq}{\coloneqq}

\theoremstyle{plain}
\newtheorem{theorem}{\sffamily Theorem}[section]
\newtheorem{proposition}[theorem]{\sffamily Proposition}
\newtheorem{lemma}[theorem]{\sffamily Lemma}
\newtheorem{corollary}[theorem]{\sffamily Corollary}
\theoremstyle{definition}
\newtheorem{definition}[theorem]{\sffamily Definition}

\newtheorem{notation}[theorem]{\sffamily Notation}
\newtheorem{example}[theorem]{\sffamily Example}

\newtheorem{remark}[theorem]{\sffamily Remark}

\newcommand{\absval}[1]{\lvert #1 \rvert}
\newcommand{\innerprod}[2]{\langle #1 , #2 \rangle}
\newcommand{\norm}[1]{\lVert #1 \rVert}
\newcommand{\set}[2]{\{ #1 \mid #2 \}}
\newcommand{\bigabsval}[1]{\bigl\vert #1 \bigr\vert}

\newcommand{\bignorm}[1]{\bigl\Vert #1 \bigr\Vert}
\newcommand{\bigset}[2]{\bigl\{ #1 \,\big\vert\, #2 \bigr\}}
\newcommand{\Bigset}[2]{\Bigl\{ #1 \,\Big\vert\, #2 \Bigr\}}

\newcommand{\Set}[2]{\left\{ #1 \,\middle\vert\, #2 \right\}}

\numberwithin{equation}{section}
\numberwithin{figure}{section}
\numberwithin{table}{section}

\newcommand*{\redCref}[1]{\bgroup\color{red}\hypersetup{linkcolor=red}\Cref{#1}\egroup}
\newcommand*{\redref}[1]{\bgroup\color{red}\hypersetup{linkcolor=red}\ref{#1}\egroup}

\newcommand*{\EndExampleText}{\hfill$\blacklozenge$}
\newcommand*{\EndExampleEquation}{\tag*{$\blacklozenge$}}

\usepackage{acro}
\DeclareAcronym{pdf}{short=PDF, long=probability density function}

\acsetup{
  cite/group = true ,
  cite/cmd = \citep ,
  cite/group/cmd = \citealp,
  cite/group/pre = {; }
}

\DeclareAcronym{OM}{short=OM, long=Onsager--Machlup}
\DeclareAcronym{gs-mode}{short=\swgs-mode, long=generalised strong mode, cite=Clason2019GeneralizedMI}
\DeclareAcronym{pgs-mode}{short=\swpgs-mode, long=partial generalised strong mode}
\DeclareAcronym{gps-mode}{short=\swgps-mode, long=generalised partial strong mode}
\DeclareAcronym{ps-mode}{short=\swps-mode, long=partial strong mode}
\DeclareAcronym{s-mode}{short=\sws-mode, long=strong mode, cite=DashtiLawStuartVoss2013}
\DeclareAcronym{gw-mode}{short=\swgw-mode, long=generalised weak mode}
\DeclareAcronym{pgw-mode}{short=\swpgw-mode, long=partial generalised weak mode}
\DeclareAcronym{pw-mode}{short=\swpw-mode, long=partial weak mode}
\DeclareAcronym{w-mode}{short=\sww-mode, long=weak mode, cite=HelinBurger2015}
\DeclareAcronym{wa-mode}{short=\swwa-mode, long=weak approximating mode}
\DeclareAcronym{wp-mode}{short=\swwp-mode, long=weak partial mode}
\DeclareAcronym{wg-mode}{short=\swwg-mode, long=weak generalised mode}
\DeclareAcronym{gwp-mode}{short=\swgwp-mode, long=generalised weak partial mode}
\DeclareAcronym{gpw-mode}{short=\swgpw-mode, long=generalised partial weak mode}
\DeclareAcronym{pwg-mode}{short=\swpwg-mode, long=partial weak generalised mode}
\DeclareAcronym{wpg-mode}{short=\swwpg-mode, long=weak partial generalised mode}
\DeclareAcronym{wgp-mode}{short=\swwgp-mode, long=weak generalised partial mode}

\DeclareAcronym{wap-mode}{short=\swwap-mode, long=weak approximating partial mode}
\DeclareAcronym{gwa-mode}{short=\swgwa-mode, long=generalised weak approximating mode}
\DeclareAcronym{wga-mode}{short=\swwga-mode, long=weak generalised approximating mode}
\DeclareAcronym{wag-mode}{short=\swwag-mode, long=weak approximating generalised mode}

\DeclareAcronym{pgwa-mode}{short=\swpgwa-mode, long=partial generalised weak approximating mode}

\DeclareAcronym{gwpa-mode}{short=\swgwpa-mode, long=generalised weak partial approximating  mode}
\DeclareAcronym{pwga-mode}{short=\swpwga-mode, long=partial weak generalised approximating mode}
\DeclareAcronym{gwap-mode}{short=\swgwap-mode, long=generalised weak approximating partial mode}
\DeclareAcronym{wpga-mode}{short=\swwpga-mode, long=weak partial generalised approximating mode}
\DeclareAcronym{wgpa-mode}{short=\swwgpa-mode, long=weak generalised partial approximating mode}

\DeclareAcronym{pwag-mode}{short=\swpwag-mode, long=partial weak approximating generalised mode}
\DeclareAcronym{wgap-mode}{short=\swwgap-mode, long=weak generalised approximating partial mode}
\DeclareAcronym{wpag-mode}{short=\swwpag-mode, long=weak partial approximating generalised mode}
\DeclareAcronym{wapg-mode}{short=\swwapg-mode, long=weak approximating partial generalised mode}

\DeclareAcronym{e-mode}{short=\swe-mode, long=exotic mode}

\usepackage[noabbrev,capitalise,nameinlink]{cleveref}

\crefname{assumption}{Assumption}{Assumptions}
\crefname{conjecture}{Conjecture}{Conjectures}
\crefname{notation}{Notation}{Notation}
\crefname{openproblem}{Open Problem}{Open Problem}

\usepackage{etex}
\usepackage[a4paper, top=88pt, bottom=88pt, left=72pt, right=72pt, headsep=16pt, footskip=28pt]{geometry}
\usepackage{hyperref}
\hypersetup{
	colorlinks,
	linkcolor=blue,
	citecolor=blue,
	urlcolor=blue,
}
\newcommand*{\arXiv}[1]{\bgroup\color{blue}\href{https://arxiv.org/abs/#1}{arXiv:#1}\egroup}
\newcommand*{\doi}[1]{\bgroup\color{blue}\href{https://doi.org/#1}{doi:#1}\egroup}
\newcommand*{\email}[1]{\bgroup\color{blue}\href{mailto:#1}{#1}\egroup}
\renewcommand*{\url}[1]{\bgroup\color{blue}\href{#1}{#1}\egroup}
\usepackage{enumitem, moreenum}
\setlist[enumerate]{nosep}
\setlist[itemize]{nosep}
\usepackage{mleftright} \mleftright

\usepackage{xpatch}
\newcommand{\proofheadfont}{\bfseries\sffamily}
\xpatchcmd{\proof}{\itshape}{\proofheadfont}{}{}

\usepackage[labelfont={sf,bf}]{caption}
\usepackage{scrlayer-scrpage, xhfill}
\automark[section]{section}
\setkomafont{pageheadfoot}{\normalcolor\sffamily}
\setkomafont{pagenumber}{\normalfont\normalsize\sffamily}
\clearpairofpagestyles
\let\oldtitle\title
\renewcommand{\title}[1]{\oldtitle{#1}\newcommand{\theshorttitle}{#1}}
\newcommand{\shorttitle}[1]{\renewcommand{\theshorttitle}{#1}}
\let\oldauthor\author
\renewcommand{\author}[1]{\oldauthor{#1}\newcommand{\theshortauthor}{#1}}
\newcommand{\shortauthor}[1]{\renewcommand{\theshortauthor}{#1}}
\cohead{\xrfill[0.525ex]{0.6pt}~\theshorttitle~\xrfill[0.525ex]{0.6pt}}
\cehead{\xrfill[0.525ex]{0.6pt}~\theshortauthor~\xrfill[0.525ex]{0.6pt}}
\newcommand{\theabstract}[1]{\par\bgroup\noindent\textbf{\textsf{Abstract.}} #1\egroup}
\newcommand{\thekeywords}[1]{\par\smallskip\bgroup\noindent\textbf{\textsf{Keywords.}}\newcommand{\and}{ $\bullet$ } #1\egroup}
\newcommand{\themsc}[1]{\par\smallskip\bgroup\noindent\textbf{\textsf{2020 Mathematics Subject Classification.}}\newcommand{\and}{ $\bullet$ } #1\egroup}
\newcommand*{\affilref}[1]{\ref{affiliation#1}}
\newcommand*{\affiliation}[3]{
	\footnotetext[#1]{\label{affiliation#2}#3}
}
\usepackage[hang,flushmargin]{footmisc}

\setlist{topsep=0.3ex, itemsep=0.3ex}

\title{Classification of small-ball modes and maximum\\a posteriori estimators in metric spaces}
\shorttitle{Classification of small-ball modes and MAP estimators}
\author{
	Ilja~Klebanov\textsuperscript{\affilref{FUB}}
	\and
	Hefin~Lambley\textsuperscript{\affilref{WarwickSingle}}
	\and
	T.~J.~Sullivan\textsuperscript{\affilref{WarwickJoint}}
}
\shortauthor{I.~Klebanov, H.~Lambley, and T.~J.~Sullivan}
\date{\today}

\begin{document}

\clearpage
\pagenumbering{arabic}
\setcounter{page}{1}

\maketitle
\affiliation{1}{FUB}{\raggedright Freie Universit{\"a}t Berlin, Arnimallee 6, 14195 Berlin, Germany (\email{klebanov@zedat.fu-berlin.de})}
\affiliation{2}{WarwickSingle}{\raggedright Mathematics Institute, University of Warwick, Coventry, CV4~7AL, United Kingdom (\email{hefin.lambley@warwick.ac.uk})}
\affiliation{3}{WarwickJoint}{\raggedright Mathematics Institute and School of Engineering, University of Warwick, Coventry, CV4~7AL, United Kingdom (\email{t.j.sullivan@warwick.ac.uk})}

\begin{abstract}\small
	\theabstract{A mode, or `most likely point', for a probability measure $\mu$ can be defined in various ways via the asymptotic behaviour of the $\mu$-mass of balls as their radius tends to zero.
Such points are of intrinsic interest in the local theory of measures on metric spaces and also arise naturally in the study of Bayesian inverse problems and diffusion processes.
Building upon special cases already proposed in the literature, this paper develops a systematic framework for defining modes through small-ball probabilities.
We propose `common-sense' axioms that such definitions should obey, including appropriate treatment of discrete and absolutely continuous measures, as well as symmetry and invariance properties.
We show that there are exactly ten such definitions consistent with these axioms, and that they are partially but not totally ordered in strength, forming a complete, distributive lattice.
We also show how this classification simplifies for well-behaved $\mu$.}
	\thekeywords{{Bayesian inverse problems}
\and
{local behaviour of measures}
\and
{maximum a posteriori estimation}
\and
{metric measure spaces}
\and
{modes}
}
	\themsc{28C15
\and
60B05
\and
62F10
\and
62F15
\and
62R20
\and
06A06
}
\end{abstract}

\section{Introduction}
\label{section:Introduction}

This article studies the problem of meaningfully defining modes, or `most likely points', for probability measures on metric spaces.
We do so by defining and analysing maps of the form
\begin{equation}
	\label{eq:intro_mode_map}
	\begin{split}
	\fM \colon \mathsf{U} \subseteq \catMetProb & \to \catSet \\
	(X, d, \mu) & \mapsto \{ \text{most likely points of $X$, or \emph{modes}, under $\mu$} \} ,
	\end{split}
\end{equation}
which we call \emph{mode maps}, where $\catMetProb$ is the class of metric probability spaces and $\catSet$ is the class of sets.
We are interested in the properties of, and relations among, such maps.
In particular, we focus on mode maps that have a `large enough' domain $\mathsf{U}$ to be interesting and also agree with the elementary cases of absolutely continuous and purely atomic measures, namely
\begin{align}
	\label{eq:intro_mode_map_LP}
	(\bR^{m}, d_{\text{Eucl}}, \rho \, \Leb{m}) & \mapsto \argmax_{x \in \bR^{m}} \rho(x) , \\
	\label{eq:intro_mode_map_AP}
	\Biggl( X, d, \sum_{j = 1}^{n} \alpha_{j} \delta_{x_{j}} \Biggr) & \mapsto \Set{ x_{j} }{ \alpha_{j} = \max_{1 \leq i \leq n} \alpha_{i} } .
\end{align}
In \eqref{eq:intro_mode_map_LP}, $\rho$ is a continuous \ac{pdf} with respect to Lebesgue measure $\Leb{m}$ on $\bR^{m}$, equipped with the Euclidean metric $d_{\text{Eucl}}$;
in \eqref{eq:intro_mode_map_AP}, $(X, d)$ is an arbitrary metric space, the coefficients $0 \leq \alpha_{j} \leq 1$ sum to $1$, the points $x_{j} \in X$ are distinct, and $\delta_{c}$ is the unit Dirac point mass at $c \in X$.

The task of determining the modes of a probability measure is of intrinsic interest in geometric measure theory on metric spaces, but also arises naturally across applied mathematics:
e.g., in Bayesian inference, modes of the posterior measure are \emph{maximum a posteriori estimators};
and in the study of random dynamical systems, modes of a diffusion process are \emph{minimum-action paths}.
For the classes of measures that typically arise in these applications---e.g., measures with a continuous density with respect to a Gaussian measure on a Banach space---there are several established notions of modes, characterised using small-ball probabilities in the inverse problems community \citep{DashtiLawStuartVoss2013,HelinBurger2015}, and as minimisers of an \emph{Onsager--Machlup functional} or \emph{Freidlin--Wentzell action} in the dynamical systems literature \citep{FreidlinWentzell1998,LiLi2021}.

These notions have different properties, and
may behave undesirably or not be defined at all outside of `nice' classes of measures \citep{LieSullivan2018,Clason2019GeneralizedMI,ayanbayev2021gamma}.
We therefore systematically study mode maps of the form \eqref{eq:intro_mode_map}, establishing a minimal collection of axioms that any sensible mode map should obey, among them \eqref{eq:intro_mode_map_LP} and \eqref{eq:intro_mode_map_AP}.
The axioms that we propose are by no means exhaustive, and a complete axiomatic characterisation of mode maps appears to be difficult, so we restrict attention to a specific class of maps that rely solely on small-ball probabilities.
We propose a large system of such maps, including the definitions from the Bayesian inverse problems literature, and we explicitly construct measures with fine structure that illustrate the subtle differences between definitions.
We also prove that, for some subclasses of measures, this system reduces further;
in particular, we find an unambiguous notion of small-ball mode for a large class of posteriors arising in Bayesian inverse problems.

\paragraph{Contributions.}
We make the following contributions to the study of modes of probability measures on metric spaces:
\begin{enumerate}[label=(C\arabic*),leftmargin=3em]
	\item
	\label{item:contributions_axioms}
	We define the notion of a mode map and propose axioms to enforce the correct handling of discrete and absolutely continuous probability measures, along with sensible behaviour under geometric transformations (\Cref{defn:axioms}).

	\item
	\label{item:contributions_small-ball_mode_map_notation}
	Since these axioms do not uniquely characterise a mode map, we study a class of maps that define modes using only the masses of balls in the small-radius limit.
	To describe this class we propose two notations.
	The notation of \Cref{defn:letter_notation} extends the established notions of generalised (\swg), strong (\sws), and weak (\sww) modes by representing combinations of these three properties and two new ones, `partial' (\swp) and `approximating' (\swa),  by strings in the alphabet $\{ \swg, \sws, \sww, \swp, \swa \}$.
	The more general notation of \Cref{defn:structured_definition} uses explicit strings of universal ($\forall$) and existential ($\exists$) quantifiers.

	\item
	\label{item:contributions_distinct_mode_maps}
	We verify the axioms of \ref{item:contributions_axioms} for the small-ball mode maps of \ref{item:contributions_small-ball_mode_map_notation}, finding that there are ten distinct definitions that satisfy our axioms (\Cref{thm:main}).
	We show that these definitions form a lattice when partially ordered by implication (\Cref{fig:periodic_table}).

	\item
	\label{item:contributions_counterexamples}
	We construct examples in $X = \bR$ and $X = \ell^{2}(\bN; \bR)$ to prove that no further implications among the notions in \ref{item:contributions_distinct_mode_maps} are possible.
	These examples make use of measures with fine structure, including measures for which the ball mass $\mu(\ball{u}{r})$ behaves increasingly irregularly as $r \to 0$ and measures for which, for some null sequence $(r_{n})_{n \in \bN}$, $\mu(\ball{u}{r_{n}})$ is asymptotically different from $\mu(\ball{u_{n}}{r_{n}})$ along some sequence $(u_{n})_{n \in \bN} \to u$.

	\item
	\label{item:contributions_reductions}
	We examine special cases in which the lattice of definitions in \ref{item:contributions_distinct_mode_maps} simplifies.
	Notably, in the application-relevant setting of measures with a density with respect to a Gaussian, we show that there is a unique notion of small-ball mode under mild conditions on the density.
	However, in general, multiple notions of mode may still arise even in this setting.
\end{enumerate}

\paragraph{Outline.}
\Cref{section:Setup_And_Notation} summarises the spaces of interest and defines the notation used throughout the article.
\Cref{sec:axioms} defines the notion of a mode map and states the axioms referenced in \ref{item:contributions_axioms}.
In \Cref{sec:small-ball_modes} we restrict attention to mode maps based on small-ball probabilities, state the systematic notations of \ref{item:contributions_small-ball_mode_map_notation}, and verify the axioms and equivalences needed to arrive at the ten mode maps described in \ref{item:contributions_distinct_mode_maps};
we defer the examples of \ref{item:contributions_counterexamples} to \Cref{section:Examples}.
\Cref{sec:Mode_definitions_coincide} addresses \ref{item:contributions_reductions}, investigating conditions under which many of the generally distinct mode maps coincide, and \Cref{section:Conclusion} gives some closing remarks.

\section{Setup and notation}
\label{section:Setup_And_Notation}

Throughout, $(X, d)$ will be a metric space.
The open ball centred at $u \in X$ with radius $r > 0$ is denoted $\Ball{u}{r} \defeq \set{ x \in X }{ d(x,u) < r }$, and $\dist{A}{B} \defeq \inf \set{ d(a,b) }{ a\in A,\, b\in B }$ denotes the distance between $A, B \subseteq X$.
The space $X$ will be equipped with its ball $\sigma$-algebra $\ballsig{X}$ generated by the open balls \citep[Section~1.7]{VanDerVaartWellner1996};
when $X$ is separable, $\ballsig{X}$ coincides with the Borel $\sigma$-algebra $\Borel{X}$ generated by the open sets.

The set of probability measures on $(X, \ballsig{X})$ is denoted $\prob{X}$.
Given $\mu \in \prob{X}$, the triple $(X, d, \mu)$ forms a metric probability space\footnote{Note that, unlike some authors, we do not insist in our definition of a metric probability space that $X$ be complete or separable.}, and $\catMetProb$ denotes the class of all such spaces.

The topological support of a probability measure $\mu$ on $(X, \Borel{X})$ is 
\begin{align}
	\label{eq:support}
	\supp(\mu) & \defeq \bigset{ x \in X }{ \text{for all $r > 0$, $\mu(\ball{x}{r}) > 0$} } \\
	\label{eq:support_open_set_version}
	& \hspace{0.5ex} = \bigset{ x \in X }{ \text{for all open neighbourhoods $U$ of $x$, $\mu(U) > 0$} } .
\end{align}
We use \eqref{eq:support} (but \emph{not} \eqref{eq:support_open_set_version}) to define the support of $\mu \in \prob{X}$ defined only on $\ballsig{X}$, noting that general open sets may not be $\ballsig{X}$-measurable.
We write $\sMass{\mu}{r} \defeq \sup_{x \in X} \mu(\Ball{x}{r})$.
When $X$ is separable, definitions \eqref{eq:support} and \eqref{eq:support_open_set_version} coincide, $\supp(\mu) \neq \emptyset$, and $\sMass{\mu}{r} > 0$ for all $r > 0$.

In the case $X = \bR^{m}$ for $m \in \bN$, we work with the Euclidean metric $d_{\text{Eucl}}$ and prescribe measures in terms of their \acp{pdf} with respect to the Lebesgue measure $\Leb{m}$.
Moreover, we denote by $\Uniform[a, b] \in \prob{\bR}$ the uniform distribution on $[a, b] \subset \bR$ and by $\dirac{c} \in \prob{X}$ the Dirac measure centred at $c \in X$.

Since ratios of ball masses will occur frequently in what follows, we define
\begin{equation}
	\label{eq:Ratio_wz}
	\Ratio{u}{v}{r}{\mu}
	\defeq
	\frac{\mu(B_r(u))}{\mu(B_r(v))}, \qquad
	\Ratio{u}{\sup}{r}{\mu}
	\defeq
	\frac{\mu(B_r(u))}{\sMass{\mu}{r}}, \qquad \text{$u, v \in X$.}
\end{equation}
If the denominator in either expression in \eqref{eq:Ratio_wz} is zero, then we set $\nicefrac{c}{0} \defeq \infty$ for $c>0$ and $\nicefrac{0}{0} \defeq 1$, with the latter convention reflecting the fact that neither $u$ nor $v$ strictly dominates the other for this specific radius.
The restriction $\mu|_{C} \in \prob{X}$ of $\mu$ to a set $C \in \ballsig{X}$ with $\mu(C) > 0$ is defined by
\[
	\mu|_{C}(E)
	\defeq
	\frac{\mu(C \cap E)}{\mu(C)},
	\qquad
	E \in \ballsig{X}.
\]
We write $\chi_{A} \colon X \to \bR$ for the indicator function of a subset $A\subseteq X$, taking value $1$ on $A$ and taking value $0$ otherwise.
In this manuscript, all radii $r$, $r_{n}$, and so forth are assumed to be strictly positive, and when we write $r \to 0$ or similar, only positive null sequences are meant.
We use the standard asymptotic notation for functions $f, g \colon \bN \to \bR$:
\begin{align*}
	f(n) \in \bigO\bigl(g(n)\bigr) & \iff \limsup_{n \to \infty} \frac{f(n)}{g(n)} < +\infty , \\
	f(n) \in \bigTheta\bigl(g(n)\bigr) &\iff \bigl( f(n) \in \bigO\bigl(g(n)\bigr) \text{ and } g(n) \in \bigO\bigl(f(n)\bigr) \bigr) \text{, and} \\
	f(n) \in \littleOmega\bigl(g(n)\bigr) & \iff \liminf_{n \to \infty} \frac{f(n)}{g(n)} = +\infty.
\end{align*}
The same notation applies to $f, g \colon (0, \infty) \to \bR$, with the limits $n \to \infty$ replaced with $r \to 0$.

\section{Mode maps: Definitions, examples, and axioms}
\label{sec:axioms}

This section makes precise the notion of a mode map $\fM$ referred to in \eqref{eq:intro_mode_map}.
We give some examples of mode maps that have proven useful in the literature, discuss their properties, and illustrate their differences.
With the aim of identifying a `correct' notion of a mode for a probability measure on a metric space, we then state axioms extending \eqref{eq:intro_mode_map_LP} and \eqref{eq:intro_mode_map_AP} that, in our view, should be imposed on any map claiming to identify `most likely points'.

\begin{definition}[Mode map]
	\label[definition]{defn:mode_map}
	We call any map $\fM \colon \mathsf{U} \subseteq \catMetProb \to \catSet$ satisfying $\fM(X, d, \mu) \subseteq X$ for every $(X, d, \mu) \in \mathsf{U}$ a \defterm{mode map}.
	Where no ambiguity arises, we write $\fM(\mu)$ instead of $\fM(X, d, \mu)$.
\end{definition}

This definition imposes no significant constraints on a mode map $\fM$;
in particular, $\fM$ need not agree with the elementary cases \eqref{eq:intro_mode_map_LP} and \eqref{eq:intro_mode_map_AP}, so one could take, e.g., one of the trivial mode maps $\fM(X, d, \mu) \defeq \emptyset$ or $\fM(X, d, \mu) \defeq X$.
From \cref{sec:small-ball_modes} onwards, we will focus on maps defined on $\catMetProb$ that meaningfully extend the elementary cases \eqref{eq:intro_mode_map_LP} and \eqref{eq:intro_mode_map_AP}.
We will examine when these mode maps agree on subclasses of $\catMetProb$ in \cref{sec:Mode_definitions_coincide}.

Before stating constraints to ensure that the notion of a `most likely point' is captured, we first give some important examples of mode maps.

\begin{example}[Most likely balls of fixed radius]
	\label[example]{ex:big-ball_modes}
	For fixed $r > 0$, one may define a mode of $\mu$ to be any point $u \in X$ for which the $\mu$-probability of the ball $\ball{u}{r}$ attains the supremal ball mass $M_{r}^{\mu} \defeq \sup_{x \in X} \mu(\ball{x}{r})$.
	This leads to the \defterm{positive-radius mode map} $\fM^{r} \colon \catMetProb \to \catSet$ given by
	\[
		\fM^{r}(\mu) \defeq \bigset{ u \in X }{ \mu(\ball{u}{r}) = M_{r}^{\mu} } .
	\]
	Points of $\fM^{r}(\mu)$ can be seen as regularised modes in which features of $\mu$ on scales smaller than $r$ are neglected.
	If $X$ is the separable dual of a Banach space, then every $\mu \in \prob{X}$ assigning zero mass to spheres has a mode of this kind, i.e., $\fM^{r}(\mu) \neq \emptyset$ \citep[Theorem~2.1]{Schmidt2024}.
	The mode map $\fM^{r}$ is not one of the trivial mode maps, but it still fails to satisfy \eqref{eq:intro_mode_map_LP} and \eqref{eq:intro_mode_map_AP}.\footnote{To see that $\fM^{r}$ does not satisfy \eqref{eq:intro_mode_map_AP}, consider $\mu \defeq \frac{1}{2} ( \dirac{-2 r \mathbin{/} 3} + \dirac{2 r \mathbin{/} 3} )$ on $\bR$, which has $0 \in \fM^{r}(\mu)$;
	similarly, for $r \leq \frac{1}{10}$, the Gaussian mixture $\mu \defeq \frac{1}{2} \bigl(N({-\frac{r}{2}}, r^{4}) + N(\frac{r}{2}, r^{4})\bigr)$ has a continuous \ac{pdf} with maxima at $\pm \frac{r}{2}$, but neither of these maximisers belongs to $\fM^{r}(\mu)$, whereas $0$ does, showing that $\fM^{r}$ does not satisfy \eqref{eq:intro_mode_map_LP}.}
	\EndExampleText
\end{example}

\begin{example}[Small-ball modes]
	\label[example]{ex:small-ball_modes}
	In the literature on Bayesian inverse problems, several notions of a mode for a probability measure $\mu$ on a metric space $X$ have been proposed.
	These notions induce very natural mode maps:
	\begin{enumerate}[label=(\alph*)]
		\item
		A \defterm{\ac{s-mode}} is any point $u \in X$ for which the $\mu$-probability of the ball $\ball{u}{r}$ asymptotically dominates the supremal ball mass $M_{r}^{\mu}$.
		Using the abbreviated notation \eqref{eq:Ratio_wz} for ball-mass ratios in which $\Ratio{u}{\sup}{r}{\mu} \defeq \mu(\ball{u}{r}) / M_{r}^{\mu}$, the associated mode map $\fM^{\sws} \colon \catMetProb \to \catSet$ is given by
		\begin{equation*}
			\fM^{\sws}(\mu) \defeq \Set{u \in X}{\liminf_{r \to 0} \Ratio{u}{\sup}{r}{\mu} \geq 1}.
		\end{equation*}

		\item
		A \defterm{\ac{w-mode}} is any point $u \in X$ for which $\mu(\ball{u}{r})$ asymptotically dominates $\mu(\ball{v}{r})$ for every other \defterm{comparison point} $v \in X \setminus \{ u \}$.
		Again using the abbreviated notation \eqref{eq:Ratio_wz} in which $\Ratio{u}{v}{r}{\mu} \defeq \mu(\ball{u}{r}) / \mu(\ball{v}{r})$, the associated mode map $\fM^{\sww} \colon \catMetProb \to \catSet$ is given by
		\begin{equation*}
			\fM^{\sww}(\mu) \defeq \Set{u \in X}{\liminf_{r \to 0} \Ratio{u}{v}{r}{\mu} \geq 1 \text{~for all $v \in X \setminus \{u\}$}}.
		\end{equation*}

		\item
		A \defterm{\ac{gs-mode}} is any point $u \in X$ such that, for every null sequence $(r_{n})_{n \in \bN}$ of radii, there exists an \defterm{approximating sequence} $(u_{n})_{n \in \bN} \to u$ for which the ball mass $\mu(\ball{u_{n}}{r_{n}})$ asymptotically dominates the supremal ball mass $M_{r_{n}}^{\mu}$.
		The associated mode map $\fM^{\swgs} \colon \catMetProb \to \catSet$ is given by
		\begin{equation*}
			\fM^{\swgs}(\mu) \defeq \Set{u \in X}{
			 \begin{aligned}
				 &\text{for all null $(r_{n})_{n \in \bN}$, there exists $(u_{n})_{n \in \bN} \to u$}  \\
				 &~~~~~~~~\text{such that~} \liminf_{n \to \infty} \Ratio{u_{n}}{\sup}{r_{n}}{\mu} \geq 1
			 \end{aligned}
		 	}.
		\end{equation*}
	\end{enumerate}
	These maps correctly identify the modes of a measure with continuous \ac{pdf}, as in \eqref{eq:intro_mode_map_LP}, and a purely atomic measure, as in \eqref{eq:intro_mode_map_AP}.
	Beyond these settings, however, the maps can behave very differently \citep{Clason2019GeneralizedMI,LieSullivan2018}, and distinguishing examples exist even in $X = \bR$ (see \Cref{example:subclass_classification}).

	We discuss the existing literature, and minor differences in our definitions, in \Cref{rk:existing_literature}.
	\EndExampleText
\end{example}

We now consider a different type of mode map arising in large-deviation theory, which characterises
modes as minimisers of some functional on the space $X$.
This mode map plays a central role in the analysis of random dynamical systems, but has the deficiency of being defined only on a subclass of $\catMetProb$.
This deficiency gives further motivation for a systematic study of more general mode maps that can be applied to any metric probability space.

\begin{example}[Minimum-action paths]
	\label[example]{ex:OM_minimisers}
	In the study of random dynamical systems, one is often interested the `most likely paths' of a diffusion conditioned on known endpoints.
	To exemplify the approach used to compute modes in this domain, consider the Brownian dynamics model \citep{LuStuartWeber2017} on the state space $\Omega = \bR^{m}$, in which the atomic configuration $x(t)$ at time $t$ is governed by the gradient flow of a potential $V \colon \Omega \to \bR$ perturbed by additive noise with temperature $\varepsilon > 0$.
	In this model, the process obeys the It\^{o} stochastic differential equation
	\begin{equation}
		\label{eq:Brownian_dynamics}
		\rd x(t) = -\nabla V(x(t)) \,\rd t + \sqrt{2\varepsilon} \,\rd w(t), \quad t \in [0, T], \quad x(0) = a \in \Omega, \quad x(T) = b \in \Omega,
	\end{equation}
	where $(w(t))_{t \geq 0}$ denotes the Wiener process and the terminal time $T > 0$ is fixed throughout.
	This induces a path measure $\mu$ on the path space $X \defeq \set{u \in C([0, T]; \Omega)}{u(0) = a, u(T) = b}$ equipped with the $\sup$-norm that, under mild conditions on $V$, is equivalent to the law of a Brownian motion conditioned on the endpoints $a$ and $b$.
	For fixed potential $V$ and temperature $\varepsilon > 0$, `most likely paths' are characterised as minimisers of the \defterm{\ac{OM} functional} $I_{V, \varepsilon} \colon X \cap H^{1}([0, T]; \Omega) \to \bR$, given, as in \citet[eq.~(1.9)]{LuStuartWeber2017}, by
	\[
		I_{V, \varepsilon}(h) \defeq \int_{0}^{T} \frac{1}{4\varepsilon} \bigabsval{\dot{h}(t)}^{2} + \frac{1}{4 \varepsilon} \bigabsval{\nabla V(h(t))}^{2} - \frac{1}{2} \Delta V(h(t)) \, \rd t .
	\]
	This induces a mode map $\fM^{\text{OM}}$, which is defined only on the subclass $\mathsf{U}$ of $\catMetProb$ consisting of all triples $(X, d, \mu)$ of the path space $X \subset C([0, T]; \Omega)$ equipped with the supremum metric $d$, with $\mu \defeq \mu_{V, \varepsilon}$ arising as the path measure of \eqref{eq:Brownian_dynamics} for some choice of potential $V$ and temperature $\varepsilon$:
	\begin{equation*}
		\fM^{\text{OM}}(\mu) \defeq \argmin_{h \in X \cap H^{1}([0, T]; \Omega)} I_{V, \varepsilon}(h) .
	\end{equation*}
	Here $\argmin$ is set-valued, allowing for the possibility of no minimisers, a unique minimiser, or multiple minimisers of $I_{V, \varepsilon}$.
	The \ac{OM} functional is derived by computing the $\mu$-probability of the ball $\ball{h}{r}$ in terms of $\mu(\ball{0}{r})$ using the Girsanov theorem \citep{DuerrBach1978};
	indeed
	\begin{equation*}
		\lim_{r \to 0} \Ratio{h}{k}{r}{\mu} = \exp\bigl(I_{V, \varepsilon}(k) - I_{V, \varepsilon}(h)\bigr)
		\quad
		\text{for } h, k \in X \cap H^{1}\bigl([0, T]; \Omega\bigr) .
	\end{equation*}
	Minimisers of $I_{V, \varepsilon}$ are closely related to \sww-modes---though in many cases $I_{V, \varepsilon}$ characterises the limiting behaviour of $\Ratio{h}{k}{r}{\mu}$ only for paths $h, k$ in a subspace $E \subset X$, as in this example where $E = X \cap H^{1}([0, T]; \Omega)$.
	Also related are minimisers of the \defterm{Freidlin--Wentzell action} \citep{FreidlinWentzell1998}, which can be viewed as `most likely paths' of \eqref{eq:Brownian_dynamics} in the zero-temperature limit $\varepsilon \to 0$.
	While such minimisers are of central importance in large-deviation theory, we cannot interpret them as modes of the path measure $\mu$ arising from \eqref{eq:Brownian_dynamics}, as we have taken the limit $\varepsilon \to 0$, and so we shall not define a `Freidlin--Wentzell mode map'.
	\EndExampleText
\end{example}

These examples illustrate that there are many plausible generalisations of `mode' to the setting of probability measures on metric spaces.
Therefore, as a first step in our search for a mode map with the most desirable properties, we propose some minimal axioms for mode maps;
we will regard any mode map satisfying them all as `meaningful' and any mode map that fails to satisfy even one of them as `meaningless'.

\begin{definition}[Axioms for mode maps]
	\label[definition]{defn:axioms}
	A mode map $\fM \colon \mathsf{U} \subseteq \catMetProb \to \catSet$ satisfies
	\begin{itemize}[leftmargin=3.5em]
		\item[(\text{LP})]
		\label{item:mode_axioms_Lebesgue}
		the \defterm{Lebesgue property} if, for any $(X, d, \mu) \in \mathsf{U}$ with $X = \bR^{m}$, $d$ the Euclidean metric, and $\mu \in \prob{\bR^{m}}$ having continuous Lebesgue \ac{pdf} $\rho$, the modes of $\mu$ are given by $\fM(\mu) = \argmax_{x \in X} \rho(x)$;

		\item[(\text{AP})]
		\label{item:mode_axioms_Dirac}
		the \defterm{atomic property} if, for any $(X, d, \mu) \in \mathsf{U}$ with
		\begin{equation*}
			\mu = \alpha_{0} \mu_{0} + \sum_{k = 1}^{K} \alpha_{k} \dirac{x_{k}}
		\end{equation*}
		for some $K\in\bN$, non-atomic $\mu_{0} \in \prob{X}$, distinct $x_{1},\dots,x_{K}\in X$, and $\alpha_{0},\dots,\alpha_{K}\in [0,1]$ with $\sum_{k=0}^{K} \alpha_{k} = 1$ and $\alpha_{0} \neq 1$, the modes of $\mu$ are given by
		\[
			\fM (\mu)
			=
			\Set{
			x_{k} }{ k=1,\dots,K \text{ with } \alpha_{k} = \max_{j = 1,\dots,K} \alpha_{j}
			};
		\]
		\item[(\text{CP})]
		\label{item:mode_axioms_CP}
		the \defterm{cloning property} if, for any $(X, d, \mu) \in \mathsf{U}$ for which $X$ is a vector space equipped with a translation-invariant metric $d$, and $\mu \in \prob{X}$ has non-empty support, and for any $\alpha \in [0,1]$ and any $b \in X$ such that $\dist{\supp(\mu)}{b + \supp(\mu)} > 0$, the space $(X, d)$ equipped with the convex combination $\nu = \alpha \mu + (1-\alpha) \mu(\quark - b)$ also lies in $\mathsf{U}$, i.e., $(X, d, \nu) \in \mathsf{U}$, and the modes of $\nu$ are given by
		\[
			\fM(\nu) 
			=
			\begin{cases}
				\fM(\mu), & \text{if } \alpha > \tfrac{1}{2} ,
				\\
				\fM(\mu)\cup (\fM(\mu)+b), & \text{if } \alpha = \tfrac{1}{2} ,
				\\
				\fM(\mu) + b, & \text{if } \alpha < \tfrac{1}{2} .
			\end{cases}
		\]
	\end{itemize}
\end{definition}

The axiom \LP\ enforces correct behaviour for measures with continuous Lebesgue \acp{pdf}, as in \eqref{eq:intro_mode_map_LP};
\AP\ slightly extends \eqref{eq:intro_mode_map_AP}, enforcing that the atom with the largest mass is the mode, and that the non-atomic part of the measure has no influence.
\CP\ enforces that, for a convex combination of two well-separated copies of a measure $\mu$, the copy with the larger weight determines the modes of the combination;
in particular, taking $\alpha = 0$ in \CP\ imposes the requirement that $\fM$ be equivariant with respect to translations.

These three axioms immediately eliminate some of the mode maps we have discussed as meaningless:
the mode map $\fM^{r} \colon \catMetProb \to \catSet$ of \cref{ex:big-ball_modes} violates all three axioms, even when considering just $(X, d, \mu) \in \catMetProb$ with $X = \bR$, and the trivial mode map $\fM(\mu) \equiv \emptyset$ defined on $\catMetProb$ violates \AP\ and \LP\ while vacuously satisfying \CP.

We point out that these axioms may vacuously hold if the domain $\mathsf{U}$ is not sufficiently rich.
For example, \LP\ holds immediately if $\mathsf{U}$ does not contain any metric probability spaces with $X = \Reals^{m}$, $m \in \Naturals$.
Our interest, however, will be determining meaningful maps on sufficiently large domains of definition $\mathsf{U}$.
Crucially, these axioms do not determine a unique meaningful mode map on $\mathsf{U} = \catMetProb$:
we shall see in \Cref{sec:small-ball_modes} that all of the mode maps from \Cref{ex:small-ball_modes} satisfy the axioms, even though their behaviour differs greatly, as do some more mode maps yet to be introduced.
While it would be highly desirable to axiomatically determine a unique mode map, this appears to be very challenging.

Therefore, this article confines itself to a more tractable problem that is mathematically rich and yet still relevant to applications:
understanding a class of \defterm{small-ball mode maps} that make use only of the $\mu$-masses of metric balls $\ball{u}{r}$, $u \in X$, in the limit $r \to 0$.
This class includes the \sws-, \sww-, and \swgs-modes along with several notions new to the literature;
notions that, from the perspective of our axiomatic approach, are as reasonable as those of \Cref{ex:small-ball_modes}.

\section{Small-ball mode maps}
\label{sec:small-ball_modes}

This section forms the core of the paper, in which we define, analyse, and classify a combinatorially large space of small-ball mode maps on $\catMetProb$.
In order to ease the transition from the mode maps of \Cref{ex:small-ball_modes} to our general setting, \Cref{subsec:s-modes} consists of a preliminary exploration of small-ball mode maps described by combinations of adjectives such as `strong' ($\sws$), `weak' ($\sww$), and `generalised' ($\swg$).
The system of small-ball maps that we establish has some deficiencies, which we discuss, but it serves as a useful warm-up for what follows.
In \Cref{subsec:structured_small-ball_modes} we introduce our comprehensive notation for small-ball mode maps.
\Cref{subsec:verification_of_axioms} contains our main result:
precisely ten out of the 282 definitions from \Cref{subsec:structured_small-ball_modes} actually satisfy the axioms of \Cref{defn:axioms} for a meaningful mode map, and we determine their partial-order (lattice) structure.
In \Cref{section:MP_discussion} we investigate some further properties that these ten meaningful small-ball mode maps do (or do not) enjoy.

\subsection{Preliminary exploration of small-ball modes}
\label{subsec:s-modes}

The common feature of the mode maps in \Cref{ex:small-ball_modes} is that a mode is characterised as a point $u$ with dominant small-ball mass, in the sense that the limit inferior of a certain ball mass ratio is greater than or equal to unity for all possible null sequences of radii $(r_{n})_{n \in \bN}$.
All that differs is the exact manner in which this limit is taken:
\begin{enumerate}[label=(\alph*)]
	\item
	an \sws-mode is a point for which the centred ball mass $\mu(\ball{u}{r_{n}})$ dominates the supremal radius-$r_{n}$ ball mass $\sMass{\mu}{r_{n}}$ over all centres as $n \to \infty$;

	\item
	a \sww-mode is a point for which the centred ball mass $\mu(\ball{u}{r_{n}})$ dominates the centred ball mass $\mu(\ball{v}{r_{n}})$ as $n \to \infty$ for all \emph{comparison points} $v \neq u$ separately;

	\item
	a \swgs-mode is a point such that, for each $(r_{n})_{n \in \bN}$, there exists an \emph{approximating sequence} $(u_{n})_{n \in \bN} \to u$ for which the non-centred ball mass $\mu(\ball{u_{n}}{r_{n}})$ dominates $\sMass{\mu}{r_{n}}$ as $n \to \infty$.
\end{enumerate}

In particular, we can interpret the $\swgs$-mode as the result of combining two concepts:
generalised ($\swg$), corresponding to the condition that `there exists an approximating sequence $(u_{n})_{n \in \bN} \to u$', and strong ($\sws$), corresponding to comparison to the supremal ball mass $M_{r_{n}}^{\mu}$.

A priori, there seems to be no reason not to combine these concepts of strong, weak, and generalised in other ways.
For example, one could consider a `\swgw-mode' or `\swwg-mode', where the order of the letters corresponds to the order of the corresponding \emph{for all} and \emph{there exists} statements.
But would such modes be the same, or even meaningful, or perhaps coincide with other mode types?
Moreover, it is natural to consider two new concepts not previously seen in the literature, motivated by the concepts we have just discussed:
\begin{itemize}
	\item
	a \emph{partial} (\swp) mode $u \in X$ would require $u$ to be dominant only along some (rather than every) null sequence $(r_{n})_{n\in\bN}$;

	\item
	an \emph{approximating} (\swa) mode would allow the comparison point $v \in X$ to have its own \emph{comparison sequence} $(v_{n})_{n \in \bN} \to v$, and the ball mass around $u$ would instead be required to dominate the non-centred ball mass $\mu(\ball{v_{n}}{r_{n}})$ over all possible $(v_{n})_{n \in \bN}$.
\end{itemize}

Many of these five adjectives, with the corresponding letters \sws, \sww, \swg, \swp, and \swa, can be combined into strings, and we stress right away that \emph{these adjectives do not generally commute}.
Further, many combinations are `grammatically incorrect':
e.g., \swa\ and \sws\ cannot be combined as there is no comparison point to be approximated;
and \swa\ cannot appear before \sww, since it would refer to a point $v$ not yet introduced.
These rules are summarised in the following definition, where, for a string $\textswab{S} = [S_{1} \ldots S_{K}]$, we write, e.g., $\sws \in \textswab{S}$ if and only if $S_{i} = \sws$ for some $i \in \{1, \dots, K\}$.

\begin{definition}[Small-ball mode maps {$\fM^{\textswab{S}}$}]
	\label[definition]{defn:letter_notation}
	Let $(X, d, \mu)$ be a metric probability space.
	Given a string \textswab{S} in the alphabet $\{ \swa, \swg, \swp, \sws, \sww \}$ containing precisely one of $\sws$ and $\sww$, with $\swa$ only ever appearing after $\sww$, and with no repeated letters, the point $u \in X$ is said to be an \defterm{\textswab{S}-mode} if it satisfies the logical sentence
	\begin{equation}
		\label{eq:S_liminf}
		\textswab{S}
		\colon\quad
		\liminf_{n \to \infty} \Ratio{u_{n}}{v_{n}}{r_{n}}{\mu} \geq 1,
	\end{equation}
	where the sequences $(u_{n})_{n \in \bN}$, $(v_{n})_{n \in \bN}$ and $(r_{n})_{n \in \bN}$ are determined by the following rules, applied by treating each letter in $\textswab{S}$ as a quantifier appearing in order of occurrence, with a special treatment of the letter $\sws$:
	
	\begin{enumerate}[label=(\alph*)]
		\item A mode can be generalised ($\swg \in \textswab{S}$), corresponding to the quantifier `there exists an approximating sequence $(u_{n})_{n \in \bN} \to u$', or non-generalised ($\swg \notin \textswab{S}$), in which case $u_{n} = u$ for all $n \in \bN$.

		\item A mode is either weak ($\sww \in \textswab{S}$), corresponding to the quantifier `for every comparison point $v \in X \setminus \{ u \}$', or strong ($\sws \in \textswab{S}$), in which case $v_{n}$ in \eqref{eq:S_liminf} must be replaced by $\sup$.
		If \sws\ does appear, then it is always placed last in the string \textswab{S} as a matter of convention.

		\item A mode with $\sww \in \textswab{S}$ can be approximating ($\swa \in \textswab{S}$), corresponding to the quantifier `for every comparison sequence $(v_{n})_{n \in \bN} \to v$', or non-approximating ($\swa \notin \textswab{S}$), in which case $v_{n} = v$ for all $n \in \bN$.
		As mentioned, the letter \swa\ must occur later than the letter \sww\ in $\textswab{S}$.

		\item A mode can be partial ($\swp \in \textswab{S}$), corresponding to the quantifier `there exists a null sequence $(r_{n})_{n \in \bN}$', in which case we often write that $u$ is a mode \emph{along} this particular null sequence of radii, or non-partial ($\swp \notin \textswab{S}$), meaning that we consider all possible null sequences, in which case the quantifier `for every null sequence $(r_{n})_{n \in \bN}$' comes first by default.
	\end{enumerate}
	We define the corresponding \defterm{mode map} $\fM^{\textswab{S}} \colon \catMetProb \to \catSet$ by
	\begin{equation}
		\label{eq:S_modemap}
		\fM^{\textswab{S}} (X, d, \mu)
		\defeq
		\bigset{u\in X}{u \text{ is an \textswab{S}-mode of } \mu }.
	\end{equation}
\end{definition}

\begin{remark}[Comparison to literature]
	\label[remark]{rk:existing_literature}
	Our definitions of $\sws$- and $\swgs$-modes follow the original proposals of \citet{DashtiLawStuartVoss2013} and \citet{Clason2019GeneralizedMI}.
	However, our definition of \acp{w-mode} subtly differs from that of \citet{HelinBurger2015}, which considered comparison points $v$ only in a subspace $E \subset X$, and used $\lim$ rather than $\liminf$.
	Our definition is the same as the \defterm{global weak mode} of \citet{ayanbayev2021gamma}, with the inconsequential difference that we do not allow the comparison point $v$ to equal the candidate mode $u$.
\end{remark}

This article's goal is to classify small-ball mode maps satisfying the axioms of \Cref{defn:axioms}, including those maps $\fM^{\textswab{S}}$ arising from \Cref{defn:letter_notation}, and the question of which mode maps satisfy the axioms is deferred to \Cref{subsec:verification_of_axioms}.
However, to illustrate our programme, we first consider a particular subclass of \textswab{S}-modes;
this subclass contains some, but not all, of the \textswab{S}-modes using only the letters $\sws$, $\sww$, $\swg$, and $\swp$.
This subclass involves just one of our two new letters and it admits a simple classification result that avoids the complexities of our full theory (\Cref{thm:main}) while still illustrating some key features.

\begin{figure}[t]
	\centering
	\begin{tikzpicture}[scale=0.9]
		\node (sws) at (0, 0) {$\sws$};
		\node (swgs) at (-2, -1) {$\swgs$};
		\node (sww) at (3, -1) {$\sww$};
		\node (swgw) at (1, -2) {$\swgw$};
		\node (swps) at (0, -3) {$\swps$};
		\node (swpgs) at (-2, -4) {$\swpgs$};
		\node (swpw) at (3, -4) {$\swpw$};
		\node (swpgw) at (1, -5) {$\swpgw$};
		\draw (sws) -- (swps) -- (swpw) -- (swpgw);
		\draw (sws) -- (swgs) -- (swpgs) -- (swpgw);
		\draw (sws) -- (sww) -- (swgw);
		\draw (swps) -- (swpgs);
		\draw (sww) -- (swpw);
		\draw[preaction={draw=white, line width=6pt}] (swgs) -- (swgw) -- (swpgw);
	\end{tikzpicture}
	\caption{The implications among the subclass of \textswab{S}-modes introduced in \Cref{example:subclass_classification} can be summarised in a Hasse diagram, in which a descending black line from $\textswab{S}_{i}$ to $\textswab{S}_{j}$ indicates that every $\textswab{S}_{i}$-mode is also an $\textswab{S}_{j}$-mode.
	The implications are all direct consequences of the respective definitions, while just three counterexamples (\Cref{example:e_but_not_pgs,example:gs_not_s_not_wp,example:ps_not_s_not_gw}) suffice to disprove the twelve converse implications.}
	\label{fig:CausalRelationsOfModeDefinitionsNew}
\end{figure}

\begin{example}[Classification of a subclass of \textswab{S}-modes]
	\label[example]{example:subclass_classification}
	Consider the \textswab{S}-mode types that use the letters \sws, \sww, \swg, and \swp\ (but not \swa) with either \sws\ or \sww\ in the final position.
	Since the existential quantifiers implicit in neighbouring `\swp' and `\swg' commute, there are eight such mode maps, namely
	\[
		\sws, \quad \swgs, \quad \swps, \quad \swpgs = \swgps, \quad \sww, \quad \swpw, \quad \swgw, \quad \text{and} \quad \swpgw = \swgpw .
	\]
	For example, given a metric probability space $(X, d, \mu)$, a point $u \in X$ is
	\begin{enumerate}[label=(\alph*)]
		\item a \ac{ps-mode} if there exists a null sequence $(r_{n})_{n \in \bN}$ with the property that $\liminf_{n \to \infty} \Ratio{u}{\sup}{r_{n}}{\mu} \geq 1$;
		\item a \ac{pgw-mode} if there exists a null sequence $(r_{n})_{n \in \bN}$ and an approximating sequence $(u_{n})_{n \in \bN}$ such that for all comparison points $v \neq u$, we have $\liminf_{n \to \infty} \Ratio{u_{n}}{v}{r_{n}}{\mu} \geq 1$.
	\end{enumerate}
	The implications among these eight mode types follow immediately from the definitions and are illustrated in \Cref{fig:CausalRelationsOfModeDefinitionsNew}.
	Moreover, these mode types are distinct, as shown by the following three distinguishing examples based on absolutely continuous measures on $X = \bR$:
	\begin{enumerate}[label=(\alph*)]
		\item
		\label{item:weak_vs_strong}
		\Cref{example:e_but_not_pgs} gives a measure with a $\sww$-mode that is not a $\swpgs$-mode;
		\item
		\label{item:generalised_vs_non_generalised}
		\Cref{example:gs_not_s_not_wp} gives a measure with a $\swgs$-mode that is not a $\swpw$-mode;
		and
		\item
		\label{item:partial_vs_non_partial}
		\Cref{example:ps_not_s_not_gw} gives a measure with a $\swps$-mode that is not a $\swgw$-mode.
	\end{enumerate}
	Even with this limited result it is clear that we cannot expect mode types to be totally ordered in strength:
	e.g., \swgs- and \swps-modes are incomparable, with neither implying the other.
	\EndExampleText
\end{example}

\begin{remark}
	\begin{enumerate}[label=(\alph*)]
		\item
		\Cref{example:e_but_not_pgs,example:gs_not_s_not_wp,example:ps_not_s_not_gw} show that differences between modes that are strong ($\sws \in \textswab{S}$) and modes that are weak ($\sww \in \textswab{S}$), and between modes that are partial ($\swp \in \textswab{S}$) and modes that are non-partial ($\swp \notin \textswab{S}$) are subtle.
		This is especially surprising for partial modes, given the replacement of the quantifier `for all null sequences' with the far less stringent quantifier `there exists a null sequence'.

		\item
		\Cref{example:e_but_not_pgs,example:gs_not_s_not_wp,example:ps_not_s_not_gw} have been carefully constructed to be as sharp as possible so that we can distinguish among the mode types using a minimal number of counterexamples.
		In terms of the partial order generated by the implications among mode types,  \Cref{example:e_but_not_pgs} distinguishes the `minimal' or `weakest' mode type using \sws\ from the `maximal' or `strongest' mode type using \sww.
		Similar remarks apply for \Cref{example:gs_not_s_not_wp,example:ps_not_s_not_gw}, and indeed we make a particular effort throughout this paper to supply such sharp counterexamples.
	\end{enumerate}
\end{remark}

Next we illustrate some further \textswab{S}-modes, including some that use the new letter $\swa$ in which comparison with the point $v$ is replaced by comparison with all sequences $(v_{n})_{n \in \bN} \to v$.
We defer all study of relations among these mode maps to the comprehensive treatment of \Cref{subsec:verification_of_axioms}.

\begin{example}[Some further \textswab{S}-modes]
	\label[example]{def:further_weak_modes}
	Let $(X, d, \mu)$ be a metric probability space.
	\begin{enumerate}[label = (\alph*)]
		\item
		A point $u\in X$ is called a \defterm{\ac{wp-mode}} if, for any comparison point $v \in X \setminus \{ u \}$, there exists a null sequence $(r_{n})_{n\in\bN}$, which we emphasise may depend on the point $v$, such that
		$\liminf_{n \to \infty} \Ratio{u}{v}{r_{n}}{\mu} \geq 1$.

		\item
		A point $u\in X$ is called a \defterm{\ac{wpag-mode}} if, for any comparison point $v \in X \setminus \{ u \}$, there exists a null sequence $(r_{n})_{n\in\bN}$ such that, for any comparison sequence $(v_{n})_{n\in\bN} \to v$, there exists an approximating sequence $(u_{n})_{n\in\bN} \to u$ such that
		$\liminf_{n \to \infty} \Ratio{u_{n}}{v_{n}}{r_{n}}{\mu} \geq 1$.
		\EndExampleText
	\end{enumerate}
\end{example}

\begin{remark}[\textswab{S}-modes versus explicit logical quantifiers]
	\label[remark]{rk:s-modes_as_quantifiers}
	\Cref{defn:letter_notation} allows us to translate a string \textswab{S} of the letters \swa, \swg, \swp, \sws, and \sww\ into a string $[Q_{1} \ldots Q_{K}]$ of universal ($\forall$) and existential ($\exists$) quantifiers $Q_{i}$ over null sequences $(r_{n})_{n \in \bN}$, approximating sequences $(u_{n})_{n \in \bN}$, comparison points $v \in X \setminus \{u\}$, and comparison sequences $(v_{n})_{n \in \bN}$.
	For example, the letter \sww\ is translated to the quantifier $Q_{i} = \text{`$\forall v \in X \setminus \{u\}$'}$;
	the letter \swp\ is translated to the quantifier $Q_{i} = \text{`$\exists (r_{n})_{n \in \bN}$'}$, but its absence means an implicit $Q_{1} = \text{`$\forall (r_{n})_{n \in \bN}$'}$ at the start of the sentence.

	The fact that one cannot obtain all `grammatically correct' strings $[Q_{1} \ldots Q_{K}]$ combining these concepts is an obvious limitation of \Cref{defn:letter_notation}.
	To address this, \Cref{subsec:structured_small-ball_modes} proposes a comprehensive notation that directly lists the quantifier string;
	this allows for many combinations not possible under \Cref{defn:letter_notation}.
	However, \Cref{defn:letter_notation} comes surprisingly close to capturing all \emph{meaningful} small-ball modes as classified by our main result, \Cref{thm:main}, other than the following deficiencies:
	
	\begin{enumerate}[label=(\alph*)]
		\item
		There is no representation of `there exists a comparison point $v \in X \setminus \{ u \}$' in \textswab{S}-mode notation.
		This turns out to be an unimportant omission, since it omits only meaningless mode maps that violate at least one of the axioms \AP, \CP, or \LP\ from \Cref{defn:axioms}.

		\item
		Neither `there exists a comparison sequence $(v_{n})_{n \in \bN} \to v$' nor `for all approximating sequences $(u_{n})_{n \in \bN} \to u$' can be expressed in \textswab{S}-mode notation.
		These too turn out to be unimportant omissions, since any meaningful mode map using them is equal to another mode map that does not use them.
		These equivalences are listed in \Cref{fig:periodic_table}(\subref{fig:table_of_21}).

		\item
		If present, the quantifier `for all null sequences $(r_{n})_{n \in \bN}$' comes first in \textswab{S}-mode notation.
		This turns out to be the sole important omission of \Cref{defn:letter_notation}:
		the systematic approach of \Cref{subsec:structured_small-ball_modes,subsec:verification_of_axioms} will uncover a single, novel, meaningful mode type strictly intermediate between \sws- and \sww-modes in which this quantifier appears in the (highly counterintuitive) final position.
		We dub this the \emph{exotic} mode owing to its unusual order of logical quantifiers and assign it a shorthand letter, $\swe$, \`a la \textswab{S}-mode notation.
	\end{enumerate}
\end{remark}

\begin{definition}
	\label[definition]{def:exotic_mode}
	Let $(X, d, \mu)$ be a metric probability space.
	A point $u \in X$ is an \defterm{\ac{e-mode}} if, for every comparison point $v \in X \setminus \{ u \}$ and comparison sequence $(v_{n})_{n \in \bN} \to v$, there exists an approximating sequence $(u_{n})_{n \in \bN} \to u$ such that, for every null sequence $(r_{n})_{n \in \bN}$, $\liminf_{n \to \infty} \Ratio{u_{n}}{v_{n}}{r_{n}}{\mu} \geq 1$.
\end{definition}

\subsection{Structured definitions of small-ball modes}
\label{subsec:structured_small-ball_modes}

\Cref{subsec:s-modes} introduced two new adjectives/qualities (partial and approximating) to complement those arising from the existing small-ball mode maps in the literature (generalised, strong, and weak);
\textswab{S}-mode notation is a simply a shorthand for certain combinations of these qualities and the logical quantifiers that they represent.
However, as \Cref{rk:s-modes_as_quantifiers} points out, \textswab{S}-mode notation does not describe \emph{all} possible logical combinations of these qualities.
Therefore, we now propose a more comprehensive notation $\fM[Q_{1} \ldots Q_{K}]$ for small-ball mode maps.
This notation directly lists the quantifiers $Q_{i}$ and we will specify rules to ensure that this leads to a valid logical sentence.
Doing so allows us to define and analyse a combinatorially large number of mode maps in a structured but easily readable way.

\begin{notation}
	\label[notation]{notation:comprehensive}
	Let $X$ be a metric space and let $u\in X$.
	We will use the following abbreviations:
	\begin{itemize}
		\item
		$\ns$: null sequence $(r_{n})_{n\in\bN} \subset (0, \infty)$ of radii;
		\item
		$\as$: approximating sequence $(u_{n})_{n\in\bN}$ converging to the candidate mode $u$ in $X$;
		\item
		$\cp$: comparison point $v \in X \setminus \{ u \}$;
		\item
		$\cpas$: comparison sequence $(v_{n})_{n\in\bN}$ converging to the $\cp$ $v$ in $X$.
	\end{itemize}
	Also, given a sequence $\mathcal{Q} = [Q_{1} \dots Q_{K}]$, we slightly abuse notation and write `$Q \in \mathcal{Q}$' when $Q = Q_{i}$ for some $i \in \{1, \dots, K\}$.
\end{notation}

\begin{definition}[Small-ball mode maps {$\fM [ Q_{1} \dots Q_{K} ] $}]
	\label[definition]{defn:structured_definition}
	Let $(X, d, \mu)$ be a metric probability space.
	Given a sequence $\mathcal{Q} = [Q_{1} \ldots Q_{K}]$, $K \leq 4$, of quantifiers
	\[
		Q_{k} \in \{ \forall \ns, \exists \ns, \forall \cp, \exists \cp, \forall \as, \exists \as, \forall \cpas, \exists \cpas \},
		\qquad
		k=1,\dots,K,
	\]
	we call $u\in X$ a \defterm{$\cQ$-mode} if
	\begin{equation}
		\label{eq:Q_liminf}
		Q_{1} \ldots Q_{K} \colon
		\qquad
		\liminf_{n \to \infty} \Ratio{\star_{1}}{\star_{2}}{r_{n}}{\mu}
		\geq
		1,
	\end{equation}
	where, recalling that the quantifiers $Q_{1}, \dots, Q_{K}$ may define $(u_{n})_{n \in \bN}$, $v$, and $(v_{n})_{n \in \bN}$, we take
	\begin{equation*}
		\star_{1} =
		\begin{cases}
			u_{n}, & \text{if $\forall \as \in \mathcal{Q}$ or $\exists \as \in \mathcal{Q}$,}
			\\
			u, & \text{otherwise;}
		\end{cases}\qquad
		\star_{2} =
		\begin{cases}
			v_{n}, & \text{if $\forall \cpas \in \mathcal{Q}$ or $\exists \cpas \in \mathcal{Q}$,}
			\\
			\sup, & \text{if neither $\forall \cp \in \mathcal{Q}$ nor $\exists \cp \in \mathcal{Q}$,}
			\\
			v, & \text{otherwise;}
		\end{cases}
	\end{equation*}
	and the sequence $\mathcal{Q}$ of quantifiers obeys the following rules:
	\begin{enumerate}[label=(\alph*)]
		\item
		\label{item:ns_must_appear}
		$\ns$ must appear: either $\forall \ns \in \mathcal{Q}$ or $\exists \ns \in \mathcal{Q}$;
		\item
		\label{item:no_cpas_before_cp}
		$\cpas$ can only appear after $\cp$: if $Q_{k} \in \{ \forall \cpas, \exists \cpas \}$ for some $k=1,\dots,K$, then $Q_{j} \in \{ \forall \cp, \exists \cp \}$ for some $j < k$;
		and
		\item each of $\ns,\cp,\as,\cpas$ appears at most once:
		e.g., if $\forall \ns \in \mathcal{Q}$, then $\exists \ns \notin \mathcal{Q}$, and so forth.
	\end{enumerate}
	We define the corresponding \defterm{mode map} $\fM [\cQ]\colon \catMetProb \to \catSet$ by
	\begin{equation}
		\label{eq:Q_modemap}
		\fM [\cQ](X, d, \mu)
		\defeq
		\bigset{u\in X}{u \text{ is a $\cQ$-mode of } \mu }.
	\end{equation}
\end{definition}

We now consider three examples of mode maps $\fM[\cQ]$.
The first and third are expressible in $\textswab{S}$-mode notation, but the second is only expressible in $\cQ$-mode notation.
As discussed in \Cref{rk:s-modes_as_quantifiers}, the added flexibility of the new notation results in some maps which appear very unnatural when compared to the existing small-ball mode maps of \Cref{ex:small-ball_modes}.

\begin{example}
	\begin{enumerate}[label=(\alph*)]
		\item
		A $[\exists \ns \exists \as]$-mode is any point $u \in X$ such that, for some \ns\ $(r_{n})_{n \in \bN}$, there exists an \as\ $(u_{n})_{n \in \bN} \to u$ such that \eqref{eq:Q_liminf} holds with $\star_{2} = \sup$.
		This can also be expressed in the \textswab{S}-mode notation of \Cref{defn:letter_notation} as a \ac{pgs-mode}.

		\item
		A $[\forall \cp \forall \cpas \exists \as \forall \ns]$-mode is precisely an \swe-mode.
		As discussed in \Cref{rk:s-modes_as_quantifiers}, \swe-modes prove to be the only meaningful mode type not expressible in \textswab{S}-mode notation.

		\item
		Moving `$\forall (r_{n})_{n \in \bN}$' from its final position in the definition of an \swe-mode to the more usual first position yields a $[\forall \ns \forall \cp \forall \cpas \exists \as]$-mode, which is precisely a \swwag-mode.
		This notion is distinct from the \swe-mode, emphasising the need to allow for mode types in which $\forall \ns$ does not appear first;
		moreover \swwag-modes turn out to violate \CP\ (\Cref{prop:mode_maps_violating_CP}\ref{item:violation_CP_e}).
		\EndExampleText
	\end{enumerate}
\end{example}

\subsection{Classification of small-ball mode maps}
\label{subsec:verification_of_axioms}

\begin{figure}[tp]
	\centering
	\begin{subfigure}{\textwidth}
        \centering
		\scalebox{0.95}{\begin{tabular}{|c|l|}
			\hline
			{Canonical representative} & {Equivalent definitions}  \\
			\hline
			$\sws = [\forall \ns]$   & $[\exists \as \forall \ns]$\\
			$\swgs = [\forall \ns \exists \as]$ & --- \\
			$\swps = [\exists \ns]$  & $[\forall \as \exists \ns]$ \\
			$\swpgs = [\exists \ns \exists \as]$ & --- \\
			$\swe = [\forall \cp \forall \cpas \exists \as \forall \ns]$ & --- \\
			$\sww = [\forall \ns \forall \cp]$  & $[\exists \as \forall \ns \forall \cp]$, $[\forall \cp \exists \as \forall \ns]$, $[\forall \cp \exists \cpas \forall \ns]$, $[\exists \as \forall \cp \exists \cpas \forall \ns]$, $[\forall \cp \exists \as \exists \cpas \forall \ns]$ \\
			$\swpw = [\exists \ns \forall \cp]$  & $[\forall \as \exists \ns \forall \cp]$\\
			$\swwp = [\forall \cp \exists \ns]$ & $[\forall \as \forall \cp \exists \ns]$, $\swwap = [\forall \cp \forall \cpas \exists \ns]$,  $[\forall \as \forall \cp \forall \cpas \exists \ns]$ \\
			$\swgwap = [\exists \as \forall \cp \forall \cpas \exists \ns]$  & --- \\
			$\swwgap = [\forall \cp \exists \as \forall \cpas \exists \ns]$  & --- \\
			\hline
		\end{tabular}}
		\vspace{1em}
		\caption{There are 282 valid mode maps $\fM[\cQ]$ satisfying  \Cref{defn:structured_definition}, as listed in \Cref{table:enumeration}.
		Of these 282 maps, 138 are trivially equal (as functions) to others and are not listed here, and only the 21 listed here satisfy the properties \AP, \LP, and \CP.
		By \Cref{prop:mode_map_equivalences}, these 21 definitions form ten equivalence classes.
		We have taken the canonical representative to be the earliest one in the enumeration of \Cref{table:enumeration}, which is also a shortest expression of the definition in the notation of \Cref{defn:structured_definition}.
		We also give the corresponding description of the mode in the alphabetical notation of \Cref{defn:letter_notation,def:exotic_mode}.
	}
		\label{fig:table_of_21}
    \end{subfigure}

    \begin{subfigure}{\textwidth}
        \centering
		\vspace{1em}
		\begin{tikzpicture}[scale=0.9]
			\draw[draw=none, use as bounding box] (-3.0, 0.5) rectangle (11.5, -6.5);
			\node (sws) at (0, 0) {$\sws$};
			\node (swe) at (2, -1) {$\swe$};
			\node (sww) at (2, -2) {$\sww$};
			\node (swgs) at (-2, -1) {$\swgs$};
			\node (swps) at (0, -2) {$\swps$};
			\node (swpw) at (2, -3) {$\swpw$};
			\node (swpgs) at (-2, -3) {$\swpgs$};
			\node (swwp) at (2, -4) {$\swwp$};
			\node (swgwap) at (0, -5) {$\swgwap$};
			\node (swwgap) at (0, -6) {$\swwgap$};
			\draw (sws) -- (swgs) -- (swpgs);
			\draw (sws) -- (swps) -- (swpgs);
			\draw[rounded corners] (swpgs) -- (-2, -4) -- (swgwap);
			\draw (sws) -- (swe) -- (sww) -- (swpw) -- (swwp) -- (swgwap) -- (swwgap);
			\draw (swps) -- (swpw);

			\node[red, anchor=west] (e_but_not_pgs) at (3.5,-0.5) {\footnotesize Example~\redref{example:e_but_not_pgs} (\swe\ but not \swpgs)};
			\draw[red, rounded corners] (e_but_not_pgs) -- (0.85,-0.5) -- (0.85, -3.425) -- (-1.0,-4.35) -- (-2.5,-4.35);

			\node[red, anchor=west] (w_not_wpg) at (3.5, -1.5) {\footnotesize Example~\redref{example:CP} (\sww\ but neither \swpgs\ nor \swe)};
			\draw[red, rounded corners] (w_not_wpg) -- (1.0,-1.5) -- (1.0,-3.5) -- (-1.0,-4.5) -- (-2.5,-4.5);

			\node[red, anchor=west]  (SuspensionBridge) at (3.5,-2.5) {\footnotesize Example~\redref{example:ps_not_s_not_gw} (\swps\ but neither \swgs\ nor \sww)};
			\draw[red, rounded corners] (SuspensionBridge) -- (2,-2.5) -- (0,-1.55) -- (-2,-2.5) -- (-2.5,-2.5);

			\node[red, anchor=west] (wp_not_pw_pgs) at (3.5,-3.5) {\footnotesize Example~\redref{example:wp_but_not_pw_or_pgs} (\swwp\ but neither \swpw\ nor \swpgs)};
			\draw[red, rounded corners] (wp_not_pw_pgs) -- (1.3,-3.5) -- (-1.0,-4.65) -- (-2.5,-4.65);

			\node[red, anchor=west] (GeneralisedModesNotStrong) at (3.5, -4.5) {\footnotesize Example~\redref{example:gs_not_s_not_wp} (\swgs\ but not \swwp)};
			\draw[red, rounded corners] (GeneralisedModesNotStrong) -- (0.9,-4.5) -- (-1.0,-3.5) -- (-1.0,-0.1);

			\node[red, anchor=west] (wgap_but_not_gwap) at (3.5, -5.5) {\footnotesize Example~\redref{example:wgap_but_not_gwap} (\swwgap\ but not \swgwap)};
			\draw[red] (wgap_but_not_gwap) -- (-2.5,-5.5);
		\end{tikzpicture}
		\caption{When partially ordered by implication, the ten distinct, meaningful, small-ball mode types form the complete, distributive lattice illustrated above.
		A descending black line from $\textswab{S}_{i}$ to $\textswab{S}_{j}$ indicates that every $\textswab{S}_{i}$-mode is also an $\textswab{S}_{j}$-mode.
		Each red line separates the lattice into two subsets:
		those mode types below the red line for which the indicated example is a mode, and those above the red line for which it is not.
		The six examples show that there are, in general, no further implications among these ten mode types.
		However, the lattice structure does simplify further for particular classes of measures, as shown in \Cref{fig:Hasse_CASIO,fig:Hasse_in_special_cases,fig:Hasse_dichotomy_2,fig:Hasse_Gaussian}.
		}
		\label{fig:Hasse_diagram_main}
    \end{subfigure}
	\caption{The ten meaningful small-ball mode definitions of \Cref{thm:main}, their equivalent formulations, and implication structure.}
	\label{fig:periodic_table}
\end{figure}

Not all of the mode maps $\fM[\cQ]$ can be expected to satisfy the axioms of \Cref{defn:axioms}.
For example, those containing the quantifier $\exists \cp$ involve comparison of the candidate mode $u \in X$ with just one \cp\ $v \in X \setminus \{ u \}$, and will inevitably violate \AP\ and \LP.

The following theorem, which is the main result of this article, completely classifies the small-ball mode maps that do satisfy the axioms \AP, \CP, and \LP, and the implications among them.
In what follows, we say two mode definitions are equivalent if the corresponding mode maps are equal as functions, and define equivalence classes of mode definitions correspondingly.

\begin{theorem}[Classification of small-ball mode maps]
	\label{thm:main}
	There are 282 small-ball mode maps $\fM [\cQ] \colon \catMetProb \to \catSet$ that satisfy \Cref{defn:structured_definition}.
	Of these, there are ten meaningful maps that satisfy the axioms \LP, \AP, and \CP\ stated in \Cref{defn:axioms}, namely with $\cQ$ given by
	\begin{align*}
		\sws &= [\forall \ns], &
		\swgs &= [\forall \ns \exists \as], &
		\swps &= [\exists \ns], &
		\swpgs &= [\exists \ns \exists \as], \\
		\swe &= [\forall \cp \forall \cpas \exists \as \forall \ns], &
		\sww &= [\forall \ns \forall \cp], &
		\swpw &= [\exists \ns \forall \cp], &
		\swwp &= [\forall \cp \exists \ns], \\
		\swgwap &= [\exists \as \forall \cp \forall \cpas \exists \ns], &
		\swwgap &= [\forall \cp \exists \as \forall \cpas \exists \ns].
	\end{align*}
	The ten meaningful definitions are listed, along with those equivalent to the ten meaningful definitions, in \Cref{fig:periodic_table}(\subref{fig:table_of_21}).
	Any definition not listed is either trivially equivalent to another in \Cref{fig:periodic_table}(\subref{fig:table_of_21}), or else it fails at least one of \LP, \AP, and \CP\ for some probability measure on $X = \bR$ or $X = \ell^{2}(\bN; \bR)$.
	When partially ordered by implication, the ten meaningful definitions form a complete, distributive lattice, as illustrated in \Cref{fig:periodic_table}(\subref{fig:Hasse_diagram_main}).
\end{theorem}

\begin{proof}
	One can directly compute that there are 282 small-ball mode maps $\fM [\cQ]$ that satisfy \Cref{defn:structured_definition};
	for completeness, they are listed in \Cref{table:enumeration}.
	Many of these are logically equivalent:
	where a string contains a pair of neighbouring universal ($\forall$) or existential ($\exists$) quantifiers, these can be permuted to obtain a logically equivalent string.
	For example,
	$[ \exists \ns \exists \as \forall \cp]$-modes (\swpgw-modes) coincide with $[ \exists \as \exists \ns \forall \cp]$-modes (\swgpw-modes).
	Under this relation there are 144 equivalence classes of valid strings, which can be verified by exhaustion.

	\Cref{prop:mode_maps_violating_AP,prop:mode_maps_violating_CP} show that all but 21 of these 144 violate \AP\ or \CP\ for some probability measure on $X = \bR$ or $X = \ell^{2}(\bN; \bR)$.
	\Cref{prop:mode_map_equivalences} shows that these 21 definitions fall into ten equivalence classes, with each definition being equivalent to the others within its group;
	we list these classes in \Cref{fig:periodic_table}(\subref{fig:table_of_21}) and select a canonical representative for each class, leaving ten essentially distinct definitions.
	\Cref{prop:mode_map_all_axioms} shows that these ten definitions satisfy properties \LP, \AP, and \CP\ for arbitrary metric probability spaces.

	The implications among the remaining ten mode types all follow trivially from their respective definitions, and \Cref{example:e_but_not_pgs,example:gs_not_s_not_wp,example:CP,example:ps_not_s_not_gw,example:wp_but_not_pw_or_pgs,example:wgap_but_not_gwap} in $X = \bR$ distinguish them.

	The completeness, distributivity, and lattice properties of the resulting partially ordered set of ten meaningful mode types can be verified by straightforward exhaustion, leading to the lattice structure displayed in \Cref{fig:periodic_table}(\subref{fig:Hasse_diagram_main}).
\end{proof}

\begin{remark}
	\label[remark]{remark:minimality_of_counterexamples}
	We use six examples to distinguish the ten meaningful small-ball mode maps.
	This collection is minimal:
	an exhaustive check reveals that, in the lattice illustrated in \Cref{fig:periodic_table}(\subref{fig:Hasse_diagram_main}), there is no collection of five downward-closed subsets $S_{1}, \dots, S_{5}$ such that, for any two distinct mode types $\fM_{i}$ and $\fM_{j}$, there is some $k$ with $\fM_{i} \notin S_{k} \ni \fM_{j}$.
\end{remark}

The next two results, \Cref{prop:mode_maps_violating_AP,prop:mode_maps_violating_CP}, show that the vast majority of the 144 (equivalence classes) of mode maps $\fM[\cQ]$ fail to satisfy \AP\ and \CP.

\begin{proposition}[Small-ball mode maps violating \AP]
	\label[proposition]{prop:mode_maps_violating_AP}
	All mode maps $\fM [\mathcal{Q}] \colon \catMetProb \to \catSet$ of the following forms violate \AP:

	\begin{enumerate}[label = (\alph*)]
		\item
		\label{item:violation_AP_a}
		any map with $\exists \cp \in \mathcal{Q}$;

		\item
		\label{item:violation_AP_b}
		any map with both
		$\forall \ns \in \mathcal{Q}$ and $\forall \as \in \mathcal{Q}$;

		\item
		\label{item:violation_AP_c}
		any map with both
		$\exists \cpas \in \mathcal{Q}$ and $\exists \ns \in \mathcal{Q}$;

		\item
		\label{item:violation_AP_d}
		any map with $\exists \ns$ appearing before $\forall \as$ in $\mathcal{Q}$;

		\item
		\label{item:violation_AP_e}
		any map with $\forall \ns$ appearing before $\exists \cpas$ in $\mathcal{Q}$.
	\end{enumerate}
\end{proposition}

\begin{proof}
	We will show that \AP\ is violated by all such mode types for
	\begin{equation*}
		\mu \defeq \frac{1}{2} \dirac{0} + \frac{1}{2} \Uniform[1, 3] \in \prob{\bR}.
	\end{equation*}
	Note that $\supp(\mu) = \{0\} \cup [1, 3]$, and that \AP\ mandates that $u = 0$ be the unique mode of $\mu$.

	\proofpartorcase{Case \ref{item:violation_AP_a}: $\exists \cp \in \mathcal{Q}$.}
	All such mode types declare every $u \in X$ to be a mode.
	To see this, take any \cp\ $v \notin \supp(\mu) \cup \{ u \}$.
	If $\forall \as$ or $\exists \as \in \mathcal{Q}$, then take $(u_{n})_{n \in \bN} \to u$ to be an arbitrary \as;
	otherwise, take $u_{n} = u$ for all $n \in \bN$.
	Similarly, if $\forall \cpas \in \mathcal{Q}$ or $\exists \cpas \in \mathcal{Q}$, then take $(v_{n})_{n \in \bN}$ to be an arbitrary \cpas;
	otherwise take $v_{n} = v$ for all $n \in \bN$.
	Then, regardless of the order of quantifiers, $\mu(\ball{v_{n}}{r_{n}}) = 0$ for all large $n$ as $v \notin \supp(\mu)$, and our convention that $\nicefrac{c}{0} \geq 1$ for all $c \geq 0$ implies that $\liminf_{n \to \infty} \Ratio{u_{n}}{v_{n}}{r_{n}}{\mu} \geq 1$,
	proving that $u$ is a mode.

	\proofpartorcase{Case \ref{item:violation_AP_b}: $\forall \ns \in \mathcal{Q}$ and $\forall \as \in \mathcal{Q}$.}
	Given the previous case, we may assume that $\exists \cp \notin \mathcal{Q}$;
	we will show that the remaining mode maps fail to declare $u = 0$ a mode.
	Moreover, it suffices to consider maps with $\forall \cp \in \mathcal{Q}$, as $\Ratio{u_{n}}{\sup}{r_{n}}{\mu} \leq \Ratio{u_{n}}{v_{n}}{r_{n}}{\mu}$ for any \cp\ and \cpas.
	Thus, take the \as\ $u_{n} \defeq - \frac{2}{n}$, the \ns\ $r_{n} \defeq \frac{1}{n}$, and the \cp\ $v \defeq 2$.
	For any \cpas\ $(v_{n})_{n \in \bN}$, including the constant \cpas\ $v_{n} = v$ to be taken if neither $\forall \cpas \in \mathcal{Q}$ nor $\exists \cpas \in \mathcal{Q}$, we have
	regardless of the order of quantifiers, that $\mu(\ball{v_{n}}{r_{n}}) > 0$ and $\mu(\ball{u_{n}}{r_{n}}) = 0$ for all large $n$, so $\liminf_{n \to \infty} \Ratio{u_{n}}{v_{n}}{r_{n}}{\mu} = 0$, meaning that $u = 0$ is not a mode.

	\proofpartorcase{Case \ref{item:violation_AP_c}: $\exists \cpas \in \mathcal{Q}$ and $\exists \ns \in \mathcal{Q}$.}
	As before we may assume that $\exists \cp \notin \mathcal{Q}$;
	as a $\cpas$ is declared, this means $\forall \cp \in \mathcal{Q}$.
	We will show that $u = 2$ is a mode for all of these types.
	Take the \ns\ $r_{n} \defeq \tfrac{1}{n}$, and note that, for any \cp\ $v \neq 0$, any \cpas\ $(v_{n})_{n \in \bN} \to v$, and any \as\ $(u_{n})_{n \in \bN} \to u$,
	\[
		\mu(\ball{v_{n}}{r_{n}}) \leq \tfrac{1}{2} r_{n} \qquad \text{and} \qquad \mu(\ball{u_{n}}{r_{n}}) = \tfrac{1}{2} r_{n}.
	\]
	Thus, all that remains is to consider the \cp\ $v = 0$, for which we take the \cpas\ $v_{n} \defeq -\tfrac{2}{n}$.
	If $\forall \as \in \mathcal{Q}$ or $\exists \as \in \mathcal{Q}$, then take $(u_{n})_{n \in \bN} \to u$ to be an arbitrary \as;
	otherwise take $u_{n} \defeq u$ for all $n \in \bN$.
	Then, for all $n$, $\mu(\ball{u_{n}}{r_{n}}) > 0$ and $\mu(\ball{v_{n}}{r_{n}}) = 0$, showing that $u = 2$ is a mode.

	\proofpartorcase{Case \ref{item:violation_AP_d}: $\mathcal{Q} = [\ldots \exists \ns \ldots \forall \as \ldots]$.}
	As in \ref{item:violation_AP_b}, we will show that $u = 0$ is not a mode, and we may assume that $\forall \cp \in \mathcal{Q}$.
	Hence, take the \cp\ $v \defeq 2$ and let the \ns\ $(r_{n})_{n \in \bN}$ be arbitrary.
	If $\forall \cpas \in \mathcal{Q}$ or $\exists \cpas \in \mathcal{Q}$, let $(v_{n})_{n \in \bN}$ be an arbitrary \cpas;
	otherwise let $v_{n} = v$.
	Consider the \as\ $u_{n} \defeq - 2 r_{n}$, for which $\mu(\ball{u_{n}}{r_{n}}) = 0$ and $\mu(\ball{v_{n}}{r_{n}}) > 0$ for all large $n$;
	this shows that $u = 0$ is not a mode.

	\proofpartorcase{Case \ref{item:violation_AP_e}: $\mathcal{Q} = [\ldots \forall \ns \ldots \exists \cpas \ldots]$.}
	The argument is similar to case \ref{item:violation_AP_c}, and it can be shown that $u = 2$ is a mode.
	As in this case, it is clear that $u$ dominates all $v \neq 0$;
	for the \cp\ $v = 0$, let $(r_{n})_{n \in \bN}$ be an arbitrary \ns\ and pick the \cpas\ $v_{n} = -2r_{n}$.
\end{proof}

The next result identifies a large number of mode types that violate \CP;
this includes some types, such as $[\forall \ns \forall \as \forall \cp \forall \cpas]$-modes, already found to be meaningless by \cref{prop:mode_maps_violating_AP}.

\begin{proposition}[Small-ball mode maps violating \CP]
	\label[proposition]{prop:mode_maps_violating_CP}
	All mode maps $\fM [\mathcal{Q}] \colon \catMetProb \to \catSet$ with $\forall \cp \in \mathcal{Q}$ of the following forms violate \CP:

	\begin{enumerate}[label=(\alph*)]
		\item
		\label{item:violation_CP_a}
		any map with $\forall \cpas$ appearing after $\ns$ in $\mathcal{Q}$, but $\exists \as \notin \mathcal{Q}$;

		\item
		\label{item:violation_CP_b}
		any map with $\exists \ns \in \mathcal{Q}$ and $\exists \as \in \mathcal{Q}$, but $\forall \cpas \notin \mathcal{Q}$;

		\item  
		\label{item:violation_CP_c}
		any map with $\forall \ns$ appearing before $\exists \as$ in $\mathcal{Q}$, but $\forall \cpas \notin \mathcal{Q}$;

		\item 
		\label{item:violation_CP_d}
		any map with both $\ns$ and $\exists \as$ appearing before $\forall \cpas$ in $\mathcal{Q}$;

		\item 
		\label{item:violation_CP_e}
		any map with both $\ns$ and $\forall \cpas$ appearing before $\exists \as$ in $\mathcal{Q}$.

	\end{enumerate}
\end{proposition}

\begin{proof}
	Consider the following absolutely continuous measure supported on $[0, 1] \subset \bR$:
	\begin{equation*}
		\mu = \rho \,\Leb{1} \in \prob{\bR},\qquad \rho(x) = \tfrac{1}{2} \absval{x}^{-\frac{1}{2}} \chi_{(0, 1)}(x).
	\end{equation*}
	We will show that \CP\ is violated by the mode types in \ref{item:violation_CP_a}--\ref{item:violation_CP_c} for the convex combination $\nu = \alpha \mu + (1 - \alpha) \mu(\quark - 2)$, with $\alpha = 0.49$.
	To do this, we note that, for all mode types under consideration, $0$ is the unique mode of $\mu$, so \CP\ demands that $2$ is the unique mode of $\nu$, but the constant approximating sequence under $\mu$ of $0$ is not optimal:
	\begin{equation*}
		\mu(\ball{0}{r}) = \sqrt{r},\qquad \mu(\ball{t}{r}) = \sqrt{r + t},\qquad t \in [-r, r], \qquad r < \tfrac{1}{2}.
	\end{equation*}
	These facts also apply, with the appropriate weighting, to the point $2$ under $\nu$.

	\proofpartorcase{Case \ref{item:violation_CP_a}: $\forall \cpas$ appears after $\ns$, but $\exists \as \notin \mathcal{Q}$.}
	We will show that $u = 2$ is not a mode.
	To do this, note that either $\forall \as \in \mathcal{Q}$, or $\as$ is not present in $\mathcal{Q}$, so we can take $u_{n} = u$ for all $n \in \bN$.
	Let $(r_{n})_{n \in \bN}$ be an arbitrary \ns, take the \cp\ $v = 0$, and take the \cpas\ $v_{n} = v + r_{n}$.
	Then
	\[
		\nu(\ball{u_{n}}{r_{n}}) = (1-\alpha) \sqrt{r_{n}},\qquad  \nu(\ball{v_{n}}{r_{n}}) = \alpha \sqrt{2r_{n}}\qquad \text{for all large $n$,}
	\]
	and the choice $\alpha = 0.49$ implies that $u = 2$ is not a mode as $\liminf_{n \to \infty} \Ratio{u_{n}}{v_{n}}{r_{n}}{\nu} < 1$.

	\proofpartorcase{Case \ref{item:violation_CP_b}: $\exists \ns \in \mathcal{Q}$ and $\exists \as \in \mathcal{Q}$, but $\forall \cpas \notin \mathcal{Q}$.}
	We will show that $u = 0$ is a mode.
	For any \cp\ $v \neq 2$, this is an immediate consequence of \cref{lem:mass_near_Lebesgue_point}:
	since $\rho$ is bounded away from $0$ and $2$, there is a constant $C > 0$ such that,
	for any \ns\ $(r_{n})_{n \in \bN}$ and \cpas\ $(v_{n})_{n \in \bN} \to v$,
	$\nu(\ball{v_{n}}{r_{n}}) \leq C r_{n}$ for all large $n$;
	but $\nu(\ball{0}{r}) = \alpha \sqrt{r}$.
	It remains to consider the \cp\ $v = 2$, for which we take the \cpas\ $v_{n} = v$, the \ns\ $r_{n} = n^{-1}$, and the \as\ $u_{n} = n^{-1}$.
	Then $\nu(\ball{u_{n}}{r_{n}}) = \alpha \sqrt{2r_{n}}$ and $\nu(\ball{v_{n}}{r_{n}}) = (1 - \alpha) \sqrt{r_{n}}$ for all large $n$;
	the choice $\alpha = 0.49$ ensures that $u = 0$ is a mode.

	\proofpartorcase{Case \ref{item:violation_CP_c}: $\forall \ns$ appears before $\exists \as$, but $\forall \cpas \notin \mathcal{Q}$.}
	This follows by a similar argument, with the only change being to take the \as\ $u_{n} = r_{n}$.

	\vspace{1ex}

	To show that \CP\ is violated for the mode types in \ref{item:violation_CP_d} and \ref{item:violation_CP_e}, we consider the measure $\mu \in \prob{\ell^{2}(\bN; \bR)}$ from \cref{ex:no_optimal_approximating_sequence}, which is supported on a ball of finite radius around the origin.
	For all such types, $u = 0$ is the unique mode of $\mu$, and any \as\ of $u$ can be strictly improved:
	there is a factor $\zeta \in (0, 1)$ such that
	\begin{equation} \label{eq:CP_improvement_property}
		\forall \ns\ (r_{n})_{n \in \bN} \, \forall \as\ (u_{n})_{n \in \bN} \to u \, \exists (u_{n}')_{n \in \bN} \to u\colon~ \mu(\ball{u_{n}}{r_{n}}) \leq \zeta \mu(\ball{u_{n}'}{r_{n}}).
	\end{equation}
	Consider the combination $\nu = \alpha \mu + (1 - \alpha) \mu(\quark - b)$, with $\alpha \in(0, \frac{1}{2})$, and with $b$ chosen such that $\supp \mu$ and $\supp \mu(\quark - b)$ are positively separated;
	note that the improvement property \eqref{eq:CP_improvement_property} applies equally to \as\ of $0$ and $b$ under $\nu$, and \CP\ demands that $b$ is the only mode of $\nu$.

	\proofpartorcase{Case \ref{item:violation_CP_d}: $\ns$ and $\exists \as$ appear before $\forall \cpas$.}
	We will show that $b$ is not a mode.
	Let $(r_{n})_{n \in \bN}$ be an arbitrary \ns\ and let $(u_{n})_{n \in \bN}$ be an arbitrary \as\ of $b$.
	Apply the property \eqref{eq:CP_improvement_property} repeatedly to find an \as\ $(u_{n}')_{n \in \bN} \to b$ with $\nu(\ball{u_{n}}{r_{n}}) < \frac{\alpha}{2(1-\alpha)} \nu(\ball{u_{n}'}{r_{n}})$;
	then

	take the \cp\ $v = 0$ and the \cpas\ $v_{n} = u_{n}' - b \to v$, for which
	\[
		\nu(\ball{v_{n}}{r_{n}}) = \alpha \mu(\ball{u_{n}' - b}{r_{n}}) > \alpha \cdot \frac{2(1-\alpha)}{\alpha} \mu(\ball{u_{n} - b}{r_{n}}) = 2\nu(\ball{u_{n}}{r_{n}}).
	\]
	We thus conclude that $b$ is not a mode, because $\liminf_{n \to \infty} \Ratio{u_{n}}{v_{n}}{r_{n}}{\nu} < \frac{1}{2}$.

	\proofpartorcase{Case \ref{item:violation_CP_e}: $\ns$ and $\forall \cpas$ appear before $\exists \as$.}
	A similar argument shows that $0$ is a mode, using the fact that the \as\ can be chosen after the \cpas\ is fixed, so \eqref{eq:CP_improvement_property} can again be used.
\end{proof}

We now reduce the collection of candidate mode maps yet further, by showing that many of the 21 definitions not already eliminated by \Cref{prop:mode_maps_violating_AP,prop:mode_maps_violating_CP} are equivalent:

\begin{proposition}[Equivalences among remaining small-ball mode maps]
	\label[proposition]{prop:mode_map_equivalences}
	Each of the following forms a class of equivalent mode maps $\fM[\cQ] \colon \catMetProb \to \catSet$, as displayed in \Cref{fig:periodic_table}(\subref{fig:table_of_21}):
	\begin{enumerate}[label = (\alph*)]
		\item
		\label{item:alternative_characterisation_strong_for_closed_balls}
		$\{\sws = [\forall \ns], [\exists \as \forall \ns]\}$;

		\item
		\label{item:alternative_characterisation_weak_for_closed_balls}
		$\bigl\{\text{(i)}~\sww =  [\forall \ns \forall \cp], \text{~(ii)}~[\exists \as \forall \ns \forall \cp], \text{~(iii)}~[\forall \cp \exists \as \forall \ns], \text{~(iv)}~[\forall \cp \exists \cpas \forall \ns], \text{~(v)}~[\exists \as \forall \cp \exists \cpas \forall \ns],$\\
		$\text{~(vi)}~[\forall \cp \exists \as \exists \cpas \forall \ns]\bigr\}$;

		\item
		\label{item:alternative_characterisation_partial_strong}
		$\{ \swps = [\exists \ns],  [\forall \as \exists \ns]\}$;

		\item
		\label{item:alternative_characterisation_partial_weak}
		$\{\swpw = [\exists \ns \forall \cp], [\forall \as \exists \ns \forall \cp]\}$;

		\item
		\label{item:alternative_characterisation_wp_for_spherically_non_atomic}
		$\bigl\{ \text{(i)~} \swwp = [\forall \cp \exists \ns], \text{~(ii)~} \swwap = [\forall \cp \forall \cpas \exists \ns], \text{~(iii)~}[\forall \as \forall \cp \exists \ns], \text{~(iv)~} [\forall \as \forall \cp \forall \cpas \exists \ns]\bigr\}$.
	\end{enumerate}
\end{proposition}

\begin{proof}
	Throughout, we rely on the elementary facts, proven in \cref{lem:ball_masses}, that
	\begin{equation}
		\label{eq:ball_mass_left_limits}
		\lim_{r \nearrow s} \mu(\ball{x}{r}) = \mu(\ball{x}{s}) \qquad \text{and} \qquad \lim_{r \nearrow s} \sMass{\mu}{r} = \sMass{\mu}{s} \text{~~for all $x \in X$ and $s > 0$,}
	\end{equation}
	and that, for any fixed $r > 0$, the ball-mass map $x \mapsto \mu(\ball{x}{r})$ is lower semicontinuous, i.e.,
	\begin{equation}
		\label{eq:ball_mass_lsc}
		\liminf_{y \to x} \mu(\ball{y}{r}) \geq \mu(\ball{x}{r}) \text{~~for all $x \in X$.}
	\end{equation}

	\proofpartorcase{Case \ref{item:alternative_characterisation_strong_for_closed_balls}.}
	Any $[\forall \ns]$-mode is, by definition, a $[\exists \as \forall \ns]$-mode;
	it remains to prove the converse.
	Suppose that $u$ is a $[\exists \as \forall \ns]$-mode and let $(u_{n})_{n \in \bN}$ be any \as\ proving this;
	let $(r_{n})_{n \in \bN}$ be any \ns.
	Using \eqref{eq:ball_mass_left_limits}, pick a \ns\ $(s_{n})_{n \in \bN}$ with $s_{n} < r_{n}$ and $\sMass{\mu}{s_{n}} \geq \bigl(1 - \tfrac{1}{n}\bigr) \sMass{\mu}{r_{n}}$;
	then, using that $u_{n} \to u$, pick a strictly increasing sequence $(k_{n})_{n \in \bN}$ such that $\ball{u_{k_{n}}}{s_{n}} \subseteq \ball{u}{r_{n}}$ for all $n \in \bN$.
	Then $u$ is a $[\forall \ns]$-mode because
	\begin{equation*}
		\liminf_{n \to \infty} \Ratio{u}{\sup}{r_{n}}{\mu} \geq \liminf_{n \to \infty} \bigl(1 - \tfrac{1}{n} \bigr) \Ratio{u_{k_{n}}}{\sup}{s_{n}}{\mu} \geq 1.
	\end{equation*}

	\proofpartorcase{Case \ref{item:alternative_characterisation_weak_for_closed_balls}.}
	The implications (i)$\Rightarrow$(ii)$\Rightarrow$(iii)$\Rightarrow$(vi) and (i)$\Rightarrow$(iv)$\Rightarrow$(v)$\Rightarrow$(vi) hold by definition;
	it remains to show (vi)$\Rightarrow$(i).
	We use a similar argument to the previous case to `delete' the quantifier $\exists \as$, and we use the lower semicontinuity \eqref{eq:ball_mass_lsc} to `delete' the quantifier $\exists \cpas$.

	Suppose that $u$ is a $[\forall \cp \exists \as \exists \cpas \forall \ns]$-mode, and
	let $(u_{n})_{n \in \bN}$ and $(v_{n})_{n \in \bN}$ be any \as\ and \cpas\ proving this.
	Then, let $(r_{n})_{n \in \bN}$ be any \ns\ and let $v$ be any \cp.
	As before, using \eqref{eq:ball_mass_left_limits}, pick a \ns\ $(s_{n})_{n \in \bN}$ such that $s_{n} < r_{n}$ and $\mu(\ball{v}{s_{n}}) \geq (1 - \tfrac{1}{n}) \mu(\ball{v}{r_{n}})$;
	with this and \eqref{eq:ball_mass_lsc}, pick a strictly increasing sequence $(k_{n})_{n \in \bN}$ such that
	\[
		\ball{u_{k_{n}}}{s_{n}} \subseteq \ball{u}{r_{n}}
		\quad
		\text{and}
		\quad
		\mu(\ball{v_{k_{n}}}{s_{n}}) \geq \bigl(1 - \tfrac{1}{n}\bigr) \mu(\ball{v}{s_{n}})
		\text{~~for all $n \in \bN$.}
	\]
	Then $u$ is a $[\forall \ns \forall \cp]$-mode because
	\begin{align*}
		\liminf_{n \to \infty} \Ratio{u}{v}{r_{n}}{\mu} &\geq \liminf_{n \to \infty} \bigl(1 - \tfrac{1}{n} \bigr)  \Ratio{u_{k_{n}}}{v}{s_{n}}{\mu}
		\geq \liminf_{n \to \infty} \bigl(1 - \tfrac{1}{n}\bigr)^{2} \Ratio{u_{k_{n}}}{v_{k_{n}}}{s_{n}}{\mu} \geq 1.
	\end{align*}

	\proofpartorcase{Case \ref{item:alternative_characterisation_partial_strong}.}
	Clearly, any $[\forall \as \exists \ns]$-mode is a $[\exists \ns]$-mode;
	it remains to prove the converse.
	Suppose that $u$ is a $[\exists \ns]$-mode along the \ns\ $(s_{n})_{n \in \bN}$ and let $(u_{n})_{n \in \bN}$ be an arbitrary \as.
	Using the lower-semicontinuity property \eqref{eq:ball_mass_lsc}, take any strictly increasing sequence $(N_{k})_{k \in \bN}$ such that
	\[
		\mu\bigl(\ball{u_{n}}{s_{k}}\bigr)
		\geq
		\bigl(1 - \tfrac{1}{k}\bigr) \mu\bigl(\ball{u}{s_{k}}\bigr) \text{~~for all $k \in \bN$ and $n \geq N_{k}$.}
	\]
	For each $n \in \bN$, let $k_{n}$ be the unique index $k$ such that $N_{k} \leq n < N_{k + 1}$, discarding the finitely many $n < N_{1}$, and take $r_{n} = s_{k_{n}}$;
	it is easily verified that $(r_{n})_{n \in \bN}$ is a \ns, and by construction
	\[
		\liminf_{n \to \infty} \Ratio{u_{n}}{\sup}{r_{n}}{\mu} = \liminf_{n \to \infty} \Ratio{u_{n}}{\sup}{s_{k_{n}}}{\mu} \geq \liminf_{n \to \infty} \bigl(1 - \tfrac{1}{k_{n}}\bigr)\Ratio{u}{\sup}{s_{k_{n}}}{\mu}.
	\]
	Note that $(s_{k_{n}})_{n \in \bN}$ is not necessarily a subsequence of $(s_{n})_{n \in \bN}$;
	but since $(k_{n})_{n \in \bN}$ is merely the sequence $(k)_{k \in \bN}$ with finitely many repetitions of each term, and using the fact that $u$ is an $[\exists \ns]$-mode along $(s_{n})_{n \in \bN}$, we conclude that
	\[
		\liminf_{n \to \infty} \Ratio{u_{n}}{\sup}{r_{n}}{\mu} \geq \liminf_{n \to \infty} \bigl(1 - \tfrac{1}{k_{n}}\bigr)\Ratio{u}{\sup}{s_{k_{n}}}{\mu} = \liminf_{k \to \infty} \bigl(1 - \tfrac{1}{k} \bigr)\Ratio{u}{\sup}{s_{k}}{\mu} \geq 1.
	\]

	\proofpartorcase{Case \ref{item:alternative_characterisation_partial_weak}.}
	This follows from a similar argument to the previous case.

	\proofpartorcase{Case \ref{item:alternative_characterisation_wp_for_spherically_non_atomic}.}
 	The implications
	(iv)$\Rightarrow$(ii)$\Rightarrow$(i) and (iv)$\Rightarrow$(iii)$\Rightarrow$(i) hold by definition;
	it remains to prove (i)$\Rightarrow$(iv).
	We use a similar argument to \ref{item:alternative_characterisation_partial_strong} to `insert' the quantifier $\forall \as$, and we use \eqref{eq:ball_mass_left_limits} to `insert' the quantifier $\forall \cpas$.

	Suppose that $u$ is a $[\forall \cp \exists \ns]$-mode;
	then pick an arbitrary \as\ $(u_{n})_{n \in \bN}$, \cp\ $v$, and \cpas\ $(v_{n})_{n \in \bN}$.
	Let $(s_{n})_{n \in \bN}$ be any \ns, which may depend on $v$, that proves that $u$ is a $[\forall \cp \exists \ns]$-mode;
	then, using \eqref{eq:ball_mass_left_limits}, pick a \ns\ $(t_{n})_{n \in \bN}$ with $t_{n} < s_{n}$ and $\mu(\ball{u}{t_{n}}) > (1 - \tfrac{1}{n}) \mu(\ball{u}{s_{n}})$.
	As in \ref{item:alternative_characterisation_partial_strong}, and using that $v_{n} \to v$, construct a strictly increasing sequence $(N_{k})_{k \in \bN}$ with
	\[
		\mu\bigl( \ball{u_{n}}{t_{k}} \bigr) \geq \bigl(1 - \tfrac{1}{k} \bigr) \mu\bigl(\ball{u}{t_{k}}\bigr) \quad \text{and} \quad
		\ball{v_{n}}{t_{k}} \subseteq \ball{v}{s_{k}} \text{~~for all $n \geq N_{k}$.}
	\]
	Define $k_{n}$ to be the unique index $k$ such that $N_{k} \leq n < N_{k + 1}$, discarding the finitely many $n < N_{1}$, and take $r_{n} \defeq t_{k_{n}}$.
	By construction,
	\begin{align*}
		\liminf_{n \to \infty} \Ratio{u_{n}}{v_{n}}{r_{n}}{\mu} &= \liminf_{n \to \infty} \Ratio{u_{n}}{v_{n}}{t_{k_{n}}}{\mu}
		\\
		&\geq \liminf_{n \to \infty} \bigl(1 - \tfrac{1}{k_{n}}\bigr)\Ratio{u}{v_{n}}{t_{k_{n}}}{\mu} \\
		&\geq \liminf_{n \to \infty} \bigl(1 - \tfrac{1}{k_{n}}\bigr)^{2} \Ratio{u}{v}{s_{k_{n}}}{\mu} \\
		&= \liminf_{k \to \infty} \bigl(1 - \tfrac{1}{k}\bigr)^{2} \Ratio{u}{v}{s_{k}}{\mu}  \geq 1,
	\end{align*}
	where the penultimate equality follows by the argument used in \ref{item:alternative_characterisation_partial_strong}.
\end{proof}

We are now in a position to establish the positive result that the remaining 21 mode maps, in the form of their ten canonical representatives, do satisfy the axioms of \Cref{defn:axioms}.

\begin{proposition}[Verification of axioms for remaining mode maps]
	\label[proposition]{prop:mode_map_all_axioms}
	The axioms \AP, \CP, and \LP\ are satisfied by the ten equivalence classes of small-ball mode maps listed in \Cref{fig:periodic_table}(\subref{fig:table_of_21}) on the domain $\mathsf{U} = \catMetProb$.
\end{proposition}

\begin{proof}
	For the purpose of the proof, we place the ten equivalence classes, listed here using their canonical representatives, into four groups:
	\begin{enumerate}[label = (\alph*)]
		\item
		\label{item:axioms_strong_modes}
		\phantom{hide} \\[-2.8ex]
		\begin{tabularx}{\textwidth}{YYYY}
			$\sws = [\forall \ns ]$,
			&
			$\swps = [\exists \ns ]$,
			&
			$\swgs = [\forall \ns \exists \as ]$,
			&
			$\swpgs = [\exists \ns \exists \as ]$;
		\end{tabularx}
		\item
		\label{item:axioms_weak_modes}
		\begin{tabularx}{\textwidth}{YYYY}
			$\sww = [\forall \ns \forall \cp]$,
			&
			$\swpw = [\exists \ns \forall \cp]$,
			&
			$\swwp = [\forall \cp \exists \ns]$;
			&
		\end{tabularx}
		\vspace{-3ex}
		\item
		\label{item:axioms_exotic}
		\phantom{hide} \\[-2.8ex]
		\begin{tabularx}{\textwidth}{YYYY}
			$\swe = [\forall\cp\forall\cpas\exists\as\forall\ns]$;
			&&&
		\end{tabularx}
		\item
		\label{item:axioms_weak_generalised_approximating_and_partial_modes}
		\begin{tabularx}{\textwidth}{YYYY}
			$\swgwap = [\exists \as \forall \cp \forall \cpas \exists \ns]$,
			&
			$\swwgap = [\forall \cp \exists \as \forall \cpas \exists \ns]$.
			&&
		\end{tabularx}
	\end{enumerate}

	\proofpartorcase{Proof of \AP\ and \LP.}
	\AP\ and \LP\ are straightforward to verify for each definition in turn.
	For \LP, the argument uses the Lebesgue differentiation theorem:
	for any $\mu \in \prob{\bR^{m}}$ with continuous \ac{pdf} $\rho$, any $\ns$ $(r_{n})_{n\in\bN}$, and any $\as$ $(u_{n})_{n\in\bN}$ of $u\in X$,
	\[
		\lim_{n\to\infty}
		\frac{\mu(\Ball{u_{n}}{r_{n}})}{\Leb{m}(\Ball{0}{r_{n}})}
		=
		\rho(u),\qquad
		\lim_{n\to\infty}
		\frac{\sMass{\mu}{r_{n}}}{\Leb{m}(\Ball{0}{r_{n}})}
		=
		\sup_{x \in X} \rho(x) \in (0, \infty].
	\]

	\proofpartorcase{Proof of \CP.}
	Let $X$ be a vector space with translation-invariant metric $d$, let $\mu \in \prob{X}$ have non-empty support, and let $\nu \defeq \alpha \mu + (1-\alpha) \mu(\quark - b)$ and $\Delta \defeq \dist{\supp(\mu)}{b + \supp(\mu)} / 2 > 0$.
	Then $\supp(\nu) = \supp(\mu) \cup (b + \supp(\mu))$ and, for any $0 < r < \Delta$ and $x \in X$,
	\begin{equation}
	\label{equ:CP_verification_preliminary_considerations}
	\nu(\Ball{x}{r})
	=
	\begin{cases}
		\alpha \, \mu(\Ball{x}{r}),
		&
		\text{if } x\in \supp(\mu) + \Ball{0}{\Delta},
		\\
		(1-\alpha) \, \mu(\Ball{x-b}{r}),
		&
		\text{if } x\in \supp(\mu) + \Ball{0}{\Delta} + b,
		\\
		0,
		&
		\text{otherwise.}
	\end{cases}
	\end{equation}
	Moreover, for such $r$, $\sMass{\nu}{r} = \max(\alpha, (1-\alpha))\, \sMass{\mu}{r}$.
	Further, for any $u \in X$, any $\as$ $(u_{n})_{n\in\bN}$ of $u$, any $\ns$ $(r_{n})_{n\in\bN}$, and for sufficiently large $n\in\bN$,
	\[
	\Ball{u_{n}}{r_{n}}
	\subseteq
	\begin{cases}
		\supp(\mu) + \Ball{0}{\Delta},
		&
		\text{if } u\in \supp(\mu),
		\\
		\supp(\mu) + \Ball{0}{\Delta} + b,
		&
		\text{if } u\in \supp(\mu) + b,
		\\
		X\setminus (\supp(\nu) + \Ball{0}{\Delta})
		\text{ with } \nu(\Ball{u_{n}}{r_{n}}) = 0,
		&
		\text{otherwise.}
	\end{cases}
	\]
	
	These observations already imply \CP\ for all mode definitions in groups~\ref{item:axioms_strong_modes} and~\ref{item:axioms_weak_modes}.

	For~\ref{item:axioms_exotic} and~\ref{item:axioms_weak_generalised_approximating_and_partial_modes}, these observations imply $\fM(\nu) \subseteq \fM(\mu)\cup (\fM(\mu)+b)$.
	To verify the other claims of \CP\ we consider three cases:

	\proofpartorcase{Case~1: $\alpha = \frac{1}{2}$.}
	Let $u$ be a mode of $\mu$ and consider the measure $\nu$.
	Then $u$ clearly dominates every \cp\ $v\neq u$ with $v \in \supp(\mu)$ in the corresponding sense.
	By shifting the corresponding \as\ and \cpas\ by $b$ and $-b$ in a straightforward way, this implies that it also dominates every point $v \in \supp(\mu) + b$; but the point $v = u+b$ requires special treatment.
	If $u$ is an $\swe$-mode of $\mu$ and $v_{n}$ is an arbitrary \cpas\ of $v = u+b$, then simply choose the \as\ $u_{n} = v_{n} -b$ and \eqref{equ:CP_verification_preliminary_considerations} implies that $u$ dominates $v = u+b$, hence, $u$ is an $\swe$-mode of $\nu$.
	If $u$ is a $\swwgap$-mode (or $\swgwap$-mode, respectively) of $\mu$, then, for the constant \as\ $u_{n} = u$ and every \cpas\ $v_{n}$ of $v = u+b$, denote $\Delta_{n} = d(v_{n},v)$ and let $(\tilde{r}_{n})_{n\in\bN}$ and $(\varepsilon_{n})_{n\in\bN}$ be the corresponding null sequences from \Cref{lemma:technical_lemma_for_SP}.
	Since $\Ball{v}{\tilde{r}_{n}+\Delta_{n}} \supseteq \Ball{v_{n}}{\tilde{r}_{n}}$, we obtain
	\[
		\liminf_{n \to \infty}
		\Ratio{u_{n}}{v_{n}}{\tilde{r}_{n}}{\nu}
		\geq
		\liminf_{n \to \infty}
		\frac{\nu( \Ball{u}{\tilde{r}_{n}} )}{\nu( \Ball{v}{\tilde{r}_{n}+\Delta_{n}} )}
		\geq
		\liminf_{n \to \infty}
		\frac{\mu( \Ball{u}{\tilde{r}_{n}} )}{\mu( \Ball{u}{\tilde{r}_{n}+\Delta_{n}} )}
		\geq
		\liminf_{n \to \infty}
		(1+\varepsilon_{n})^{-1}
		=
		1,
	\]
	proving that $u$ dominates $v = u+b$, hence, $u$ is an $\swwgap$-mode (or $\swgwap$-mode, respectively) of $\nu$.
	In summary, $\fM(\mu) \subseteq \fM(\nu)$, which, by symmetry, also implies $\fM(\mu) + b \subseteq \fM(\nu)$, proving \CP\ in this case.

	\proofpartorcase{Case~2: $\alpha > \frac{1}{2}$.}
	The proof of $\fM(\mu) \subseteq \fM(\nu)$ is identical to Case 1.
	For the reverse inclusion we assume that there is a mode $u \in \supp(\mu)+b$ of $\nu$ and derive a contradiction, thereby proving \CP.
	For mode types in group~\ref{item:axioms_exotic} consider the $\cp$ $v = u-b$ and the constant $\cpas$ $v_{n}=v$.
	By assumption there exists an $\as$ $(u_{n})_{n\in\bN}$ of $u$ such that, for any $\ns$ $(r_{n})_{n\in\bN}$, $\liminf_{n\to\infty} \Ratio{u_{n}}{v}{r_{n}}{\nu} \geq 1$.
	On the other hand, denoting $\Delta_{n} = d(u_{n},u)$ and using \Cref{lemma:technical_lemma_for_SP}, there exist null sequences $(\tilde{r}_{n})_{n\in\bN}$ and $(\varepsilon_{n})_{n\in\bN}$ such that
	$\mu( \Ball{u}{\tilde{r}_{n}+\Delta_{n}} )
	\leq
	(1+\varepsilon_{n}) \, \mu( \Ball{u}{\tilde{r}_{n}} )$.
	Therefore, since $\Ball{u}{\tilde{r}_{n}+\Delta_{n}} \supseteq \Ball{u_{n}}{\tilde{r}_{n}}$,
	\[
		\liminf_{n \to \infty}
		\Ratio{u_{n}}{u}{\tilde{r}_{n}}{\mu}
		\leq
		\liminf_{n \to \infty}
		\frac{\mu( \Ball{u}{\tilde{r}_{n}+\Delta_{n}} )}{\mu( \Ball{u}{\tilde{r}_{n}} )}
		\leq
		\liminf_{n \to \infty}
		1+\varepsilon_{n}
		=
		1.
	\]
	Since $\Ratio{u_{n}}{v}{\tilde{r}_{n}}{\nu} = \frac{\alpha}{1-\alpha}\, \Ratio{u_{n}}{u}{\tilde{r}_{n}}{\mu}$ for sufficiently large $n\in\bN$ and $\alpha < \frac{1}{2}$, this yields a contradiction.
	For mode types in group~\ref{item:axioms_weak_generalised_approximating_and_partial_modes} consider the $\cp\ v=u-b$ and choose, for any $\as$ $(u_{n})_{n\in\bN}$ of $u$, the $\cpas$ $v_{n} = u_{n} - b$ of $v$.
	Then, for any $\ns$ $(r_{n})_{n\in\bN}$,
	$\Ratio{u_{n}}{v_{n}}{r_{n}}{\nu}
	\leq
	\frac{\alpha}{1-\alpha} < 1$ for sufficiently large $n\in\bN$, contradicting the assumption.

	\proofpartorcase{Case~3: $\alpha < \frac{1}{2}$.}
	Follows from Case 2 by symmetry.
	\qedhere
\end{proof}

\begin{remark}[The role of \LP\ in \Cref{thm:main}]
	\Cref{prop:mode_maps_violating_AP,prop:mode_maps_violating_CP} eliminate meaningless mode maps using only \AP\ and \CP;
	the remaining maps all satisfy \LP\ (\Cref{prop:mode_map_all_axioms}).
	This elimination process requires just three examples:
	one measure on $\bR$ for \AP, along with a measure with Lebesgue \ac{pdf} on $\bR$ and a measure on $\ell^{2}(\bN; \bR)$ (\Cref{ex:no_optimal_approximating_sequence}) for \CP.
	We could have instead eliminated mode maps violating \LP\ first, but all such maps turn out to violate either \AP\ or \CP.
	Nevertheless, \LP\ is a desirable property in its own right that requires verification.
\end{remark}

\subsection{Further properties of meaningful small-ball mode maps}
\label{section:MP_discussion}

We now show that---on the subclass $\mathsf{U} = \catSepMetProb$ of separable metric probability spaces, but not in general---the ten meaningful small-ball mode maps identified in \Cref{thm:main} enjoy an additional desirable property, namely that the modes of $\mu$ always lie in the support of $\mu$. 
However, we also establish a negative result that shows that none of these ten mode maps satisfy a \emph{merging property} under which the modes of $\mu$ on $X$ can be easily related to the modes of $\mu$ on well-separated positive-mass subsets of $X$.

\begin{definition}
	We say that a mode map $\fM \colon \mathsf{U} \subseteq \catMetProb \to \catSet$ satisfies
	\begin{itemize}[leftmargin=3.5em]
		\item[(\text{SP})]
		\label{item:def_mode_axioms_SP}
		the \defterm{support property} if $\fM(X, d, \mu) \subseteq \supp(\mu)$ for every $(X, d, \mu) \in \mathsf{U}$.
	\end{itemize}
\end{definition}

\begin{proposition}[Support property]
	\label[proposition]{proposition:meaningful_support_property}
	The restriction of each of the ten meaningful small-ball mode maps of \Cref{thm:main} to the domain $\mathsf{U} = \catSepMetProb \subset \catMetProb$ satisfies \SP.
\end{proposition}

\begin{proof}
	It suffices to prove the claim for the weakest meaningful mode type, namely \swwgap-modes, i.e., $[\forall \cp \exists \as \forall \cpas \exists \ns]$-modes.
	To this end, let $\mu \in \prob{X}$, let $u \in X \setminus \supp(\mu)$, and let $(u_{n})_{n \in \bN}$ be an arbitrary \as\ of $u$.
	In any separable metric probability space $(X, d, \mu)$, it is always the case that $\supp(\mu) \neq \emptyset$ \citep[Theorem~12.14]{AliprantisBorder2006}, so choose any \cp\ $v \in \supp(\mu)$, fix the constant \cpas\ $v_{n} = v$, and let the \ns\ $(r_{n})_{n \in \bN}$ be arbitrary.
	Since $X \setminus \supp(\mu)$ is open, for all sufficiently large $n$, $\Ball{u_{n}}{r_{n}} \cap \supp(\mu) = \varnothing$ and hence $\mu(\Ball{u_{n}}{r_{n}}) = 0$.
	Therefore, $u$ is not a \swwgap-mode, since
	\[
		\liminf_{n \to \infty} \Ratio{u_{n}}{v_{n}}{r_{n}}{\mu} = \liminf_{n \to \infty} \frac{0}{\mu(\Ball{v}{r_{n}})} = 0 < 1 .
		\qedhere
	\]
\end{proof}

\begin{remark}[\SP\ for non-separable spaces]
	It is necessary to restrict to separable metric probability spaces in the previous proposition;
	on non-separable spaces, the support property \SP\ may fail, as there exist metric probability spaces $(X, d, \mu)$ for which 
	\begin{equation*}
		\supp(\mu) = \bigset{x \in X}{\mu(\ball{x}{r}) > 0 \text{~for all $r > 0$}} = \emptyset.
	\end{equation*}
	For example, consider the infinite product space $X \defeq (0, 1)^{\bN} \subset \ell^{\infty}$ endowed with the supremum metric and the probability measure $\mu$ defined on $\ballsig{X}$ as the countably infinite product of one-dimensional Lebesgue measure.
	In this case, $\mu(\ball{x}{r}) = 0$ for all $x \in X$ and $0 \leq r \leq \tfrac{1}{2}$ and so $\supp(\mu) = \varnothing$ and $\sMass{\mu}{r} = 0$;
	however, due to our convention that $\nicefrac{0}{0} \defeq 1$, every point of $X$ is a \sws-mode of $\mu$.
	We emphasise here that we do not define $\mu$ on the Borel $\sigma$-algebra, and indeed characterising the support of Borel probability measures on non-separable metric spaces is a sensitive issue \citep[see][Section~7.2]{Bogachev2007_II}.
\end{remark}

The axiom \CP\ enforces a relationship between the modes of $\nu \in \prob{X}$ and those of its restrictions $\nu|_{A}$, $\nu|_{B}$ to positive-mass subsets $A, B \in \ballsig{X}$ with $\dist{A}{B} > 0$, provided that, for some $\mu \in \prob{X}$ with $\supp(\mu) \neq \emptyset$,
\begin{equation*}
	\nu = \alpha \mu + (1-\alpha) \mu(\quark - b),\qquad\alpha \in [0, 1],\qquad A = \supp(\mu),\qquad B = b + \supp(\mu).
\end{equation*}
In particular, \CP\ requires that, if $\fM(\nu|_{A}) \neq \emptyset$, then the set $\fM(\nu)$ of modes of the convex combination must also be non-empty.
We now consider an extension of this property in which $\nu|_{A}$ and $\nu|_{B}$ need not be shifted copies of $\mu$;
we again seek to encode the property that the merged measure $\nu$ has at least one mode if $\nu|_{A}$ and $\nu|_{B}$ have modes.
Unlike in \CP\ we shall not prescribe the exact value of $\fM(\nu)$;
we ask only that it be non-empty if $\fM(\nu|_{A})$ and $\fM(\nu|_{B})$ are.

\begin{definition}[Merging property]
	\label[definition]{def:mode_axioms_2}
	Suppose that $\mathsf{U} \subseteq \catMetProb$ is closed under restrictions to positive-mass subsets, i.e.,
	\begin{equation*}
		(X, d, \mu) \in \mathsf{U},~A \in \ballsig{X},~\mu(A) > 0 \implies (X, d, \mu|_{A}) \in \mathsf{U}.
	\end{equation*}
	Then we say that a mode map $\fM \colon \mathsf{U}  \to \catSet$ satisfies
	\begin{itemize}[leftmargin=3.5em]
		\item[(\text{MP})]
		\label{item:def_mode_axioms_rMP}
		the \defterm{merging property} if, for any $(X, d, \mu) \in \mathsf{U}$ and any subsets $A,B \in \ballsig{X}$ with $\mu(A)>0$, $\mu(B)>0$, and $\dist{A}{B} > 0$,
		\[
			\bigl( \fM(\mu|_{A}) \neq \emptyset \text{ and } \fM(\mu|_{B}) \neq \emptyset \bigr) \implies \fM(\mu|_{A\cup B}) \neq \emptyset .
		\]
	\end{itemize}
\end{definition}

It turns out that \rMP\ fails for all meaningful mode types identified in \Cref{thm:main}.
Before proving this, let us discuss one counterexample in more detail.
\Cref{example:ps_not_s_not_gw} was first used in our elementary classification result, \Cref{example:subclass_classification}, and provides an absolutely continuous measure $\mu \in \prob{\bR}$ with a $\swps$-mode that is not a $\swgw$-mode, and thus not an $\sws$-mode.
The density of this measure involves three singularities centred at $-2$, $0$, and $+2$, with the property that
\[
	\liminf_{r \to 0} \Ratio{-2}{+2}{r}{\mu} < 1 < \limsup_{r \to 0} \Ratio{-2}{+2}{r}{\mu}.
\]
Restricting to the intervals $A = (-3, -1)$ and $B = (1, 3)$, it can be shown that $\fM^{\sws}(\mu|_{A}) = \{-2\}$ and $\fM^{\sws}(\mu|_{B}) = \{+2\}$;
but $\fM^{\sws}(\mu|_{A \cup B}) = \emptyset$.
Thus, owing to the oscillation of the ratio $\Ratio{-2}{+2}{r}{\mu}$, this measure $\mu$ has the unsatisfactory property that the two modes of the restricted measures vanish when we take the union of the intervals $A$ and $B$.
We use oscillatory properties of this kind to show that \emph{none} of the ten meaningful mode types satisfies \rMP\ for $X = \bR$:

\begin{proposition}[Failure of \rMP\ for small-ball mode maps]
	\label[proposition]{prop:all_modes_violate_MP}
	None of the ten meaningful small-ball mode maps of \Cref{thm:main} satisfies \rMP on $\mathsf{U} = \catMetProb$.
\end{proposition}

\begin{proof}
	\Cref{example:ps_not_s_not_gw} is a counterexample for all mode types of \Cref{thm:main} that are not partial ($\sws$-, $\swe$-, $\sww$-, and $\swgs$-modes):
	taking $A = (-3, -1)$ and $B = (1, 3)$, the points $u_{A} \defeq -2$ and $u_{B} \defeq 2$ are $\sws$-modes of $\mu|_{A}$ and $\mu|_{B}$, respectively;
	but neither point is even a $\swgs$- or $\sww$-mode of $\mu|_{A \cup B}$.

	\Cref{example:SuspensionBridgeExtended} is a counterexample for all modes in \Cref{thm:main} that are partial ($\swps$-, $\swpgs$-, $\swpw$-, $\swwp$-, $\swgwap$-, and $\swwgap$-modes).
	Again, using the implications in \Cref{fig:periodic_table}(\subref{fig:Hasse_diagram_main}), it is sufficient to show that, for $A = (-\infty,-\frac{1}{2})$ and $B = (\frac{1}{2},\infty)$, $u_{A}=-1$ and $u_{B} = 1$ are $\swps$-modes of $\mu|_{A}$ and $\mu|_{B}$, respectively, and $u_{A}$ and $u_{B}$ are not \swwgap-modes of $\mu|_{A \cup B}$.
	Both statements follow directly from the derivations within \Cref{example:SuspensionBridgeExtended}.
\end{proof}

\begin{remark}[Connection to order theory]
	The \acp{w-mode} of $\mu \in \prob{X}$ are the greatest elements of a particular preorder on $X$ induced by $\mu$, the \defterm{small-radius limiting preorder} \citep[Definition~5.1]{LambleySullivan2022}.
	The failure of \rMP\ for \sww-modes is intimately connected with the non-totality of this preorder:
	the preorders for $\mu|_{A}$ and $\mu|_{B}$ may each be total and admit a greatest element (\sww-mode), while these properties fail for $\mu|_{A \cup B}$.
 	Moreover, the points identified by the mode map $\fM^{r}$ of \Cref{ex:big-ball_modes} correspond to the greatest elements of a total preorder, the \defterm{positive-radius preorder} \citep[Definition~4.1]{LambleySullivan2022}.
\end{remark}

\section{Conditions for mode maps to agree}
\label{sec:Mode_definitions_coincide}

While the ten mode definitions in \Cref{thm:main} are provably distinct in general, they often coincide in specific settings.
For example, when $X$ is a finite set, every \swwgap-mode is also an \sws-mode, and so all ten meaningful mode types coincide.
In this section, we consider various subclasses of $\catMetProb$ on which the ten meaningful small-ball mode maps of \cref{thm:main} coincide, and the relations among the remaining definitions.
As in \Cref{sec:small-ball_modes}, we say that two mode definitions are equivalent if they are equal as functions, and consider equivalence classes of definitions under this relation.

We emphasise that, in principle, some mode maps rejected as meaningless in \cref{sec:small-ball_modes} may well satisfy our axioms on a subclass of $\catMetProb$.
Since our goal is to find mode maps that maximally generalise the elementary cases \eqref{eq:intro_mode_map_LP} and \eqref{eq:intro_mode_map_AP}, we will consider only mode maps that are meaningful on all of $\catMetProb$, and then study equivalences between such mode types.

\Cref{sec:CASIO} studies a notion of optimality that makes certain \swg- and \swa-modes coincide with their non-\swg\ and non-\swa\ counterparts.
\Cref{sec:OM} considers measures admitting an Onsager--Machlup functional, which has the effect of making certain \swp-modes coincide with their non-\swp\ counterparts.
\Cref{sec:dichotomies} investigates the situation in which the existence of a stronger mode causes certain weaker modes to become stronger than they were a priori known to be.
Finally, \Cref{sec:Gaussian} applies the results of the previous three subsections to measures on Banach spaces that are dominated by a Gaussian measure, which is a class of considerable relevance in applications to inverse problems and diffusion processes.

\subsection{Constant optimal approximating sequences remove \texorpdfstring{\swa}{a}- and \texorpdfstring{\swg}{g}-modes}
\label{sec:CASIO}

\begin{definition}
	\label{def:CASIO}
	Let $(X, d, \mu)$ be a metric probability space.
	We say that the \defterm{constant approximating sequence is optimal} at $u \in X$ if the constant \as\ $(u)_{n \in \bN}$ dominates every other \as\ of $u$ along every possible \ns\ of radii, i.e.,
	\begin{equation}
		\label{eq:CASIO}
		\forall \ns\ (r_{n})_{n \in \bN} \, \forall \as \ (u_{n})_{n \in \bN} \to u\colon
		\quad
		\liminf_{n \to \infty} \Ratio{u}{u_{n}}{r_{n}}{\mu}
		\geq
		1.
	\end{equation}
\end{definition}

From the perspective of small-ball mode maps, measures $\mu$ satisfying \eqref{eq:CASIO} at every $u \in X$ are very well behaved.
As \Cref{lemma:CASIO_for_ac_in_Rm} shows, any $\mu \in \prob{\bR^{m}}$ with continuous, positive \ac{pdf} $\rho$ belongs to this class;
conversely, \eqref{eq:CASIO} can easily fail on the boundary of $\supp(\mu)$ even if $\rho$ is continuous.
Informally, for $\mu$ satisfying \eqref{eq:CASIO}, there is nothing to be gained by approximating a candidate mode, and so the terms $\swg = \exists \as$ and $\swa = \forall \cpas$ become redundant, vastly simplifying the lattice of mode types.
In fact, as we shall see, to remove \swg\ and \swa\ in this way it is enough that \eqref{eq:CASIO} should hold at every \swwgap-mode of $\mu$;
while it is no longer obvious that $\swa$ is redundant, this turns out to be the case for the ten meaningful mode maps.
The only additional advantage conferred by having \eqref{eq:CASIO} hold everywhere is that $\forall \cpas$ can then be removed entirely and \swe- and \sww-modes then coincide.

\begin{theorem}[Optimality of the constant \as\ and removal of \texorpdfstring{\swa}{a}- and \texorpdfstring{\swg}{g}-modes]
	\label{thm:modes_coincide_CASIO}
	On the subclass 
	\begin{equation*}
		\Bigset{(X, d, \mu) \in \catMetProb}{\text{\eqref{eq:CASIO} holds at every $\swwgap$-mode of $\mu$}},
	\end{equation*}
	the meaningful small-ball mode maps of \cref{thm:main} fall into six equivalence classes, partially ordered by implication, illustrated in  \Cref{fig:Hasse_CASIO}, namely
	\begin{equation*}
		\{ \sws, \swgs \}, \quad \{ \swps, \swpgs \},\quad \{ \swe\}, \quad \{\sww\}, \quad \{\swpw\}, \text{~~~~and~~~~} \{\swwp, \swgwap, \swwgap\}.
	\end{equation*}
	Moreover, on the subclass
	\begin{equation*}
		\Bigset{(X, d, \mu) \in \catMetProb}{\text{\eqref{eq:CASIO} holds at every $x \in X$}},
	\end{equation*}
	the mode types $\sww$ and $\swe$ coincide, yielding a lattice of five equivalence classes of mode types partially ordered by implication.
\end{theorem}

\begin{figure}[t]
	\centering
	\begin{tikzpicture}[scale=0.9]
		\draw[draw=none, use as bounding box] (-3.0,0.5) rectangle (5.0,-5.5);
		\node (sws) at (0, 0) {$\sws = \swgs$};
		\node (swe) at (2, -1.25) {$\swe$};
		\node (sww) at (2, -2.5) {$\sww$};
		\node (swps) at (-2, -1.875) {$\swps = \swpgs$};
		\node (swpw) at (0, -3.75) {$\swpw$};
		\node (swwp) at (0, -5) {$\swwp = \swgwap = \swwgap$};
		\draw[rounded corners] (sws) -- (-2, -1.25) -- (swps) -- (-2, -2.5) -- (swpw);
		\draw (sws) -- (swe) -- (sww) -- (swpw) -- (swwp);

		\draw[red, rounded corners] (-2.0, -0.4) -- (2.5, -0.4);
		\node[red, anchor=west] at (2.5, -0.4) {\footnotesize Ex.~\redref{ex:Gaussian_swps_not_sws}};

		\draw[red, rounded corners] (2.5, -1.875) -- (0.3, -1.875) -- (-2, -3.3125);
		\node[red, anchor=west] at (2.5, -1.875) {\footnotesize Ex.~\redref{example:CP}};

		\draw[red, rounded corners] (-2.0, -3.125) -- (1.64, -0.85) -- (2.5, -0.85);
		\node[red, anchor=west] at (2.5, -0.85) {\footnotesize Ex.~\redref{ex:Gaussian_swe_not_swps}};

		\draw[red, rounded corners] (2.5, -3.125) -- (2.0, -3.125) -- (-2.0, -0.625);
		\node[red, anchor=west] at (2.5, -3.125) {\footnotesize Ex.~\redref{example:ps_not_s_not_gw}};

		\draw[red] (-2.0, -4.375) -- (2.5, -4.375);
		\node[red, anchor=west] at (2.5, -4.375) {\footnotesize Ex.~\redref{example:wp_but_not_pw_or_pgs}};
	\end{tikzpicture}
	\caption{The lattice of distinct, meaningful, small-ball mode types from \Cref{fig:periodic_table}(\subref{fig:Hasse_diagram_main}) simplifies on the subclass of measures satisfying the optimality condition \eqref{eq:CASIO} at each \swwgap-mode (\Cref{thm:modes_coincide_CASIO}).
	Also, $\swe = \sww$ on the subclass for which the constant \as\ is optimal everywhere.}
	\label{fig:Hasse_CASIO}
\end{figure}

\begin{proof}
	The essential idea is that the optimality condition \eqref{eq:CASIO} ensures that any definition containing `$\exists \as$' is equivalent to the corresponding definition with `$\exists \as$' deleted.
	To see that $\swwgap$-, $\swgwap$-, and $\swwp$-modes coincide, suppose that $u$ is a \swwgap-mode, i.e., a $[\forall \cp \exists \as \forall \cpas \exists \ns]$-mode.
		Let $v \in X$ be arbitrary.
		By the hypothesis that $u$ is a \swwgap-mode, there exists an \as\ $(u_{n})_{n \in \bN} \to u$ such that, for $v_{n} \equiv v$, there is an \ns\ $(r_{n})_{n \in \bN}$ for which
		\begin{equation}
			\label{eq:step_in_wgap_implies_wp}
			\liminf_{n \to \infty} \Ratio{u_{n}}{v_{n}}{r_{n}}{\mu} \geq 1 .
		\end{equation}
		Hence, $u$ is also a \swwp-mode, since
		\begin{align*}
			\liminf_{n \to \infty} \Ratio{u}{v}{r_{n}}{\mu}
			& = \liminf_{n \to \infty} \bigl( \Ratio{u}{u_{n}}{r_{n}}{\mu} \Ratio{u_{n}}{v_{n}}{r_{n}}{\mu} \bigr) && \text{(since $v_{n} \equiv v$)} \\
			& \geq \underbrace{ \liminf_{n \to \infty} \Ratio{u}{u_{n}}{r_{n}}{\mu} }_{\geq 1 \text{ by optimality}} \underbrace{ \liminf_{n \to \infty} \Ratio{u_{n}}{v_{n}}{r_{n}}{\mu} }_{\geq 1 \text{ by \eqref{eq:step_in_wgap_implies_wp}}} \geq 1 .
		\end{align*}

		To see that $\swpgs$- and $\swps$-modes coincide, suppose that $u$ is a \swpgs-mode, i.e., a $[\exists \ns \exists \as]$-mode.
		By hypothesis, there exists an \ns\ $(r_{n})_{n \in \bN}$ and an \as\ $(u_{n})_{n \in \bN} \to u$ such that
		\begin{equation}
			\label{eq:step_in_pgs_implies_ps}
			\liminf_{n \to \infty} \Ratio{u_{n}}{\sup}{r_{n}}{\mu} \geq 1 .
		\end{equation}
		Hence, $u$ is also a \swps-mode, since
		\begin{align*}
			\liminf_{n \to \infty} \Ratio{u}{\sup}{r_{n}}{\mu}
			& \geq \underbrace{ \liminf_{n \to \infty} \Ratio{u}{u_{n}}{r_{n}}{\mu} }_{\geq 1 \text{ by optimality}} \underbrace{ \liminf_{n \to \infty} \Ratio{u_{n}}{\sup}{r_{n}}{\mu} }_{\geq 1 \text{ by \eqref{eq:step_in_pgs_implies_ps}}} \geq 1 .
		\end{align*}

		The proof that $\swgs$- and $\sws$-modes coincide is similar to the previous cases and is omitted.

		To see that, under the additional hypothesis \eqref{eq:CASIO}, $\swe$- and $\sww$-modes coincide,
		let $u$ be a \sww-mode and consider any \cp\ $v \in X \setminus \{ u \}$ with \cpas\ $(v_{n})_{n \in \bN} \to v$.
		Using the fact that a \sww-mode is equivalently characterised as a $[\forall \cp \exists \as \forall \ns]$-mode (\Cref{prop:mode_map_equivalences}\ref{item:alternative_characterisation_weak_for_closed_balls}), let $(u_{n})_{n \in \bN}$ be an \as\ of $u$ such that, for any \ns\ $(r_{n})_{n \in \bN}$,
		\begin{equation}
			\label{eq:step_in_w_implies_e}
			\liminf_{n \to \infty} \Ratio{u_{n}}{v}{r_{n}}{\mu} \geq 1 .
		\end{equation}
		Then, appealing to optimality \eqref{eq:CASIO} at the arbitrary \cp\ $v$, we see that $u$ is an $\swe$-mode as
		\begin{align*}
			\liminf_{n \to \infty} \Ratio{u_{n}}{v_{n}}{r_{n}}{\mu}
			& \geq \underbrace{ \liminf_{n \to \infty} \Ratio{u_{n}}{v}{r_{n}}{\mu} }_{\geq 1 \text{ by \eqref{eq:step_in_w_implies_e}}} \underbrace{ \liminf_{n \to \infty} \Ratio{v}{v_{n}}{r_{n}}{\mu} }_{\geq 1 \text{ by optimality}} \geq 1 . \qedhere
		\end{align*}
\end{proof}

\subsection{Onsager--Machlup functionals remove certain \texorpdfstring{\swp}{p}-modes}
\label{sec:OM}

\begin{definition}
	Let $(X, d, \mu)$ be a metric probability space, and let $\emptyset \neq E \subseteq X$.
	A function $I \colon E \to \bR$ is called an \defterm{\acl{OM} (\acs{OM}) functional} for $\mu$ if, for all $u, v \in E$,
	\begin{equation}
		\label{eq:OM_functional}
		\lim_{r \to 0} \Ratio{u}{v}{r}{\mu} = \exp\bigl( I(v) - I(u) \bigr).
	\end{equation}
	We call an OM functional \defterm{\OMexhaustive} if \defterm{property~$M(\mu, E)$} holds:
	for all $u \in E$ and $v \notin E$,
	$\lim_{r \to 0} \Ratio{v}{u}{r}{\mu} = 0$ \citep{ayanbayev2021gamma}.
\end{definition}

An \ac{OM} functional can be viewed as a formal negative log-density for $\mu$ on the subspace $E$, and property $M(\mu, E)$, which is implicit in much of the literature, ensures that $E$ is large enough that points outside of $E$ can be ignored as they have negligible small-ball mass.
The advantages of the \ac{OM} perspective on modes are already well known in the literature:

\begin{proposition}[{\citealp[Proposition~4.1]{ayanbayev2021gamma}}]
	\label[proposition]{prop:OM_minimisers_w-modes}
	Let $(X, d, \mu)$ be a metric probability space.
	If $\mu$ admits an \OMexhaustive\ \ac{OM} functional $I \colon E \to \bR$, then the \acp{w-mode} of $\mu$ coincide with the global minimisers of $I$, where $I$ is understood to take the value ${+ \infty}$ on $X \setminus E$.
\end{proposition}

The next result shows that, for measures admitting an \OMexhaustive\ \ac{OM} functional, the lattice of mode types simplifies:
\swpw-modes coincide with ordinary \sww-modes.
However, the existence of an \OMexhaustive\ \ac{OM} functional does not entirely void the class of partial modes:
\Cref{subsec:Gaussian_s_ps_w} exhibits a measure dominated by a non-degenerate Gaussian, satisfying the hypotheses of both \Cref{thm:modes_coincide_CASIO,thm:modes_coincide_OMpM}, with a \ac{ps-mode} that is not a \ac{s-mode}.

\begin{theorem}[\OMExhaustive\ \ac{OM} functionals and certain \swp-modes]
	\label{thm:modes_coincide_OMpM}
	On the subclass 
	\begin{equation*}
		\Bigset{(X, d, \mu) \in \catMetProb}{ \text{$\mu$ admits an exhaustive \ac{OM} functional} },
	\end{equation*}
	the meaningful small-ball mode maps of \cref{thm:main} fall into eight equivalence classes, partially ordered by implication, illustrated in \Cref{fig:Hasse_in_special_cases}(\subref{subfig:Hasse_in_special_cases_OMpM}), namely
	\begin{equation*}
		\{\sws\}, \quad \{\swgs\}, \quad \{\swpgs\}, \quad \{\swe\}, \quad \{\swps\}, \quad \{\sww, \swpw, \swwp\}, \quad \{\swgwap\}, \text{~~~~and~~~~} \{\swwgap\}.
	\end{equation*}
\end{theorem}

\begin{figure}[t]
	\centering
	\begin{subfigure}[t]{0.48\textwidth}
		\centering
		\begin{tikzpicture}[scale=0.9]
			\draw[draw=none, use as bounding box] (-3.0,0.5) rectangle (5.0,-5.5);
			\node (sws) at (0, 0) {$\sws$};
			\node (swe) at (2, -1) {$\swe$};
			\node (sww) at (2, -3) {$\sww = \swpw = \swwp$};
			\node (swgs) at (-2, -1) {$\swgs$};
			\node (swps) at (0, -2) {$\swps$};
			\node (swpgs) at (-2, -3) {$\swpgs$};
			\node (swgwap) at (0, -4) {$\swgwap$};
			\node (swwgap) at (0, -5) {$\swwgap$};
			\draw (sws) -- (swps);
			\draw (swps) -- (swpgs);
			\draw (sws) -- (swgs);
			\draw (swe) -- (sww);
			\draw (swps) -- (sww);
			\draw (sww) -- (swgwap);
			\draw (sws) -- (swe);
			\draw (swgwap) -- (swwgap);
			\draw (swpgs) -- (swgwap);
			\draw (swgs) -- (swpgs);

			\draw[red, rounded corners] (-1, -0.2) -- (-1, -2.5) -- (1.4, -3.7) -- (2.5, -3.7);
			\node[red, anchor=west] at (2.5, -3.7) {\footnotesize Ex.~\redref{example:gs_not_s_not_wp}};

			\draw[red, rounded corners] (-2.5, -2) -- (-2, -2)  -- (2.0, 0.0) -- (2.5, 0.0);
			\node[red, anchor=west] at (2.5, 0) {\footnotesize Ex.~\redref{ex:Gaussian_swps_not_sws}};

			\draw[red, rounded corners] (2.5, -0.5) -- (1.3, -0.5) -- (1, -0.65) -- (1, -2.3) -- (-1.8, -3.7) -- (-2.5, -3.7);
			\node[red, anchor=west] at (2.5, -0.5) {\footnotesize Ex.~\redref{ex:Gaussian_swe_not_swps}};

			\draw[red, rounded corners] (2.5, -2.0) -- (2.0, -2.0) -- (-1.4, -3.7);
			\node[red, anchor=west] at (2.5, -2.0) {\footnotesize Ex.~\redref{example:CP}};

			\draw[red] (-1.4, -4.5) -- (2.5, -4.5);
			\node[red, anchor=west] at (2.5, -4.5) {\footnotesize Ex.~\redref{example:wgap_but_not_gwap}};
		\end{tikzpicture}
		\caption{\raggedright On the subclass of metric probability spaces where $\mu$ has an \OMexhaustive\ \ac{OM} functional $I \colon E \to \bR$ (\Cref{thm:modes_coincide_OMpM}).}
		\label{subfig:Hasse_in_special_cases_OMpM}
	\end{subfigure}
	\hspace{1em}
	\begin{subfigure}[t]{0.48\textwidth}
		\centering
		\begin{tikzpicture}[scale=0.9]
			\draw[draw=none, use as bounding box] (-3.0,1.0) rectangle (4.0,-5.0);
			\node (sws) at (0, 0) {$\sws = \swgs$};
			\node (swps) at (-2, -2.0) {$\swps = \swpgs$};
			\node (swe) at (2, -2.0) {$\swe$};
			\node (sww) at (0, -4.0) {$\sww = \cdots = \swwgap$};
			\draw (sws) -- (swps);
			\draw (sws) -- (swe);
			\draw (swps) -- (sww);
			\draw (swe) -- (sww);

			\draw[red, rounded corners] (-2.9, -2.75) -- (-0.75, -2.75) -- (0.75, -1.25) -- (2.5, -1.25);
			\node[red, anchor=west] at (2.5, -1.25) {\footnotesize Ex.~\redref{ex:Gaussian_swe_not_swps}};

			\draw[red, rounded corners] (-2.9, -0.75) -- (2.5, -0.75);
			\node[red, anchor=west] at (2.5, -0.75) {\footnotesize Ex.~\redref{ex:Gaussian_swps_not_sws}};

			\draw[red, rounded corners] (2.5, -3.25) -- (-2.9, -3.25);
			\node[red, anchor=west] at (2.5, -3.25) {\footnotesize Ex.~\redref{example:CP}};
		\end{tikzpicture}
		\caption{\raggedright On the subclass of metric probability spaces where (\subref{subfig:Hasse_in_special_cases_OMpM}) holds and the constant \as\ is optimal at each \swwgap-mode (\Cref{cor:modes_coincide_CASIO_and_OMpM}).
		Also, $\swe = \sww$ if the constant \as\ is optimal everywhere.}
		\label{subfig:Hasse_in_special_cases_CASIO_and_OMpM}
	\end{subfigure}
	\caption{The lattice of distinct, meaningful, small-ball mode types from \Cref{fig:periodic_table}(\subref{fig:Hasse_diagram_main}) simplifies for measures possessing an \OMexhaustive\ \ac{OM} functional, and possibly also satisfying the optimality condition \eqref{eq:CASIO}.}
	\label{fig:Hasse_in_special_cases}
\end{figure}

\begin{proof}
	It suffices to show that any $\swwp$-mode is a $\sww$-mode.
	So, suppose that $I \colon E \to \Reals$ is an \OMexhaustive\ \ac{OM} functional and let $u \in X$ be a $\swwp$-mode;
	property $M(\mu, E)$ implies that $u \in E$.
	Take an arbitrary \cp\ $v \in X$.
	If $v \notin E$, then property $M(\mu, E)$ ensures that, for any \ns\ $(s_{n})_{n \in \bN}$, $\liminf_{n \to \infty} \Ratio{u}{v}{s_{n}}{\mu} \geq 1$.
	Otherwise, if $v \in E$, then since $u$ is a $\swwp$-mode we may find an \ns\ $(r_{n})_{n \in \bN}$, possibly depending on $v$, such that
	\begin{equation*}
		\liminf_{n \to \infty} \Ratio{u}{v}{r_{n}}{\mu} \geq 1.
	\end{equation*}
	Since both $u$ and $v$ lie in $E$, we know that, for any \ns\ $(s_{n})_{n \in \bN}$,
	\begin{equation*}
		\lim_{n \to \infty} \Ratio{u}{v}{s_{n}}{\mu} = \exp\bigl(I(v) - I(u)\bigr),
	\end{equation*}
	and this limit is seen to be at least unity by taking $s_{n} = r_{n}$.
	Thus, $u$ is a $\sww$-mode.
\end{proof}

\begin{corollary}
	\label[corollary]{cor:modes_coincide_CASIO_and_OMpM}
	On the subclass 
	\begin{equation*}
		\Set{(X, d, \mu) \in \catMetProb}{
		\begin{aligned}
			&\text{$\mu$ admits an exhaustive OM functional and}\\
			&\text{\eqref{eq:CASIO} holds at every $\swwgap$-mode of $\mu$}
		\end{aligned}},
	\end{equation*}
	 the meaningful small-ball mode maps of \cref{thm:main} fall into four equivalence classes, partially ordered by implication, illustrated in \Cref{fig:Hasse_in_special_cases}(\subref{subfig:Hasse_in_special_cases_CASIO_and_OMpM}), namely
	\begin{equation*}
		\{\sws, \swgs\},\quad \{ \swps, \swpgs \}, \quad \{\swe\}, \text{~~~~and~~~~} \{\sww, \swpw, \swwp, \swgwap, \swwgap\}.
	\end{equation*}
	Moreover, on the subclass
	\begin{equation*}
		\Set{(X, d, \mu) \in \catMetProb}{
		\begin{aligned}
			&\text{$\mu$ admits an exhaustive OM functional and}\\
			&\text{\eqref{eq:CASIO} holds at every $x \in X$}
		\end{aligned}},
	\end{equation*}
	the mode types $\sww$ and $\swe$ coincide, yielding a lattice of three mode types totally ordered by implication.
\end{corollary}

\begin{proof}
	This follows straightforwardly from \Cref{thm:modes_coincide_CASIO,thm:modes_coincide_OMpM}.
\end{proof}

Recall that every \sws-mode is a \sww-mode, and that every \swps-mode is a \swpw-mode.
In the situation considered in this subsection, we obtain the following unexpected relationship between \emph{partial} strong modes and \emph{non-partial} weak modes, which would ordinarily be incomparable, as well as a condition under which \swps-modes and \sws-modes coincide:

\begin{corollary}[\swps-, \sww-, and \sws-modes]
	\label[corollary]{cor:ps_implies_w_or_even_s}
	Let $(X, d, \mu) \in \catMetProb$ and suppose that $\mu$ admits an exhaustive \ac{OM} functional.
	\begin{enumerate}[label=(\alph*)]
		\item
		\label{item:ps_implies_w}
		Every \swps-mode of $\mu$ is also a \sww-mode.\footnote{Indeed, this part holds whenever the \swpw- and \sww-modes of $\mu$ coincide, for which the existence of an \OMexhaustive\ \ac{OM} functional is sufficient (by \Cref{thm:modes_coincide_OMpM}) but not necessary.}

		\item
		\label{item:ps_implies_s}
		If, for each \ns\ $(r_{n})_{n \in \bN}$, $\mu$ has a \swps-mode along some subsequence of $(r_{n})_{n \in \bN}$, then its $\swps$- and $\sws$-modes coincide.
	\end{enumerate}
\end{corollary}

\begin{proof}
	Every \swps-mode is a \swpw-mode (by definition), and every \swpw-mode is a \sww-mode under the stated hypotheses (by \Cref{thm:modes_coincide_OMpM}), which establishes \ref{item:ps_implies_w}.

	For \ref{item:ps_implies_s}, let $I \colon E \to \Reals$ be an exhaustive OM functional.
	By assumption, for each \ns\ $\underline{r} = (r_{n})_{n \in \bN}$, there exists a subsequence $\underline{r}' = (r_{n}')_{n \in \Naturals}$ and a \swps-mode $u^{\underline{r}'} \in X$ such that $\liminf_{n \to \infty} \Ratio{u^{\underline{r}}}{\sup}{r_{n}'}{\mu} \geq 1$.
	Part \ref{item:ps_implies_w} implies that each such $u^{\underline{r}'}$ is a \sww-mode and, by \Cref{prop:OM_minimisers_w-modes}, \sww-modes all lie in $E$ and are global minimisers of $I$.
	So, fix one \swps-mode $u^{\underline{s}} \in E$ of this form.
	Let $\underline{r} = (r_{n})_{n \in \bN}$ be any \ns\ and let $u^{\underline{r}'} \in E$ be an associated \swps-mode along some subsequence $\underline{r}'$.
	Then, using the elementary fact that, if every subsequence of a real sequence has a further subsequence converging to the same limit, then the whole sequence converges to that limit, we conclude that $u^{\underline{s}}$ is in fact an \sws-mode, since
	\begin{align*}
		\liminf_{n \to \infty} \Ratio{u^{\underline{s}}}{\sup}{r_{n}'}{\mu}
		& = \lim_{n \to \infty} \Ratio{u^{\underline{s}}}{u^{\underline{r}'}}{r_{n}'}{\mu} \liminf_{n \to \infty} \Ratio{u^{\underline{r}'}}{\sup}{r_{n}'}{\mu} \\
		& = \frac{\exp(- \min_{E} I)}{\exp(- \min_{E} I)} \liminf_{n \to \infty} \Ratio{u^{\underline{r}'}}{\sup}{r_{n}'}{\mu} \geq 1 .
		\qedhere
	\end{align*}
\end{proof}

\subsection{Coincidences induced by the existence of a stronger mode}
\label{sec:dichotomies}

It is natural to ask whether a single $\mu$ could simultaneously possess modes of all ten meaningful types.
The answer to this question is no:
by the strong--weak dichotomy of \citet[Lemma~2.2]{Lambley2023}, if $\mu$ has an \sws-mode, then any \sww-modes must in fact be \sws-modes as well.
That is, the existence of a mode of a stronger kind grants additional strength to modes that are, a priori, only known to be of a weaker kind.
The next theorem states results of a similar flavour:

\begin{theorem}[Coincidence of mode types when a stronger mode exists]
	\label{thm:s--ps--w_dichotomy}
	Consider the following subclasses of $\catMetProb$:
	
	\begin{enumerate}[label=(\alph*)]
		\item
		\label{item:s--ps--w_dichotomy_s--w}
		On
		the subclass
		\begin{equation*}
			\Bigset{(X, d, \mu) \in \catMetProb}{\text{$\mu$ has a $\sws$-mode}},
		\end{equation*}
		the ten meaningful small-ball mode types of \cref{thm:main} fall into four distinct equivalence classes, illustrated in \Cref{fig:Hasse_dichotomy_2}(\subref{subfig:Hasse_dichotomy_exists_s}), namely
		\[
			\{ \sws, \swe, \sww \}, \quad
			\{ \swgs \}, \quad
			\{ \swps, \swpw, \swwp \}, \quad \text{and} \quad
			\{ \swpgs, \swgwap, \swwgap \} .
		\]

		\item
		On the subclass
		\begin{equation*}
			\Bigset{(X, d, \mu) \in \catMetProb}{\text{$\mu$ has a $\sww$-mode}},
		\end{equation*}
		the mode types \swwp\ and \swpw\ coincide.

		\item
		\label{item:s--ps--w_dichotomy_ps--w}
		Let $(X, d, \mu) \in \catMetProb$.
		If $\mu$ has a \swps-mode along the \ns\ $(r_{n})_{n \in \bN}$, then any \sww-mode is a \swps-mode along $(r_{n})_{n \in \bN}$.		

		\item
		\label{item:s--ps--w_dichotomy_ps--s}
		On the subclass
		\begin{equation*}
			\Bigset{(X, d, \mu) \in \catMetProb}{\text{for every \ns, $\mu$ has a \swps-mode along that \ns}},
		\end{equation*}	
		the mode types \sww, \swe, and \sws\ coincide; moreover, the mode types \swps\ and \swpw\ coincide.
	\end{enumerate}
\end{theorem}

\begin{proof}
	\begin{enumerate}[label=(\alph*)]
		\item
		The claim that \sws-, \swe-, and \sww-modes coincide is precisely the strong--weak dichotomy of \citet[Lemma~2.2]{Lambley2023}:
		let $u \in X$ be a \sww-mode and let $v \in X$ be an \sws-mode; then $u$ is an \sws-mode because
		\begin{equation*}
			\liminf_{r \to 0} \Ratio{u}{\sup}{r}{\mu} \geq \liminf_{r \to 0} \Ratio{u}{v}{r}{\mu} \liminf_{r \to 0} \Ratio{v}{\sup}{r}{\mu} \geq 1.
		\end{equation*}

		Secondly, let $u$ be a \swwp-mode and let $v$ be a \sws-mode.
		Then, since $u$ is a \swwp-mode, there exists some \ns\ $(r_{n})_{n \in \bN}$ such that $\liminf_{n \to \infty} \Ratio{u}{v}{r_{n}}{\mu} \geq 1$.
		Then $u$ is a \swps-mode, since
		\[
			\liminf_{n \to \infty} \Ratio{u}{\sup}{r_{n}}{\mu} \geq \liminf_{n \to \infty} \Ratio{u}{v}{r_{n}}{\mu} \liminf_{n \to \infty} \Ratio{v}{\sup}{r_{n}}{\mu} \geq 1 .
		\]

		Thirdly, let $u \in X$ be a \swwgap-mode and let $v \in X$ be an \sws-mode.
		Then, for this \cp\ $v$, there exists an \as\ $(u_{n})_{n \in \bN} \to u$ such that, for the constant \cpas\ $v_{n} = v$, there exists an \ns\ $(r_{n})_{n \in \bN}$ satisfying $\liminf_{n \to \infty} \Ratio{u_{n}}{v}{r_{n}}{\mu} \geq 1$.
		Hence, $u$ is a \swpgs-mode, since, along this \as\ and \ns,
		\[
			\liminf_{n \to \infty} \Ratio{u_{n}}{\sup}{r_{n}}{\mu}
			\geq
			\liminf_{n \to \infty} \Ratio{u_{n}}{v}{r_{n}}{\mu}
			\liminf_{n \to \infty} \Ratio{v}{\sup}{r_{n}}{\mu}
			\geq
			1.
		\]

		Finally, \Cref{example:ps_not_s_not_gw,example:gs_not_s_not_wp} show that the four classes of modes are, in general, distinct.

		\item
		Let $u$ be a \swwp-mode and let $w$ be a \sww-mode.
		Then there exists an \ns\ $(r_{n})_{n \in \bN}$, possibly depending upon $w$, such that $\liminf_{n \to \infty} \Ratio{u}{w}{r_{n}}{\mu} \geq 1$.
		But then $u$ is also a \swpw-mode since, for this \ns\ and any \cp\ $v \in X$,
		\[
			\liminf_{n \to \infty} \Ratio{u}{v}{r_{n}}{\mu} \geq \liminf_{n \to \infty} \Ratio{u}{w}{r_{n}}{\mu} \liminf_{n \to \infty} \Ratio{w}{v}{r_{n}}{\mu} \geq 1 .
		\]

		\item
		Let $u$ be a \sww-mode, and let $v$ be a \swps-mode along the \ns\ $(r_{n})_{n \in \bN}$.
		Then $u$ is also a \swps-mode along that \ns, since
		\begin{equation*}
			\liminf_{n \to \infty} \Ratio{u}{\sup}{r_{n}}{\mu} \geq \liminf_{n \to \infty} \Ratio{u}{v}{r_{n}}{\mu} \liminf_{n \to \infty} \Ratio{v}{\sup}{r_{n}}{\mu} \geq 1.
		\end{equation*}

		\item
		By the hypotheses and \ref{item:s--ps--w_dichotomy_ps--w}, any $\sww$-mode is a $\swps$-mode along every \ns\ $(r_{n})_{n \in \Naturals}$, and is thus both an $\sws$- and $\swe$-mode.
		To see that any \swpw-mode is also a \swps-mode, let $u$ be a \swpw-mode along the \ns\ $(r_{n})_{n \in \bN}$ and let $v$ be a $\swps$-mode along this \ns.
		Then $u$ is also a \swps-mode along the same \ns, since
		\begin{equation*}
			\liminf_{n \to \infty} \Ratio{u}{\sup}{r_{n}}{\mu} \geq \liminf_{n \to \infty} \Ratio{u}{v}{r_{n}}{\mu} \liminf_{n \to \infty} \Ratio{v}{\sup}{r_{n}}{\mu} \geq 1.
			\qedhere
		\end{equation*}
	\end{enumerate}
\end{proof}

\begin{corollary}
	\label[corollary]{cor:dichotomy_exists_s_CASIO_OMpM}
	\begin{enumerate}[label=(\alph*)]
		\item
		\label{item:dichotomy_exists_s_CASIO}
		On the subclass
		\begin{equation*}
			\Bigset{(X, d, \mu) \in \catMetProb}{\text{$\mu$ has a $\sws$-mode and \eqref{eq:CASIO} holds at every $\swwgap$-mode of $\mu$}},
		\end{equation*}
		
		the meaningful small-ball modes of \cref{thm:main} fall into two equivalence classes, namely $\{ \sws, \swe, \sww, \swgs \}$ and $\{ \swps, \swpgs, \swpw, \swwp, \swgwap, \swwgap \}$, as illustrated in \Cref{fig:Hasse_dichotomy_2}(\subref{subfig:Hasse_dichotomy_exists_s_CASIO}).

		\item
		\label{item:dichotomy_exists_s_OMpM}
		On the subclass 
		\begin{equation*}
			\Bigset{(X, d, \mu) \in \catMetProb}{\text{$\mu$ has a $\sws$-mode and admits an exhaustive \ac{OM} functional}},
		\end{equation*}
		
		the meaningful small-ball modes of \cref{thm:main} fall into three equivalence classes, namely $\{ \sws, \swe, \sww, \swpw, \swwp, \swps \}$, $\{ \swgs \}$, and $\{ \swpgs, \swgwap, \swwgap \}$, as illustrated in \Cref{fig:Hasse_dichotomy_2}(\subref{subfig:Hasse_dichotomy_exists_s_OMpM}).

		\item
		\label{item:dichotomy_exists_s_CASIO_OMpM}
		On the subclass 
		\begin{equation*}
			\Set{(X, d, \mu) \in \catMetProb}{\begin{aligned}
				&\text{$\mu$ has a $\sws$-mode and admits an exhaustive \ac{OM} functional and }\\
				&\text{\eqref{eq:CASIO} holds at every $\swwgap$-mode of $\mu$}
			\end{aligned}},
		\end{equation*}
		
		all ten meaningful small-ball modes of $\mu$ coincide, as illustrated in \Cref{fig:Hasse_dichotomy_2}(\subref{subfig:Hasse_dichotomy_exists_s_CASIO_OMpM}).
	\end{enumerate}
\end{corollary}

\begin{figure}[t]
	\centering
	\begin{subfigure}[t]{0.45\textwidth}
		\centering
		\begin{tikzpicture}[scale=0.9]
			\draw[draw=none, use as bounding box] (-3.0,0.5) rectangle (4.5,-4.5);
			\node (sws) at (0,0) {$\exists \, \sws = \swe = \sww$};

			\node[red,anchor=west] at (2.5, -0.6) {\footnotesize Ex.~\redref{example:ps_not_s_not_gw}};
			\draw[red, rounded corners] (2.5, -0.6) -- (1.5, -0.6) -- (-1.5, -3.4) -- (-2.5, -3.4);

			\node[red, anchor=west] at (2.5, -3.4) {\footnotesize Ex.~\redref{example:gs_not_s_not_wp}};
			\draw[red, rounded corners] (-2.5, -0.6) -- (-1.5, -0.6) -- (1.5, -3.4) -- (2.5, -3.4);

			\node (swps) at (2,-2) {$\swps = \swpw = \swwp$};
			\node (swgs) at (-2,-2) {$\swgs$};
			\node (swpgs) at (0,-4) {$\swpgs = \swgwap =\swwgap$};
			\draw (sws) -- (swps) -- (swpgs);
			\draw (sws) -- (swgs) -- (swpgs);
		\end{tikzpicture}
		\caption{On the subclass of metric probability spaces where $\mu$ has an \sws-mode.}
		\label{subfig:Hasse_dichotomy_exists_s}
	\end{subfigure}
	\hspace{1em}
	\begin{subfigure}[t]{0.45\textwidth}
		\centering
		\begin{tikzpicture}[scale=0.9]
			\draw[draw=none, use as bounding box] (-2.5, 1.5) rectangle (4.0, -3.5);
			\node (sws) at (0,0) {$\exists \, \sws = \swe = \sww = \swgs$};
			\node (swps) at (0,-2) {$\swps = \swpgs = \cdots = \swwgap$};
			\draw (sws) -- (swps);

			\node[red, anchor=west] at (2.5, -1) {\footnotesize Ex.~\redref{example:ps_not_s_not_gw}};
			\draw[red] (-2.5, -1) -- (2.5, -1);
		\end{tikzpicture}
		\caption{\raggedright On the subclass of metric probability spaces where $\mu$ has an \sws-mode and the constant \as\ is optimal at each \swwgap-mode.}
		\label{subfig:Hasse_dichotomy_exists_s_CASIO}
	\end{subfigure}
	\hspace{1em}
	\begin{subfigure}[t]{0.45\textwidth}
		\centering
		\begin{tikzpicture}[scale=0.9]
			\draw[draw=none, use as bounding box] (-3.0, 0.5) rectangle (4.5, -3.5);
			\node (sws) at (0,0) {$\exists \, \sws = \swe = \cdots = \swps$};
			\node (swgs) at (0,-1.5) {$\swgs$};
			\node (swpgs) at (0,-3) {$\swpgs = \swgwap = \swwgap$};
			\draw (sws) -- (swgs) -- (swpgs);

			\node[red, anchor=west] at (2.5, -0.75) {\footnotesize Ex.~\redref{example:gs_not_s_not_wp}};
			\draw[red] (-2.5, -0.75) -- (2.5, -0.75);

			\node[red, anchor=west] at (2.5, -2.25) {\footnotesize Ex.~\redref{example:pgs_not_gs_under_s_mode_and_OM}};
			\draw[red] (-2.5, -2.25) -- (2.5, -2.25);
		\end{tikzpicture}
		\caption{\raggedright On the subclass of metric probability spaces where $\mu$ has an \OMexhaustive\ \ac{OM} functional and an $\sws$-mode.}
		\label{subfig:Hasse_dichotomy_exists_s_OMpM}
	\end{subfigure}
	\hspace{1em}
	\begin{subfigure}[t]{0.45\textwidth}
		\centering
		\begin{tikzpicture}[scale=0.9]
			\draw[draw=none, use as bounding box] (-2.5, 2.0) rectangle (4.0, -2.0);
			\node (sws) at (0,0) {$\exists \, \sws = \swe = \cdots = \swwgap$};
		\end{tikzpicture}
		\caption{\raggedright On the subclass of metric probability spaces where $\mu$ has an \OMexhaustive\ \ac{OM} functional and an $\sws$-mode, and the constant \as\ is optimal at each \swwgap-mode.}
		\label{subfig:Hasse_dichotomy_exists_s_CASIO_OMpM}
	\end{subfigure}
		\caption{By \Cref{thm:s--ps--w_dichotomy,cor:dichotomy_exists_s_CASIO_OMpM}, the lattice of distinct, meaningful, small-ball mode types from \Cref{fig:periodic_table}(\subref{fig:Hasse_diagram_main}) simplifies if $\mu$ possesses an \sws-mode and an \OMexhaustive\ \ac{OM} functional and/or the optimality condition \eqref{eq:CASIO} holds.}
	\label{fig:Hasse_dichotomy_2}
\end{figure}

\begin{proof}
	This follows immediately from combining \Cref{thm:modes_coincide_CASIO,thm:modes_coincide_OMpM,thm:s--ps--w_dichotomy}.
\end{proof}

\subsection{Gaussian-dominated measures on Banach spaces}
\label{sec:Gaussian}

Many measures $\mu \in \prob{X}$ of practical interest in Bayesian inference and the study of diffusion processes are absolutely continuous with respect to a Gaussian measure $\mu_{0}$ on a separable Banach space $X$.
It is often advantageous to express the Radon--Nikodym derivative $\rho$ of $\mu$ with respect to $\mu_{0}$ in exponential form as
\begin{equation}
	\label{eq:density_and_potential}
	\mu(\rd u)
	=
	\rho(u) \, \mu_{0}(\rd u)
	=
	\frac{1}{Z} \exp(-\Phi(u)) \, \mu_{0}(\rd u),
	\qquad
	Z
	\defeq
	\int_{X} \exp(-\Phi(u)) \, \mu_{0}(\rd u),
\end{equation}
for some continuous $\Phi \colon X \to \bR$, called the \defterm{potential} of $\mu$ with respect to $\mu_{0}$, and normalising constant $0 < Z < \infty$.
This representation enforces positivity of $\rho$ and also simplifies statements of regularity conditions on $\mu$ through assumptions on $\Phi$.

If $X$ has finite dimension, then \eqref{eq:density_and_potential} implies that $\mu$ has a continuous, strictly positive Lebesgue density;
thus all mode maps must coincide as the set of modes is prescribed by \LP.
But, in the case that $X$ has infinite dimension---as frequently arises in applications---we show in this section that \Cref{cor:modes_coincide_CASIO_and_OMpM} applies to such measures, reducing the number of mode definitions to just three distinct ones.
In doing so, we give the first examples in the literature distinguishing among \sws-, \swps-, and \sww-modes for $\mu$ of the form \eqref{eq:density_and_potential} with merely continuous potential $\Phi$.
Furthermore, under mild conditions on $\Phi$, which we explore in \Cref{thm:Gaussian_prior_combined}\ref{item:Gaussian_prior_better_potential}, even these three definitions coincide, giving an unambiguous notion of small-ball mode for such measures when $X$ has infinite dimension.
Conversely, the distinguishing examples for merely continuous $\Phi$ show that the conditions on $\Phi$ used by \citet[Proposition~4.1]{Lambley2023} and \Cref{thm:Gaussian_prior_combined}\ref{item:Gaussian_prior_better_potential} are not universally satisfied.
The resulting lack of canonical mode map in the case of continuous $\Phi$ emphasises the need for further research into the `correct' definition for applications communities.

We first briefly recall some important properties of Gaussian measures and discuss how these properties transfer to the reweighted measure \eqref{eq:density_and_potential};
we refer the reader to standard references, e.g., \citet{Bogachev1998Gaussian}, for further details.
Let $X$ be a separable Banach space, which for simplicity we assume to be a vector space over $\bR$.
A measure $\mu_{0}$ on $X$ is a (centred, non-degenerate) Gaussian measure if the pushforward measure $\mu_{0} \circ f^{-1}$ on $\bR$ is a (mean-zero, positive-variance) normal distribution for each continuous linear functional $f \colon X \to \bR$.
Since $\mu_{0}$ does not charge spheres, one can equivalently consider closed or open balls under $\mu_{0}$, because $\mu_{0}(\ball{u}{r}) = \mu_{0}(\cball{u}{r})$ for all $u \in X$ and $r > 0$.
The behaviour of $\mu_{0}$ is fully determined by a compactly embedded Hilbert subspace $(E, \norm{\quark}_{E})$ called the \emph{Cameron--Martin space}.

The map $I_{0} \colon E \to \bR$ defined by $I_{0}(u) \defeq \frac{1}{2} \norm{u}_{E}^{2}$ is vital in understanding small-ball probabilities under $\mu_{0}$, as
$I_{0}$ is an \ac{OM} functional for $\mu_{0}$ \citep[see][Section~4.7]{Bogachev1998Gaussian}.
Moreover, it forms part of a quantitative estimate of the ball-mass ratio known as the \emph{explicit Anderson inequality} \citep[Proposition~3.1]{Lambley2023}:
\begin{equation}
	\label{eq:explicit_Anderson_Gaussian}
	\Ratio{u}{0}{r}{\mu_{0}}
	\leq
	\exp\bigg(-\min_{h \in E \cap \cball{u}{r}} I_{0}(h)\bigg)
	\text{~~for any $u \in X$ and $r > 0$}.
\end{equation}
Note that, since $E \cap \cball{u}{r}$ is $E$-weakly closed and $h \mapsto \norm{h}_{E}^{2}$ is $E$-weakly lower semicontinuous, the minimum in \eqref{eq:explicit_Anderson_Gaussian} is attained.
Property $M(\mu_{0}, E)$ is a straightforward consequence of \eqref{eq:explicit_Anderson_Gaussian}, so balls centred in the subspace $E$ are dominant in the small-radius limit, even though $\mu_{0}(E) = 0$.
In fact, even more is true:
by \citet[Corollary~3.3(a)]{Lambley2023}, whenever $(u_{n})_{n \in \bN}$ is an \as\ in $X$ of some $u$, and $(r_{n})_{n \in \bN}$ is any \ns,
\begin{equation}
	\label{eq:Lambley2023Corollary3.3a}
	\lim_{n \to \infty} u_{n} = u \in X \setminus E \implies \lim_{n \to \infty} \min_{h \in E \cap \cball{u_{n}}{r_{n}}} I_{0}(h) = \infty ,
\end{equation}
and so \emph{sequences} of balls whose centres have limit outside $E$ are negligible in the small-radius limit.
We appeal to \eqref{eq:explicit_Anderson_Gaussian} and \eqref{eq:Lambley2023Corollary3.3a} often in \Cref{thm:Gaussian_prior_combined} and the supporting \Cref{lemma:cts_pot_Gaussian_EAI_and_wgap}.

\begin{theorem}[Small-ball modes for Gaussian prior and well-behaved potential]
	\label{thm:Gaussian_prior_combined}
	Consider the following subclasses of $\catMetProb$,
	\begin{equation*}
		\mathsf{U}^{j} \defeq \Set{(X, d, \mu) \in \catMetProb}{
			\begin{aligned}
				&\text{$X$ is a separable Banach space, $\mu_{0} \in \prob{X}$ is}\\
				&\text{a centred non-degenerate Gaussian,}\\
				&\text{$\Phi \colon X \to \bR$ satisfies $(C_{j})$, and $\mu$ is given by \eqref{eq:density_and_potential}}  
		\end{aligned}},
	\end{equation*}
	with the following conditions on the potential $\Phi$:
	\begin{enumerate}[label = ($C_{\arabic*}$)]
		\item
		\label{item:Phi_merely_continuous}
		$\Phi \colon X \to \bR$ is continuous;
		\item
		\label{item:Phi_uniformly_continuous}
		$\Phi \colon X \to \bR$ is uniformly continuous;
		\item
		\label{item:Phi_lower_bound}
		for all $\eta > 0$, there exists $K(\eta) > 0$ such that, for all $u \in X$, $\Phi(u) \geq K(\eta) - \eta \norm{u}_{X}^{2}$.
	\end{enumerate}
	\vspace{1ex}
	Then the ten meaningful small-ball mode maps of \cref{thm:main} are classified as follows:
	\begin{enumerate}[label=(\alph*)]
		\item
		\label{item:Gaussian_prior_continuous_potential}
		On $\mathsf{U}^{1}$ the ten meaningful small-ball mode types fall into three equivalence classes, totally ordered by implication, illustrated in \Cref{fig:Hasse_Gaussian}(\subref{subfig:Hasse_Gaussian_cts_potential}), namely
		\[
		\{ \sws, \swgs \}, \quad
		\{ \swps, \swpgs \}, \quad \text{and} \quad
		\{ \swe, \sww, \swpw, \swwp, \swgwap, \swwgap \}.
		\]
		
		\item
		\label{item:Gaussian_prior_better_potential}
		On $\mathsf{U}^{2} \cup \mathsf{U}^{3}$, all ten meaningful small-ball mode maps coincide, as illustrated in \Cref{fig:Hasse_Gaussian}(\subref{subfig:Hasse_Gaussian_very_good_potential}).
		Further, each $(X,d,\mu) \in \mathsf{U}^{3}$ has an \sws-mode.
	\end{enumerate}	
\end{theorem}

\begin{figure}[t]
	\centering
	\begin{subfigure}[t]{0.48\textwidth}
		\centering
		\begin{tikzpicture}[scale=0.9]
			\draw[draw=none, use as bounding box] (-3.0,0.5) rectangle (3.0,-3.5);
			\node (sws) at (0, 0) {$\sws = \swgs$};
			\node (swps) at (0, -1.5) {$\swps = \swpgs$};
			\node (swe) at (0, -3) {$\swe = \sww = \cdots = \swwgap$};
			\draw (sws) -- (swps) -- (swe);

			\draw[red, rounded corners] (-2.0, -0.75) -- (2.0, -0.75);
			\node[red, anchor=east] at (-2.0, -0.75) {\footnotesize Ex.~\redref{ex:Gaussian_swps_not_sws}};

			\draw[red, rounded corners] (-2.0,-2.25) -- (2.0,-2.25);
			\node[red, anchor=east] at (-2.0,-2.25) {\footnotesize Ex.~\redref{ex:Gaussian_swe_not_swps}};
		\end{tikzpicture}
		\caption{\raggedright On the subclass $\mathsf{U}^{1}$ where $\mu$ has a continuous potential with respect to a non-degenerate Gaussian prior $\mu_{0}$ (\Cref{thm:Gaussian_prior_combined}\ref{item:Gaussian_prior_continuous_potential}).}
		\label{subfig:Hasse_Gaussian_cts_potential}
	\end{subfigure}
	\hspace{1em}
	\begin{subfigure}[t]{0.48\textwidth}
		\centering
		\begin{tikzpicture}[scale=0.9]
			\draw[draw=none, use as bounding box] (-3.0,2.0) rectangle (3.0,-2.0);
			\node (sws) at (0, 0) {$\sws = \swe = \cdots = \swwgap$};
		\end{tikzpicture}
		\caption{\raggedright If $\mu$ has a uniformly continuous or quadratically-lower-bounded potential with respect to $\mu_{0}$ (\Cref{thm:Gaussian_prior_combined}\ref{item:Gaussian_prior_better_potential}), or if $X$ has finite dimension.}
		\label{subfig:Hasse_Gaussian_very_good_potential}
	\end{subfigure}
	\caption{The lattice of distinct, meaningful, small-ball mode types from \Cref{fig:periodic_table}(\subref{fig:Hasse_diagram_main}) simplifies considerably for Gaussian-dominated measures on Banach spaces.}
	\label{fig:Hasse_Gaussian}
\end{figure}

\begin{proof}
	Let $(X, d, \mu) \in \mathsf{U}^{1}$, that is, the potential $\Phi$ is continuous by \ref{item:Phi_merely_continuous}.
	By \Cref{lemma:cts_pot_Gaussian_EAI_and_wgap}\ref{item:cts_pot_Gaussian_OM_and_M}, $\mu$ admits an \ac{OM} functional $I \defeq \Phi + I_{0} \colon E \to \bR$ and property $M(\mu, E)$ is satisfied.
	By \Cref{lemma:cts_pot_Gaussian_EAI_and_wgap}\ref{item:cts_pot_Gaussian_wgap} and \ref{item:cts_pot_Gaussian_CASIO}, the constant \as\ is optimal in the sense of \eqref{eq:CASIO} at every point of the Cameron--Martin space $E$, which contains all the \swwgap-modes of $\mu$.
	Therefore, the claimed coincidences of mode types follow from \Cref{cor:modes_coincide_CASIO_and_OMpM}---with the exception of $\sww = \swe$, which we now verify as a special case.

	Let $u \in X$ be a \sww-mode of $\mu$, which by \Cref{prop:mode_map_equivalences}\ref{item:alternative_characterisation_weak_for_closed_balls} is equivalently a $[\forall \cp \exists \as \forall \ns]$-mode.
	To show that $u$ is also an \swe-mode, let $v \in X$ and let $(v_{n})_{n \in \bN}$ be an arbitrary \cpas\ of $v$, and consider the cases $v \in E$ and $v \in X \setminus E$ separately.

	\proofpartorcase{The case $v \in E$.}
	Since $u$ is a $[\forall \cp \exists \as \forall \ns]$-mode, there exists an \as\ $(u_{n}^{v})_{n \in \bN}$ such that, for any \ns\ $(r_{n})_{n \in \bN}$, $\liminf_{n \to \infty} \Ratio{u_{n}^{v}}{v}{r_{n}}{\mu} \geq 1$.
	As already noted, \Cref{lemma:cts_pot_Gaussian_EAI_and_wgap}\ref{item:cts_pot_Gaussian_CASIO} shows that the constant \as\ is optimal in the sense of \eqref{eq:CASIO} at every point of $E$.
	Therefore,
	\[
		\liminf_{n \to \infty} \Ratio{u_{n}^{v}}{v_{n}}{r_{n}}{\mu} \geq \underbrace{ \liminf_{n \to \infty} \Ratio{u_{n}^{v}}{v}{r_{n}}{\mu} }_{\geq 1} \underbrace{ \liminf_{n \to \infty} \Ratio{v}{v_{n}}{r_{n}}{\mu} }_{\geq 1 \text{ by optimality}} \geq 1.
	\]

	\proofpartorcase{The case $v \in X \setminus E$.}
	In this case, combining \eqref{eq:Lambley2023Corollary3.3a} with the explicit Anderson inequality \eqref{eq:explicit_Anderson_Gaussian} implies that $\Ratio{0}{v_{n}}{r_{n}}{\mu_{0}} \to \infty$, and an elementary estimate using the continuity of $\Phi$ at $v$ and $0$ shows that $\Ratio{0}{v_{n}}{r_{n}}{\mu} \to \infty$ as well.
	Now, since $u$ is a $[\forall \cp \exists \as \forall \ns]$-mode, there exists an \as\ $(u_{n}^{0})_{n \in \bN}$ such that, for any \ns\ $(r_{n})_{n \in \bN}$, $\liminf_{n \to \infty} \Ratio{u_{n}^{0}}{0}{r_{n}}{\mu} \geq 1$.
	Thus,
	\[
		\liminf_{n \to \infty} \Ratio{u_{n}^{0}}{v_{n}}{r_{n}}{\mu} \geq \underbrace{ \liminf_{n \to \infty} \Ratio{u_{n}^{0}}{0}{r_{n}}{\mu} }_{\geq 1} \underbrace{ \lim_{n \to \infty} \Ratio{0}{v}{r_{n}}{\mu} }_{= \infty} = \infty > 1.
	\]
	Therefore, the \sww-mode $u$ is in fact an \swe-mode, as claimed.
	
	\vspace{2ex}

	Finally, \Cref{ex:Gaussian_swe_not_swps,ex:Gaussian_swps_not_sws} provide measures with continuous potentials with respect to a non-degenerate Gaussian prior where, in the former, the measure has an \swe-mode that is not a \swps-mode, and, in the latter, the measure has a \swps-mode that is not an \sws-mode.
		
	This finishes the proof of \ref{item:Gaussian_prior_continuous_potential}.
	Let us now turn to \ref{item:Gaussian_prior_better_potential} and consider the cases $(X, d, \mu) \in \mathsf{U}^{2}$ and $(X, d, \mu) \in \mathsf{U}^{3}$, that is, conditions \ref{item:Phi_uniformly_continuous} and \ref{item:Phi_lower_bound} on the potential $\Phi$, separately:
	
	\proofpartorcase{$(X, d, \mu) \in \mathsf{U}^{2}$.}
	In view of \ref{item:Gaussian_prior_continuous_potential}, it suffices to show that every \sww-mode of $\mu$ is also an \sws-mode.	
	Let $u$ be a \sww-mode of $\mu$, and let $(u_{r})_{r > 0}$ be any net such that $\Ratio{u_{r}}{\sup}{r}{\mu} \geq 1 - \varepsilon(r)$ for some increasing function $\varepsilon \colon (0, \infty) \to (0, \infty)$ with $\lim_{r \to 0} \varepsilon(r) = 0$.
	Then
	\begin{align*}
		\liminf_{r \to 0} \Ratio{u}{\sup}{r}{\mu}
		=
		\liminf_{r \to 0} \Ratio{u}{u_{r}}{r}{\mu}
		\geq \liminf_{r \to 0} \exp\left(-\sup_{\ball{u}{r} \cap E} \Phi + \inf_{\ball{u_{r}}{r} \cap E} \Phi \right) \Ratio{u}{u_{r}}{r}{\mu_{0}}.
	\end{align*}
	The explicit Anderson inequality \eqref{eq:explicit_Anderson_Gaussian} implies that
	\begin{align}
		\notag
		\liminf_{r \to 0} \Ratio{u}{\sup}{r}{\mu} & \geq \liminf_{r \to 0} \exp\left(-\sup_{\ball{u}{r} \cap E} \Phi + \inf_{\ball{u_{r}}{r} \cap E} \Phi + \min_{\cball{u_{r}}{r} \cap E} I_{0} \right) \Ratio{u}{0}{r}{\mu_{0}} \notag \\
		\label{eq:Gaussian_s_w_coincide}
		& = \exp\left(-I_{0}(u) -\sup_{\ball{u}{r} \cap E} \Phi + \inf_{\ball{u_{r}}{r} \cap E} \Phi + \min_{\cball{u_{r}}{r} \cap E} I_{0} \right) .
	\end{align}
	Since $\Phi$ is uniformly continuous by \ref{item:Phi_uniformly_continuous}, it possesses a modulus of continuity $\varpi\colon (0, \infty) \to (0,\infty)$.
	For each $r > 0$, let $h_{r} \in \argmin_{\cball{u_{r}}{r} \cap E} I_{0}$.
	Then
	\begin{align*}
		\exp\left(-\sup_{\ball{u}{r} \cap E} \Phi \right) & \geq \exp \bigl(-\Phi(u) - \varpi(r) \bigr), \text{~and}\\
		\exp\left(\inf_{\ball{u_{r}}{r} \cap E} \Phi + I_{0}(h_{r}) \right) & \geq \exp \bigl( \Phi(h_{r}) + I_{0}(h_{r}) - \varpi(r) \bigr) .
	\end{align*}
	Using these lower bounds in \eqref{eq:Gaussian_s_w_coincide}, and the fact that $u$ globally minimises the \ac{OM} functional $I(u) = \Phi(u) + \frac{1}{2} \norm{u}_{E}^{2}$ for $\mu$, we conclude that $u$ is an \sws-mode, since
	\begin{equation*}
		\liminf_{r \to 0} \Ratio{u}{\sup}{r}{\mu} \geq \liminf_{r \to 0} \exp \bigl( I(h_{r}) - I(u) - 2 \varpi(r) \bigr) \geq \lim_{r \to 0} \exp(-2 \varpi(r)) = 1 .
	\end{equation*}
	
	\proofpartorcase{$(X, d, \mu) \in \mathsf{U}^{3}$.}
	The lower bound \ref{item:Phi_lower_bound} implies that $\mu$ must have an \sws-mode \citep[Theorem~1.1]{Lambley2023}.
	Therefore, combining \ref{item:Gaussian_prior_continuous_potential} with \Cref{thm:s--ps--w_dichotomy}\ref{item:s--ps--w_dichotomy_s--w} and \Cref{cor:dichotomy_exists_s_CASIO_OMpM}\ref{item:dichotomy_exists_s_CASIO_OMpM}, all ten meaningful small-ball mode types must coincide for $\mu$.
\end{proof}

Thus, in the case of a continuous potential $\Phi$, only three distinct mode definitions remain: \sws, \swps, and \sww;
when the potential $\Phi$ is even better behaved, these three classes collapse into one.

The two conditions \ref{item:Phi_uniformly_continuous} and \ref{item:Phi_lower_bound} on the potential $\Phi$ in \Cref{thm:Gaussian_prior_combined} act as different ways of preventing $\mu$ from concentrating large amounts of mass relative to the prior $\mu_{0}$ in a small region in exactly the way that the measures in \Cref{subsec:Gaussian_s_ps_w} exploit:
uniform continuity \ref{item:Phi_uniformly_continuous} prevents $\Phi$ from varying too quickly, whereas the lower bound \ref{item:Phi_lower_bound} allows $\Phi$ to have an arbitrary modulus of continuity so long as it does not assign too much weight relative to the prior.

As a result of \Cref{thm:Gaussian_prior_combined}\ref{item:Gaussian_prior_better_potential}, we may conclude that all mode types in \Cref{fig:periodic_table}(\subref{fig:Hasse_diagram_main}) coincide for a very broad class of posterior measures arising in Bayesian inverse problems.
The following example illustrates the very common case of a quadratic misfit, which corresponds to an inverse problem with finite-dimensional Gaussian observational noise.
(Note that the case of infinite-dimensional data and observational noise is more subtle, as discussed by, e.g., \citet[Remark~3.8]{Stuart2010}, \citet[Remark~9]{Lasanen2012}, and \citet{KasanickyMandel2017}.)

\begin{example}[Unique notion of small-ball mode for some Bayesian inverse problems]
	Consider the inverse problem of recovering an unknown in a real, separable, and possibly infinite-dimensional Banach space $X$ from a finite-dimensional observation corrupted by additive noise.
	Suppose that the observed data, conditional upon a possible value $u \in X$ of the unknown, is given by
	\begin{equation*}
		y = G(u) + \eta ,
		\quad
		\text{i.e.}
		\quad
		y|u \sim N(G(u), \Gamma) ,
	\end{equation*}
	where the forward/observation operator $G \colon X \to \bR^{J}$ is Lipschitz and $\eta$ is Gaussian with mean zero and symmetric, positive-definite covariance matrix $\Gamma \in \bR^{J \times J}$, independent of $u$.
	For a given observed value of $y$, this setting corresponds to the potential
	\begin{equation*}
		\Phi(u) = \tfrac{1}{2} \bignorm{ \Gamma^{-1 \mathbin{/} 2} ( G(u) - y ) }^{2} .
	\end{equation*}
	If a centred non-degenerate Gaussian prior $\mu_{0} \in \prob{X}$ is posited, then the corresponding Bayesian posterior is $\mu(\rd u) = \frac{1}{Z} \exp (-\Phi(u)) \, \mu_{0}(\rd u)$ for an appropriate normalisation constant $Z$.
	There is an umambiguous notion of small-ball MAP estimator, i.e., a small-ball mode of the posterior $\mu$:
	since $\Phi$ is globally bounded below, all mode types in \Cref{fig:periodic_table}(\subref{fig:Hasse_diagram_main}) coincide by \Cref{thm:Gaussian_prior_combined}\ref{item:Gaussian_prior_better_potential}.
	\EndExampleText
\end{example}

\section{Examples and counterexamples}
\label{section:Examples}

This section provides examples and counterexamples to the concepts and implications discussed in \Cref{sec:small-ball_modes,sec:Mode_definitions_coincide}.
In the analysis of these examples, we make frequent use of the following elementary fact:

\begin{lemma}
	\label[lemma]{lem:mass_near_Lebesgue_point}
	If $\mu \in \prob{\bR}$ has a \ac{pdf} $\rho \colon \bR \to [0, \infty]$ that is bounded by $C > 0$ on some ball $\ball{u}{R}$ with $R > 0$, then, for any sequences $(u_{n})_{n \in \bN} \to u$ and $(r_{n})_{n \in \bN} \to 0$,
	\begin{equation}
		\label{eq:mass_near_Lebesgue_point}
		\text{for all large enough $n$,} \quad \mu(\ball{u_{n}}{r_{n}}) \equiv \int_{\ball{u_{n}}{r_{n}}} \rho(x) \, \rd x \leq 2 C r_{n} .
	\end{equation}
\end{lemma}

\subsection{Distinguishing weak, exotic, and strong modes}

We begin with an example of a \sww-mode that is not an \sws-mode, complementing earlier examples of \citet[Example~4.4]{LieSullivan2018} and \citet[Example~B.5]{ayanbayev2021gamma}.
The idea behind the construction of the measure $\mu$ is that it has exactly one singularity at $u = 0$, which, naturally, dominates each other point in terms of small ball probabilities $\mu(\Ball{u}{r})$ as $r \to 0$, making it a \sww-mode, but, for each fixed radius $r > 0$, there exists a point $v$ such that the centred ball mass $\mu(\Ball{v}{r})$ is greater, meaning that $u$ cannot be an \sws-mode.
Furthermore, `\sww' here can be strengthened to `\swe', and `\sws' can be weakened to `\swpgs'.

\begin{example}[$\swe$ but not $\swpgs$, even with an \OMexhaustive\ OM functional]
	\label[example]{example:e_but_not_pgs}
	Consider the truncated $\absval{x}^{-1/4}$-singularity $\sigma(x) \defeq \tfrac{3}{4} \absval{ x }^{-1 \mathbin{/} 4} \chi_{[-\nicefrac{1}{2}, \nicefrac{1}{2}]} (x)$ and let $\tau(x) \defeq \max \{ 0, 2 - 4 \absval{x} \}$;
	then let $\tau_{k}(x) = k\tau(k^{3} x)$ and
	let $\mu \in \prob{\bR}$ have the unnormalised \ac{pdf} $\rho$ shown in \Cref{fig:e_and_w_distinguishing}(\subref{figure:e_but_not_pgs}) and given by
	\[
		\rho(x) \defeq \sigma(x) + \sum_{k \in \bN} \tau_{k}(x - k).
	\]
	Taking $Z \defeq \int_{\bR} \rho(x) \,\rd x < \infty$, we note that
	$\mu(\ball{0}{r}) = \tfrac{1}{Z} r^{3 \mathbin{/} 4}$ for $0 \leq r \leq \frac{1}{2}$;
	moreover we note that $\tau_{k}$ is supported on $\ball{0}{\frac{k^{-3}}{2}}$ and $\mu(\ball{k}{r}) = \tfrac{1}{Z} k^{-2}$ for radii $\tfrac{1}{2} k^{-3} \leq r \leq \tfrac{1}{2}$.

	\proofpartorcase{$u = 0$ is an \swe-mode and the unique \swwgap-mode.}
	Let $v \neq 0$ and $(v_{n})_{n \in \bN} \to v$ be arbitrary.
	Consider the constant \as\ $(u)_{n \in \bN} \to u = 0$ and let $(r_{n})_{n \in \bN}$ be any \ns.
	Since $\rho$ is bounded by some $C = C(\rho, v)$ on $\ball{v}{\absval{v} \mathbin{/} 2}$, \Cref{lem:mass_near_Lebesgue_point} yields
	\begin{align*}
		\liminf_{n \to \infty} \Ratio{u}{v_{n}}{r_{n}}{\mu} \geq  \liminf_{n \to \infty} \frac{\tfrac{1}{Z} r_{n}^{3/4}}{\tfrac{1}{Z} C r_{n}}  = \infty \geq 1 .
	\end{align*}
	A minor variation of this argument (using the \cp\ $v = 0$) shows that no $u \neq 0$ can be a \swwgap-mode.

	\proofpartorcase{The constant \as\ is optimal at at all \swwgap-modes.}
	The \ac{pdf} $\rho$ is symmetric and unimodal within $(-\tfrac{1}{2}, \tfrac{1}{2})$, so \Cref{lemma:CASIO_for_ac_in_Rm}\ref{item:CASIO_at_local_symmetric_modes} implies that the constant \as\ of $u = 0$ is optimal in the sense of \eqref{eq:CASIO}, and $u = 0$ has just been shown to be the only \swwgap-mode of $\mu$.

	\proofpartorcase{\OMExhaustive\ \ac{OM} functional.}
	The measure $\mu$ has an \ac{OM} functional $I \colon E \to \bR$ on $E = \{ 0 \}$ given by $I(0) \defeq 0$, and property $M(\mu, E)$ follows from \Cref{lem:mass_near_Lebesgue_point}.

	\proofpartorcase{$u = 0$ is not a \swpgs-mode.}
	By \Cref{thm:modes_coincide_CASIO}, $\swps$- and $\swpgs$-modes are equivalent for this $\mu$,.
	Thus, it suffices to show that $u = 0$ is not a \swps-mode, i.e., that for any \ns\ $(r_{n})_{n \in \bN}$, $\liminf_{n \to \infty} \Ratio{u}{\sup}{r_{n}}{\mu} < 1$.
	We shall show the stronger statement that $\lim_{r \to 0} \Ratio{u}{\sup}{r}{\mu} = 0$.
	Set $k(r) \defeq \lceil (2 r)^{-1 \mathbin{/} 3} \rceil$, so that $\tfrac{1}{2} k(r)^{-3} \leq r < \tfrac{1}{2} ( k(r) - 1 )^{-3}$.
	Then
	\begin{align*}
		\liminf_{r \to 0} \Ratio{u}{\sup}{r}{\mu}
		\leq \liminf_{r \to 0} \Ratio{u}{k(r)}{r}{\mu}
		= \liminf_{r \to 0} \frac{\frac{1}{Z} r^{3 \mathbin{/} 4}}{\frac{1}{Z} k(r)^{-2}}
		\leq \liminf_{r \to 0} r^{3/4} \bigl((2r)^{-1/3} + 1\bigr)
		= 0.
		\EndExampleEquation
	\end{align*}
\end{example}

\begin{figure}[t!]
	\centering
	\begin{subfigure}[t]{0.49\textwidth}
		\centering
		\begin{tikzpicture}[scale=0.7]
			\begin{axis}[
				axis lines=middle,
				xlabel={$x$},
				ylabel={$\rho(x)$},
				xtick={1, 2, 3},
				ytick={1, 2, 3, 4},
				xmin=-0.2, xmax=3.4,
				ymin=0, ymax=4.0,
				samples=1024,
				domain=-0.5:3.0,
				restrict y to domain=0:4,
				clip=false,
				enlargelimits,
				xscale=1.5,
				xlabel style={at={(axis cs:3.8,0.0)}, anchor=west},
				ylabel style={at={(axis cs:0.0,4.5)}, anchor=south}
			]
				\addplot[domain=-0.5:-0.001, thick, blue] {0.75 * (-x)^(-1/4)};
				\addplot[domain=0.001:0.5, thick, blue] {0.75 * x^(-1/4)};
				\draw[thick, blue] (axis cs:-0.6, 0.0) -- (axis cs:-0.5, 0.0) -- (axis cs:-0.5, 0.891);
				\draw[thick, blue] (axis cs:0.5, 0.891) -- (axis cs:0.5, 0.0) -- (axis cs:1.0, 1) -- (axis cs:1.5, 0.0) -- (axis cs:1.5, 0.0) -- (axis cs:1.9375, 0.0) -- (axis cs:2.0, 2) -- (axis cs:2.0625, 0.0) -- (axis cs:2.9815,0.0) -- (axis cs:3.0, 3.0) -- (axis cs:3.0185,0.0) -- (axis cs:3.5,0.0);
			\end{axis}
		\end{tikzpicture}
		\subcaption{\raggedright Part of the unnormalised \ac{pdf} $\rho$ defined in \Cref{example:e_but_not_pgs}, which has an \swe-mode at $0$ but does not have any \swpgs-modes.}
		\label{figure:e_but_not_pgs}
	\end{subfigure}
	\begin{subfigure}[t]{0.49\textwidth}
		\centering
		\begin{tikzpicture}[scale=0.7]
			
			\begin{axis}[
				axis lines=middle,
				xlabel={$x$},
				ylabel={$\rho(x)$},
				xtick={1, 2, 3},
				ytick={1, 2, 3, 4},
				xmin=-0.2, xmax=3.4,
				ymin=0, ymax=4.0,
				samples=1024,
				domain=-0.5:3.0,
				restrict y to domain=0:4,
				clip=false,
				enlargelimits,
				xscale=1.5,
				xlabel style={at={(axis cs:3.8,0.0)}, anchor=west},
				ylabel style={at={(axis cs:0.0,4.5)}, anchor=south}
			]
				\addplot[domain=-0.5:-0.001, thick, blue] {0.75 * (-x)^(-1/4)};
				\addplot[domain=0.001:0.5, thick, blue] {0.75 * x^(-1/4)};
				
				\draw[thick, blue] (axis cs:-0.6, 0.0) -- (axis cs:-0.5, 0.0) -- (axis cs:-0.5, 0.891);
				\draw[thick, blue] (axis cs:0.0, 3.0) -- (axis cs:0.0, 4.0);
				\draw[thick, blue] (axis cs:0.5, 0.891) -- (axis cs:0.5, 0.0) -- (axis cs:1.0, 0.0) -- (axis cs:1.0, 3.9);
				\addplot[domain=1.001:1.5, thick, blue] {2*0.75*(1 - 0.25) * (x - 1)^(-1/4)};
				\draw[thick, blue] (axis cs:1.5, 1.337) -- (axis cs:1.5, 0.0) -- (axis cs:2.0, 0.0) -- (axis cs:2.0, 3.98);
				\addplot[domain=2.001:2.25, thick, blue] {2*0.75*(1 - 0.125) * (x - 2)^(-1/4)};
				\draw[thick, blue] (axis cs:2.25, 1.856) -- (axis cs:2.25, 0.0) -- (axis cs:3.0, 0.0) -- (axis cs:3.0, 3.95);
				\addplot[domain=3.001:3.125, thick, blue] {2*0.75*(1 - 0.0833) * (x - 3)^(-1/4)};
				\draw[thick, blue] (axis cs:3.125, 2.313) -- (axis cs:3.125, 0.0) -- (axis cs:3.6, 0.0);
			\end{axis}
		\end{tikzpicture}
		\subcaption{\raggedright Part of the unnormalised \ac{pdf} $\rho$ defined in \cref{example:CP}, which has a $\sww$-mode at $0$ that is neither a $\swpgs$- nor an $\swe$-mode.}
		\label{fig:CP_violation}
	\end{subfigure}
	\caption{\Cref{example:e_but_not_pgs,example:CP} serve as distinguishing examples in the lattices of mode types, making use of (possibly one-sided) $\absval{\quark}^{-1/4}$-singularities and hat functions.}
	
	\label{fig:e_and_w_distinguishing}
\end{figure}

\begin{example}[$\sww$ but neither $\swpgs$ nor $\swe$, even with an \OMexhaustive\ OM functional]
	\label[example]{example:CP}
	Define the one-sided, truncated $\absval{x}^{-1 \mathbin{/} 4}$-singularity $\rho_{r}(x) = \frac{3}{4} x^{-1 \mathbin{/} 4} \chi_{(0, r]}$, $r > 0$, and define $\mu \in \prob{\bR}$ through the unnormalised \ac{pdf} shown in \Cref{fig:e_and_w_distinguishing}(\subref{fig:CP_violation}) and given by
	\begin{equation*}
		\rho(x) = \rho_{2^{-1}}(x) + \rho_{2^{-1}}(-x) + 2\sum_{k = 1}^\infty \biggl(1-\frac{1}{4k}\biggr)\rho_{2^{-k}}(x - k).
	\end{equation*}

	\proofpartorcase{Ball masses under $\rho_{r}$.}
	A straightforward calculation using the density $\rho_{r}$ reveals that
	\begin{equation}
		\label{equ:CP_violation_example_properties}
		\int_{\ball{0}{s}} \rho_{r}(x) \,\rd x
		=
		s^{3/4}
		~~
		\text{for $0 \leq s \leq r$, and}
		~~
		\int_{\ball{s}{s}} \rho_{r}(x) \,\rd x
		=
		(2s)^{3/4}
		~~
		\text{for $0 \leq s \leq \tfrac{r}{2}$.}
	\end{equation}
	These facts will be used to construct non-centred approximating sequences which dominate the constant approximating sequence centred at $0$.

	\proofpartorcase{$\rho$ is normalisable.}
	Using \eqref{equ:CP_violation_example_properties} and the definition of $\rho$, it is straightforward to check that the normalisation constant $Z = \int_{\bR} \rho(x) \,\rd x < \infty$, so $\mu \in \prob{\bR}$.

	\proofpartorcase{\OMExhaustive\ \ac{OM} functional.}
	The measure $\mu$ has an \ac{OM} functional $I \colon E \to \bR$ defined on $E = \bN \cup \{0\}$, given by
	\begin{equation*}
		I(x) = \begin{cases}
			0, & \text{if $x = 0$,}\\
			-\log \left(1 - \frac{1}{4x}\right), & \text{if $x \in \bN$,}
		\end{cases}
	\end{equation*}
	which can be seen using elementary calculations using the definition \eqref{eq:OM_functional} of an \ac{OM} functional and the density $\rho$;
	property $M(\mu, E)$ follows because $\mu(\ball{u}{r}) \in \bigO(r)$ for $u \notin E$ by \eqref{eq:mass_near_Lebesgue_point}.

	\proofpartorcase{$u = 0$ is the unique $\sww$- and $\swwgap$-mode.}
	Since $u = 0$ is the unique minimiser of $I$ and property $M$ holds, it is the unique $\sww$-mode by \Cref{prop:OM_minimisers_w-modes};
	it is also a $\swwgap$-mode by \Cref{thm:main}.

	To see that there is no other $\swwgap$-mode,
	first note that no $u \notin \bN$ can be a $\swwgap$-mode:
	given any \as\ $(u_{n})_{n \in \bN}$ of such $u$ and any \ns\ $(r_{n})_{n \in \bN}$, we have $\mu(\ball{u_{n}}{r_{n}}) \in \bigO(r_{n})$ by \eqref{eq:mass_near_Lebesgue_point};
	however, the constant \cpas\ $(0)_{n \in \bN}$ satisfies $\mu(\ball{0}{r_{n}}) \in \littleOmega(r_{n})$, and so $\liminf_{n \to \infty} \Ratio{u_{n}}{0}{r_{n}}{\mu} = 0$.
	It remains to show that no $u \in \bN$ is a $\swwgap$-mode.
	To see this, take the \cp\ $v = u + 1$ and the \cpas\ $v_{n} = u_{n} + 1$, and let $(r_{n})_{n \in \bN}$ be any \ns.
	Then $u$ is not a $\swwgap$-mode because
	\[
		\Ratio{u_{n}}{v_{n}}{r_{n}}{\mu} = \frac{1 - \tfrac{1}{4u}}{1 - \tfrac{1}{4(u+1)}} < 1 \text{~~~for all large enough $n$.}
	\]

	\proofpartorcase{The constant $\as$ is optimal at all $\swwgap$-modes.}
	The point $u = 0$ is the unique $\swwgap$-mode, and, since $\rho$ is symmetric and has convex superlevel sets on $\ball{0}{1/2}$, \Cref{lemma:CASIO_for_ac_in_Rm}\ref{item:CASIO_at_local_symmetric_modes} implies that the constant \as\ is optimal here.

	\proofpartorcase{$u = 0$ is neither an $\swe$-mode nor a $\swpgs$-mode.}
	Let $(r_{n})_{n \in \bN}$ be an arbitrary \ns\ and note that
		the constant \as\ $u_{n} = u$ is optimal.
		Then, take the \cp\ $v = 1$ and the \cpas\ $v_{n} = 1 + r_{n}$;
		 we see that
		\begin{equation} \label{eq:example_CP_ball_ratio_e}
			\Ratio{u_{n}}{v_{n}}{r_{n}}{\mu} = \frac{2r_{n}^{3/4}}{2\bigl(1-\tfrac{1}{4}\bigr) (2r_{n})^{3/4}} < 1 \text{~~for all large enough $n$.}
		\end{equation}
		This immediately implies that $u = 0$ is not an \ac{e-mode};
		as \eqref{eq:example_CP_ball_ratio_e} holds for all \ns\ and $\Ratio{u_{n}}{\sup}{r_{n}}{\mu} \leq \Ratio{u_{n}}{v_{n}}{r_{n}}{\mu} < 1$, the claim also follows for \acp{pgs-mode}.
	\EndExampleText
\end{example}

\subsection{Distinguishing some partial and non-partial modes; the merging property}
\label{subsec:examples_partial--non-partial}

The following example illustrates the difference between partial and non-partial modes by considering two mild perturbations of the integrable singularity $\absval{ x }^{-1 \mathbin{/} 2}$ on $\bR$.
The perturbations, centred in this example at $\pm 2$, are chosen in such a way that the mass of the radius-$r$ ball centred at $\pm 2$ is $\bigTheta(\sqrt{r})$, but the \emph{ratio} of these two ball masses oscillates either side of unity as $r \to 0$.
This makes the points $\pm 2$ into \swps-modes but not \sws-modes, nor even \swgw-modes, even when we include an additional unperturbed $\absval{ x }^{-1 \mathbin{/} 2}$ singularity at the origin that is an \sws-mode.
This oscillation phenomenon also has the side effect of disproving the merging property \rMP\ for non-partial modes;
the more sophisticated \Cref{example:SuspensionBridgeExtended} will disprove \rMP\ for partial modes.

\begin{example}[{$\swps$ but not $\swgw$, even with an $\sws$-mode;
	meaningful non-$\swp$-modes fail \rMP}]
	\label[example]{example:ps_not_s_not_gw}
	Consider the \ac{pdf} $\sigma(x) \defeq \frac{1}{4} \absval{x}^{-1 \mathbin{/} 2} \chi_{[{-1}, 1]} (x)$ with the property that $\int_{-r}^{r} \sigma(x) \, \rd x = \sqrt{r}$ for $0 \leq r \leq 1$.
	Now let $\sigma_{\even}$ be the symmetric density supported on $[-1, 1]$ that is uniquely determined by the property that $r \mapsto \int_{-r}^{r} \sigma_{\even}(x) \, \rd x$ agrees with $\sqrt{r}$ for $r = r_{n} \defeq 2^{-2 n}$, $n \in \bN \cup \{ 0 \}$, and interpolates linearly in between these knots;
	define $\sigma_{\odd}$ similarly, but supported on $[-\tfrac{1}{2}, \tfrac{1}{2}]$ and with agreement at $r = 2^{-2 n - 1}$ for $n \in \bN \cup \{ 0 \}$.
	Note that $\sigma_{\text{even}}$ has unit mass by construction, but $\sigma_{\text{odd}}$ does not.
	Finally, let $\mu \in \prob{\bR}$ have the unnormalised \ac{pdf} shown in \Cref{fig:sigma_even_odd_plots}(\subref{subfig:ps_but_neither_gs_nor_w}) and given by
	\[
		\rho(x) \defeq \frac{1}{3} \sigma_{\even}(x + 2) + \frac{1}{3} \sigma(x) + \frac{1}{3} \sigma_{\odd}(x - 2),
	\]
	and let $Z \defeq \int_{\bR} \rho(x) \,\rd x < \infty$ be the normalising constant.

	\proofpartorcase{The constant \as\ is optimal at all \swwgap-modes.}
	Every point $x \in \bR \setminus \{ {-2}, 0, 2 \}$ is a Lebesgue point for $\mu$, with $\mu(\ball{x}{r}) \in \bigO(r)$, by \eqref{eq:mass_near_Lebesgue_point}.
	On the other hand, $\mu(\ball{u}{r}) \in \bigTheta(\sqrt{r})$ for $u \in \{ {-2}, 0, 2 \}$, so it follows that these points $u$ are the only possible \swwgap-modes of $\mu$.
	By the same appeal to Anderson's theorem \citep[Theorem~1]{Anderson1955} as is used in \Cref{lemma:CASIO_for_ac_in_Rm}\ref{item:CASIO_at_local_symmetric_modes},
	\begin{equation}
		\label{eq:Anderson_consequence_for_ps_not_s_not_gw}
		\forall \as\, (u_{n})_{n \in \bN} \to u \in \{ {-2}, 0, 2 \} , \,
		\forall \ns\, (r_{n})_{n \in \bN} : \quad
		\mu(\ball{u_{n}}{r_{n}}) \leq \mu(\ball{u}{r_{n}}) \text{ eventually,}
	\end{equation}
	and in particular the constant \as\ of $u$ is optimal at all \swwgap-modes in the sense of \eqref{eq:CASIO}.

	\proofpartorcase{$u = 0$ is an \sws-mode.}
	Simply observe that, for each $0 < r < 1$, $\mu(\ball{0}{r}) = \frac{1}{3Z} \sqrt{r} = \sMass{\mu}{r}$.

	\proofpartorcase{$u = \pm 2$ are \swps-modes.}
	For $r_{n} \defeq 2^{-2 n}$, $\mu(\ball{-2}{r_{n}}) = \frac{1}{3Z} \sqrt{r_{n}} = \sMass{\mu}{r_{n}}$, and so $u = {-2}$ is a \swps-mode along this \ns.
	A similar argument shows that $u = 2$ is a \swps-mode along the \ns\ $(\frac{r_{n}}{2})_{n \in \bN}$.

	\proofpartorcase{$u = \pm 2$ are not \swgw-modes.}
	A direct calculation shows that
	\begin{equation} \label{eq:ps_but_neither_gs_nor_w_ball-mass_ratio}
		\Ratio{-2}{2}{2^{-2n}}{\mu} = \frac{2\sqrt{2}}{3} < 1 < \frac{3}{2\sqrt{2}} = \Ratio{-2}{2}{2^{-2n-1}}{\mu} \quad \text{for all $n \in \bN$}.
	\end{equation}
	This shows that $\liminf_{r \to 0} \Ratio{-2}{2}{r}{\mu} < 1 < \limsup_{r \to 0} \Ratio{-2}{2}{r}{\mu}$.
	Now, to see that $u = {-2}$ is not a \swgw-mode, let us again consider $r_{n} \defeq 2^{-2 n}$.
	Then, for any \as\ $(u_{n})_{n \in \bN} \to u = {-2}$, we consider the \cp\ $v = 2$.
	Then, using \eqref{eq:Anderson_consequence_for_ps_not_s_not_gw},
	\begin{equation*}
		\liminf_{n \to \infty} \Ratio{u_{n}}{v}{\frac{r_{n}}{2} }{\mu} \leq \liminf_{n \to \infty} \Ratio{u}{v}{\frac{r_{n}}{2}}{\mu}  < 1 .
	\end{equation*}
	Thus, $u = -2$ is not a \swgw-mode;
	this means that it is neither a \swgs-mode nor a \sww-mode.
	The argument for $u = 2$ is similar, with \cp\ $v = {-2}$ and \ns\ $(r_{n})_{n \in \bN}$.

	\proofpartorcase{Merging properties.}
	Take $A \defeq (-3, -1)$ and $B \defeq (1, 3)$.
	Then $u_{A} \defeq {-2}$ and $u_{B} \defeq 2$ are \sws-modes of $\mu|_{A}$ and $\mu|_{B}$ respectively, while almost all arguments regarding $\mu$ above also apply to $\mu|_{A \cup B}$:
	its only \swwgap-modes are $u_{A}$ and $u_{B}$, and they are \swps-modes, but neither is a \swgw-mode.
	This disproves \rMP\ for all meaningful non-partial modes (\Cref{prop:all_modes_violate_MP}).
	\EndExampleText
\end{example}

\begin{figure}[t!]
	\centering
	\begin{subfigure}[t]{0.49\textwidth}
		\centering
		\begin{tikzpicture}[scale=0.7]
			\begin{axis}[
				axis lines=middle,
				xlabel={$x$},
				ylabel={$\rho(x)$},
				xtick={-2,2},
				ytick={1,2,3,4},
				xmin=-2.2, xmax=2.8,
				ymin=0, ymax=3.9,
				samples=1024,
				domain=-1.0:1.0,
				restrict y to domain=0:4,
				clip=false,
				enlargelimits,
				xscale=1.5,
				xlabel style={at={(axis cs:3.3,0.0)}, anchor=west},
				ylabel style={at={(axis cs:0.0,4.3)}, anchor=south}
			]
				\addplot[domain=-1:-0.001, thick, blue] {0.25 * (-x)^(-1/2)};
				\addplot[domain=0.001:1, thick, blue] {0.25 * x^(-1/2)};
				\draw[thick, blue] (axis cs:0.0,3) -- (axis cs:0.0, 4.2);
				\draw[thick, blue] (axis cs:-1, 0.0) -- (axis cs:-1, 0.25);
				\draw[thick, blue] (axis cs:1, 0.25) -- (axis cs:1.0, 0.0);

				\draw[thick, accent4] (axis cs:3, 0) -- (axis cs:3, 0.3333333333333333) -- (axis cs:2.25, 0.3333333333333333);
				\draw[thick, accent4] (axis cs:1, 0) -- (axis cs:1, 0.3333333333333333) -- (axis cs:1.75, 0.3333333333333333);
				\draw[thick, accent4] (axis cs:2.25, 0.3333333333333333) -- (axis cs:2.25, 0.6666666666666666) -- (axis cs:2.0625, 0.6666666666666666);
				\draw[thick, accent4] (axis cs:1.75, 0.3333333333333333) -- (axis cs:1.75, 0.6666666666666666) -- (axis cs:1.9375, 0.6666666666666666);
				\draw[thick, accent4] (axis cs:2.0625, 0.6666666666666666) -- (axis cs:2.0625, 1.3333333333333333) -- (axis cs:2.015625, 1.3333333333333333);
				\draw[thick, accent4] (axis cs:1.9375, 0.6666666666666666) -- (axis cs:1.9375, 1.3333333333333333) -- (axis cs:1.984375, 1.3333333333333333);
				\draw[thick, accent4] (axis cs:2.015625, 1.3333333333333333) -- (axis cs:2.015625, 2.6666666666666665) -- (axis cs:2.00390625, 2.6666666666666665);
				\draw[thick, accent4] (axis cs:1.984375, 1.3333333333333333) -- (axis cs:1.984375, 2.6666666666666665) -- (axis cs:1.99609375, 2.6666666666666665);
				\draw[thick, accent4] (axis cs:2.00390625, 2.6666666666666665) -- (axis cs:2.00390625, 4.2) -- (axis cs:2.0009765625, 4.2);
				\draw[thick, accent4] (axis cs:1.99609375, 2.6666666666666665) -- (axis cs:1.99609375, 4.2) -- (axis cs:1.9990234375, 4.2);
				\draw[thick, accent3] (axis cs:-1.5, 0) -- (axis cs:-1.5, 0.9428090415820634) -- (axis cs:-1.875, 0.9428090415820634);
				\draw[thick, accent3] (axis cs:-2.5, 0) -- (axis cs:-2.5, 0.9428090415820634) -- (axis cs:-2.125, 0.9428090415820634);
				\draw[thick, accent3] (axis cs:-1.875, 0.9428090415820634) -- (axis cs:-1.875, 1.8856180831641267) -- (axis cs:-1.96875, 1.8856180831641267);
				\draw[thick, accent3] (axis cs:-2.125, 0.9428090415820634) -- (axis cs:-2.125, 1.8856180831641267) -- (axis cs:-2.03125, 1.8856180831641267);
				\draw[thick, accent3] (axis cs:-1.96875, 1.8856180831641267) -- (axis cs:-1.96875, 3.7712361663282534) -- (axis cs:-1.9921875, 3.7712361663282534);
				\draw[thick, accent3] (axis cs:-2.03125, 1.8856180831641267) -- (axis cs:-2.03125, 3.7712361663282534) -- (axis cs:-2.0078125, 3.7712361663282534);
				\draw[thick, accent3] (axis cs:-1.9921875, 3.7712361663282534) -- (axis cs:-1.9921875, 4.2) -- (axis cs:-1.998046875, 4.2);
				\draw[thick, accent3] (axis cs:-2.0078125, 3.7712361663282534) -- (axis cs:-2.0078125, 4.2) -- (axis cs:-2.001953125, 4.2);
				\draw[thick, accent3] (axis cs:-1.998046875, 4.2) -- (axis cs:-1.998046875, 4.2) -- (axis cs:-1.99951171875, 4.2);
				\draw[thick, accent3] (axis cs:-2.001953125, 4.2) -- (axis cs:-2.001953125, 4.2) -- (axis cs:-2.00048828125, 4.2);
			\end{axis}
		\end{tikzpicture}
		\subcaption{\raggedright Part of the unnormalised \ac{pdf} $\rho$ from \Cref{example:ps_not_s_not_gw}, which consists of singularities $\sigma_{\text{even}}$ (in black, centred at $2$) and $\sigma_{\text{odd}}$ (in orange, centred at $-2$) and has \swps-modes at $\pm 2$ that are neither \swgs- nor \sww-modes, even with an \sws-mode at $0$.}
		\label{subfig:ps_but_neither_gs_nor_w}
	\end{subfigure}
	\begin{subfigure}[t]{0.49\textwidth}
		\centering
		\begin{tikzpicture}[scale=0.7]
			\begin{axis}[
				axis lines=middle,
				xlabel={$x$},
				ylabel={$\rho(x)$},
				xtick={-5,-4,-3,-1, 1,3,4,5},
				ytick={1, 2, 3, 4},
				xmin=-4.6, xmax=4.6,
				ymin=0, ymax=3.9,
				samples=1024,
				domain=-5.5:5.5,
				restrict y to domain=0:3,
				clip=false,
				enlargelimits,
				xscale=1.5,
				xlabel style={at={(axis cs:5.6,0.0)}, anchor=west},
				ylabel style={at={(axis cs:0.0,4.3)}, anchor=south}
			]
				\draw[thick, accent4] (axis cs:0, 0) -- (axis cs:0, 0.3333333333333333) -- (axis cs:-0.75, 0.3333333333333333);
				\draw[thick, accent4] (axis cs:-2, 0) -- (axis cs:-2, 0.3333333333333333) -- (axis cs:-1.25, 0.3333333333333333);
				\draw[thick, accent4] (axis cs:-0.75, 0.3333333333333333) -- (axis cs:-0.75, 0.6666666666666666) -- (axis cs:-0.9375, 0.6666666666666666);
				\draw[thick, accent4] (axis cs:-1.25, 0.3333333333333333) -- (axis cs:-1.25, 0.6666666666666666) -- (axis cs:-1.0625, 0.6666666666666666);
				\draw[thick, accent4] (axis cs:-0.9375, 0.6666666666666666) -- (axis cs:-0.9375, 1.3333333333333333) -- (axis cs:-0.984375, 1.3333333333333333);
				\draw[thick, accent4] (axis cs:-1.0625, 0.6666666666666666) -- (axis cs:-1.0625, 1.3333333333333333) -- (axis cs:-1.015625, 1.3333333333333333);
				\draw[thick, accent4] (axis cs:-0.984375, 1.3333333333333333) -- (axis cs:-0.984375, 2.6666666666666665) -- (axis cs:-0.99609375, 2.6666666666666665);
				\draw[thick, accent4] (axis cs:-1.015625, 1.3333333333333333) -- (axis cs:-1.015625, 2.6666666666666665) -- (axis cs:-1.00390625, 2.6666666666666665);
				\draw[thick, accent4] (axis cs:-0.99609375, 2.6666666666666665) -- (axis cs:-0.99609375, 4.2) -- (axis cs:-0.9990234375, 4.2);
				\draw[thick, accent4] (axis cs:-1.00390625, 2.6666666666666665) -- (axis cs:-1.00390625, 4.2) -- (axis cs:-1.0009765625, 4.2);
				\draw[thick, accent3] (axis cs:1.5, 0) -- (axis cs:1.5, 0.9428090415820634) -- (axis cs:1.125, 0.9428090415820634);
				\draw[thick, accent3] (axis cs:0.5, 0) -- (axis cs:0.5, 0.9428090415820634) -- (axis cs:0.875, 0.9428090415820634);
				\draw[thick, accent3] (axis cs:1.125, 0.9428090415820634) -- (axis cs:1.125, 1.8856180831641267) -- (axis cs:1.03125, 1.8856180831641267);
				\draw[thick, accent3] (axis cs:0.875, 0.9428090415820634) -- (axis cs:0.875, 1.8856180831641267) -- (axis cs:0.96875, 1.8856180831641267);
				\draw[thick, accent3] (axis cs:1.03125, 1.8856180831641267) -- (axis cs:1.03125, 3.7712361663282534) -- (axis cs:1.0078125, 3.7712361663282534);
				\draw[thick, accent3] (axis cs:0.96875, 1.8856180831641267) -- (axis cs:0.96875, 3.7712361663282534) -- (axis cs:0.9921875, 3.7712361663282534);
				\draw[thick, accent3] (axis cs:1.0078125, 3.7712361663282534) -- (axis cs:1.0078125, 4.2) -- (axis cs:1.001953125, 4.2);
				\draw[thick, accent3] (axis cs:0.9921875, 3.7712361663282534) -- (axis cs:0.9921875, 4.2) -- (axis cs:0.998046875, 4.2);
				\draw[thick, accent3] (axis cs:1.001953125, 4.2) -- (axis cs:1.001953125, 4.2) -- (axis cs:1.00048828125, 4.2);
				\draw[thick, accent3] (axis cs:0.998046875, 4.2) -- (axis cs:0.998046875, 4.2) -- (axis cs:0.99951171875, 4.2);
				\draw[thick, accent4] (axis cs:3.25, 0) -- (axis cs:3.25, 0.6678910949120377) -- (axis cs:3.0625, 0.6678910949120377);
				\draw[thick, accent4] (axis cs:2.75, 0) -- (axis cs:2.75, 0.6678910949120377) -- (axis cs:2.9375, 0.6678910949120377);
				\draw[thick, accent4] (axis cs:3.0625, 0.6678910949120377) -- (axis cs:3.0625, 1.3357821898240754) -- (axis cs:3.015625, 1.3357821898240754);
				\draw[thick, accent4] (axis cs:2.9375, 0.6678910949120377) -- (axis cs:2.9375, 1.3357821898240754) -- (axis cs:2.984375, 1.3357821898240754);
				\draw[thick, accent4] (axis cs:3.015625, 1.3357821898240754) -- (axis cs:3.015625, 2.6715643796481507) -- (axis cs:3.00390625, 2.6715643796481507);
				\draw[thick, accent4] (axis cs:2.984375, 1.3357821898240754) -- (axis cs:2.984375, 2.6715643796481507) -- (axis cs:2.99609375, 2.6715643796481507);
				\draw[thick, accent4] (axis cs:3.00390625, 2.6715643796481507) -- (axis cs:3.00390625, 4.2) -- (axis cs:3.0009765625, 4.2);
				\draw[thick, accent4] (axis cs:2.99609375, 2.6715643796481507) -- (axis cs:2.99609375, 4.2) -- (axis cs:2.9990234375, 4.2);
				\draw[thick, accent3] (axis cs:-2.875, 0) -- (axis cs:-2.875, 1.8890812892256394) -- (axis cs:-2.96875, 1.8890812892256394);
				\draw[thick, accent3] (axis cs:-3.125, 0) -- (axis cs:-3.125, 1.8890812892256394) -- (axis cs:-3.03125, 1.8890812892256394);
				\draw[thick, accent3] (axis cs:-2.96875, 1.8890812892256394) -- (axis cs:-2.96875, 3.778162578451279) -- (axis cs:-2.9921875, 3.778162578451279);
				\draw[thick, accent3] (axis cs:-3.03125, 1.8890812892256394) -- (axis cs:-3.03125, 3.778162578451279) -- (axis cs:-3.0078125, 3.778162578451279);
				\draw[thick, accent3] (axis cs:-2.9921875, 3.778162578451279) -- (axis cs:-2.9921875, 4.2) -- (axis cs:-2.998046875, 4.2);
				\draw[thick, accent3] (axis cs:-3.0078125, 3.778162578451279) -- (axis cs:-3.0078125, 4.2) -- (axis cs:-3.001953125, 4.2);
				\draw[thick, accent3] (axis cs:-2.998046875, 4.2) -- (axis cs:-2.998046875, 4.2) -- (axis cs:-2.99951171875, 4.2);
				\draw[thick, accent3] (axis cs:-3.001953125, 4.2) -- (axis cs:-3.001953125, 4.2) -- (axis cs:-3.00048828125, 4.2);
				\draw[thick, accent4] (axis cs:4.0625, 0) -- (axis cs:4.0625, 1.4095999522231526) -- (axis cs:4.015625, 1.4095999522231526);
				\draw[thick, accent4] (axis cs:3.9375, 0) -- (axis cs:3.9375, 1.4095999522231526) -- (axis cs:3.984375, 1.4095999522231526);
				\draw[thick, accent4] (axis cs:4.015625, 1.4095999522231526) -- (axis cs:4.015625, 2.819199904446305) -- (axis cs:4.00390625, 2.819199904446305);
				\draw[thick, accent4] (axis cs:3.984375, 1.4095999522231526) -- (axis cs:3.984375, 2.819199904446305) -- (axis cs:3.99609375, 2.819199904446305);
				\draw[thick, accent4] (axis cs:4.00390625, 2.819199904446305) -- (axis cs:4.00390625, 4.2) -- (axis cs:4.0009765625, 4.2);
				\draw[thick, accent4] (axis cs:3.99609375, 2.819199904446305) -- (axis cs:3.99609375, 4.2) -- (axis cs:3.9990234375, 4.2);
				\draw[thick, accent3] (axis cs:-3.96875, 0) -- (axis cs:-3.96875, 3.986950739908898) -- (axis cs:-3.9921875, 3.986950739908898);
				\draw[thick, accent3] (axis cs:-4.03125, 0) -- (axis cs:-4.03125, 3.986950739908898) -- (axis cs:-4.0078125, 3.986950739908898);
				\draw[thick, accent3] (axis cs:-3.9921875, 3.986950739908898) -- (axis cs:-3.9921875, 4.2) -- (axis cs:-3.998046875, 4.2);
				\draw[thick, accent3] (axis cs:-4.0078125, 3.986950739908898) -- (axis cs:-4.0078125, 4.2) -- (axis cs:-4.001953125, 4.2);
				\draw[thick, accent3] (axis cs:-3.998046875, 4.2) -- (axis cs:-3.998046875, 4.2) -- (axis cs:-3.99951171875, 4.2);
				\draw[thick, accent3] (axis cs:-4.001953125, 4.2) -- (axis cs:-4.001953125, 4.2) -- (axis cs:-4.00048828125, 4.2);
				\draw[thick, accent4] (axis cs:5.015625, 0) -- (axis cs:5.015625, 2.827884347081491) -- (axis cs:5.00390625, 2.827884347081491);
				\draw[thick, accent4] (axis cs:4.984375, 0) -- (axis cs:4.984375, 2.827884347081491) -- (axis cs:4.99609375, 2.827884347081491);
				\draw[thick, accent4] (axis cs:5.00390625, 2.827884347081491) -- (axis cs:5.00390625, 4.2) -- (axis cs:5.0009765625, 4.2);
				\draw[thick, accent4] (axis cs:4.99609375, 2.827884347081491) -- (axis cs:4.99609375, 4.2) -- (axis cs:4.9990234375, 4.2);
				\draw[thick, accent3] (axis cs:-4.9921875, 0) -- (axis cs:-4.9921875, 4.2) -- (axis cs:-4.998046875, 4.2);
				\draw[thick, accent3] (axis cs:-5.0078125, 0) -- (axis cs:-5.0078125, 4.2) -- (axis cs:-5.001953125, 4.2);
				\draw[thick, accent3] (axis cs:-4.998046875, 4.2) -- (axis cs:-4.998046875, 4.2) -- (axis cs:-4.99951171875, 4.2);
				\draw[thick, accent3] (axis cs:-5.001953125, 4.2) -- (axis cs:-5.001953125, 4.2) -- (axis cs:-5.00048828125, 4.2);
			\end{axis}
		\end{tikzpicture}
		\subcaption{\raggedright Part of the unnormalised \ac{pdf} $\rho$ from \cref{example:SuspensionBridgeExtended}, which consists of truncated, reweighted copies of $\sigma_{\text{even}}$ (in black, centred at $-1$ and $3$, $4$, $\ldots$) and $\sigma_{\text{odd}}$ (in orange, centred at $1$ and $-3$, $-4$, $\ldots$), and proves that meaningful $\swp$-modes violate \rMP.}
		\label{subfig:meaningful_p-modes_violate_MP}
	\end{subfigure}
	\caption{\Cref{example:ps_not_s_not_gw,example:SuspensionBridgeExtended} disprove \rMP\ for all meaningful mode types (\cref{prop:all_modes_violate_MP}). \Cref{example:ps_not_s_not_gw} also serves as a distinguishing example in the lattices of mode types.}
	\label{fig:sigma_even_odd_plots}
\end{figure}

This example disproves \rMP\ for all meaningful non-$\swp$-modes.
To disprove \rMP\ for meaningful $\swp$-modes, we consider a variant with countably many copies of the singularities $\sigma_{\text{even}}$ and $\sigma_{\text{odd}}$.

\begin{example}[Meaningful $\swp$-modes fail \rMP]
	\label[example]{example:SuspensionBridgeExtended}
	Recall the densities $\sigma_{\text{even}}$ and $\sigma_{\text{odd}}$ from \cref{example:ps_not_s_not_gw}, and define truncated versions
	\begin{equation*}
		\sigma_{\text{even}, k}(x) \defeq \sigma_{\text{even}}(x) \chi_{[-2^{-2k}, 2^{-2k}]}(x), \qquad \sigma_{\text{odd}, k}(x) \defeq \sigma_{\text{odd}}(x) \chi_{[-2^{-2k-1}, 2^{-2k-1}]}(x).
	\end{equation*}
	These densities have mass $2^{-k}$ and $2^{-k-\nicefrac{1}{2}}$ respectively.
	Then let $\beta > \bigl(\tfrac{3}{2\sqrt{2}} - 1\bigr)^{-1}$ and let $\mu \in \prob{\bR}$ have the unnormalised \ac{pdf} shown in \Cref{fig:sigma_even_odd_plots}(\subref{subfig:meaningful_p-modes_violate_MP}) and given by
	\begin{equation*}
		\rho(x) \defeq \sigma_{\text{even}}(x + 1) + \sigma_{\text{odd}}(x - 1) + \sum_{k = 1}^{\infty} \biggl(\frac{3}{2\sqrt{2}} - \beta^{-k} \biggr) \Bigl[ \sigma_{\text{even},k}(x - k - 2) + \sigma_{\text{odd},k}(x + k + 2) \Bigr].
	\end{equation*}

	\proofpartorcase{$\mu$ is a probability measure.}
	The normalising constant $Z \defeq \int_{\bR} \rho(x) \,\rd x$ is finite because, owing to the choice of truncation radius, the masses of the $\sigma_{\text{even},k}$ and $\sigma_{\text{odd},k}$ are summable.

	\proofpartorcase{$-1$ is a $\swps$-mode of $\mu|_{A}$, $A = (-\infty, -\frac{1}{2})$.}
	Note that, for $r \leq \frac{1}{4}$, any ball $\ball{u}{r}$, $u \in X$ intersects the support of at most one of the singularities, and if $\ball{u}{r}$ intersects the support of the singularity centred at $k$, then $\mu|_{A}(\ball{u}{r}) \leq \mu|_{A}(\ball{k}{r})$.
	Thus
	\begin{equation*}
		\Ratio{-1}{\sup}{r}{\mu|_{A}} = \sup_{k \in \{-1\} \cup \{-3, -4, \dots\}} \Ratio{-1}{k}{r}{\mu|_{A}} \text{~~for all sufficiently small $r$.}
	\end{equation*}
	But, using the definition of $\rho$ and \eqref{eq:ps_but_neither_gs_nor_w_ball-mass_ratio}, we have, along the distinguished \ns\ $r_{n} \defeq 2^{-2n-1}$ and for $k \in \{-3, -4, \dots\}$ that
	\begin{equation*}
		\Ratio{-1}{k}{r_{n}}{\mu|_{A}} \geq \frac{1}{\bigl(\tfrac{3}{2\sqrt{2}} - \beta^{-k+2}\bigr)} \cdot \frac{\int_{-r_{n}}^{r_{n}} \sigma_{\text{even}}(x) \,\rd x}{\int_{-r_{n}}^{r_{n}} \sigma_{\text{odd}}(x) \,\rd x} \geq \frac{1}{\bigl(\tfrac{3}{2\sqrt{2}} - \beta^{-k+2}\bigr)} \cdot \frac{2\sqrt{2}}{3}.
	\end{equation*}
	From this we conclude that $\liminf_{n \to \infty} \Ratio{-1}{\sup}{r_{n}}{\mu|_{A}} \geq 1$, i.e., that $-1$ is a $\swps$-mode of $\mu|_{A}$.

	\proofpartorcase{$1$ is a $\swps$-mode of $\mu|_{B}$, $B = (\frac{1}{2}, \infty)$.}
	This follows by a similar argument with the \ns\ $r_{n} \defeq 2^{-2n}$.

	\proofpartorcase{$\mu|_{A \cup B}$ has no $\swwgap$-modes.}
	First, note that if $u \notin \bZ \setminus \{0, -2, 2\}$, then $\mu|_{A \cup B}(\ball{u_{n}}{r_{n}}) \in \bigO(r_{n})$ for any \as\ $(u_{n})_{n \in \bN}$ and \ns\ $(r_{n})_{n \in \bN}$ by \cref{lem:mass_near_Lebesgue_point}, whereas $\mu|_{A \cup B}(\ball{1}{r_{n}}) \in \bigTheta(\sqrt{r_{n}})$;
	this shows that no such $u$ can be a $\swwgap$-mode.
	Next we will prove that no $u \in \{1, 3, 4, \dots\}$ is a \swwgap-mode;
	in all cases the constant \as\ of $u$ is optimal.
	By construction we have, for all \ns\ $(r_{n})_{n \in \bN}$, that
	\begin{equation*}
		\liminf_{n \to \infty} \Ratio{1}{-3}{r_{n}}{\mu} = \liminf_{n \to \infty} \frac{1}{\bigl(\tfrac{3}{2\sqrt{2}} - \beta^{-1}\bigr)} \cdot \frac{\int_{-r_{n}}^{r_{n}} \sigma_{\text{odd}}(x) \,\rd x}{\int_{-r_{n}}^{r_{n}} \sigma_{\text{odd}}(x) \,\rd x} < 1,
	\end{equation*}
	showing that $u = 1$ is not a \swwgap-mode.
	Otherwise, if $u \geq 3$, then taking the \cp\ $v = u + 1$ yields
	\begin{equation*}
		\liminf_{n \to \infty} \Ratio{u}{v}{r_{n}}{\mu} = \liminf_{n \to \infty} \frac{\bigl(\tfrac{3}{2\sqrt{2}} - \beta^{-u+2} \bigr)}{\bigl(\tfrac{3}{2\sqrt{2}} - \beta^{-u+1}\bigr)} \cdot \frac{\int_{-r_{n}}^{r_{n}} \sigma_{\text{even}}(x) \,\rd x}{\int_{-r_{n}}^{r_{n}} \sigma_{\text{even}}(x) \,\rd x} < 1.
	\end{equation*}
	A similar argument holds for $u < 0$, proving that $\mu|_{A \cup B}$ has no \swwgap-modes.
	Together, these properties disprove \rMP\ for all meaningful $\swp$-modes (\cref{prop:all_modes_violate_MP}).
	\EndExampleText
\end{example}

Recall that, for any measure with an $\sws$-mode and an \OMexhaustive\ \ac{OM} functional, all $\swps$-modes are $\sws$-modes (\cref{cor:dichotomy_exists_s_CASIO_OMpM}), so \cref{example:ps_not_s_not_gw} cannot admit an \ac{OM} functional satisfying property $M$.
Thus, while the following example, which provides a \swpgs-mode that is not a \swgs-mode, would seem to be weaker than \cref{example:ps_not_s_not_gw}, it is indeed necessary to distinguish between mode types for measures with an \OMexhaustive\ \ac{OM} functional and an $\sws$-mode (\cref{fig:Hasse_dichotomy_2}(\subref{subfig:Hasse_dichotomy_exists_s_OMpM})).

The main idea, which will also be used in \cref{section:distinguishing_wgap--gwap}, is that, for mode types containing $\swg$ and $\swp$, $u \in X$ is a mode if $\mu(\ball{u_{n}}{r_{n}})$ dominates along some \as\ $(u_{n})_{n \in \bN}$ and some \ns\ $(r_{n})_{n \in \bN}$;
but this mass can be spread so that the centred ball mass $\mu(\ball{u}{r_{n}})$ is not dominant along any \ns\ $(r_{n})_{n \in \bN}$.
We exploit this, using a density composed of step functions, to construct a measure $\mu$ with an \OMexhaustive\ \ac{OM} functional $I \colon E \to \bR$ possessing a $\swpgs$-mode $u \notin E$.
While the \ac{OM} functional prevents oscillatory behaviour of the ball mass for points in $E$, it imposes no restriction on points outside of $E$, allowing us to distinguish $\swpgs$- and $\swgs$-modes in this setting.

\begin{example}[$\swpgs$ but not $\swgs$, even with an \sws-mode and an \OMexhaustive\ \ac{OM} functional]
	\label[example]{example:pgs_not_gs_under_s_mode_and_OM}
	Consider $\mu \in \prob{\bR}$ defined through the unnormalised \ac{pdf} shown in \Cref{fig:distinguishing_p}(\subref{subfig:distinguishing_pgs_and_gs}) and given by
	\begin{align*}
		\rho(x)
		\defeq
		\frac{\absval{x-c}^{-1 \mathbin{/} 2}}{2} \chi_{\Bigl[{-\tfrac{\absval{c}}{2}},\tfrac{\absval{c}}{2}\Bigr]}(x-c)
		&+
		\sum_{k \in \bN}
		4^{k} \chi_{[2^{-k}-R_{k}, 2^{-k}+R_{k}]} (x),\quad c \defeq -\tfrac{1}{8},\quad R_{k} \defeq 4^{-2k}.
	\end{align*}
	It is easy to see that $\rho$ is normalisable, i.e., that $Z \defeq \int_{\bR} \rho(x) \, \rd x < \infty$, and that $\mu(\Ball{c}{r}) = \tfrac{2}{Z} \sqrt{r}$ for $0 < r \leq \frac{\absval{c}}{2}$.
	Note that, for $S_{n} \defeq 4^{-2n-1}$ and $u_{n} \in [-\frac{1}{4}, \frac{1}{4}]$, we have
	\begin{align}
		\mu\bigl(\ball{u_{n}}{r}\bigr) &\leq \mu\bigl(\ball{2^{-n}}{r}\bigr) = \tfrac{2}{Z} \cdot 4^{n}r &&\text{for $S_{n} \leq r \leq R_{n}$,} \label{eq:pgs_but_not_gw_1} \\
		\mu\bigl(\ball{u_{n}}{r}\bigr) &\leq \mu\bigl(\ball{2^{-n-1}}{r}\bigr) = \tfrac{2}{Z} \cdot 4^{n+1}R_{n+1} = \tfrac{2}{Z} \cdot 4^{-n-1} &&\text{for $R_{n+1} < r \leq S_{n}$.} \label{eq:pgs_but_not_gw_2}
	\end{align}

	\proofpartorcase{$c$ is an \sws-mode.}
	This follows as the previous facts imply $\mu(\ball{c}{r}) = \sMass{\mu}{r}$ for $r \in (0, \tfrac{\absval{c}}{2})$.

	\proofpartorcase{$u = 0$ is a \swpgs-mode.}
	This follows because, for the \ns\ $R_{n} \defeq 4^{-2n}$ and \as\ $u_{n} \defeq 2^{-n}$, we have
	\[
		\mu(\Ball{u_{n}}{R_{n}})
		=
		\tfrac{2}{Z} \cdot 4^{-n}
		=
		\tfrac{2}{Z} \cdot \sqrt{R_{n}}
		=
		\mu(\Ball{c}{R_{n}})
		=
		\sMass{\mu}{R_{n}} \text{~~for all sufficiently large $n$.}
	\]

	\proofpartorcase{$u = 0$ is not a \swgs-mode.}
	Take the \ns\ $S_{n} \defeq 4^{-2n-1}$ and let $(u_{n})_{n \in \bN}$ be any \as\ of $u = 0$, where we assume without loss of generality that $u_{n} \in [-\frac{1}{4}, \frac{1}{4}]$ for all $n \in \bN$.
	Then \eqref{eq:pgs_but_not_gw_1} and \eqref{eq:pgs_but_not_gw_2} imply that $u = 0$ is not a $\swgs$-mode, because
	\[
		\mu\bigl(\Ball{u_{n}}{S_{n}}\bigr) \leq \tfrac{2}{Z} \cdot 4^{-n-1} = \tfrac{1}{Z} \sqrt{S_{n}} = \tfrac{1}{2} \mu(\ball{c}{S_{n}}) \text{~~for all sufficiently large $n$.}
	\]

	\proofpartorcase{\OMExhaustive\ \ac{OM} functional.}
	An \ac{OM} functional $I \colon E \to \bR$ for $\mu$ on $E = \{ {c} \}$ is trivially given by $I(c) \defeq 0$.
	To prove property $M(\mu, E)$ it is sufficient, by 	\eqref{eq:mass_near_Lebesgue_point}, to compare $c \in E$ to the point $0$, for which
	$\mu(\Ball{0}{r})
	\leq
	\tfrac{8}{3Z}\, 4^{-K} \leq \tfrac{32}{3Z} r^{2}$ and $\mu(\ball{c}{r}) = \tfrac{2}{Z} \sqrt{r}$
	whenever $r \leq 2^{-K}$.
	\EndExampleText
\end{example}

\begin{figure}[t!]
	\centering
	\begin{subfigure}[t]{0.49\textwidth}
		\centering
		\begin{tikzpicture}[scale=0.65]
			\begin{axis}[
				axis lines=middle,
				xlabel={$x$},
				ylabel={$\rho(x)$},
				xtick={-0.125,0, 0.0625, 0.125, 0.25, 0.5},
				xticklabels={$c$,{},$2^{-4}$, $2^{-3}$,$2^{-2}$,$2^{-1}$},
				ytick={16,64,256},
				xmin=-0.13, xmax=0.515,
				ymin=23, ymax=250,
				samples=1024,
				domain=-0.5:3.0,
				restrict y to domain=0:256,
				clip=false,
				enlargelimits,
				xscale=1.5,
				xlabel style={at={(axis cs:0.6,0.0)}, anchor=west},
				ylabel style={at={(axis cs:0.0,270)}, anchor=south}
			]
				\draw[blue, thick] (axis cs:-0.2, 0) -- (axis cs:-0.1875, 0) -- (axis cs:-0.1875, 2);
				\addplot[domain=-0.1875:-0.1251, thick, blue] {0.5 * (-(x+0.125))^(-1/2)};
				\addplot[domain=-0.1249:-0.0625, thick, blue] {0.5 * (x+0.125)^(-1/2)};
				\draw[blue, thick] (axis cs:-0.125, 32) -- (axis cs:-0.125, 270);
				\draw[blue, thick] (axis cs:-0.0625, 2) -- (axis cs:-0.0625, 0) -- (axis cs:0, 0);

				\draw[blue, thick] (axis cs:0.570000, 0) -- (axis cs:0.562500, 0);
				\draw[blue, thick] (axis cs:0.562500, 0) -- (axis cs:0.562500, 4) -- (axis cs:0.437500, 4) -- (axis cs:0.437500, 0);
				\draw[blue, thick] (axis cs:0.437500, 0) -- (axis cs:0.253906, 0);
				\draw[blue, thick] (axis cs:0.253906, 0) -- (axis cs:0.253906, 16) -- (axis cs:0.246094, 16) -- (axis cs:0.246094, 0);
				\draw[blue, thick] (axis cs:0.246094, 0) -- (axis cs:0.125244, 0);
				\draw[blue, thick] (axis cs:0.125244, 0) -- (axis cs:0.125244, 64) -- (axis cs:0.124756, 64) -- (axis cs:0.124756, 0);
				\draw[blue, thick] (axis cs:0.124756, 0) -- (axis cs:0.062515, 0);
				\draw[blue, thick] (axis cs:0.062515, 0) -- (axis cs:0.062515, 256) -- (axis cs:0.062485, 256) -- (axis cs:0.062485, 0);
				\draw[blue, thick] (axis cs:0.062485, 0) -- (axis cs:0.031251, 0);
				\draw[blue, thick] (axis cs:0.031251, 0) -- (axis cs:0.031251, 270) -- (axis cs:0.031249, 270) -- (axis cs:0.031249, 0);
				\draw[blue, thick] (axis cs:0.031249, 0) -- (axis cs:0, 0);
			\end{axis}
		\end{tikzpicture}
		\subcaption{\raggedright Part of the unnormalised \ac{pdf} $\rho$ from \Cref{example:pgs_not_gs_under_s_mode_and_OM}, for which $u = 0$ is a \swpgs-mode but not a \swgs-mode, even in the presence of an \sws-mode at $c = -\frac{1}{8}$. To simplify the plot, only the first five of the countably many step functions around $u = 0$ are shown, and the heights are truncated.}
		\label{subfig:distinguishing_pgs_and_gs}
	\end{subfigure}
	\begin{subfigure}[t]{0.49\textwidth}
		\centering
		\begin{tikzpicture}[scale=0.65]
			\begin{axis}[
				axis lines=middle,
				xlabel={$\log r$},
				xtick={0, 1, 2, 3, 4, 5, 6},
				ytick={0},
				xmin=-0.2, xmax=6.5,
				ymin=0.6, ymax=1.4,
				samples=512,
				domain=0.0:6.5,
				restrict y to domain=0.5:1.5,
				clip=false,
				enlargelimits,
				xscale=1.5,
				xlabel style={at={(axis cs:7.2,0.525)}, anchor=west},
			]
				\fill[gray!10] (axis cs:3.926,0.535) -- (axis cs:3.926,1.4) -- (axis cs:5.497,1.4) -- (axis cs:5.497,0.535) -- cycle;
				\addplot[domain=0.0:6.283, very thick, accent3] {1 + 0.3 * sin(deg(x))};
				\addplot[domain=0.0:6.283, very thick, accent1] {1 + 0.3 * sin(deg(x - 2.094395))};
				\addplot[domain=0.0:6.283, very thick, accent2] {1 + 0.3 * sin(deg(x - 4.188790))};
				\draw[accent1, very thick] (axis cs:1.0,1.5) -- (axis cs:1.5,1.5);
				\node[anchor=west] at (axis cs:1.5,1.5) {$\Ratio{0}{2}{\mu}{r}$};
				\draw[accent2, very thick] (axis cs:3.0,1.5) -- (axis cs:3.5,1.5);
				\node[anchor=west] at (axis cs:3.5,1.5) {$\Ratio{0}{4}{\mu}{r}$};
				\draw[accent3, very thick] (axis cs:5.0,1.5) -- (axis cs:5.5,1.5);
				\node[anchor=west] at (axis cs:5.5,1.5) {$\Ratio{0}{6}{\mu}{r}$};
				\node[anchor=east] at (axis cs:-0.2,1.3) {$1 + \alpha$};
				\draw[dashed] (axis cs:-0.1,1.3) -- (axis cs:6.4,1.3);
				\node[anchor=east] at (axis cs:-0.2,1.0) {$1$};
				\draw[dashed] (axis cs:-0.1,1.0) -- (axis cs:6.4,1.0);
				\node[anchor=east] at (axis cs:-0.2,0.85) {$1 - \tfrac{\alpha}{2}$};
				\draw[dashed] (axis cs:-0.1,0.85) -- (axis cs:6.4,0.85);
				\node[anchor=east] at (axis cs:-0.2,0.7) {$1 - \alpha$};
				\draw[dashed] (axis cs:-0.1,0.7) -- (axis cs:6.4,0.7);
			\end{axis}
		\end{tikzpicture}
		\caption{\raggedright One period of the oscillating ball-mass ratios $\Ratio{0}{v}{r}{\mu}$, $v \in \{2, 4, 6\}$, for the measure $\mu$ of \Cref{example:wp_but_not_pw_or_pgs}, with $\log r \in [ \frac{5 \pi}{4}, \frac{7 \pi}{4} ]$ shaded.
		Any \ns\ $(r_{n})_{n \in \bN}$ with $\liminf_{n \to \infty} \Ratio{0}{v}{r_{n}}{\mu} \geq 1$ for $v = 2$ and $4$ has $\log r_{n} \in [ \frac{5 \pi}{4}, \frac{7 \pi}{4} ] + 2 \pi \bZ$ eventually, so $\liminf_{n \to \infty} \Ratio{0}{6}{r_{n}}{\mu} < 1 - \frac{\alpha}{2}$. This shows that $0$ is a \swwp-mode but neither a \swpw- nor a \swpgs-mode.}
		\label{subfig:wp_but_not_pw_or_pgs}
	\end{subfigure}
	\caption{\Cref{example:pgs_not_gs_under_s_mode_and_OM,example:wp_but_not_pw_or_pgs} serve as examples distinguishing among various $\swp$-modes in the lattice of mode types.}
	\label{fig:distinguishing_p}
\end{figure}

To distinguish $\swwp$-modes from $\swpw$- and $\swpgs$-modes, we give another example with singularities in the spirit of \cref{example:ps_not_s_not_gw}.
In that example, the singularities $\sigma_{\text{even}}$ and $\sigma_{\text{odd}}$ were specified through their ball mass centred at the origin along distinguished sequences $(2^{-2n})_{n \in \bN}$ and $(2^{-2n-1})_{n \in \bN}$, respectively.
To show that the modes we construct are \emph{not} $\swpw$- or $\swpgs$-modes, we require finer control of the ball mass along all null sequences;
we do this by specifying the ball mass  using a trigonometric function chosen to give the desired oscillatory behaviour.

\begin{example}[$\swwp$ but neither $\swpw$ nor $\swpgs$]
	\label[example]{example:wp_but_not_pw_or_pgs}

	Consider the \ac{pdf} $\sigma(x) \defeq \tfrac{1}{2 \sqrt{2}} \absval{x}^{-1 \mathbin{/} 2} \chi_{[-\frac{1}{2}, \frac{1}{2}]} (x)$ with $\int_{-r}^{r} \sigma(x) \, \rd x = \min ( 1, \sqrt{2 r} )$.
	For $\alpha \in (0, \tfrac{1}{3})$ and $\theta \in \bR$, let $\sigma_{\alpha, \theta} \colon \bR \to [0, \infty]$ be the unique symmetric \ac{pdf} satisfying
	\[
		\int_{-r}^{r} \sigma_{\alpha, \theta}(x) \, \rd x = \min \left( 1, \frac{\sqrt{2 r}}{1 + \alpha \sin ( \log r  - \theta )} \right) .
	\]

	Note that the constraints on $\alpha$ ensure that $\sigma_{\alpha, \theta}$ is a non-negative function supported within $[-1, 1]$.
	Now consider $\mu \in \prob{\bR}$ with \ac{pdf} $\rho$ given by
	\[
		\rho(x) \defeq \frac{1}{4} \sigma(x) + \frac{1}{4} \sigma_{\alpha, \frac{2 \pi}{3}} (x - 2) + \frac{1}{4} \sigma_{\alpha, \frac{4 \pi}{3}} (x - 4) + \frac{1}{4} \sigma_{\alpha, 0} (x - 6) .
	\]

	\proofpartorcase{$u = 0$, $2$, $4$, and $6$ are the only \swwgap-modes; optimality of the constant \as\ at \swwgap-modes.}
	The proof that these are the only \swwgap-modes of $\mu$, and that the constant \as\ is optimal at these points, is similar to the arguments used in \Cref{example:ps_not_s_not_gw}, and hence is omitted.
	Moreover, it is straightforward to verify that the constant \as\ is optimal away from these points.

	\proofpartorcase{$u = 0$ is a \swwp-mode.}
	The standard comparison between the mass near the singularity at $u = 0$ and the mass near any $v \in \bR \setminus \{ 0, 2, 4, 6 \}$ shows that $\liminf_{r \to 0} \Ratio{0}{v}{r}{\mu} = \infty$.
	By construction,
	\begin{equation}
		\label{eq:wp_but_not_pw_or_pgs_ratio}
		\Ratio{0}{v}{r}{\mu} = 1 + \alpha \sin \bigl( \log r  - \tfrac{2 v \pi}{3} \bigr) \text{~~for $v = 2$, $4$, and $6$.}
	\end{equation}
	For each such $v$, we may take $r_{n} \defeq \exp\bigl( \tfrac{2 v \pi}{3} - 4 n \pi\bigr)$, along which $\Ratio{0}{v}{r_{n}}{\mu} \equiv 1$.
	Thus, for any $v \in \bR \setminus \{ 0 \}$, there is an \ns\ $(r_{n})_{n \in \bN}$ along which $\liminf_{n \to \infty} \Ratio{0}{v}{r_{n}}{\mu} \geq 1$, showing that $u = 0$ is a \swwp-mode.

	\proofpartorcase{$u = 0$ is not a \swpw-mode.}
	We claim that there is no \ns\ $(r_{n})_{n \in \bN}$ along which $\liminf_{n \to \infty} \Ratio{0}{v}{r_{n}}{\mu} \geq 1$ for all $v \in \{ 2, 4, 6 \}$.
	To see this, suppose without loss of generality that $(r_{n})_{n \in \bN}$ is an \ns\ along which $\liminf_{n \to \infty} \Ratio{0}{v}{r_{n}}{\mu} \geq 1$ for $v = 2$ and $4$.
	From \eqref{eq:wp_but_not_pw_or_pgs_ratio}, this would imply that, for all large enough $n$, $\log r_{n} \in [\frac{5 \pi}{4}, \frac{7 \pi}{4}] + 2 \pi \bZ$;
	however, again from \eqref{eq:wp_but_not_pw_or_pgs_ratio}, for such $r_{n}$, $\Ratio{0}{6}{r_{n}}{\mu} < 1 - \frac{\alpha}{2}$, as shown in \Cref{fig:distinguishing_p}(\subref{subfig:wp_but_not_pw_or_pgs}).
	Thus, if $u = 0$ dominates $2$ and $4$ along some \ns, then it cannot dominate $6$;
	similar arguments apply to the other two cases.
	Hence, $u = 0$ cannot be a \swpw-mode.

	\proofpartorcase{$u = 0$ is not a \swpgs-mode.}
	Since the constant \as\ is optimal at $u = 0$, it is a \swpgs-mode if and only if it is a \swps-mode (\Cref{thm:modes_coincide_CASIO}).
	However, since we have just established that $u = 0$ cannot be a \swpw-mode, it certainly cannot be a \swps-mode, since every \swps-mode is necessarily a \swpw-mode.
	\EndExampleText
\end{example}

\subsection{Distinguishing generalised and non-generalised modes}
\label{sec:distinguishing_g_and_non_g}

The next examples illustrate the difference between generalised and non-generalised modes, even if an \sws-mode is present and the measure has an \OMexhaustive\ \ac{OM} functional.

\begin{example}[$\swgs$ but not $\swwp$, even with an \sws-mode and an \OMexhaustive\ \ac{OM} functional]
	\label[example]{example:gs_not_s_not_wp}

	Consider $\mu \defeq \Uniform[0, 1] \in \prob{\bR}$ with \ac{pdf} $\rho \defeq \chi_{[0, 1]}$.
	Let $E \defeq \supp(\mu) = [0, 1]$.

	\proofpartorcase{Each $u \in (0, 1)$ is an \sws-mode.}
	Let $u \in (0, 1)$ be arbitrary.
	Then, for all small enough $r > 0$, $\mu(\ball{u}{r}) = 2 r = \sMass{\mu}{r}$, showing that $u$ is an \sws-mode.

	\proofpartorcase{$u = 0$ and $u = 1$ are \swgs-modes but not \sws-modes.}
	Fix $u = 0$;
	the case $u = 1$ is similar.
	Since $\mu(\ball{u}{r}) = r$ for all small enough $r$, it follows that $\liminf_{r \to 0} \Ratio{u}{\sup}{r}{\mu} = \frac{1}{2} < 1$, and so $u$ is not an \sws-mode.
	However, given any \ns\ $(r_{n})_{n \in \bN}$, consider the \as\ $(u_{n})_{n \in \bN}$ given by $u_{n} \defeq r_{n}$.
	Eventually, $0 < r_{n} < \frac{1}{2}$ and $\ball{u_{n}}{r_{n}} = (0, 2 r_{n}) \subset [0, 1]$.
	Hence, $u$ is a \swgs-mode, since $\liminf_{n \to \infty} \Ratio{u_{n}}{\sup}{r_{n}}{\mu} = 1$.

	\proofpartorcase{$u = 0$ and $u = 1$ are not \swwp-modes.}
	
	Consider the \cp\ $v \defeq \tfrac{1}{2}$.
	Then, for $r < \tfrac{1}{2}$, $\mu(\ball{u}{r}) = r$ but $\mu(\ball{v}{r}) = 2 r$, and so, for any \ns\ $(r_{n})_{n \in \bN}$, $\liminf_{n \to \infty} \Ratio{u}{v}{r_{n}}{\mu} = \tfrac{1}{2} < 1$.

	\proofpartorcase{\OMExhaustive\ \ac{OM} functional.}
	Note that an \ac{OM} functional $I \colon E \to \bR$ for $\mu$ is given by
	\[
		I(x) \defeq
		\begin{cases}
			0, & \text{for $x \in (0, 1)$,} \\
			\log 2, & \text{for $x \in \{ 0, 1 \}$.}
		\end{cases}
	\]
	Property $M(\mu, E)$ clearly holds, since small enough balls centred outside $E$ have zero mass.
	\EndExampleText
\end{example}

\subsection{\texorpdfstring{Distinguishing \swwgap- and \swgwap-modes}{Distinguishing wgap- and gwap-modes}}
\label{section:distinguishing_wgap--gwap}

\begin{figure}[t!]
	\centering
	\begin{subfigure}[t]{0.63\textwidth}
		\centering
		\begin{tikzpicture}[scale=0.8]
			\begin{axis}[
				axis lines=middle,
				xlabel={$x$},
				ylabel={$\rho(x)$},
				xtick={0,0.046868,0.093636,0.185669,0.345703,0.75, 1, 1.046868,1.093636,1.185669,1.345703,1.75, 2, 2.046868,2.093636,2.185669,2.345703,2.75},
				xticklabels={{}, {}, {}, $m_{2}$, $m_{1}$, $m_{0}$, 1,{},{},{},{},{},2,{},{},{},{},{}},
				ytick={4,8,16,32,64,128,256},
				yticklabels={{},{},16,32,64,128,256},
				xmin=0.2, xmax=2.55,
				ymin=0, ymax=252,
				samples=1024,
				domain=-0.5:3.0,
				restrict y to domain=0.5:256,
				clip=false,
				enlargelimits,
				xscale=1.7,
				yscale=1.3,
				xlabel style={at={(axis cs:2.67,-15)}, anchor=west},
				ylabel style={at={(axis cs:0.0,275)}, anchor=south}
			]
				\draw[dashed,gray!50] (axis cs:0,4) -- (axis cs:2.75, 4);
				\draw[dashed,gray!50] (axis cs:0,8) -- (axis cs:2.75, 8);
				\draw[dashed,gray!50] (axis cs:0,16) -- (axis cs:2.75, 16);
				\draw[dashed,gray!50] (axis cs:0,32) -- (axis cs:2.75, 32);
				\draw[dashed,gray!50] (axis cs:0,64) -- (axis cs:2.75, 64);
				\draw[dashed,gray!50] (axis cs:0,128) -- (axis cs:2.75, 128);
				\draw[dashed,gray!50] (axis cs:0,256) -- (axis cs:2.75, 256);
	
				\draw[blue, thick] (axis cs:0.062485, 0.0) -- (axis cs:-0.05, 0.0);
				\draw[blue, thick] (axis cs:0.75, 0.0) -- (axis cs:1.0, 0.0) -- (axis cs:1.062485, 0.0);
				\draw[blue, thick] (axis cs:1.75, 0.0) -- (axis cs:2.0, 0.0) -- (axis cs:2.062485, 0.0);
				
				\draw[blue, thick] (axis cs:0.750000, 0.000000) -- (axis cs:0.750000, 0.895105) -- (axis cs:0.562500, 0.895105) -- (axis cs:0.562500, 4.000000) -- (axis cs:0.562500, 4.000000) -- (axis cs:0.437500, 4.000000) -- (axis cs:0.437500, 0.895105) -- (axis cs:0.345703, 0.895105);
				\draw[blue, thick] (axis cs:0.345703, 0.895105) -- (axis cs:0.345703, 0.410585) -- (axis cs:0.253906, 0.410585) -- (axis cs:0.253906, 16.000000) -- (axis cs:0.253906, 16.000000) -- (axis cs:0.246094, 16.000000) -- (axis cs:0.246094, 0.410585) -- (axis cs:0.185669, 0.410585);
				\draw[blue, thick] (axis cs:0.185669, 0.410585) -- (axis cs:0.185669, 0.170681) -- (axis cs:0.125244, 0.170681) -- (axis cs:0.125244, 64.000000) -- (axis cs:0.125244, 64.000000) -- (axis cs:0.124756, 64.000000) -- (axis cs:0.124756, 0.170681) -- (axis cs:0.093636, 0.170681);
				\draw[blue, thick] (axis cs:0.093636, 0.170681) -- (axis cs:0.093636, 0.083579) -- (axis cs:0.062515, 0.083579) -- (axis cs:0.062515, 256.000000) -- (axis cs:0.062515, 256.000000) -- (axis cs:0.062485, 256.000000) -- (axis cs:0.062485, 0.083579) -- (axis cs:0.046868, 0.083579);
				\draw[accent3, thick] (axis cs:1.750000, 0.000000) -- (axis cs:1.750000, 1.790210) -- (axis cs:1.562500, 1.790210) -- (axis cs:1.562500, 8.000000) -- (axis cs:1.562500, 8.000000) -- (axis cs:1.437500, 8.000000) -- (axis cs:1.437500, 1.790210) -- (axis cs:1.345703, 1.790210);
				\draw[blue, thick] (axis cs:1.345703, 1.790210) -- (axis cs:1.345703, 0.410585) -- (axis cs:1.253906, 0.410585) -- (axis cs:1.253906, 16.000000) -- (axis cs:1.253906, 16.000000) -- (axis cs:1.246094, 16.000000) -- (axis cs:1.246094, 0.410585) -- (axis cs:1.185669, 0.410585);
				\draw[accent3, thick] (axis cs:1.185669, 0.410585) -- (axis cs:1.185669, 0.341362) -- (axis cs:1.125244, 0.341362) -- (axis cs:1.125244, 128.000000) -- (axis cs:1.125244, 128.000000) -- (axis cs:1.124756, 128.000000) -- (axis cs:1.124756, 0.341362) -- (axis cs:1.093636, 0.341362);
				\draw[blue, thick] (axis cs:1.093636, 0.341362) -- (axis cs:1.093636, 0.083579) -- (axis cs:1.062515, 0.083579) -- (axis cs:1.062515, 256.000000) -- (axis cs:1.062515, 256.000000) -- (axis cs:1.062485, 256.000000) -- (axis cs:1.062485, 0.083579) -- (axis cs:1.046868, 0.083579);
				\draw[blue, thick] (axis cs:2.750000, 0.000000) -- (axis cs:2.750000, 0.895105) -- (axis cs:2.562500, 0.895105) -- (axis cs:2.562500, 4.000000) -- (axis cs:2.562500, 4.000000) -- (axis cs:2.437500, 4.000000) -- (axis cs:2.437500, 0.895105) -- (axis cs:2.345703, 0.895105);
				\draw[accent3, thick] (axis cs:2.345703, 0.895105) -- (axis cs:2.345703, 0.821171) -- (axis cs:2.253906, 0.821171) -- (axis cs:2.253906, 32.000000) -- (axis cs:2.253906, 32.000000) -- (axis cs:2.246094, 32.000000) -- (axis cs:2.246094, 0.821171) -- (axis cs:2.185669, 0.821171);
				\draw[blue, thick] (axis cs:2.185669, 0.821171) -- (axis cs:2.185669, 0.170681) -- (axis cs:2.125244, 0.170681) -- (axis cs:2.125244, 64.000000) -- (axis cs:2.125244, 64.000000) -- (axis cs:2.124756, 64.000000) -- (axis cs:2.124756, 0.170681) -- (axis cs:2.093636, 0.170681);
				\draw[accent3, thick] (axis cs:2.093636, 0.170681) -- (axis cs:2.093636, 0.167158) -- (axis cs:2.062515, 0.167158) -- (axis cs:2.062515, 270.000000) -- (axis cs:2.062515, 270.000000) -- (axis cs:2.062485, 270.000000) -- (axis cs:2.062485, 0.167158) -- (axis cs:2.046868, 0.167158);
			\end{axis}
		\end{tikzpicture}
		\subcaption{\raggedright Part of the unnormalised \ac{pdf} $\rho$ from \Cref{example:wgap_but_not_gwap}.  The density around $0$ is a sum of the functions $\rho_{k}$ (see (\subref{subfig:wgap--gwap_rho_k})), as in \Cref{example:pgs_not_gs_under_s_mode_and_OM}. The densities around $1$ and $2$ are defined similarly, with the $\rho_{k}$ of odd (respectively even) index, shown in orange, receiving double mass. The density $\rho$ also has a $\absval{\quark}^{-1 \mathbin{/} 2}$-singularity centred at $c = -\frac{1}{8}$ (not shown). Only the first four $\rho_{k}$ are shown, and the heights are truncated.}
		\label{subfig:wgap--gwap_density}
	\end{subfigure}
	\begin{subfigure}[t]{0.36\textwidth}
		\centering
		\begin{tikzpicture}[scale=0.8]
			\begin{axis}[
				axis lines=middle,
				xlabel={$x$},
				ylabel={$\rho_{k}(x)$},
				xtick={0, 0.345703, 0.4375, 0.5, 0.5625, 0.75},
				xticklabels = {{}, $m_{k}$, $a_{k}$, $2^{-k}$, $b_{k}$, $m_{k-1}$},
				ytick={0,0.895105,4},
				yticklabels = {{}, $\beta_{k}$, $4^{k}$},
				xmin=0.33, xmax=0.75,
				ymin=0.0, ymax=4.2,
				samples=512,
				domain=0.0:6.5,
				restrict y to domain=0.5:1.5,
				clip=false,
				enlargelimits,
				xscale=0.95,
				yscale=1.3,
				xlabel style={at={(axis cs:0.79,0.0)}, anchor=west},
				ylabel style={at={(axis cs:0.29,4.6)}, anchor=south}
			]
				\fill[gray!10] (axis cs:0.5625,4.40) -- (axis cs:0.5625,0.02) -- (axis cs:0.4375,0.02) -- (axis cs:0.4375,4.40) -- cycle;
				\node[black] at (axis cs:0.5, 4.21) {$A_{k}$};
				\fill[gray!30] (axis cs:0.5625,4.40) -- (axis cs:0.5625,0.02) -- (axis cs:0.75,0.02) -- (axis cs:0.75,4.40) -- cycle;
				\fill[gray!30] (axis cs:0.4375,4.40) -- (axis cs:0.4375,0.02) -- (axis cs:0.345703,0.02) -- (axis cs:0.345703,4.40) -- cycle;
				\node[black] at (axis cs:0.39, 4.21) {$B_{k}$};
				\node[black] at (axis cs:0.65, 4.21) {$B_{k}$};
				\node[black] at (axis cs:0.547, 4.73) {$C_{k}$};
				\draw [decorate,decoration={brace,amplitude=5pt}]
				  (axis cs:0.332,4.42) -- (axis cs:0.763,4.42);
				\draw[blue, thick] (axis cs:0.750000, 0.000000) -- (axis cs:0.750000, 0.895105) -- (axis cs:0.562500, 0.895105) -- (axis cs:0.562500, 4.000000) -- (axis cs:0.562500, 4.000000) -- (axis cs:0.437500, 4.000000) -- (axis cs:0.437500, 0.895105) -- (axis cs:0.345703, 0.895105) -- (axis cs:0.345703, 0.0);
			\end{axis}
		\end{tikzpicture}
		\caption{\raggedright Unnormalised density $\rho_{k}$, which takes value $4^{k}$ on $A_{k} = [a_{k}, b_{k}]$ (light gray), takes value $\beta_{k}$ on $B_{k} = (m_{k}, a_{k}) \cup (b_{k}, m_{k-1}]$ (dark gray), and is zero outside of $C_{k} = A_{k} \cup B_{k}$.}
		\label{subfig:wgap--gwap_rho_k}
	\end{subfigure}
	\caption{\Cref{example:wgap_but_not_gwap} has a \swwgap-mode at $u = 0$ that is not a \swgwap-mode, and serves as a distinguishing example in the lattices of mode types.}
	\label{fig:wgap_but_not_gwap}
\end{figure}

To distinguish \swwgap-modes from \swgwap-modes, we construct an absolutely continuous measure on $X = \bR$ with three important singularities at the points $0$, $1$, and $2$, and a further singularity at $c \defeq -\frac{1}{8}$ not relevant to the main argument, such that:
\begin{enumerate}[label=(\alph*)]
	\item
	\label{item:wgap--gwap_property_not_gwap}
	for any \as\ of $u = 0$, there is a \cpas\ of either $v = 1$ or $v = 2$ that dominates the \as\ along all \ns; but

	\item
	\label{item:wgap--gwap_property_wgap}
	for each \cp\ separately, there is a choice of \as\ and \ns\ that dominates all \cpas.
\end{enumerate}
We wish to do this in the presence of an \OMexhaustive\ \ac{OM} functional, so that the measure we construct can serve as a distinguishing example in \Cref{thm:modes_coincide_OMpM}/\Cref{fig:Hasse_in_special_cases}(\subref{subfig:Hasse_in_special_cases_OMpM}).
Our starting point is the unnormalised \ac{pdf} $\rho$ from \Cref{example:pgs_not_gs_under_s_mode_and_OM}, which took the form
\begin{equation}
	\label{eq:pgs_not_gs_density}
	\rho(x) \defeq \sigma(x - c)\chi_{\Bigl[-\tfrac{\absval{c}}{2}, \tfrac{\absval{c}}{2} \Bigr]}(x - c) + \sum_{k \in \bN} \rho_{k}(x),
\end{equation}
with $\sigma \defeq \tfrac{1}{2} \absval{\quark}^{-1/2}$, $\rho_{k} \defeq 4^{k} \chi_{[2^{-k} - R_{k}, 2^{-k} + R_{k}]}$, $R_{k} \defeq 4^{-2k}$, and $c \defeq -\tfrac{1}{8}$.
Recall that, in that example, an \OMexhaustive\ \ac{OM} functional can be defined on $E = \{c\}$;
in particular, the singularity at $c$ turns out to be an \sws-mode and the singularity at $0$ is merely a \swpgs-mode.

In \Cref{example:wgap_but_not_gwap} that follows, we will modify \eqref{eq:pgs_not_gs_density} by placing, at the points $1$ and $2$, weighted copies of the singularity centred at $0$, with weighting sequences $(w^{\text{odd}}_{k})_{k \in \bN}$ and $(w^{\text{even}}_{k})_{k \in \bN}$, giving the unnormalised \ac{pdf}
\begin{equation} \label{eq:wgap--gwap_density}
	\rho(x) \defeq \sigma(x - c)\chi_{\Bigl[-\tfrac{\absval{c}}{2}, \tfrac{\absval{c}}{2} \Bigr]}(x - v) + \sum_{k \in \bN} \rho_{k}(x) + w^{\text{odd}}_{k} \rho_{k}(x - 1) + w^{\text{even}}_{k} \rho_{k}(x - 2),
\end{equation}
for which we will take the specific choices of weights
\begin{equation} \label{eq:wgap--gwap_weights}
	w^\text{odd}_{k} \defeq \begin{cases}
		2, & \text{if $k$ is odd,}\\
		1, & \text{if $k$ is even,}
	\end{cases}\qquad\qquad
	w^\text{even}_{k} \defeq \begin{cases}
		2, & \text{if $k$ is even,}\\
		1, & \text{if $k$ is odd}.
	\end{cases}
\end{equation}
With some minor modifications, this density gives the properties we desire.
First, $u = 0$ satisfies property \ref{item:wgap--gwap_property_wgap}, making it a \swwgap-mode.
Second, owing to the increased weight of the step functions at $1$ and $2$, the point $v$ is not an \sws-mode, unlike \Cref{example:pgs_not_gs_under_s_mode_and_OM};
this is essential, as otherwise \Cref{thm:s--ps--w_dichotomy} would imply that all \swwgap-modes are \swgwap-modes.
Third, the weighting at $1$ and $2$ is designed to give \ref{item:wgap--gwap_property_not_gwap}, with the intuition that any \as\ can be translated from $u = 0$ to either $v = 1$ or $v = 2$ to give a \cpas\ that receives double the weight infinitely often.
Making this rigorous, particularly when $r$ is much larger than the width of the step function, is technical;
moreover our convention that $\nicefrac{0}{0} = 1$ causes issues in light of the gaps between the supports of the $\rho_{k}$.
To address this we modify the $\rho_{k}$ to fill the gaps between their supports:
taking
\begin{equation} \label{eq:wgap--gwap_notation_1}
	A_{k} \defeq [a_{k}, b_{k}],\quad  a_{k} \defeq 2^{-k} - R_{k},\quad b_{k} \defeq 2^{-k} + R_{k}, \quad R_{k} \defeq 4^{-2k}, \quad k \in \bN,
\end{equation}
we extend the support of $\rho_{k}$ to contain the interval
\begin{equation} \label{eq:wgap--gwap_notation_2}
	C_{k} \defeq (m_{k}, m_{k - 1}],\qquad m_{k} \defeq \tfrac{1}{2} \bigl(a_{k} + b_{k+1}\bigr),\qquad k \in \bN,\qquad m_{0} \defeq \tfrac{3}{4}
\end{equation}
by adding a uniform density of mass $4^{-k}$ to $B_{k} \defeq C_{k} \setminus A_{k}$.
Doing so gives us a measure $\mu \in \prob{\bR}$ satisfying both \ref{item:wgap--gwap_property_wgap} and \ref{item:wgap--gwap_property_not_gwap}, as we now describe.

\begin{example}[$\swwgap$ but not $\swgwap$, even with an \OMexhaustive\ \ac{OM} functional]
	\label[example]{example:wgap_but_not_gwap}
	Using the notation of \eqref{eq:wgap--gwap_notation_1}--\eqref{eq:wgap--gwap_notation_2},
	define $\mu \in \prob{\bR}$ through the unnormalised \ac{pdf} shown in \Cref{fig:wgap_but_not_gwap} and given by \eqref{eq:wgap--gwap_density}--\eqref{eq:wgap--gwap_weights}, with
	\begin{equation*}
		\sigma(x) \defeq \frac{1}{2} \absval{x}^{-1/2},\qquad \rho_{k}(x) \defeq 4^{k} \chi_{A_{k}} + \beta_{k} \chi_{B_{k}},\qquad \beta_{k} \defeq \frac{4^{-k}}{\Leb{1}(B_{k})},\qquad c \defeq -\frac{1}{8}.
	\end{equation*}

	It is easy to see that $\rho$ is normalisable, i.e., that $Z \defeq \int_{\bR} \rho(x) \,\rd x < \infty$.
	Moreover,
	\begin{equation} \label{eq:wgap--gwap_ball_masses}
		\mu(A_{k}) =  \tfrac{2}{Z} \cdot 4^{-k}, \quad
		\mu(B_{k}) = \tfrac{1}{Z} \cdot 4^{-k},\quad
		\mu(C_{k}) = \tfrac{3}{Z} \cdot 4^{-k},\quad
		\mu(\Ball{c}{r}) = \tfrac{1}{Z} \cdot \sqrt{r}
		\text{~~for $r \leq \tfrac{1}{2}$},
	\end{equation}
	and, for $k \geq 2$,
	a direct calculation reveals that $2^{-k} \leq \beta_{k} \leq 2^{-k+2}$.

	\proofpartorcase{$u = 0$ is a \swwgap-mode.}
    Considering first the \cp\ $v=-1$, for which the constant \cpas\ is optimal, choose the \as\ $u_{n} = 2^{-n}$ and the \ns\ $R_{n} = 4^{-2n}$ to obtain
    \[
	\liminf_{n \to \infty} \Ratio{u_{n}}{v_{n}}{R_{n}}{\mu}
    \geq
    \frac{2\cdot 4^{-2n} \cdot 4^{n}}{\sqrt{4^{-2n}}}
    =
    2
	\text{~~for any \cpas\ $(v_{n})_{n \in \bN} \to v$.}
    \]
    By \eqref{eq:mass_near_Lebesgue_point}, the same \as\ and \ns\ can be used for any \cp\ $v \notin \{ -1,0,1,2 \}$.
    For the \cp\ $v = 1$, choose the \as\ $u_{n} = 2^{-2n}$ and the \ns\ $r_{n} = 4^{-4n}$.
    We must show that, for any \cpas\ $v_{n} \to v$, $\liminf_{n \to \infty} \Ratio{u_{n}}{v_{n}}{r_{n}}{\mu} \geq 1$.
   	To see this we will show that $v_{n}^{\ast} = 1 + 2^{-2n}$ is an optimal \as\ of $v$ for this \ns.
	Note that $\mu(\Ball{v_{n}^{\ast}}{r_{n}}) = \tfrac{2}{Z} \cdot 4^{-2n}$ and let $(v_{n})_{n \in \bN}$ be another \as\ of $v=1$;
	without loss of generality, we assume $\ball{v_{n}}{r_{n}} \subseteq [\frac{3}{4}, \frac{3}{2}]$ for all $n \in \bN$.
	If $\ball{v_{n}}{r_{n}} \subseteq [\frac{3}{4}, 1 + m_{2n}]$, then
	\[
	\mu(\Ball{v_{n}}{r_{n}})
	\leq
	\frac{2}{Z} \sum_{k \geq 2n+1} \mu(C_{k})
	=
	\frac{2}{Z} \sum_{k \geq 2n+1} 3 \cdot 4^{-k}
	=
	\frac{2}{Z} \cdot 4^{-2n}
	=
	\mu(\Ball{v_{n}^{\ast}}{r_{n}}).
	\]
	Otherwise it must be the case that $1 + m_{2n} < v_{n} + r_{n}$, which implies that $\ball{v_{n}}{r_{n}} \subseteq [1 + m_{2n} - 2r_{n}, \frac{3}{2}]$.
	Since $1 + m_{2n} - 2r_{n} \geq 1 + R_{2n+1}$, the density $\rho$ takes value $\beta_{2n+1}$ on $\ball{v_{n}}{r_{n}} \cap C_{2n+1}$ and attains its maximum value, $4^{2n}$, within $\ball{v_{n}}{r_{n}} \cap C_{2n}$.
	Thus
    \[
	\mu(\Ball{v_{n}}{r_{n}})
    \leq
	\frac{2}{Z} \cdot 4^{-4n} \sup_{x \in [1+m_{2n}-2r_{n},\nicefrac{3}{2}]} \rho(x)
    =
	\frac{2}{Z} \cdot 4^{-4n} \cdot 4^{2n}
    =
    \mu(\Ball{v_{n}^{\ast}}{r_{n}}).
    \]

    Thus, the \as\ $v_{n}^{\ast}$ is optimal and, since $\Ratio{u_{n}}{v_{n}^{\ast}}{r_{n}}{\mu} = 1$, this proves the claim.
    A similar argument works for the \cp\ $v = 2$ by choosing the \as\ $u_{n} = 2^{-2n-1}$ and the \ns\ $r_{n} = 4^{-4n-2}$.

	\proofpartorcase{$u = 0$ is not a \swgwap-mode}
	To show that $u = 0$ is not a $\swgwap$-mode, we show that
	\begin{equation*}
		\forall \as~(u_{n})_{n \in \bN} \to u,~\exists \cp~v \in X,~\exists \cpas~(v_{n})_{n \in \bN} \to v,~\forall \ns~(r_{n})_{n \in \bN}:\quad \liminf_{n \to \infty} \Ratio{u_{n}}{v_{n}}{r_{n}}{\mu} < 1.
	\end{equation*}
	To this end, let the \as\ $(u_{n})_{n \in \bN} \to u$ be arbitrary.
	Either this sequence satisfies $u_{n} \leq 0$ for all $n$ sufficiently large, or else $u_{n} > 0$ infinitely often.

	\proofpartorcase{Case 1: $u_{n} \leq 0$ for all sufficiently large $n$.}
	Choose the \cp\ $v=1$ and the constant \cpas\ $v_{n} = 1$, let $(r_{n})_{n \in \bN}$ be an arbitrary \ns, and take $k_{n} \defeq \max \set{k \in \bN \text{ even}}{m_{k} \leq r_{n}}$.
	Then a direct calculation of the ball masses yields
	\begin{align*}
		\Ratio{u_{n}}{v_{n}}{r_{n}}{\mu} \leq \Ratio{0}{1}{r_{n}}{\mu}
		 &= \frac{\mu\bigl( [m_{k_{n}}, r_{n}) \bigr) + \sum_{k \geq k_{n} + 1} \mu(C_{k})}{\mu\bigl( [m_{k_{n}}, r_{n}) \bigr) + \sum_{k \geq k_{n} + 1} w^{\text{odd}}_{k} \mu(C_{k})} \\
		 &=  \frac{\mu\bigl( [m_{k_{n}}, r_{n}) \bigr) + \tfrac{1}{Z} \cdot 4^{-k_{n}}}{\mu\bigl( [m_{k_{n}}, r_{n}) \bigr) + \tfrac{9}{5Z} \cdot 4^{-k_{n}}} &&\text{(by \eqref{eq:wgap--gwap_ball_masses}).}
	\end{align*}
	But, noting that the choice of $k_{n}$ ensures that $0 \leq \mu\bigl([m_{k_{n}}, r_{n})\bigr) \leq \tfrac{1}{Z} \bigl(\mu(C_{k_{n}}) + \mu(C_{k_{n} -1})\bigr) = \tfrac{15}{Z} \cdot 4^{-k_{n}}$, we conclude from the previous inequality that
	\begin{equation*}
		\liminf_{n \to \infty} \Ratio{u_{n}}{v_{n}}{r_{n}}{\mu} \leq \max_{x \in [0, 15]}\frac{x + 1}{x + \tfrac{9}{5}} < 1.
	\end{equation*}

	\proofpartorcase{Case 2: $u_{n} > 0$ infinitely often.}
	Passing to a subsequence if necessary, we may assume $u_{n} \in (0, \frac{1}{4}]$ for all $n \in \bN$ and let $k_{n}$ be the unique index such that $u_{n} \in C_{k_{n}}$.
	Either infinitely many $k_{n}$ are odd, or infinitely many are even;
	we prove the odd case with the \cp\ $v = 1$, and note that the even case follows similarly with the \cp\ $v = 2$.

	Passing to a further subsequence if necessary, we may assume that all $k_{n}$ are odd.
	Then, choose the \cp\ $v=1$, the \cpas\ $v_{n} = v + u_{n}$, and let the \ns\ $(r_{n})_{n \in \bN}$ be arbitrary.
	Now let
	\begin{equation*}
		M^{\text{odd}}_{n} \defeq \mu\Biggl( \ball{u_{n}}{r_{n}} \cap \biguplus_{\text{$k$ odd}} C_{k} \Biggr),\qquad M^{\text{even}}_{n} \defeq \mu\Biggl( \ball{u_{n}}{r_{n}} \cap \biguplus_{\text{$k$ even}} C_{k} \Biggr),\qquad \gamma_{n} \defeq \frac{M^{\text{even}}_{n}}{M^{\text{odd}}_{n}}.
	\end{equation*}
	Since $u_{n} \in C_{k_{n}}$, with the index $k_{n}$ being odd, and $r_{n} > 0$, it is clear that $M^{\text{odd}}_{n} > 0$.
	To evaluate $\Ratio{u_{n}}{v_{n}}{r_{n}}{\mu}$, we apply a technical result, proven in \Cref{lemma:wgap_gwap_technical_step}, which shows that for centres $u_{n} \in C_{k_{n}}$ with $k_{n}$ odd, the ratio $\gamma_{n}$ cannot become too large, meaning that $M^{\text{odd}}_{n}$ never becomes negligible compared to $M^{\text{even}}_{n}$, regardless of the radius $r_{n}$.
	Precisely, $\gamma_{n} \leq 8$ for all $n \in \bN$, so
	\[
	\Ratio{u_{n}}{v_{n}}{r_{n}}{\mu}
	=
	\frac{M^{\text{odd}}_{n}  + M^{\text{even}}_{n}}{2 M^{\text{odd}}_{n} + M^{\text{even}}_{n}}
	=
	\frac{(1 + \gamma_{n}) M^{\text{odd}}_{n}}{(2 + \gamma_{n}) M^{\text{odd}}_{n}}
	\leq
	\frac{9}{10},
	\]
	proving that $\liminf_{n \to \infty} \Ratio{u_{n}}{v_{n}}{r_{n}}{\mu} < 1$, as required.

	\proofpartorcase{\OMExhaustive\ \ac{OM} functional.}
	Note that an \ac{OM} functional $I \colon E \to \bR$ for $\mu$ on $E = \{ -1 \}$ is trivially given by $I = 0$.
	To prove property $M(\mu, E)$ it is clearly sufficient, by 	\eqref{eq:mass_near_Lebesgue_point}, to compare $-1 \in E$ to the points $x = 0,1,2$.
	Since
	\[
	\mu(\Ball{x}{r})
	\leq
	\frac{2}{Z} \sum_{k\geq K} \mu(C_{k})
	\leq
	\frac{2}{Z} \sum_{k\geq K} 3\cdot 4^{-k}
	=
	\frac{2}{Z} \cdot 4^{-K+1}
	\quad
	\text{whenever
	$r \leq 2^{-K}$},
	\]
	we have, for sufficiently small $r>0$, $\mu(\Ball{x}{r}) \leq 2\cdot 4^{- \lfloor -\log_{2} r \rfloor +1} \leq 32 r^2$ and thereby
	\[
	\liminf_{r \to 0} \Ratio{x}{-1}{r}{\mu}
	\leq
	\liminf_{r \to 0} \frac{32 r^2}{\sqrt{r}}
	=
	0.
	\EndExampleEquation
	\]
\end{example}

\subsection{\texorpdfstring{Gaussian-dominated measures with distinct \sws-, \swps-, and \swe-modes}{Gaussian-dominated measures with distinct s-, ps-, and e-modes}}
\label{subsec:Gaussian_s_ps_w}

In support of \Cref{thm:Gaussian_prior_combined}\ref{item:Gaussian_prior_continuous_potential}, this section constructs a measure on $X = \ell^{2}(\bN; \bR)$ with a continuous density with respect to a Gaussian measure $\mu_{0}$, and having an \swe-mode (equivalently, \sww-mode) that is not a \swps-mode;
and another measure having a \swps-mode that is not an \sws-mode.
As discussed in \Cref{sec:Gaussian}, such measures must necessarily be defined on infinite-dimensional $X$, so we again work on the sequence space $X = \ell^{2}(\bN; \bR)$.

We will emulate \Cref{example:e_but_not_pgs}'s construction of an absolutely continuous measure $\mu \in \prob{\bR}$ with a $\sww$-mode (and, indeed, an $\swe$-mode) that is not a $\swpgs$-mode.
That $\mu$ has \ac{pdf}
\begin{equation}
	\label{eq:e_but_not_pgs_rho}
	\rho(x) = \sigma(x) + \sum_{n \in \bN} \tau_{n}(x - c_{n}),
\end{equation}
where $\sigma \propto \absval{\quark}^{-1/4} \chi_{\ball{0}{1 \mathbin{/} 2}}$, the densities $(\tau_{n})_{n \in \bN}$ are bounded with support contained in $\ball{0}{1 \mathbin{/} 2}$, and the centres $c_{n} = n \in \bN$ are chosen so that the supports are disjoint.
Since $\sigma$ is a singularity and each $\tau_{n}$ is bounded, $\mu(\ball{0}{r})$ dominates $\mu(\ball{c_{n}}{r})$ for every $n \in \bN$ in the limit $r \to 0$;
but the $\tau_{n}$ are chosen such that, for any fixed $r > 0$, the radius-$r$ ball mass at some centre $c_{n}$ strictly dominates that at zero.
In the following examples we give similar constructions, reweighting the following countable product measure with respect to the canonical basis $(e_{k})_{k \in \bN}$ of $X = \ell^{2}(\bN; \bR)$:
\begin{equation}
	\label{eq:Gaussian_product_measure}
	\mu_{0} = \bigotimes_{k \in \bN} N\bigl(0, k^{-2}\bigr) \in \prob{X}.
\end{equation}
The measure $\mu_{0}$ is a centred non-degenerate Gaussian on $X$ with positive, self-adjoint, trace-class covariance operator $C = \sum_{k \in \bN} k^{-2} (e_{k} \otimes e_{k})$ on $X$.
Writing $x_{k} \defeq \innerprod{x}{e_{k}}$, the Cameron--Martin space $E \subsetneq X$, Cameron--Martin norm $\norm{ \quark }_{E}$, and OM functional $I_{0} \colon E \to \bR$ for $\mu_{0}$ are
\begin{equation} \label{eq:Gaussian_CM_norm}
	E \defeq \bigset{x \in X}{\norm{x}_{E} < \infty},\qquad \norm{x}_{E}^{2} \defeq \bignorm{C^{-1 \mathbin{/} 2} x}_{X}^{2} = \sum_{k \in \bN} k^{2} x_{k}^{2},\qquad I_{0}(x) \defeq \frac{1}{2} \norm{x}_{E}^{2}.
\end{equation}
Instead of specifying a density with respect to $\mu_{0}$ \`a la \eqref{eq:e_but_not_pgs_rho}, we specify a probability measure $\mu \propto \exp(-\Phi) \mu_{0}$ with a continuous potential $\Phi$ with respect to $\mu_{0}$, as in \eqref{eq:density_and_potential}, of the form
\begin{equation} \label{eq:Gaussian_potential}
	\Phi(x) \defeq \sigma(x) + \sum_{n \in \bN} \tau_{n}\bigl(x - c_{n}\bigr) \zeta_{R_{n}}\bigl(\norm{x - c_{n}}_{X}\bigr),
\end{equation}
where $\zeta_{r}$ is a continuous cutoff function taking value one on $[0, R_{n}]$ and value zero outside of $[0, 2R_{n}]$, and with the functions $\sigma$ and $(\tau_{n})_{n \in \bN}$, the centres $(c_{n})_{n \in \bN}$, and the radii $(R_{n})_{n \in \bN}$ to be specified.
In contrast to \eqref{eq:e_but_not_pgs_rho}, the density of $\mu$ with respect to $\mu_{0}$ is greatest when $\Phi$ is negative, and $\mu$ has unit density when $\Phi = 0$.

The choice of $\sigma$ and $\tau_{n}$ is influenced by two factors not present in \cref{example:e_but_not_pgs}.
First, we must take care in translating the $\tau_{n}$ since $\mu_{0}$ is not uniform;
we will prescribe a form for $\tau_{n}$ allowing us to translate along the $e_{1}$-axis with no effect on the mass.
Second, we cannot take $\sigma$ to be a singularity as $\Phi$ must be continuous;
instead we take $\sigma$ to be a smoothed step function bounded below by $-\tfrac{1}{2}$,
and we exploit the infinite dimension of $X$ to construct $\tau_{n}$ which dominate $\sigma$ on balls of large enough radius $r$, but not in the limit $r \to 0$.

To allow $\tau_{n}$ to be translated along the $e_{1}$-axis, we note that, while there is no Lebesgue measure on $X$, the measure $\mu_{0}$ can be viewed as having a fictitious Lebesgue density in terms of the OM functional $I_{0}$, which can be written explicitly using \eqref{eq:Gaussian_CM_norm}:
\begin{equation}
	\label{eq:Gaussian_fictitious_density}
	\mu_{0}(\rd x) \mathrel{\text{``=''}}
	\exp\Bigl(-I_{0}(x) \Bigr) \,\rd x \defeq \exp\Biggl(-\frac{1}{2} \sum_{k = 1}^{\infty} k^{2} x_{k}^{2} \Biggr) \,\rd x.
\end{equation}
Since $I_{0}(x) = \infty$ for $\mu_{0}$-almost all $x \in X$, the identity \eqref{eq:Gaussian_fictitious_density} cannot be interpreted literally, but the calculations that we describe can be made rigorous using the Cameron--Martin theorem (see \cref{lem:w_but_not_s_JK_properties}).
Reweighting \eqref{eq:Gaussian_fictitious_density} by the potential $\tau_{n}(x) \defeq -x_{1}^{2}/2 - \tilde{\tau}_{n}(x_{2}, x_{3}, \dots)$ results in a measure that is `uniform' in the $e_{1}$-axis as it does not depend on $x_{1}$:
\begin{equation*}
	\exp\bigl(-\tau_{n}(x) \bigr) \,\mu_{0}(\rd x)
	\defeq \exp\Biggl( \tilde{\tau}_{n}\bigl(x_{2}, x_{3}, \dots\bigr) -\frac{1}{2} \sum_{k = 2}^{\infty} k^{2} x_{k}^{2} \Biggr) \,\rd x.
\end{equation*}
This allows us to construct elementary potentials $\tau_{n}$ centred at the origin, with the translation in \eqref{eq:Gaussian_potential} having no effect on the ball mass so long as $c_{n}$ lies on the $e_{1}$-axis.
It remains to choose $\tilde{\tau}_{n}$ such that $\tau_{n}$ dominates $\sigma$ at positive radius but not in the small-radius limit.
We do this by 
taking an increasing sequence $(K_{n})_{n \in \bN}$, to be specified shortly, and defining $\tilde{\tau}_{n}$ to be an approximation to $I_{0}$ truncated at $K_{n}$ terms:
\begin{equation*}
	\tilde{\tau}_{n}(x_{2}, x_{3}, \dots) \defeq \min\left(1, \frac{1}{2} \sum_{k = 2}^{K_{n}} k^{2} x_{k}^{2}\right).
\end{equation*}
This creates a tension between two properties of the potentials $\tau_{n}$ that we can exploit.
First, since each $\tau_{n}$ is continuous, and $\tau_{n}(0) = 0$, it behaves like the zero function on balls $\ball{0}{r}$ as $r \to 0$.
Second, as $X$ has infinite dimension, $\norm{x}_{E}^{2} = \infty$ for $\mu_{0}$-almost all $x \in X$, and so the superlevel sets
\begin{equation*}
	A_{n} \defeq \Set{x \in X}{\frac{1}{2} \sum_{k = 2}^{n} k^{2} x_{k}^{2} \geq 1}
\end{equation*}
satisfy $\mu_{0}(A_{n}) \nearrow 1$.
Thus, for fixed $r > 0$, we can choose $K_{n}$ such that $\tau_{n} \leq -1$ on $\ball{0}{r}$ with arbitrarily high $\mu_{0}$-probability.
In \cref{ex:Gaussian_swe_not_swps,ex:Gaussian_swps_not_sws} we will use two different strategies to choose the sequences $(K_{n})_{n \in \bN}$ and $(R_{n})_{n \in \bN}$ to achieve the desired behaviour:
\begin{enumerate}[label=(\alph*)]
	\item
	In \cref{ex:Gaussian_swe_not_swps}, we will choose $R_{n} = n^{-2}$, $n \in \bN$, so that $\mu$ is normalisable;
	then we will choose $(K_{n})_{n \in \bN}$ so that
	\begin{equation} \label{eq:Gaussian_swe_not_swps_K_n}
		r \in (R_{n+1}, R_{n}] \implies \tau_{n} \leq -1 \text{~with $\mu_{0}$-probability at least $1 - (3n)^{-1}$ on $\ball{0}{r}$.}
	\end{equation}
	Since $\sigma \geq -\tfrac{1}{2}$, this ensures that $0$ does not attain the supremal ball mass for any $r > 0$, so it is not a $\swps$-mode along any \ns.

	\item
		In \cref{ex:Gaussian_swps_not_sws} we choose $(R_{n})_{n \in \bN}$, $(K_{n})_{n \in \bN}$, and a further sequence $(r_{n})_{n \in \bN}$ using an algorithm.
		Taking $R_{1} = 1$, we repeat the following steps for $n \in \bN$:
		choose $K_{n}$ such that
	\begin{align}
		r = R_{n} &\implies \tau_{n} \leq -1 \text{~with $\mu_{0}$-probability at least $1 - (3n)^{-1}$ on $\ball{0}{r}$.}  \label{eq:Gaussian_swps_not_sws_K_n}
	\end{align}
	Arguing similarly to the previous example, \eqref{eq:Gaussian_swps_not_sws_K_n} will guarantee that $0$ is not a $\swps$-mode along the \ns\ $(R_{n})_{n \in \bN}$.
	Then, using continuity of $\tau_{n}$, choose $r_{n} < R_{n}$ such that
	\begin{align}
		r \leq r_{n} &\implies \tau_{n} \geq -\tfrac{1}{2} \text{~on $\ball{0}{r}$.} \label{eq:Gaussian_swps_not_sws_r_n}
	\end{align}
	This ensures that $\mu(\ball{0}{r_{n}}) \geq \mu(\ball{c_{m}}{r_{n}})$, $m \leq n$.
	Finally, choose the truncation radius $R_{n + 1} < r_{n}$ small enough that $\mu(\ball{0}{r_{n}})$ dominates an upper bound on the ball masses $\mu(\ball{c_{m}}{r_{n}})$, $m > n$.
	Along with \eqref{eq:Gaussian_swps_not_sws_r_n}, this shows that $0$ is a $\swps$-mode along $(r_{n})_{n \in \bN}$.
\end{enumerate}

\begin{notation}
	\label[notation]{notation:Gaussian_s_ps_w}
	Equip $X = \ell^{2}(\bN; \bR)$ with its usual metric, the standard basis $(e_{n})_{n \in \bN}$, and the
	product measure $\mu_{0}$ defined in \eqref{eq:Gaussian_product_measure}.
	Define $\mu \in \prob{X}$ to be of the form \eqref{eq:density_and_potential} with
	potential \eqref{eq:Gaussian_potential}, taking $c_{n} \defeq 6n e_{1}$, with the sequences $(K_{n})_{n \in \bN} \nearrow +\infty$ and  $R_{n} \leq n^{-2} \searrow 0$ both yet to be specified, and with
	\begin{align}
		\sigma(x) &\defeq -\frac{1}{2} \zeta_{1}\bigl(\norm{x}_{X}\bigr),\\
		\tau_{n}(x) &\defeq -\frac{x_{1}^{2}}{2} - \min\left(1, \frac{1}{2} \sum_{k = 2}^{K_{n}} k^{2} x_{k}^{2} \right) \geq -I_{0}(x),\\
		\zeta_{r}(s)
		&\defeq
		\begin{cases}
			1, &\text{if } 0 \leq s \leq r, \\
			2 - s/r, &\text{if } r \leq s \leq 2r, \\
			0, &\text{otherwise.}
		\end{cases}
	\end{align}
\end{notation}

The centres $c_{n}$ are chosen such that balls of radius $r \leq R_{1}$ intersect at most one of the balls $\ball{0}{2}$ and $\ball{c_{n}}{2R_{n}}$, $n \in \bN$.

\begin{example}[$\swe$ but not $\swps$ for Gaussian-dominated measures with continuous potential]
	\label[example]{ex:Gaussian_swe_not_swps}

	We construct an instance of the measure $\mu$ defined in \cref{notation:Gaussian_s_ps_w} with $R_{n} = n^{-2}$ for $n \in \bN$, choosing $(K_{n})_{n \in \bN}$ such that \eqref{eq:Gaussian_swe_not_swps_K_n} holds.
	We do this by taking $K_{n}$ large enough that $\mu_{0}(A_{K_{n}}^{\complement} \cap \ball{0}{R_{n}}) \leq (3n)^{-1} \mu_{0}(\ball{0}{R_{n+1}})$, which is possible as $\mu_{0}(A_{K_{n}}) \to 1$ as $K_{n} \to \infty$.

	\proofpartorcase{$\mu$ is a probability measure.}
	Using the definition of $\Phi$ and the fact that $R_{n} \leq n^{-2}$, we see that
	\begin{align*}
		 \int_{\ball{c_{n}}{2R_{n}}} \exp(-\Phi(x)) \, \mu_{0}(\rd x)
		&=
		\int_{\ball{0}{2R_{n}}} \exp\Bigl(-\underset{\leq 0}{\underbrace{\tau_{n}(x - c_{n})}} \,\underset{\leq 1}{\underbrace{\zeta_{R_{n}}\bigl(\norm{x - c_{n}}_{X}\bigr)}} \Bigr) \,\mu_{0}(\rd x) \\
		&\leq
		\int_{\ball{0}{2R_{n}}} \exp \big( -\tau_{n}(x) \big) \, \mu_{0}(\rd x) \hspace{4em} \text{(\cref{lem:w_but_not_s_JK_properties})}
		\\
		&\leq
		\int_{[-2R_{n},2R_{n}] \times \bR^{\bN}} \exp \left( \frac{x_{1}^{2}}{2} + 1 \right) \, \mu_{0}(\rd x)
		\leq
		\frac{4}{\sqrt{2\pi}} \exp(1) n^{-2},
	\end{align*}
	where the last line follows from writing $\mu_{0}(\rd x) = \tfrac{1}{\sqrt{2\pi}} \exp(-x_{1}^{2}/2) \Leb{1}(\rd x_{1}) \otimes \tilde{\mu}_{0}(\rd \tilde{x})$ for some probability measure $\tilde{\mu}_{0}$, with $\tilde{x} = (x_{2}, x_{3}, \dots)$.
	From this, and the fact that $\Phi \geq -\tfrac{1}{2}$ outside of the balls $\ball{c_{n}}{2R_{n}}$, it follows that $Z$ is finite, since
	\begin{align*}
		Z
		&=
		\int_{X \setminus \bigcup_{n} \ball{c_{n}}{2R_{n}}} \exp\bigl(-\Phi(x)\bigr) \,\mu_{0}(\rd x) + \sum_{n \in \bN} \int_{\ball{c_{n}}{2R_{n}}} \exp\bigl(-\Phi(x)\bigr) \, \mu_{0}(\rd x)
		\\
		&\leq \exp\bigl(\tfrac{1}{2}\bigr)\mu_{0} \left( X \setminus \bigcup_{n} \ball{c_{n}}{2R_{n}} \right) +
		\frac{4}{\sqrt{2\pi}} \exp\bigl(1\bigr) \sum_{n \in \bN} n^{-2}
		<
		\infty.
	\end{align*}

	\proofpartorcase{The constant \as\ is optimal at all \swwgap-modes; \OMexhaustive\ \ac{OM} functional.}
	By \Cref{lemma:cts_pot_Gaussian_EAI_and_wgap}, the measure $\mu$ has \ac{OM} functional $I(h) = I_{0}(h) + \Phi(x)$ defined on the Cameron--Martin space $E$ of $\mu_{0}$, and property $M(\mu, E)$ holds; moreover the constant \as\ is optimal at all \swwgap-modes.

	\proofpartorcase{$u = 0$ is not a $\swps$-mode.}
	Let $n \in \bN$ and $R_{n + 1} < r \leq R_{n}$.
	Then
	\begin{align*}
		\mu(\ball{c_{n}}{r})
		& = \frac{1}{Z} \int_{\ball{c_{n}}{r}} \exp\Bigl(-\tau_{n} \bigl(x - c_{n}\bigr) \underset{=1}{\underbrace{\zeta_{R_{n}}\bigl(\norm{x - c_{n}}_{X}\bigr)}}\Bigr) \,\mu_{0}(\rd x) \\
		& = \frac{1}{Z} \int_{\ball{0}{r}} \exp\Bigl(-\tau_{n} \bigl(x\bigr) \Bigr) \,\mu_{0}(\rd x) &&\text{(\cref{lem:w_but_not_s_JK_properties})}\\
		& \geq \frac{1}{Z} \int_{\ball{0}{r} \cap A_{K_{n}}} \exp\bigl(1\bigr) \,\mu_{0}(\rd x) &&\text{($\tau_{n} \leq -1$ on $A_{K_{n}}$)} \\
		& \geq \frac{1}{Z} \exp\bigl(1\bigr) \bigl(1 - (3n)^{-1}\bigr) \mu_{0}(\ball{0}{r}).   &&\text{(by \eqref{eq:Gaussian_swe_not_swps_K_n})}
	\end{align*}
	Again using the definition of $\Phi$, and the fact that $\sigma \geq -\tfrac{1}{2}$, we obtain that, for $r \leq R_{n}$,
	\begin{equation*}
		\mu(\ball{0}{r}) = \frac{1}{Z} \int_{\ball{0}{r}} \exp\Bigl( -\sigma(x) \Bigr) \,\mu_{0}(\rd x)  \leq \frac{1}{Z} \exp\bigl(\tfrac{1}{2} \bigr)\mu_{0}(\ball{0}{r}).
	\end{equation*}
	Taking the ratio of these ball masses shows that $u$ is not a $\swps$-mode:
	for any $n \in \bN$,
	\begin{equation*}
		R_{n + 1} < r \leq R_{n} \implies
		\Ratio{0}{\sup}{r}{\mu} \leq \Ratio{0}{c_{n}}{r}{\mu} \leq \frac{\exp\bigl(-\frac{1}{2}\bigr)}{1 - 3^{-1}}<  1.
	\end{equation*}

	\proofpartorcase{$u = 0$ is an $\swe$-mode.}
	Minimisers of $I$ are precisely the \acp{w-mode} of $\mu$ (\Cref{prop:OM_minimisers_w-modes}), and in this setting $\sww$- and $\swe$-modes coincide.
	From the definition \eqref{eq:Gaussian_potential} it follows that $I \geq -\tfrac{1}{2}$, and it is easily verified that $I(0) = -\tfrac{1}{2}$.
	Thus, since $u = 0$ minimises $I$, it is an $\swe$-mode.
	\EndExampleText
\end{example}

\begin{example}[$\swps$ but not $\sws$ for Gaussian-dominated measures with continuous potential]
	\label[example]{ex:Gaussian_swps_not_sws}
	Again, take $\mu$ to be an instance of the measure from \cref{notation:Gaussian_s_ps_w}, with parameters chosen as follows.

	\proofpartorcase{Algorithm to choose $(K_{n})_{n \in \bN}$, $(R_{n})_{n \in \bN}$, and $(r_{n})_{n \in \bN}$.}
	We repeat the following steps, initialising with $n = 1$ and $R_{1} = 1$:

	\begin{enumerate}[label=(\alph*)]
		\item
			Pick $K_{n} \in \bN$ such that $\mu_{0}(A_{K_{n}} \cap \ball{0}{R_{n}}) \geq (1 - (3n)^{-1}) \mu_{0}(\ball{0}{R_{n}})$, so that \eqref{eq:Gaussian_swps_not_sws_K_n} holds.

		\item
			Using that $\tau_{n}$ is continuous and $\tau_{n}(0) = 0$, pick $r_{n} < R_{n}$ such that $\tau_{n} \geq -\tfrac{1}{2}$ on $\ball{0}{r_{n}}$;
		this ensures that \eqref{eq:Gaussian_swps_not_sws_r_n} holds.

		\item
			
			Pick $0 < R_{n + 1} < \min\bigl( r_{n} , (n + 1)^{-2}\bigr)$ small enough that $\mu(\ball{0}{r_{n}})$ dominates a crude upper bound on the ball mass $\mu(\ball{c_{m}}{r_{n}})$, $m > n$:
		\begin{equation}
			\label{eq:w_but_not_s_rn_definition}
			\mu(\ball{0}{r_{n}}) = \frac{1}{Z} \exp\bigl(\tfrac{1}{2}\bigr) \mu_{0}(\ball{0}{r_{n}}) \geq \frac{1}{Z} \exp(5) \mu_{0}(\ball{0}{2R_{n + 1}}) + \frac{1}{Z} \mu_{0}(\ball{0}{r_{n}}).
		\end{equation}
		This is possible because $r \mapsto \mu_{0}(\ball{0}{r})$ is continuous and strictly increasing;
		moreover the upper bound $R_{n} \leq n^{-2}$ ensures that $\mu$ is normalisable.
	\end{enumerate}

	\proofpartorcase{$\mu$ is a probability measure; the constant \as\ is optimal at all \swwgap-modes; \OMexhaustive\ \ac{OM} functional.}
	These claims all follow using the same argument as in \cref{ex:Gaussian_swe_not_swps}.

	\proofpartorcase{$u = 0$ is a $\swps$-mode along the \ns\ $(r_{n})_{n \in \bN}$, and thus an $\swe$-mode.}
	To see this, we show that $\rcdf{0}{r_{n}} \geq \rcdf{x}{r_{n}}$ for any $x \in X$.
	There are three cases to consider, owing to the separation properties of $\Phi$:
	$\ball{x}{r_{n}}$ intersects $\ball{0}{2}$; $\ball{x}{r_{n}}$ intersects $\ball{c_{m}}{2R_{m}}$ for some $m \in \bN$; or $\ball{x}{r}$ intersects no such ball.

	First, suppose that $\ball{x}{r_{n}}$ intersects $\ball{0}{2}$.
	Then, using the definition of $\mu$ and the Anderson inequality \eqref{eq:explicit_Anderson_Gaussian},
	\begin{equation*}
		\rcdf{x}{r_{n}} \leq  \frac{1}{Z} \exp\bigl(\tfrac{1}{2}\bigr) \mu_{0}(\ball{x}{r_{n}}) \leq \rcdf{0}{r_{n}}.
	\end{equation*}
	Second, suppose that $\ball{x}{r_{n}}$ intersects some ball $\ball{c_{m}}{2R_{m}}$, $m \in \bN$.
	If $m > n$, then, using the fact that $\Phi = 0$ on $\ball{x}{r_{n}} \setminus \ball{c_{m}}{2R_{m}}$,
	\begin{align*}
		\rcdf{x}{r_{n}} &\leq \frac{1}{Z} \int_{\ball{c_{m}}{2R_{m}}} \exp\Bigl(-\underset{\leq 0}{\underbrace{\tau_{m}(x)}} \, \underset{\leq 1}{\underbrace{\chi_{R_{m}}\bigl(\norm{x - c_{m}}_{X}\bigr)}} \Bigr) \,\mu_{0}(\rd x) + \frac{1}{Z} \mu_{0}(\ball{x}{r_{n}} \setminus \ball{c_{m}}{2R_{m}}) \\
		&\leq \frac{1}{Z} \int_{\ball{0}{2R_{m}}} \exp\bigl(-\tau_{m}(x)\bigr) \,\mu_{0}(\rd x) + \frac{1}{Z} \mu_{0}(\ball{x}{r_{n}} \setminus \ball{c_{m}}{2R_{m}}) \text{~~~~~~~~~(\cref{lem:w_but_not_s_JK_properties})}\\
		&\leq \frac{1}{Z} \exp(5) \mu_{0}(\ball{0}{2R_{m}}) + \frac{1}{Z} \mu_{0}(\ball{x}{r_{n}})  \text{~~~~~~~~~~~~~~~~~~~~~~~~~~($\tau_{m} \geq -5$ on $\ball{0}{2R_{m}}$)} \\
		&\leq \frac{1}{Z} \exp(5) \mu_{0}(\ball{0}{2R_{n+1}}) + \frac{1}{Z} \mu_{0}(\ball{0}{r_{n}}) \leq \rcdf{0}{r_{n}}. \text{~~~~~~~~~(by \eqref{eq:explicit_Anderson_Gaussian} and \eqref{eq:w_but_not_s_rn_definition})}
	\end{align*}
	On the other hand, if $m \leq n$, then
	\begin{align*}
		\rcdf{x}{r_{n}} &\leq \frac{1}{Z} \int_{\ball{x}{r_{n}}} \exp\Bigl(-\underset{\leq 0}{\underbrace{\tau_{m}(x)}}\, \underset{\leq 1}{\underbrace{\chi_{R_{m}}\bigl(\norm{x - c_{m}}_{X}\bigr) }}\Bigr) \,\mu_{0}(\rd x) \\
		&\leq \frac{1}{Z} \int_{\ball{0}{r_{n}}} \exp\bigl(-\tau_{m}(x) \bigr)\,\mu_{0}(\rd x) \text{~~~~~~~~~~~~~~~~~~~~~~~~~~~~~~~~~~~~~(\cref{lem:w_but_not_s_JK_properties})} \\
		&\leq \frac{1}{Z} \exp\bigl(\tfrac{1}{2} \bigr)\mu_{0}(\ball{0}{r_{n}}) \text{~~~~~~~~~~~~~~~~~~~~~~~~~~~~~~~~~~~~(as $\tau_{m} \geq -\tfrac{1}{2}$ on $\ball{0}{r_{m}}$)} \\
		&\leq \mu(\ball{0}{r_{n}}).
	\end{align*}
	Otherwise, it must be the case that $\Phi = 0$ on $\ball{x}{r_{n}}$, and so
	\begin{equation*}
		\rcdf{x}{r_{n}} = \frac{1}{Z} \mu_{0}(\ball{x}{r_{n}}) \leq \frac{1}{Z} \mu_{0}(\ball{0}{r_{n}}) \leq \rcdf{0}{r_{n}}.
	\end{equation*}
	Thus $\Ratio{u}{x}{r_{n}}{\mu} \geq 1$ for any $x \in X$ and $n \in \bN$, and so $u = 0$ is a $\swps$-mode along $(r_{n})_{n \in \bN}$.
	The fact that $u$ is also an $\swe$-mode follows from the implications discussed in \Cref{thm:Gaussian_prior_combined}\ref{item:Gaussian_prior_continuous_potential}.

	\proofpartorcase{$u = 0$ is not an $\sws$-mode.}
	This follows by an argument analogous to that of \cref{ex:Gaussian_swe_not_swps}, examining the \ns\ $(R_{n})_{n \in \bN}$.
	\EndExampleText
\end{example}

\begin{remark}
	\begin{enumerate}[label=(\alph*)]
		\item
		In \cref{ex:Gaussian_swe_not_swps,ex:Gaussian_swps_not_sws}, the \ac{OM} functional $I \defeq \Phi + I_{0}$ for $\mu$ fails to be $E$-coercive:
		there is no choice of $A \in \bR$ and $c > 0$ for which $A + I(u) \geq c \norm{u}_{E}^{2} \text{~for all $u \in E$.}$
		This follows from considering the sequence $c_{n} \defeq 6ne_{1}$ with $\norm{c_{n}}_{E} \to +\infty$ and $I(c_{n}) = 0$.
		Thus, while the measure $\mu$ does have a \sww-mode, the mild condition on $\Phi$ used in Proposition~4.1 of \citet{Lambley2023} to ensure the existence of \sww-modes does not apply.

		\item
		While the potential $\Phi$ is locally uniformly continuous and indeed Lipschitz on bounded sets, it satisfies neither \ref{item:Phi_uniformly_continuous} nor \ref{item:Phi_lower_bound} in \Cref{thm:Gaussian_prior_combined}.
		First, $\Phi$ is not uniformly continuous \ref{item:Phi_uniformly_continuous}, because the modulus of continuity of the functions $\tau_{n}$ diverges as $n \to \infty$.
		Second, $\Phi$ does not satisfy the lower bound \ref{item:Phi_lower_bound}, because, for the centres $c_{n} \defeq 6n e_{1}$,
		\begin{equation*}
			\Phi(c_{n}) = -\frac{(6n)^{2}}{2} = -\frac{\norm{c_{n}}_{X}^{2}}{2},
		\end{equation*}
		and so it cannot be the case that $\Phi(x) \geq K(\eta) - \eta \norm{x}_{X}^{2}$ for all $\eta > 0$.
	\end{enumerate}
\end{remark}

\subsection{A measure with no optimal approximating sequence}
\label{sec:example_no_optimal_as}

\begin{figure}[tb]
	\newcommand{\alphfig}{15}
	\newcommand{\Radiusfig}{2}
	\newcommand{\ScaleAlphfig}{1.9}
	\newcommand{\ScaleRadiusfig}{3}
	\newcommand{\axisrange}{18}
	\definecolor{accent_1}{HTML}{F6CF00}
	\definecolor{accent_2}{HTML}{56B4E9}
	\definecolor{accent_3}{HTML}{01503B}
	\definecolor{accent_4}{HTML}{D55E00}
	\begin{subfigure}[t]{0.49\textwidth}
		\begin{tikzpicture}[scale=0.95]
			\clip (-1,-1) rectangle + (8.5,8.5);
			\draw[->] (-\axisrange*1em,0) -- (\axisrange*1em + 0.2em,0);
			\node[right] at (\axisrange*1em-1em, -1em) {$x_{1}$}; 
			\draw[->] (0,-\axisrange*1em) -- (0,\axisrange*1em + 0.2em) node[above] {$x_{2}$}; 

			\fill[gray, fill opacity=0.2, even odd rule]
			(0,0) circle (\alphfig*1em + \Radiusfig*1em)
			(0,0) circle (\alphfig*1em - \Radiusfig*1em)
			(0,0) circle (\alphfig/\ScaleAlphfig*1em + \Radiusfig/\ScaleRadiusfig*1em)
			(0,0) circle (\alphfig/\ScaleAlphfig*1em - \Radiusfig/\ScaleRadiusfig*1em)
			(0,0) circle (\alphfig/\ScaleAlphfig^2*1em + \Radiusfig/\ScaleRadiusfig^2*1em)
			(0,0) circle (\alphfig/\ScaleAlphfig^2*1em - \Radiusfig/\ScaleRadiusfig^2*1em)
			(0,0) circle (\alphfig/\ScaleAlphfig^3*1em + \Radiusfig/\ScaleRadiusfig^3*1em)
			(0,0) circle (\alphfig/\ScaleAlphfig^3*1em - \Radiusfig/\ScaleRadiusfig^3*1em);

			\fill[fill=accent_3, fill opacity=0.7, draw=black, dashed, line width=0.1em] (0.4*\alphfig*1em,0.7*\alphfig*1em) circle (\Radiusfig*1em);
			\fill[fill=accent_3, fill opacity=0.7, draw=black, dashed, line width=0.1em] (0.08*\alphfig/\ScaleAlphfig*1em,0.87*\alphfig/\ScaleAlphfig*1em) circle (\Radiusfig/\ScaleRadiusfig*1em);

			\draw (\alphfig*1em,0) circle (\Radiusfig*1em);
			\draw (0,\alphfig*1em) circle (\Radiusfig/\ScaleRadiusfig*1em);
			\draw (0,\alphfig/\ScaleAlphfig*1em) circle (\Radiusfig/\ScaleRadiusfig*1em);

			\node at (\alphfig*1em, \Radiusfig*1em + 0.6em) {$B_{1,1}$};
			\node at (\Radiusfig/\ScaleRadiusfig*1em + 0.5em, \alphfig*1.07em) {$B_{1,2}$};
			\node at (\Radiusfig/\ScaleRadiusfig*1em + 0.5em, \alphfig/\ScaleAlphfig*1.14em) {$B_{2,2}$};

			\node [gray] at (0.85*\alphfig*1em, 0.85*\alphfig*1em) {$A_{1}$};
			\node [gray] at (0.85*\alphfig/\ScaleAlphfig*1em, 0.85*\alphfig/\ScaleAlphfig*1em) {\small $A_{2}$};
			\node [gray] at (0.85*\alphfig/\ScaleAlphfig^2*1em, 0.85*\alphfig/\ScaleAlphfig^2*1em) {\scriptsize $A_{3}$};
			\node [gray] at (0.9*\alphfig/\ScaleAlphfig^3*1em, 0.9*\alphfig/\ScaleAlphfig^3*1em) {\tiny $A_{4}$};

			\draw [line width = 0.20em, accent_1] (\alphfig*1em - \Radiusfig*1em, 0em) -- (\alphfig*1em + \Radiusfig*1em, 0em);
			\draw [line width = 0.20em, accent_2] (0em, \alphfig*1em - \Radiusfig/\ScaleRadiusfig*1em) -- (0em, \alphfig*1em + \Radiusfig/\ScaleRadiusfig*1em);
			\draw [line width = 0.20em, accent_2] (0em, \alphfig/\ScaleAlphfig*1em - \Radiusfig/\ScaleRadiusfig*1em) -- (0em, \alphfig/\ScaleAlphfig*1em + \Radiusfig/\ScaleRadiusfig*1em);

			\draw [thin] (\alphfig*1em, -0.2em) -- (\alphfig*1em, 0.2em);
			\draw [thin] (\alphfig/\ScaleAlphfig*1em, -0.2em) -- (\alphfig/\ScaleAlphfig*1em, 0.2em);
			\draw [thin] (\alphfig/\ScaleAlphfig^2*1em, -0.2em) -- (\alphfig/\ScaleAlphfig^2*1em, 0.2em);

			\node at (\alphfig*1em, 0.5em) {$\alpha_1$};
			\node at (\alphfig/\ScaleAlphfig*1em, 0.5em) {\small $\alpha_2$};
			\node at (\alphfig/\ScaleAlphfig^2*1em, 0.5em) {\scriptsize $\alpha_3$};

			\draw [decorate,decoration={brace,amplitude=5pt,mirror}] (\alphfig*1em - \Radiusfig*1em, 0em) -- (\alphfig*1em + \Radiusfig*1em, 0em) node [black,midway,yshift=-1.1em] {$2 R_1$};
			
			\draw [decorate,decoration={brace,amplitude=5pt}] (0em, \alphfig*1em - \Radiusfig/\ScaleRadiusfig*1em) -- (0em, \alphfig*1em + \Radiusfig/\ScaleRadiusfig*1em) node [black,left,xshift=-0.4em,yshift=-0.6em] {\scriptsize $2 R_2$};
			\draw [decorate,decoration={brace,amplitude=5pt}] (0em, \alphfig/\ScaleAlphfig*1em - \Radiusfig/\ScaleRadiusfig*1em) -- (0em, \alphfig/\ScaleAlphfig*1em + \Radiusfig/\ScaleRadiusfig*1em) node [black,left,xshift=-0.4em,yshift=-0.6em] {\scriptsize $2 R_2$};
		\end{tikzpicture}
		\caption{The support of $\mu$ (shown as thick lines) lies on the coordinate axes $\set{\lambda e_{j}}{\lambda \in \bR}$, $j \in \bN$, and is contained in the balls $B_{i,j}$ (solid black border).
		The balls $B_{i, j}$, $j \geq i$, lie in the annulus $A_{i}$ (in gray).
		Any ball $\ball{u}{r}$ with $r \leq R_{i}$ (green with dashed border) intersecting the annulus $A_{i}$ cannot intersect any other $A_{k}$, $k \neq i$, and can intersect at most one ball $B_{i, j}$.}
	\end{subfigure}
	\begin{subfigure}[t]{0.49\textwidth}
		\includegraphics[clip, trim=0cm 2cm 0cm 0cm, width=0.95\linewidth]{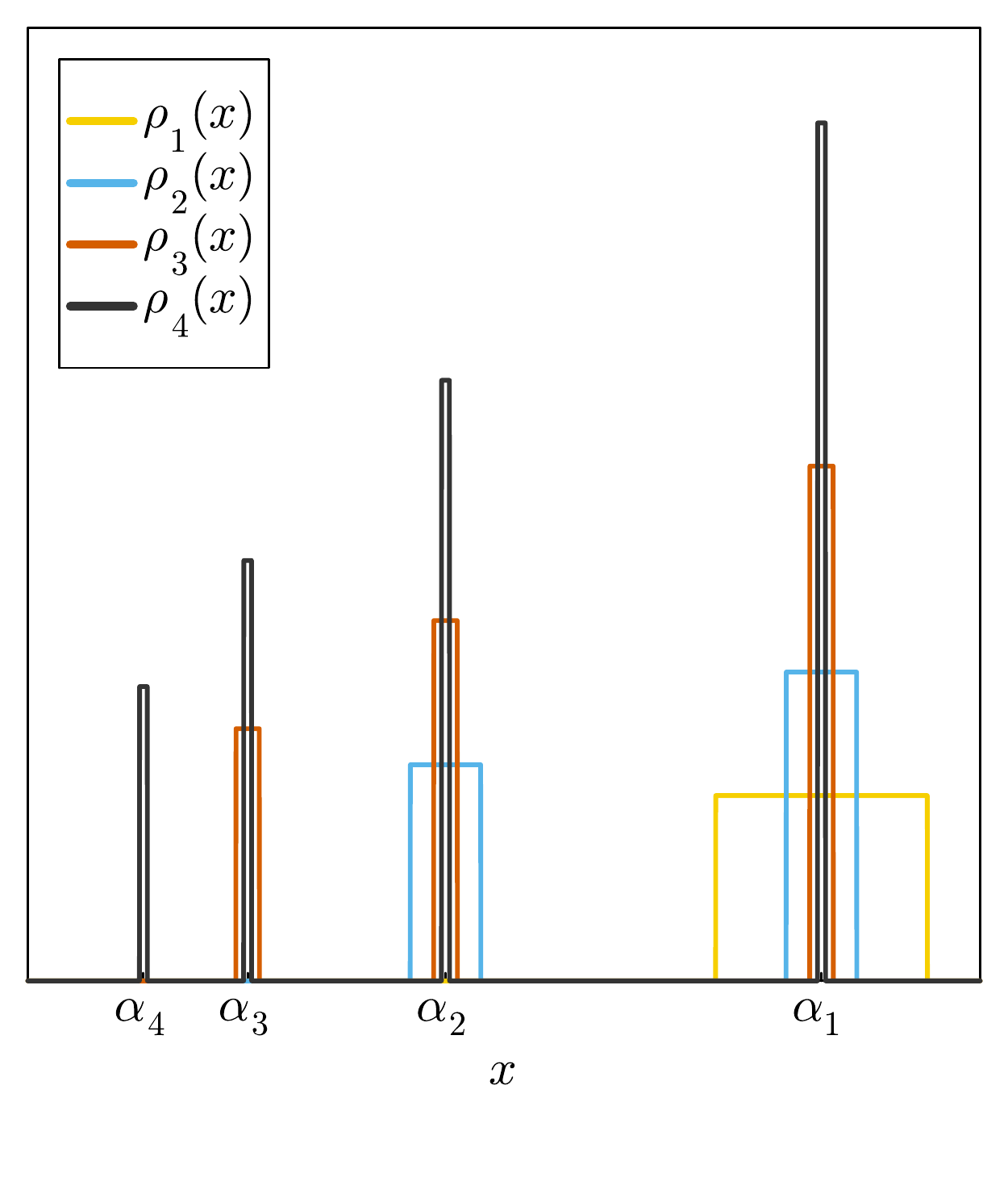}
		\caption{Unnormalised density $\rho_{j}$ of $\mu$ restricted to the coordinate axis ${\set{\lambda e_{j}}{\lambda > 0}}$ for $j = 1,\dots, 4$.
		Each density consists of $j$ step functions of equal width, and the height of the step function at $\alpha_{i} e_{j}$ is $\zeta^{i} \theta^{-j}$.
		Moving a ball centred at $\alpha_{i} e_{j}$ to $\alpha_{i - 1} e_{j}$ gains a factor $\zeta^{-1}$ of mass, and, when the ball radius is small, moving from $\alpha_{i} e_{j}$ to $\alpha_{i} e_{j + 1}$ gains a factor $\theta^{-1}$ of mass.}
	\end{subfigure}
	\label{fig:Example_with_annuli}
	\caption{Schematic diagram (not to scale) of the probability measure $\mu$ in \Cref{ex:no_optimal_approximating_sequence}.}
\end{figure}

To show that the mode types listed in parts \ref{item:violation_CP_d} and \ref{item:violation_CP_e} of \Cref{prop:mode_maps_violating_CP} violate \CP, it is necessary to construct a particular probability measure $\mu$ on $X = \ell^{2}(\bN; \bR)$ such that $u = 0$ is dominant, yet any \as\ of $u = 0$ can be strictly beaten by another.
This is, in some sense, the opposite situation to that considered in \Cref{sec:CASIO}:
not only is the constant \as\ not optimal, there is \emph{no} optimal \as!
Taking a convex combination of two copies of $\mu$ allows us to disprove \CP:
the copy with lesser weight can still dominate using the ability to strictly improve any \as.

Given the standard basis $(e_{n})_{n \in \bN}$ of $\ell^{2}(\bN; \bR)$ and a particular decreasing, positive null sequence $(\alpha_{n})_{n \in \bN}$,  our example involves placing weighted one-dimensional uniform measures of length $2R_{i}$ along the coordinate axes at the centres $\alpha_{i} e_{j}$ with $j \geq i$ and $i \in \bN$.
Each of these has support contained in the $\ell^{2}$-ball $B_{i,j} \defeq \ball{\alpha_{i}e_{j}}{R_{j}}$, and we organise them, through measures $\mu_{i}$, into annuli $A_{i}$, $i \in \bN$ around the origin:
\begin{equation*}
	A_{i}
	\defeq
	\ball{0}{\alpha_{i} + R_{i}} \setminus \ball{0}{\alpha_{i} - R_{i}}.
\end{equation*}
By definition, the mass of $\mu$ within each annulus $A_{i}$ is distributed across the balls $B_{i, j}$, $j \geq i$, and the annuli are separated by a positive distance.
The example is constructed so that balls centred in annuli further from the origin have greater mass;
this creates a tension for approximating sequences $(u_{n})_{n \in \bN}$ of $u = 0$, which must eventually converge to $u = 0$ but forgo mass by doing so.
As we will show shortly, it will be enough to understand the behaviour of the measure $\mu$ for balls centred at $\alpha_{i} e_{j}$, $j \geq i$.
For such balls, we emphasise the two essential properties of the construction.
\begin{enumerate}[label=(\alph*)]
	\item
	\label{item:moving_out_gives_zeta_mass}
	Moving the ball centre outwards from $\alpha_{i} e_{j}$ to $\alpha_{i-1} e_{j}$ increases the ball mass by at least a factor $\zeta^{-1} > 1$, regardless of the radius $r$.

	\item
	\label{item:moving_axis_gives_theta_mass}
	For balls of radius $r \leq R_{j+1}$, moving the ball centre from $\alpha_{i} e_{j}$ to $\alpha_{i} e_{j + 1}$ increases the ball mass by a factor $\theta^{-1} > 1$.
\end{enumerate}
Property \ref{item:moving_out_gives_zeta_mass} can be used to show that any \as\ $(u_{n})_{n \in \bN} \to 0$ can be strictly beaten by another sequence with a factor $\zeta^{-1}$ more mass.
Property \ref{item:moving_axis_gives_theta_mass} can be used to construct sequences $(u_{n})_{n \in \bN} \to 0$ with lots of mass asymptotically:
for any \ns\ $(r_{n})_{n \in \bN}$, there exists $(u_{n})_{n \in \bN} \to u = 0$ with $\orcdf{u_{n}}{r_{n}} \in \littleOmega(r_{n})$, whereas for $(v_{n})_{n \in \bN} \to v \neq 0$, it is always the case that $\orcdf{v_{n}}{r_{n}} \in \bigO(r_{n})$ as $n \to \infty$.

\begin{example}
	\label[example]{ex:no_optimal_approximating_sequence}
	Let $X = \ell^{2}(\bN; \bR)$ with its usual metric and standard basis $(e_{n})_{n \in \bN}$.
	Fix parameters $\theta \in (0, \nicefrac{1}{2}]$ and $\zeta \in (0, \nicefrac{\theta}{4}]$.
	Define for $j \in \bN$ the countable product measures
	\begin{align*}
		\sigma_{j}
		&\defeq
		\theta^{j}
		\Biggl(\bigotimes_{k = 1}^{j - 1} \dirac{0}\Biggr)
		\otimes
		\Uniform \left[ - R_{j}, R_{j} \right]
		\otimes
		\Biggl( \bigotimes_{k = j + 1}^{\infty} \dirac{0} \Biggr),
		&
		R_{j}
		&\defeq
		\tfrac{1}{2} \theta^{2j}.
	\end{align*}
	Moreover, for $i, j \in \bN$, let $B_{i,j} \defeq \ball{\alpha_{i} e_{j}}{R_{j}}$ and set
	\begin{align*}
		\mu_{i}
		&\defeq
		\sum_{j \geq i} \sigma_{j}(\quark - \alpha_{i} e_{j}),
		&
		\alpha_{i}
		&\defeq
		\frac{4}{1 - \theta^{2}}\theta^{2i},
		&
		Z
		&\defeq
		\frac{\zeta \theta}{(1 - \theta)(1 - \zeta \theta)}.
	\end{align*}
	Finally, define $\mu \in \prob{X}$ by
	\begin{equation*}
		\mu
		\defeq
		\frac{1}{Z} \sum_{i \geq 1} \zeta^{i} \mu_{i}
		.
	\end{equation*}

	\proofpartorcase{$\mu$ is a probability measure with finite normalising constant $Z$.}
	This follows because
	\begin{equation*}
		\mu(X)
		=
		\frac{1}{Z} \sum_{i \geq 1} \zeta^{i} \sum_{j \geq i} \underset{\defeq \theta^{j}}{\underbrace{\sigma_{j}(X)}}
		=
		\frac{1}{Z} \sum_{i \geq 1} \zeta^{i} \sum_{j \geq i} \theta^{j}
		=
		\frac{1}{Z(1-\theta)} \sum_{i \geq 1} \zeta^{i} \theta^{i}
		=
		\frac{\zeta\theta}{Z(1-\theta)(1-\zeta\theta)}
		=
		1.
	\end{equation*}

	\proofpartorcase{For any $u \in X$ and $r \leq R_{i}$,  $\ball{u}{r}$ intersects at most one annulus $A_{i}$ and one ball $B_{i,j}$, $j \geq i$.}
	The claim on annuli follows from the fact that $\dist{A_{i-1}}{A_{i}} > 2R_{i-1}$, which arises from an elementary calculation using the values of $\alpha_{j}$ and $R_{j}$.
	The fact on balls follows as the centres $\alpha_{i} e_{j}$ of $B_{i, j}$, $j \geq i$ are all a distance $\sqrt{2} \alpha_{i} > 3R_{i}$ apart, so any ball of radius $r \leq R_{i}$ can intersect at most one of the balls $B_{i, j}$.

	\proofpartorcase{For any \ns\ $(r_{n})_{n \in \bN}$ and any \as\ $(u_{n})_{n \in \bN} \to u \neq 0$, $\mu(\ball{u_{n}}{r_{n}}) \in \bigO(r_{n})$ as $n \to \infty$.}
	Either $u \notin \supp(\mu)$, in which case this is immediate, or else $u \in \closure{B_{i,j}}$ for some $j \geq i$.
	In this case, the separation properties of the previous part imply that $\ball{u_{n}}{r_{n}}$ intersects only $B_{i,j}$ for all large enough $n$.
	Thus $\mu(\ball{u_{n}}{r_{n}}) \leq \frac{2}{Z} \zeta^{i} \theta^{-j} r_{n} \in \bigO(r_{n})$ as $n \to \infty$.

	\proofpartorcase{For any \ns\ $(r_{n})_{n \in \bN}$, there exists an \as\ $(u_{n})_{n \in \bN} \to u = 0$ with $\mu(\ball{u_{n}}{r_{n}}) \in \littleOmega(r_{n})$ as $n \to \infty$.}
	Consider the following sequences, discarding the (at most finitely many) undefined terms:
	\begin{align*}
		j_{n}
		&\defeq
		\max \Set{k \in \bN}{r_{n} \leq R_{k}}
		\xrightarrow[n\to\infty]{} \infty,
		\\
		i_{n}
		&\defeq
		\max \Set{k \in \bN}{\zeta^{k} \geq \theta^{j_{n}/2},~j_{n} \geq k}
		\xrightarrow[n\to\infty]{} \infty,
		\\
		u_{n}
		&\defeq
		\alpha_{i_{n}} e_{j_{n}} \xrightarrow[n\to\infty]{} u=0.
	\end{align*}
	By construction, $\mu(\ball{\alpha_{i_{n}} e_{j_{n}}}{r_{n}})
	=
	\frac{2}{Z} \zeta^{i_{n}} \theta^{-j_{n}} r_{n}
	\geq
	\frac{2}{Z} \theta^{-j_{n}/2} r_{n} \in \littleOmega(r_{n}) \text{~as $n \to \infty$.}$

	\proofpartorcase{Any \as\ of $u = 0$ can be strictly improved.}
	Precisely, for any \ns\ $(r_{n})_{n \in \bN}$ and any \as\ $(u_{n})_{n \in \bN} \to u = 0$, there exists another \as\ $(u_{n}')_{n \in \bN} \to u = 0$ with $(u_{n}')_{n \in \bN} \subseteq \set{\alpha_{i} e_{j}}{j \geq i,~i \in \bN}$
	and with a factor $\zeta^{-1} > 1$ more mass:
	\begin{equation} \label{eq:gain_factor_zeta}
		\mu(\ball{u_{n}}{r_{n}}) \leq \zeta \mu(\ball{u_{n}'}{r_{n}}) \text{~for all large enough $n$.}
	\end{equation}
	This step is technical and proven in \cref{lem:wag_wga_properties}.
	The idea is that, for any \as\ $(u_{n})_{n \in \bN} \to 0$, one can find another \as\ $(\tilde{u}_{n})_{n \in \bN} \subset \set{\alpha_{i}e_{j}}{j \geq i \in \bN}$ with $\mu(\ball{u_{n}}{r_{n}}) \leq \mu(\ball{\tilde{u}_{n}}{r_{n}})$;
	then one can move each centre $\tilde{u}_{n} = \alpha_{i_{n}} e_{j_{n}}$ to $\alpha_{i_{n} - 1} e_{j_{n}}$ to gain a factor $\zeta^{-1}$ of mass, proving \eqref{eq:gain_factor_zeta}.
	\EndExampleText
\end{example}

\begin{remark}[On the necessity of working in non-locally-compact spaces]
	The ball-mass properties around $u = 0$ and $v \neq 0$ mean that, for any \ns\ $(r_{n})_{n \in \bN}$, any \cp\ $v \in X \setminus \{u\}$, and any \cpas\ $(v_{n})_{n \in \bN}$, there exists an \as\ $(u_{n})_{n \in \bN} \to u$ such that
	\begin{equation} \label{eq:no_optimal_approximating_sequence_dominant}
		\liminf_{n \to \infty} \Ratio{u_{n}}{v_{n}}{r_{n}}{\mu} \geq 1.
	\end{equation}
	But there is no optimal choice of \as\ for $u = 0$:
	for any \ns\ $(r_{n})_{n \in \bN}$ and \as\ $(u_{n})_{n \in \bN} \to u$, there exists a further \as\ $(u_{n}')_{n \in \bN} \to u$ such that
	\begin{equation} \label{eq:no_optimal_approximating_sequence_nonoptimal}
		\liminf_{n \to \infty} \Ratio{u_{n}}{u_{n}'}{r_{n}}{\mu} < 1.
	\end{equation}
	\Cref{lem:no_locally_compact_example} shows that there is no probability measure $\mu$ on a locally compact $X$ satisfying both properties.
	
	Thus, it is natural to consider measures on $X = \ell^{2}(\bN; \bR)$, which has infinite dimension and is therefore not locally compact, yet is still amenable to concrete computations.
\end{remark}

\section{Closing remarks}
\label{section:Conclusion}

Defining a mode for a probability measure without a continuous Lebesgue \ac{pdf}, for example in infinite-dimensional Banach spaces or more general metric spaces, is a problem of intrinsic interest in non-smooth analysis and has considerable relevance to applications.

This paper has taken a systematic and axiomatic approach to the definition of small-ball modes, and has shown that there are precisely ten small-ball mode maps that are meaningful in the sense of correctly handling purely atomic measures (\AP), those with continuous Lebesgue \ac{pdf} (\LP), and an elementary cloning property (\CP).
These ten mode maps include the strong mode of \citet{DashtiLawStuartVoss2013}, the weak mode of \citet{HelinBurger2015}, and the generalised strong mode of \citet{Clason2019GeneralizedMI}, along with seven new notions---among them an exotic mode strictly intermediate between the strong and weak modes.
We have completely classified the logical implications among these ten mode maps for general metric probability spaces.
We have also complemented the general theory with a significant, albeit not exhaustive, study of special cases involving well-behaved measures.
In particular, for measures on separable Banach spaces that are absolutely continuous with respect to a non-degenerate Gaussian measure and have a potential that is either uniformly continuous or quadratically lower bounded, all ten mode maps coincide, and the notion of mode becomes unambiguous (at least within the context of \Cref{defn:structured_definition}).

It is remarkable that the atomic property (\AP) and cloning property (\CP) serve to eliminate all but ten definitions of small-ball mode, and that these ten all satisfy the Lebesgue property (\LP) and, for separable metric spaces, the support property (\SP), yet all fail to satisfy the merging property (\rMP).
Therefore, from our axiomatic perspective, there appears to be no reason to privilege any one of the ten as \emph{the} definition of a mode.

We hope that this landscape will motivate the formulation of additional axioms---with either theoretical or applied motivation---that may single out some, or one, of the ten definitions as most appropriate in a well-defined context.
It may also inspire proposals that are not based on small-ball asymptotics, or that approach such asymptotics in fundamentally different ways from \Cref{defn:structured_definition}, potentially yielding mode maps that satisfy properties---such as the merging property (\rMP)---which all ten meaningful small-ball mode maps violate.

\appendix

\section{Proofs and supporting results}
\label[appendix]{appendix:supplementary_material_proofs}

\subsection{Supporting results for \texorpdfstring{\Cref{sec:small-ball_modes}}{Section \ref{sec:small-ball_modes}}}
\label[appendix]{section:Supporting_Technical_Results}

\begin{lemma}
	\label[lemma]{lemma:technical_lemma_for_SP}
	Let $(X, d, \mu)$ be a metric probability space and let $u \in \supp(\mu)$.
	For every null sequence $(\Delta_{n})_{n\in\bN}$ there exist null sequences $(r_{n})_{n\in\bN}$, $(\tilde{r}_{n})_{n\in\bN}$, and $(\varepsilon_{n})_{n\in\bN}$ such that
	\[
		\mu( \Ball{u}{r_{n}-\Delta_{n}} )
		\geq
		(1-\varepsilon_{n}) \, \mu( \Ball{u}{r_{n}} ),
		\qquad
		\mu( \Ball{u}{\tilde{r}_{n}+\Delta_{n}} )
		\leq
		(1+\varepsilon_{n}) \, \mu( \Ball{u}{\tilde{r}_{n}} ).
	\]
\end{lemma}

\begin{proof}
	We will first show that there exist null sequences $(s_{k})_{k\in\bN}$ and $(\tilde{s}_{k})_{k\in\bN}$ and a strictly increasing sequence $(n_{k})_{k\in\bN}$ such that, for each $n \geq n_{k}$,
	\begin{equation}
		\label{equ:technical_lemma_for_SP_first_step}
		\mu \bigl(  \Ball{u}{s_{k}} \setminus \Ball{u}{s_{k}-\Delta_{n}}  \bigr)
		\leq
		k^{-1} \mu( \Ball{u}{s_{k}} ),
		\quad
		\mu \bigl(  \Ball{u}{\tilde{s}_{k}+\Delta_{n}} \setminus \Ball{u}{\tilde{s}_{k}}  \bigr)
		\leq
		k^{-1} \mu( \Ball{u}{\tilde{s}_{k}} ).
	\end{equation}
	Since $\mu$ is a finite measure, this holds trivially for $k=1$ with $s_{1} = 1$ and $n_{1} = 1$ and sufficiently large $\tilde{s}_{1}>0$, while for $k\geq 2$ we can construct $s_{k}$ and $n_{k}$ recursively:
	Set $\hat{s} = s_{k-1}/2$, $\check{s} = s_{k-1}/4$, $Q = \mu( \Ball{u}{\hat{s}} )/\mu( \Ball{u}{\check{s}} ) \geq 1$, $J = \lceil{Qk}\rceil$ and $\delta = (\hat{s} - \check{s}) / J > 0$.
	Since the annuli $A_{j} \defeq \Ball{u}{\check{s} + j\delta} \setminus \Ball{u}{\check{s} + (j-1)\delta} \subseteq \Ball{u}{\hat{s}}$, $j=1,\dots,J$, are disjoint,
	\[
		\sum_{j=1}^{J} \mu(A_{j})
		\leq
		\mu(\Ball{u}{\hat{s}})
		=
		Q \, \mu(\Ball{u}{\check{s}}).
	\]
	Hence, there exists $j^{\ast} \in \{1,\dots,J\}$ such that
	$\mu(A_{j^{\ast}})
	\leq
	J^{-1} Q \, \mu(\Ball{u}{\check{s}})$.
	Choose $s_{k} \defeq \check{s} + j^{\ast} \delta$ and $\tilde{s}_{k} \defeq \check{s} + (j^{\ast}-1) \delta$.
	Further, since $u_{n} \to u$ as $n\to\infty$, there exists $n_{k} \geq n_{k-1}$ such that, for each $n\geq n_{k}$, $\Delta_{n} \leq \delta$.
	Since $J^{-1} Q \leq k^{-1}$ and $\Ball{u}{s_{k}} \setminus \Ball{u}{s_{k}-\Delta_{n}} \subseteq A_{j^{\ast}}$ as well as $\Ball{u}{\tilde{s}_{k}+\Delta_{n}} \setminus \Ball{u}{\tilde{s}_{k}} \subseteq A_{j^{\ast}}$ for $n\geq n_{k}$, this proves \eqref{equ:technical_lemma_for_SP_first_step}.
	(Note that $(s_{k})_{k\in\bN}$ and $(\tilde{s}_{k})_{k\in\bN}$ are null sequences since $s_{k} \leq s_{k-1}/2$, $\tilde{s}_{k} \leq s_{k-1}/2$.)

	Now, choosing $r_{n} = s_{k}$, $\tilde{r}_{n} = \tilde{s}_{k}$ and $\varepsilon_{n} = k^{-1}$ whenever $n_{k} \leq n < n_{k+1}$ for some (unique) $k\in\bN$, the above construction yields, for each $n \in \bN$,
	\begin{align*}
		\mu \bigl(  \Ball{u}{r_{n}} \setminus \Ball{u}{r_{n}-\Delta_{n}}  \bigr)
		& \leq
		\varepsilon_{n} \mu( \Ball{u}{r_{n}} )
		\text{ and} \\
		\mu \bigl(  \Ball{u}{\tilde{r}_{n}+\Delta_{n}} \setminus \Ball{u}{\tilde{r}_{n}}  \bigr)
		& \leq
		\varepsilon_{n} \mu( \Ball{u}{\tilde{r}_{n}} ).
	\end{align*}
	The claim then follows because
	\begin{align*}
		\Ball{u}{r_{n}}
		& =
		\Ball{u}{r_{n}-\Delta_{n}} \uplus \bigl( \Ball{u}{r_{n}} \setminus \Ball{u}{r_{n}-\Delta_{n}} \bigr)
		\text{ and} \\
		\Ball{u}{\tilde{r}_{n}+\Delta_{n}}
		& =
		\Ball{u}{\tilde{r}_{n}} \uplus \bigl( \Ball{u}{\tilde{r}_{n}+\Delta_{n}} \setminus \Ball{u}{\tilde{r}_{n}} \bigr). \qedhere
	\end{align*}
\end{proof}

\begin{lemma}[Continuity of ball masses]
	\label[lemma]{lem:ball_masses}
	Let $(X, d, \mu)$ be a metric probability space.
	\begin{enumerate}[label=(\alph*)]
		\item
		\label{item:left_continuity}
		For each $x \in X$ and $s > 0$, $\lim_{r \nearrow s} \mu(\Ball{x}{r}) = \mu(\Ball{x}{s})$ and $\lim_{r \nearrow s} \sMass{\mu}{r} = \sMass{\mu}{s}$.

		\item
		\label{item:lower_semicontinuity}
		For each $r>0$, the map $x \mapsto \mu(\Ball{x}{r})$ is lower semicontinuous on $X$, i.e., for any convergent sequence $(x_{k})_{k\in\bN} \to x$ in $X$, $\liminf_{k\to\infty} \mu(\Ball{x_{k}}{r}) \geq \mu(\Ball{x}{r})$.
	\end{enumerate}
\end{lemma}

\begin{proof}
	Since $\Ball{x}{r}$ is monotonically increasing in $r$ with $\Ball{x}{s} = \bigcup_{r < s} \Ball{x}{r}$, the first identity follows directly from the continuity of measures.
	Now assume that the second identity does not hold, i.e., there exists $\varepsilon>0$ such that, for each $r < s$, $\sMass{\mu}{r} \leq \sMass{\mu}{s} - \varepsilon$.
	By definition of the supremum, there exists $x_{\ast} \in X$ such that $\mu(\Ball{x_{\ast}}{s}) > \sMass{\mu}{s} - \varepsilon$.
	Hence, using the first identity, we obtain the contradiction
	\[
		\sMass{\mu}{s} - \varepsilon
		\geq
		\limsup_{r \nearrow s} \sMass{\mu}{r}
		\geq
		\lim_{r \nearrow s} \mu(\Ball{x_{\ast}}{r})
		=
		\mu(\Ball{x_{\ast}}{s})
		>
		\sMass{\mu}{s} - \varepsilon,
	\]
	and this proves claim \ref{item:left_continuity}.
	For claim \ref{item:lower_semicontinuity}, see \citet[SM1.1]{LambleySullivan2022}.
\end{proof}

\subsection{Supporting results for \texorpdfstring{\Cref{sec:Mode_definitions_coincide}}{Section \ref{sec:Mode_definitions_coincide}}}

\begin{lemma}[Optimality of the constant approximating sequence for absolutely continuous measures]
	\label[lemma]{lemma:CASIO_for_ac_in_Rm}
	Let $\mu \in \prob{\bR^{m}}$ have Lebesgue \ac{pdf} $\rho \colon \bR^{m} \to [0, \infty)$ and fix $u \in \bR^{m}$
	\begin{enumerate}[label=(\alph*)]
		\item
		\label{item:CASIO_at_local_symmetric_modes}
		If, for some $R > 0$, $\rho|_{\ball{u}{R}}$ is symmetric about $u$ and its superlevel sets are convex, then the constant \as\ of $u$ is optimal in the sense of \eqref{eq:CASIO}, i.e.
		\begin{equation}
			\tag{\ref{eq:CASIO}}
			\forall \ns \, (r_{n})_{n \in \bN} \, \forall \as \, (u_{n})_{n \in \bN} \to u\colon
			\quad
			\liminf_{n \to \infty} \Ratio{u}{u_{n}}{r_{n}}{\mu}
			\geq
			1.
		\end{equation}

		\item
		\label{item:CASIO_outside_support}
		If $u \notin \supp(\mu)$, then \eqref{eq:CASIO} holds.

		\item
		\label{item:CASIO_at_positive_continuity_points}
		If $\rho(u) > 0$ and $\rho$ is continuous at $u$, then \eqref{eq:CASIO} holds.

		\item
		\label{item:no_CASIO_on_boundary_of_support}
		If $u \in \supp(\mu)$ but $\rho(u) = 0$, then \eqref{eq:CASIO} may fail.
	\end{enumerate}
\end{lemma}

\begin{proof}
	Fix $u \in \bR^{m}$, let $(u_{n})_{n \in \bN}$ be an \as\ of $u$, and let $(r_{n})_{n \in \bN}$ be any \ns.
	\begin{enumerate}[label=(\alph*)]
		\item
		For all large enough $n$, $r_{n} < R$ and $\ball{u_{n}}{r_{n}} \subset \ball{u}{R}$.
		Then, for such $n$,
		\begin{align*}
			\mu(\ball{u_{n}}{r_{n}})
			= \int_{(u_{n} - u) + \ball{u}{r_{n}}} \rho(x) \, \rd x
			\leq \int_{\ball{u}{r_{n}}} \rho(x) \, \rd x
			= \mu(\ball{u}{r_{n}}) ,
		\end{align*}
		where the inequality follows from Anderson's theorem \citep[Theorem~1]{Anderson1955}.
		Hence, for all large enough $n$, $\Ratio{u}{u_{n}}{r_{n}}{\mu} \geq 1$, from which \eqref{eq:CASIO} follows immediately.

		\item
		Suppose that $u \notin \supp(\mu)$.
		Then, since $\supp(\mu)$ is closed, for all large enough $n$, $\ball{u}{r_{n}}$ and $\ball{u_{n}}{r_{n}}$ are both contained within $X \setminus \supp(\mu)$, and then \eqref{eq:CASIO} follows from the convention that $\nicefrac{0}{0} \defeq 1$.

		\item
		Let $\varpi \colon [0, \infty) \to [0, \infty)$ be a modulus of continuity for $\rho$ at $u$.
		Then
		\begin{align*}
			\Ratio{u}{u_{n}}{r_{n}}{\mu}
			= \frac{\int_{\ball{u}{r_{n}}} \rho(x) \, \rd x}{\int_{\ball{u_{n}}{r_{n}}} \rho(x) \, \rd x}
			\geq \frac{\rho(u) - \varpi(r_{n})}{\rho(u_{n})+ \varpi(r_{n})}
			\geq \frac{\rho(u) - \varpi(r_{n})}{\rho(u) + \varpi(\norm{u - u_{n}}) + \varpi(r_{n})} .
		\end{align*}
		Under the assumption that $\rho(u) > 0$, and using the fact that $\lim_{r \to 0} \varpi(r) = 0$, the right-hand side converges to $1$ as $n \to \infty$, which establishes \eqref{eq:CASIO}.

		\item
		Consider the case $m - 1$ and $\rho(x) \defeq x^{2}$ on a sufficiently small interval around $u = 0$, and the $\as$ $u_{n} \defeq \sqrt{r_{n}}$.
		In this situation,
		\begin{equation*}
			\Ratio{u}{u_{n}}{r_{n}}{\mu}
			= \frac{\int_{-r_{n}}^{+r_{n}} x^{2} \, \rd x}{\int_{\sqrt{r_{n}} - r_{n}}^{\sqrt{r_{n}} + r_{n}} x^{2} \, \rd x}
			= \frac{2 r_{n}^{3}}{(\sqrt{r_{n}} + r_{n})^{3} - (\sqrt{r_{n}} - r_{n})^{3}}
			= \frac{r_{n}}{r_{n} + 3}
			\to 0 \text{ as } n \to \infty .
			\qedhere
		\end{equation*}
	\end{enumerate}
\end{proof}

\begin{lemma}[Transfer of ball mass ratios]
	\label[lemma]{lemma:reweighted_ball_mass_ratio}
	Let $(X, d, \mu_{0})$ be a metric probability space.
	Let $\Phi \colon X \to \bR$ be measurable with $Z \defeq \int_{X} \exp(-\Phi(x)) \, \mu_{0}(\rd x)$ finite and define $\mu \in \prob{X}$ by $\mu (\rd x) \defeq Z^{-1} \exp(-\Phi(x)) \, \mu_{0} (\rd x)$.
	Suppose that $\Phi$ is continuous at $u, v \in X$ with modulus of continuity $\varpi$.
	Let $u_{n} \to u$, $v_{n} \to v$, and $r_{n} \to 0$.
	Then, for each $n \in \bN$,
	\begin{align}
		\label{eq:reweighted_ball_mass_ratio_1}
		\Ratio{u_{n}}{v_{n}}{r_{n}}{\mu} & \leq \exp \bigl( 2 \varpi(r_{n}) + \varpi(d(u_{n}, u)) + \varpi(d(v_{n}, v)) \bigr) \frac{\exp(-\Phi(u))}{\exp(-\Phi(v))} \Ratio{u_{n}}{v_{n}}{r_{n}}{\mu_{0}} , \\
		\label{eq:reweighted_ball_mass_ratio_2}
		\Ratio{u_{n}}{v_{n}}{r_{n}}{\mu} & \geq \exp \bigl( - 2 \varpi(r_{n}) - \varpi(d(u_{n}, u)) - \varpi(d(v_{n}, v)) \bigr) \frac{\exp(-\Phi(u))}{\exp(-\Phi(v))} \Ratio{u_{n}}{v_{n}}{r_{n}}{\mu_{0}} .
	\end{align}
	Furthermore, in the limit as $n \to \infty$,
	\begin{align}
		\label{eq:reweighted_ball_mass_ratio_3}
		\liminf_{n \to \infty} \Ratio{u_{n}}{v_{n}}{r_{n}}{\mu} = \frac{\exp(-\Phi(u))}{\exp(-\Phi(v))} \liminf_{n \to \infty} \Ratio{u_{n}}{v_{n}}{r_{n}}{\mu_{0}} ,
	\end{align}
	and similarly for the limit superior and the limit (if it exists).
\end{lemma}

\begin{proof}
	For \eqref{eq:reweighted_ball_mass_ratio_1}, observe that
	\begin{align*}
		\Ratio{u_{n}}{v_{n}}{r_{n}}{\mu}
		& = \frac{\int_{\ball{u_{n}}{r_{n}}} \exp(-\Phi(x)) \, \mu_{0}(\rd x)}{\int_{\ball{v_{n}}{r_{n}}} \exp(-\Phi(x)) \, \mu_{0}(\rd x)} \\
		& \leq \frac{\exp(- \Phi(u) + \varpi(r_{n}) + \varpi(d(u_{n}, u)))}{\exp(- \Phi(v) - \varpi(r_{n}) - \varpi(d(v_{n}, v)))} \frac{\mu(\ball{u_{n}}{r_{n}})}{\mu(\ball{v_{n}}{r_{n}})} \\
		& = \exp \bigl( 2 \varpi(r_{n}) + \varpi(d(u_{n}, u)) + \varpi(d(v_{n}, v)) \bigr) \frac{\exp(-\Phi(u))}{\exp(-\Phi(v))} \Ratio{u_{n}}{v_{n}}{r_{n}}{\mu_{0}} .
	\end{align*}
	The lower bound \eqref{eq:reweighted_ball_mass_ratio_2} is proved similarly.
	Now take the limit inferior of both sides of \eqref{eq:reweighted_ball_mass_ratio_1} and \eqref{eq:reweighted_ball_mass_ratio_2}, noting that $\lim_{t \to 0} \varpi(t) = 0$, to obtain \eqref{eq:reweighted_ball_mass_ratio_3}.
\end{proof}

\begin{lemma}[\ac{OM} and optimality properties of Gaussian-dominated measures]
	\label[lemma]{lemma:cts_pot_Gaussian_EAI_and_wgap}
	Let $\mu_{0}$ be a centred, non-degenerate Gaussian measure on a separable Banach space $X$ with Cameron--Martin space $E \subseteq X$ and \ac{OM} functional $I_{0}(h) \defeq \frac{1}{2} \norm{ h }_{E}^{2}$.
	Let $\Phi \colon X \to \bR$ be continuous with $Z \defeq \int_{X} \exp(-\Phi(u)) \, \mu_{0}(\rd u)$ finite and define $\mu \in \prob{X}$ by $\mu (\rd u) \defeq Z^{-1} \exp(-\Phi(u)) \, \mu_{0} (\rd u)$.
	\begin{enumerate}[label=(\alph*)]
		\item
		\label{item:cts_pot_Gaussian_OM_and_M}
		The function $I \defeq \Phi + I_{0} \colon E \to \bR$ is an \OMexhaustive\ \ac{OM} functional for $\mu$.

		\item
		\label{item:cts_pot_Gaussian_wgap}
		Every \swwgap-mode of $\mu$ lies in $E$.

		\item
		\label{item:cts_pot_Gaussian_CASIO}
		The constant \as\ is optimal in the sense of \eqref{eq:CASIO} at every $u \in E$.

		\item
		\label{item:cts_pot_Gaussian_finite_dimension}
		If $X$ has finite dimension, then \eqref{eq:CASIO} holds everywhere in $X = E$.
	\end{enumerate}
\end{lemma}

\begin{proof}
	Claim \ref{item:cts_pot_Gaussian_OM_and_M} is well known:
	see, e.g.\ \citet[Proposition~2.6]{Lambley2023}.

	For \ref{item:cts_pot_Gaussian_wgap}, let $u \in X \setminus E$.
	To see that $u$ is not a \swwgap-mode of $u$, we show that, for the \cp\ $v = 0 \in E$ and every \as\ $(u_{n})_{n \in \bN} \to u$, the constant \cpas\ $(v)_{n \in \bN} \to v$ is such that
	\[
		\liminf_{n \to \infty} \Ratio{u_{n}}{v}{r_{n}}{\mu} < 1 \text{~~for every \ns\ $(r_{n})_{n \in \bN}$}.
	\]
	Recall from \eqref{eq:Lambley2023Corollary3.3a} that, whenever $(u_{n})_{n \in \bN}$ converges in $X$ to a limit $u$ outside $E$, it follows that $\min_{h \in E \cap \cball{u}{r_{n}}} I_{0}(h) \to \infty$.
	Therefore, $u$ is not a $\swwgap$-mode because
	\begin{align*}
		\lim_{n \to \infty} \Ratio{u_{n}}{v}{r_{n}}{\mu}
		& = \frac{\exp(-\Phi(u))}{\exp(-\Phi(0))} \lim_{n \to \infty} \Ratio{u_{n}}{0}{r_{n}}{\mu_{0}} && \text{(by \eqref{eq:reweighted_ball_mass_ratio_3})} \\
		& \leq \frac{\exp(-\Phi(u))}{\exp(-\Phi(0))} \exp \biggl( - \lim_{n \to \infty} \min_{h \in E \cap \cball{u_{n}}{r_{n}}} I_{0}(h) \biggr) && \text{(by \eqref{eq:explicit_Anderson_Gaussian})} \\
		& = 0. && \text{(by \eqref{eq:Lambley2023Corollary3.3a})}
	\end{align*}

	For \ref{item:cts_pot_Gaussian_CASIO}, we first verify that, for $u \in E$ and $u_{n} \to u$ in $X$,
	\begin{equation}
		\label{eq:Gaussian_optimal_sequences_limit}
		\limsup_{n \to \infty} \Ratio{u_{n}}{0}{r_{n}}{\mu_{0}} \leq \exp\left(-\tfrac{1}{2} \norm{u}_{E}^{2} \right).
	\end{equation}
	To do this, take $h_{n} \in \argmin_{h \in E \cap \cball{u_{n}}{r_{n}}} \norm{h}_{E}^{2}$.
	If $(h_{n})_{n \in \bN}$ has no bounded sequence in $E$, then taking limits in \eqref{eq:explicit_Anderson_Gaussian} yields that $\lim_{n \to \infty} \Ratio{u_{n}}{0}{r_{n}}{\mu_{0}} = 0$, thus proving \eqref{eq:Gaussian_optimal_sequences_limit}.
	Otherwise, $(h_{n})_{n \in \bN}$ has a bounded subsequence, and thus we pass to an $E$-weakly convergent subsequence of $(h_{n})_{n \in \bN}$ and corresponding subsequence of $(r_{n})_{n \in \bN}$ without relabelling.
	As $E$ is compactly embedded in $X$, the $E$-weak limit of $(h_{n})_{n \in \bN}$ must be $u$, and thus by the weak lower semicontinuity of the $E$-norm, $\norm{u}_{E}^{2} \leq \liminf_{n \to \infty} \norm{h_{n}}^{2}_{E}$;
	applying this by taking limits in \eqref{eq:explicit_Anderson_Gaussian} proves \eqref{eq:Gaussian_optimal_sequences_limit}.

	We now see that the constant \as\ of $u$ is optimal in the sense of \eqref{eq:CASIO}, because
	\begin{align*}
		\liminf_{n \to \infty} \Ratio{u}{u_{n}}{r_{n}}{\mu}
		& = \frac{\exp(-\Phi(u))}{\exp(-\Phi(u))} \liminf_{n \to \infty} \Ratio{u}{u_{n}}{r_{n}}{\mu_{0}} && \text{(by \eqref{eq:reweighted_ball_mass_ratio_3})} \\
		& = 1 \cdot \lim_{n \to \infty} \Ratio{u}{0}{r_{n}}{\mu_{0}} \cdot \liminf_{n \to \infty} \Ratio{0}{u_{n}}{r_{n}}{\mu_{0}} \\
		& = \left. \exp(-I_{0}(u)) \middle / \limsup_{n \to \infty} \Ratio{u_{n}}{0}{r_{n}}{\mu_{0}} \right. && \text{($I_{0}$ is $\mu_{0}$'s OM functional)} \\
		& \geq 1 && \text{(by \eqref{eq:Gaussian_optimal_sequences_limit}).}
	\end{align*}

	Claim \ref{item:cts_pot_Gaussian_finite_dimension} follows from \Cref{lemma:CASIO_for_ac_in_Rm}, since the finite dimension of $X$, non-degeneracy of $\mu_{0}$, and continuity of $\Phi$ together imply that $\mu$ has a positive and continuous Lebesgue \ac{pdf}.
\end{proof}

\subsection{Supporting results for \texorpdfstring{\Cref{section:Examples}}{Section \protect\ref{section:Examples}}}

\begin{lemma}[Ratio of even and odd contributions to ball mass in \cref{example:wgap_but_not_gwap}]
	\label[lemma]{lemma:wgap_gwap_technical_step}

	Consider the measure $\mu$ from \cref{example:wgap_but_not_gwap} and the associated notation \eqref{eq:wgap--gwap_notation_1}--\eqref{eq:wgap--gwap_notation_2}.
	Suppose that $u \in C_{k}$ for some odd $k > 1$, and define, for all $S \in \Borel{\bR}$,
	\begin{equation*}
		M^{\text{odd}}(S) \defeq \mu\Biggl(S \cap \biguplus_{\text{$j$ odd}} C_{j}\Biggr),\quad
		M^{\text{even}}(S) \defeq \mu\Biggl(S \cap \biguplus_{\text{$j$ even}}  C_{j}\Biggr),\quad
		\gamma(S) \defeq \frac{M^{\text{odd}}(S)}{M^{\text{even}}(S)}.
	\end{equation*}
	Then, for all balls $\ball{u}{r} \subseteq [-\frac{1}{2}, \frac{1}{2}]$ with $r > 0$, we have  $\gamma(\ball{u}{r}) \leq 8$.
\end{lemma}

\begin{proof}
	Let $(a, b) \defeq \ball{u}{r}$.
	We will use a simple algorithm to find endpoints $a'$ and $b'$ such that
	\[
		\gamma\bigl(\ball{u}{r}\bigr) = \gamma\bigl((a, b)\bigr) \leq \gamma\bigl((a', b')\bigr) \leq 8,
	\]
	with the aim of choosing $a'$ and $b'$ such that $[a', b'] = C_{2 j_{1}} \cup \cdots \cup C_{2j_{2}} \text{~with $j_{1} < j_{2}$}$;
	in this case the claim $\gamma([a', b']) \leq 8$ readily holds.
	The idea, stated in Case~1, is to move $a' < a$ if $a \in \closure{C_{j}}$, $j$ even, and to move $a' > a$ if $a \in \interior{C_{j}}$, $j$ odd;
	similar steps are used for $b'$.
	This algorithm fails if $\ball{u}{r}$ does not contain the interior of any $C_{j}$, $j$ odd, because the endpoints it would select satisfy $M^{\text{odd}}((a', b')) = 0$, making $\gamma((a', b'))$ ill-defined;
	we handle this separately in Case~2.

	\proofpartorcase{Case 1: $\interior{C_{\ell}} \subseteq \ball{u}{r}$ for some odd $\ell \in \bN$.}
	Set
	\begin{equation*}
		a' = \begin{cases}
			\inf C_{k} = m_{k},  & \text{if $a \in \closure{C_{k}}$, $k$ even,} \\
			\sup C_{k} = m_{k-1}, & \text{if $a \in \interior{C_{k}}$, $k$ odd,}
		\end{cases}\qquad
		b' = \begin{cases}
			\sup C_{k} = m_{k - 1}, & \text{if $b \in \closure{C_{k}}$, $k$ even,} \\
			\inf C_{k} = m_{k}, & \text{if $b \in \interior{C_{k}}$, $k$ odd.}
		\end{cases}
	\end{equation*}
	These choices ensure that $0 < M^{\text{odd}}((a', b')) \leq M^{\text{odd}}((a, b))$ and $M^{\text{even}}((a', b')) \geq M^{\text{even}}((a, b))$;
	moreover,
	by construction, $(a', b') = C_{2 j_{1}}  \cup \cdots \cup C_{2 j_{2}}$ for some $j_{1},j_{2} \in \bN$ with $j_{2} > j_{1}$.
	Hence,
	\[
	\gamma\bigl((a, b) \bigr)
	\leq
	\gamma\bigl((a', b')\bigr)
	=
	\frac{\sum_{m=j_{1}}^{j_{2}} 3 \cdot 4^{-2m}}{\sum_{m=j_{1}}^{j_{2} - 1} 3 \cdot 4^{-(2m+1)}}
	=
	4 \cdot \frac{1 - 16^{-(j_{2}-j_{1}+1)}}{1 - 16^{-(j_{2}-j_{1})}}
	\leq
	4 \cdot \frac{16}{15}
	\leq
	8.
	\]

	\proofpartorcase{Case 2: $\ball{u}{r}$ does not contain the interior of any $C_{j}$, $j$ odd.}
	By construction, $u \in C_{k} = (m_{k-1}, m_{k}]$ but $\ball{u}{r}$ does not contain $\interior{C_{k}}$, so at least one of $a$ and $ b$ must lie in $C_{k}$.
	If both $a$ and $b$ lie in $C_{k}$, then $M^{\text{even}}((a, b)) = 0$, from which the result trivially follows, so two cases remain.

	\proofpartorcase{Case 2a: $a \in C_{k}$ but $b \notin C_{k}$.}
	Since $\ball{u}{r}$ does not contain the interior of any $C_{j}$, $j$ odd, we must have that $b < m_{k-3}$, i.e., $\ball{u}{r}$ intersects at most $C_{k}$, $C_{k-1}$, and $C_{k-2}$.
	Thus
	\begin{align}
		M^{\text{odd}}\bigl(\ball{u}{r}\bigr) &\geq M^{\text{odd}}\bigl(\ball{u}{r} \cap C_{k} \bigr) = 4^{k} \Leb{1}\bigl(\ball{u}{r} \cap A_{k}\bigr) + \beta_{k} \Leb{1}\bigl(\ball{u}{r} \cap B_{k}\bigr), \label{eq:wgap--gwap_technical_F} \\
		M^{\text{even}}\bigl(\ball{u}{r}\bigr) &= M^{\text{even}}\bigl(\ball{u}{r} \cap C_{k-1} \bigr) = 4^{k-1}\Leb{1}\bigl(\ball{u}{r} \cap A_{k-1}\bigr) +  \beta_{k-1} \Leb{1}\bigl(\ball{u}{r} \cap B_{k-1}\bigr). \label{eq:wgap--gwap_technical_G}
	\end{align}
	Next we measure the intersections of $\ball{u}{r}$ with the sets $A_{k} = [a_{k}, b_{k}]$ and $A_{k-1} = [a_{k-1}, b_{k-1}]$;
	recall that the length of $A_{k-1}$ is $4^{2}$-times the length of $A_{k}$.
	First suppose that
	$b \leq a_{k-1} + (b_{k} - a_{k})$;
	then, since $u \leq m_{k - 1}$, and $m_{k-1}$ is the midpoint between the right endpoint $b_{k}$ of $A_{k}$ and the left endpoint $a_{k-1}$ of $A_{k-1}$, we must have
	$\Leb{1}(\ball{u}{r} \cap A_{k-1}) \leq \Leb{1}(\ball{u}{r} \cap A_{k})$.
	Otherwise, $\Leb{1}(\ball{u}{r} \cap A_{k-1}) \leq \Leb{1}(A_{k-1}) = 4^{2} \Leb{1}(A_{k}) = 4^{2}\Leb{1}(\ball{u}{r} \cap A_{k})$.

	Finally we measure the intersections of $\ball{u}{r}$ with $B_{k} = (m_{k}, a_{k}) \cup (b_{k}, m_{k-1}]$ and $B_{k-1} = (m_{k-1}, a_{k-1}) \cup (b_{k-1}, m_{k-2}]$.
	It follows from the facts that $u \in (b_{k}, m_{k-1}]$ and $b_{k} - a_{k} \leq b_{k-1} - a_{k-1}$ that $\Leb{1}(\ball{u}{r} \cap B_{k-1}) \leq \Leb{1}(\ball{u}{r} \cap B_{k})$.

	Using these facts, and that $\beta_{k} \leq 8 \beta_{k-1}$ for $k > 1$, \eqref{eq:wgap--gwap_technical_F} and \eqref{eq:wgap--gwap_technical_G} yield
	\begin{align*}
		\gamma\bigl(\ball{u}{r} \bigr)
		 &\leq \frac{4^{k-1} \cdot 4^{2} \Leb{1}\bigl(\ball{u}{r} \cap A_{k}\bigr) + 8\beta_{k} \Leb{1}\bigl(\ball{u}{r} \cap B_{k}\bigr)}{4^{k} \Leb{1}\bigl(\ball{u}{r} \cap A_{k}\bigr) + \beta_{k} \Leb{1}\bigl(\ball{u}{r} \cap B_{k}\bigr)} \leq 8.
	\end{align*}

	\proofpartorcase{Case 2b: $a \notin C_{k}$ but $b \in C_{k}$.}
	The argument is similar to the previous case:
	we must have that $a > m_{k+2}$, i.e., $\ball{u}{r}$ intersects at most $C_{k}$, $C_{k+1}$, and $C_{k+2}$;
	thus
	\begin{align*}
		M^{\text{odd}}\bigl(\ball{u}{r}\bigr) &\geq M^{\text{odd}}\bigl(\ball{u}{r} \cap C_{k} \bigr) = 4^{k} \Leb{1}\bigl(\ball{u}{r} \cap A_{k}\bigr) + \beta_{k} \Leb{1}\bigl(\ball{u}{r} \cap B_{k}\bigr),\\
		M^{\text{even}}\bigl(\ball{u}{r}\bigr) &= M^{\text{even}}\bigl(\ball{u}{r} \cap C_{k+1} \bigr) = 4^{k+1}\Leb{1}\bigl(\ball{u}{r} \cap A_{k+1}\bigr) +  \beta_{k+1} \Leb{1}\bigl(\ball{u}{r} \cap B_{k+1}\bigr).
	\end{align*}
	This time, as the length of $A_{k+1}$ is less than that of $A_{k}$, we have $\Leb{1}(\ball{u}{r} \cap A_{k+1}) \leq \Leb{1}(\ball{u}{r} \cap A_{k})$ and $\Leb{1}(\ball{u}{r} \cap B_{k+1}) \leq 2 \Leb{1}(\ball{u}{r} \cap B_{k})$.
	Using these facts, and proceeding as in the previous case, we see that $\gamma(\ball{u}{r}) \leq 8$.
\end{proof}

\begin{lemma}
	\label[lemma]{lem:w_but_not_s_JK_properties}

	In the setting of \cref{notation:Gaussian_s_ps_w}:

	\begin{enumerate}[label=(\alph*)]
		\item
		for any continuous potential $\Phi \colon X \to \bR$ of the form $\Phi(x) = -\frac{x_{1}^{2}}{2} - \tilde{\Phi}(x_{2}, x_{3}, \dots)$, and any $\lambda \in \bR$,
		\begin{equation*}
			\int_{\ball{\lambda e_{1}}{r}} \exp\bigl(-\Phi(x) \bigr) \,\mu_{0}(\rd x) = \int_{\ball{0}{r}} \exp\bigl(-\Phi(x) \bigr) \,\mu_{0}(\rd x).
		\end{equation*}

		\item
		for the potential $\tau_{n}$,
		\begin{equation} \label{eq:Gaussian_radius-r_maximiser}
			\int_{\ball{x}{r}} \exp\bigl(-\tau_{n}(x) \bigr)  \,\mu_{0}(\rd x) \leq \int_{\ball{0}{r}} \exp\bigl(-\tau_{n}(x)\bigr) \,\mu_{0}(\rd x) \text{~~for any $x \in X$.}
		\end{equation}

	\end{enumerate}
\end{lemma}

\begin{proof}
	Let $\mu_{h} \defeq \mu_{0}(\quark + h)$ denote the shift of $\mu_{0}$ by $h \in X$, and recall the Cameron--Martin theorem asserts that $\mu_{h}$ and $\mu_{0}$ are equivalent if and only if $h \in E$, with the density of $\mu_{h}$ with respect to $\mu_{0}$ given by the Cameron--Martin formula \citep[Corollary~2.4.3]{Bogachev1998Gaussian}:
	\begin{equation*}
		\mu_{h}(\rd x) = \exp\left(-I_{0}(h) - \innerprod{C^{-1}h}{x}_{\ell^{2}}\right) \, \mu_{0}(\rd x).
	\end{equation*}
	For the first part,
	we use this formula with $h = \lambda e_{1}$ to obtain that
	\begin{align*}
		\int_{\ball{\lambda e_{1}}{r}} \exp\bigl(-\Phi(x) \bigr) \, \mu_{0}(\rd x) &=\int_{\ball{0}{r}} \exp\bigl(-\Phi(x + \lambda e_{1})\bigr)\,\mu_{\lambda e_{1}}(\rd x) \nonumber\\
		&= \int_{\ball{0}{r}} \exp\bigl(-\Phi(x + \lambda e_{1})\bigr) \exp\left(-I_{0}(\lambda e_{1}) - \innerprod{C^{-1} (\lambda e_{1})}{x}_{\ell^{2}}\right) \,\mu_{0}(\rd x) \\
		&= \int_{\ball{0}{r}} \exp\left(\frac{1}{2}\bigl(\lambda^{2} + 2\lambda x_{1}\bigr) - \Phi(x) \right) \exp\left( - \frac{\lambda^{2}}{2} - \lambda x_{1} \right) \,\mu_{0}(\rd x)\\
		&=
		\int_{\ball{0}{r}} \exp \bigl( -\Phi(x) \bigr) \, \mu_{0}(\rd x).
	\end{align*}
	For the second part, it suffices to prove \eqref{eq:Gaussian_radius-r_maximiser} for all $x \in E$ as $E$ is dense in $X$.
	Noting that
	\begin{align*}
		\tau_{n}(x + h) &= \frac{(x_{1} + h_{1})^{2}}{2} + \min\left(1, \sum_{k = 2}^{K_{n}} \frac{k^{2} (x_{k} + h_{k})^{2}}{2}\right) \\
		&\leq \frac{x_{1}^{2}}{2} + \min\left(1, \sum_{k = 2}^{K_{n}} \frac{k^{2} x_{k}^{2}}{2} \right) + \sum_{k = 1}^{K_{n}} \left(\frac{k^{2} h_{k}^{2}}{2} + k^{2} h_{k} x_{k}\right) \\
		&\leq \tau_{n}(x) + I_{0}(h) + \innerprod{C^{-1} h}{x}_{\ell^{2}}.
	\end{align*}
	Using this with the Cameron--Martin formula proves the claim.
\end{proof}

\begin{lemma}
	\label[lemma]{lem:wag_wga_properties}
	With the notation introduced in \Cref{ex:no_optimal_approximating_sequence}, we have the following results:
	\begin{enumerate}[label=(\alph*)]
		\item
		\label{item:wag_wga_counterexample_maximizer_ball_probability}
		\emph{(Maximisers of the radius-$r$ ball mass under $\mu_{i}$ exist.)}
		For any $r \leq R_{i}$, a maximiser of the radius-$r$ ball mass under $\mu_{i}$ exists and is given by
		\begin{equation*}
			u_{i}^{\star}(r)
			\defeq
			\left.
			\begin{cases}
				\alpha_{i} e_{j}, & \text{if } \theta^{-1}R_{j+1} \leq r < \theta^{-1} R_{j} \text{ for some $j > i$,}
				\\
				\alpha_{i} e_{i}, & \text{otherwise,}
			\end{cases}
			\right\}
			\in
			\argmax_{u \in X} \mu_{i}(\ball{u}{r}).
		\end{equation*}

		\item
		\label{item:improving_approximating_sequences_alignment}
		\emph{(Approximating sequences centred away from $\alpha_{i} e_{j}$ can always be matched by those centred only at $\alpha_{i} e_{j}$.)}
		For any \ns\ $(r_{n})_{n \in \bN}$ and any \as\ $(u_{n})_{n \in \bN} \to u = 0$, there exists a further sequence $(u_{n}')_{n \in \bN} \to u$ with $(u_{n}')_{n \in \bN} \subseteq \set{\alpha_{i} e_{j}}{j \geq i,~i \in \bN}$ and
		\begin{equation*}
			\orcdf{u_{n}}{r_{n}} \leq \orcdf{u_{n}'}{r_{n}} \text{~~for all large enough $n$.}
		\end{equation*}

		\item
		\label{item:improving_approximating_sequences_outwards}
		\emph{(Any approximating sequence centred at $\alpha_{i} e_{j}$ can be strictly improved.)}
		For any \ns\ $(r_{n})_{n \in \bN}$ and any \as\ $(u_{n}')_{n \in \bN} \to u$ with $(u_{n}')_{n \in \bN} \subseteq \set{\alpha_{i} e_{j}}{j \geq i,~i \in \bN}$, there exists a further sequence $(v_{n})_{n \in \bN} \to u$ with $(v_{n})_{n \in \bN} \subseteq \set{\alpha_{i} e_{j}}{j \geq i,~i \in \bN}$ and
		\begin{equation*}
			\orcdf{u_{n}'}{r_{n}} \leq \zeta \orcdf{v_{n}}{r_{n}} \text{~~for all large enough $n$.}
		\end{equation*}
	\end{enumerate}
\end{lemma}

\begin{proof}
	\begin{enumerate}[label=(\alph*)]
		\item
	First observe from that any ball $\ball{u}{r}$ can intersect at most one ball $B_{i, j}$, $j \geq i$.
	Either $\ball{u}{r}$ intersects no ball at all, or we may assume that $\ball{u}{r}$ intersects $B_{i,j}$, in which case it suffices to consider the case $u = \alpha_{i} e_{j}$ since $\mu_{i}(\ball{u}{r}) \leq \mu_{i}(\ball{\alpha_{i} e_{j}}{r})$.
	Thus, at least one of the centres $(\alpha_{i} e_{j})_{j \geq i}$ must attain the supremal radius-$r$ ball mass $\sMass{\mu_{i}}{r}$.
	By construction, $\mu_{i}(\ball{\alpha_{i} e_{j}}{r}) = 2 \theta^{-j} \min(r, R_{j})$, so
	\begin{equation*}
		\mu_{i}(\ball{\alpha_{i} e_{j}}{r})
		\geq
		\mu_{i}(\ball{\alpha_{i} e_{j+1}}{r})
		\qquad \Longleftrightarrow \qquad
		r
		\geq
		\theta^{-1} R_{j+1}.
	\end{equation*}
	If $\theta^{-1} R_{j + 1} \leq r < \theta^{-1} R_{j}$ for some $j > i$, then evidently the centre $\alpha_{i} e_{j}$ strictly maximises the radius-$r$ ball mass over all choices $(\alpha_{i} e_{j})_{j \geq i}$;
	otherwise if $r \geq \theta^{-1} R_{j}$ then the maximising centre is $\alpha_{i} e_{i}$, proving the claim.

		\item

	First note that
	\begin{align}
		\label{equ:measure_of_specific_balls}
		\mu(B_{i,j})
		&=
		Z^{-1} \zeta^{i} \mu_{i}(B_{i,j})
		=
		Z^{-1} \zeta^{i} \sigma_{j}(X)
		=
		Z^{-1} \zeta^{i} \theta^{j},
		\\[1ex]
		\label{equ:measure_of_annulus}
		\mu(A_{i})
		&=
		\mu \Biggl( \bigcup_{j \geq i} B_{i, j} \Biggr)
		=
		Z^{-1} \, \frac{(\zeta\theta)^{i}}{1-\theta}
		=
		\frac{1-\zeta\theta}{\zeta\theta}\, (\zeta\theta)^{i}
		=
		(\zeta\theta)^{i-1} - (\zeta\theta)^{i},
		\\
		\label{equ:measure_of_final_annuli}
		\mu \Biggl( \bigcup_{k \geq i} A_{k} \Biggr)
		&=
		(\zeta\theta)^{i-1}.
	\end{align}
	Moreover, for any $u \in X$ and $i,j\in\bN$, $\sigma_{j}(\ball{u}{r}) \in \bigO(r)$ and $\mu_{i}(\ball{u}{r}) \in \bigO(r)$ as $r \to 0$.

	For the finitely many $n \in \bN$ such that $r_{n} > R_{1}$ or $\norm{u_{n}} > \alpha_{1}$, we may set $u_{n}' = u_{n}$, so assume without loss of generality that $r_{n} \leq R_{1}$ and $\norm{u_{n}} \leq \alpha_{1}$ for all $n \in \bN$.
	Then define
	\begin{align*}
		u_{n}'
		&\defeq
		\begin{cases}
			\alpha_{j_{n}} e_{j_{n}}, & \text{if } \norm{u_{n}} < \alpha_{j_{n}} - 2 R_{j_{n}},\\
			u_{i_{n}}^{\star}(r_{n}), & \text{otherwise,}
		\end{cases}\\
		i_{n}
		&\defeq \min \Bigl( \argmin_{k \in \bN} \dist{u_{n}}{A_{k}}\Bigr),
		\\
		j_{n}
		&\defeq
		\max \Set{k \in \bN}{r_{n} \leq R_{k}}.
	\end{align*}
	
	We claim that the sequence $(u_{n}')_{n \in \bN} \to 0$ satisfies the required properties.
	To prove that $\mu(\ball{u_{n}}{r_{n}}) \leq \mu(\ball{u_{n}'}{r_{n}})$, we examine the two cases in the definition of $u_{n}'$ separately.
	First let $\norm{u_{n}} < \alpha_{j_{n}} - 2 R_{j_{n}}$.
	Then $\ball{u_{n}}{r_{n}}$ intersects none of the annuli $A_{k}$, $k \leq j_{n}$, and so
	\begin{equation*}
		\mu(\ball{u_{n}}{r_{n}})
		\leq
		\mu \Biggl( \bigcup_{j \geq j_{n}+1} A_{j} \Biggr)
		=
		(\zeta\theta)^{j_{n}}
	\end{equation*}
	by \eqref{equ:measure_of_final_annuli}.
	On the other hand, using \eqref{equ:measure_of_specific_balls} and the fact that $\theta^{2} = R_{j_{n} + 1}/R_{j_{n}} < r_{n} / R_{j_{n}}$,
	\begin{equation*}
		\mu(\ball{u_{n}'}{r_{n}})
		\geq
		\frac{r_{n}}{R_{j_{n}}} \mu( B_{j_{n},j_{n}} )
		>
		\theta^{2} \cdot Z^{-1} (\zeta \theta)^{j_{n}}.
	\end{equation*}
	With this and the fact that $\theta, \zeta \leq \tfrac{1}{2}$, we obtain the desired inequality from $\theta^{2} Z^{-1} = (1 - \theta)(1 - \zeta\theta) \frac{\theta}{\zeta} > \frac{\theta}{4\zeta} \geq 1$.
	Now suppose that $\norm{u_{n}} \geq \alpha_{j_{n}} - 2 R_{j_{n}}$.
	Using the separation properties of the annuli $A_{i}$ and balls $B_{i,j}$, since $r_{n} \leq R_{j_{n}}$, the ball $\ball{u_{n}}{r_{n}}$ could only possibly intersect one of the annuli $A_{k}$, $1 \leq k \leq j_{n}$, which, by the definition of $i_{n}$, has to be $A_{i_{n}}$ (or none at all, in which case there is nothing to show).
	Hence, by \Cref{lem:wag_wga_properties}\ref{item:wag_wga_counterexample_maximizer_ball_probability},
	\begin{equation*}
		\mu(\ball{u_{n}}{r_{n}})
		=
		Z^{-1} \zeta^{i_{n}} \, \mu_{i_{n}}(\ball{u_{n}}{r_{n}})
		\leq
		Z^{-1} \zeta^{i_{n}} \, \mu_{i_{n}}(\ball{u_{i_{n}}^{\star}(r_{n})}{r_{n}})
		=
		\mu(\ball{u_{n}'}{r_{n}}).
	\end{equation*}

	\item
	Since $(u_{n}')_{n \in \bN} \subseteq \set{\alpha_{i} e_{j}}{i \leq j \in \bN}$, we can write it in the form $u_{n}' = \alpha_{p_{n}} e_{q_{n}}$ for some indices $p_{n},q_{n} \in \bN$ and define
	\begin{equation*}
		v_{n}
		\defeq
		\begin{cases}
			u_{n}', & \text{if } p_{n} = 1,
			\\
			\alpha_{p_{n} - 1} e_{q_{n}}, & \text{otherwise.}
		\end{cases}
	\end{equation*}
	As $u_{n}' \to 0$, it is obvious that for all sufficiently large $n$, we must have $p_{n} > 1$ and thus $\mu(\ball{u_{n}'}{r_{n}}) \leq \zeta \, \mu(\ball{v_{n}}{r_{n}})$, since moving from the annulus $A_{p_{n}}$ to the annulus $A_{p_{n} - 1}$ gains a factor at least $\zeta^{-1}$ of mass. \qedhere
	\end{enumerate}
\end{proof}

\begin{lemma}
	\label[lemma]{lem:no_locally_compact_example}
	Let $X$ be a locally compact metric space, and let $\mu \in \prob{X}$.
	Suppose that $u \in X$ is dominant in the sense of \eqref{eq:no_optimal_approximating_sequence_dominant}.
	
	Then \eqref{eq:no_optimal_approximating_sequence_nonoptimal} fails:
	there exist sequences $(r_{n})_{n \in \bN} \to 0$ and $(u_{n})_{n \in \bN} \to u$ such that, for any $(v_{n})_{n \in \bN} \to u$, $\liminf_{n \to \infty} \Ratio{u_{n}}{v_{n}}{r_{n}}{\mu} \geq 1$.
\end{lemma}

\begin{proof}
	Let $(r_{n})_{n \in \bN} \to 0$ be arbitrary and fix $R > 0$ such that $\ball{u}{R}$ is precompact.
	Take a sequence $(w_{n})_{n \in \bN}$ with
	\begin{equation*}
		\liminf_{n \to \infty} \Ratio{w_{n}}{\sup}{r_{n}}{\mu|_{\ball{u}{R}}} \geq 1 - \frac{1}{n}.
	\end{equation*}
	Then $(w_{n})_{n \in \bN} \subseteq \ball{u}{R}$ has a convergent subsequence in $X$;
	pass to this subsequence and the corresponding subsequence of $(r_{n})_{n \in \bN}$ without relabelling.

	By \eqref{eq:no_optimal_approximating_sequence_dominant}, there must exist a sequence $(u_{n})_{n \in \bN} \to u$ with $\liminf_{n \to \infty} \Ratio{u_{n}}{w_{n}}{r_{n}}{\mu} \geq 1$.
	But this implies that $(u_{n})_{n \in \bN}$ is an optimal \as\ along the sequence of radii $(r_{n})_{n \in \bN}$:
	let $(v_{n})_{n \in \bN} \to u$ be arbitrary and note that we may assume that $\ball{u_{n}}{r_{n}}$ and $\ball{v_{n}}{r_{n}}$ are contained in $\ball{u}{R}$ for all sufficiently large $n$.
	Then
	\begin{align*}
		& \liminf_{n \to \infty} \Ratio{u_{n}}{v_{n}}{r_{n}}{\mu} \\
		& \quad = \liminf_{n \to \infty} \Ratio{u_{n}}{v_{n}}{r_{n}}{\mu|_{\ball{u}{R}}} \\
		& \quad \geq \liminf_{n \to \infty} \Ratio{u_{n}}{w_{n}}{r_{n}}{\mu|_{\ball{u}{R}}}
		\lim_{n \to \infty} \Ratio{w_{n}}{\sup}{r_{n}}{\mu|_{\ball{u}{R}}} \liminf_{n \to \infty} \Ratio{\sup}{v_{n}}{r_{n}}{\mu|_{\ball{u}{R}}}
		= 1.
		\qedhere
	\end{align*}
\end{proof}

\section{Enumeration of grammatically correct small-ball mode definitions}
\label[appendix]{section:Enumeration}

\Cref{table:enumeration} enumerates all 282 definitions of a mode that are `grammatically correct' as per \Cref{defn:structured_definition}.
The first column gives the index $1 \leq i \leq 282$ of each definition, using a lexicographic ordering of the strings of length at most $4$ in the alphabet $\{ \forall \ns, \exists \ns, \forall \as, \exists \as, \forall \cp, \exists \cp, \forall \cpas, \exists \cpas \}$, ignoring those strings that do not satisfy \Cref{defn:structured_definition}.
The resulting definition is then stated in the second column, along with \Cref{defn:letter_notation}'s notation for this mode type, if any.
The third column gives the `logical representative', defined to be the least $j \leq i$ such that definitions $i$ and $j$ are logically equivalent, in the sense that one is obtained from the other by commuting adjacent universal/existential quantifiers over the same terms;
\cmark\ indicates that definition $i$ is its own logical representative, and further information is reported in this case only.
The fourth column shows whether definition $i$ satisfies \LP, \AP, and \CP:
\xmark\ shows that at least one fails, whereas \cmark\ shows that they all hold, by \Cref{prop:mode_maps_violating_AP,prop:mode_maps_violating_CP,prop:mode_map_all_axioms};
no further information is reported for meaningless modes with a \xmark.
Finally, the fifth column gives the `canonical representative', defined to be the least $j \leq i$ such that definitions $i$ and $j$ are equivalent according to \Cref{prop:mode_map_equivalences};
again, \cmark\ indicates that definition $i$ is its own canonical representative.

Thus, the 144 definitions with \cmark\ in the third column are the logical representatives of the grammatically correct definitions.
Those with \cmark\ in the fourth column are the 21 meaningful definitions shown in \Cref{fig:periodic_table}(\subref{fig:table_of_21}).
The ten definitions with \cmark\ in the last column are the canonical representatives listed in \Cref{fig:periodic_table}(\subref{fig:table_of_21}) and whose implication lattice is shown in \Cref{fig:periodic_table}(\subref{fig:Hasse_diagram_main}).

\begin{center}
\scriptsize

\begin{longtable}{|r|lc|c|c|c|}
	\caption{All 282 grammatically correct mode definitions satisfying \Cref{defn:structured_definition}.}
	\label{table:enumeration} \\
	\hline
	$i$ & Definition $i$ & & Logical & \LP, \AP, \CP\ (Propositions & Canonical representative \\
	~ & ~ & ~ & representative & \ref{prop:mode_maps_violating_AP}, \ref{prop:mode_maps_violating_CP}, and \ref{prop:mode_map_all_axioms}) & (\Cref{{prop:mode_map_equivalences}}) \\
	\hline
	\endfirsthead
	\hline
	$i$ & Definition $i$ & & Logical & \LP, \AP, \CP\ (Propositions & Canonical representative \\
	~ & ~ & ~ & representative & \ref{prop:mode_maps_violating_AP}, \ref{prop:mode_maps_violating_CP}, and \ref{prop:mode_map_all_axioms}) & (\Cref{{prop:mode_map_equivalences}}) \\
	\hline
	\endhead
	\hline
	\endfoot
	\hline
	\endlastfoot
	1 & $[ \forall \ns ]$ & (\sws) & \cmark\ & \cmark\ by \ref{prop:mode_map_all_axioms} & \cmark \\
	2 & $[ \exists \ns ]$ & (\swps) & \cmark\ & \cmark\ by \ref{prop:mode_map_all_axioms} & \cmark \\
	3 & $[ \forall \ns \forall \as ]$ & & \cmark\ & \xmark\ by \ref{prop:mode_maps_violating_AP}(\hyperref[item:violation_AP_b]{b}) & --- \\
	4 & $[ \forall \ns \exists \as ]$ & (\swgs) & \cmark\ & \cmark\ by \ref{prop:mode_map_all_axioms} & \cmark \\
	5 & $[ \forall \ns \forall \cp ]$ & (\sww) & \cmark\ & \cmark\ by \ref{prop:mode_map_all_axioms} & \cmark \\
	6 & $[ \forall \ns \exists \cp ]$ & & \cmark\ & \xmark\ by \ref{prop:mode_maps_violating_AP}(\hyperref[item:violation_AP_a]{a}) & --- \\
	7 & $[ \exists \ns \forall \as ]$ & & \cmark\ & \xmark\ by \ref{prop:mode_maps_violating_AP}(\hyperref[item:violation_AP_d]{d}) & --- \\
	8 & $[ \exists \ns \exists \as ]$ & (\swpgs) & \cmark\ & \cmark\ by \ref{prop:mode_map_all_axioms} & \cmark \\
	9 & $[ \exists \ns \forall \cp ]$ & (\swpw) & \cmark\ & \cmark\ by \ref{prop:mode_map_all_axioms} & \cmark \\
	10 & $[ \exists \ns \exists \cp ]$ & & \cmark\ & \xmark\ by \ref{prop:mode_maps_violating_AP}(\hyperref[item:violation_AP_a]{a}) & --- \\
	11 & $[ \forall \as \forall \ns ]$ & & 3 & --- & --- \\
	12 & $[ \forall \as \exists \ns ]$ & & \cmark\ & \cmark\ by \ref{prop:mode_map_all_axioms} & 2 \\
	13 & $[ \exists \as \forall \ns ]$ & & \cmark\ & \cmark\ by \ref{prop:mode_map_all_axioms} & 1 \\
	14 & $[ \exists \as \exists \ns ]$ & & 8 & --- & --- \\
	15 & $[ \forall \cp \forall \ns ]$ & (\sww) & 5 & --- & --- \\
	16 & $[ \forall \cp \exists \ns ]$ & (\swwp) & \cmark\ & \cmark\ by \ref{prop:mode_map_all_axioms} & \cmark \\
	17 & $[ \exists \cp \forall \ns ]$ & & \cmark\ & \xmark\ by \ref{prop:mode_maps_violating_AP}(\hyperref[item:violation_AP_a]{a}) & --- \\
	18 & $[ \exists \cp \exists \ns ]$ & & 10 & --- & --- \\
	19 & $[ \forall \ns \forall \as \forall \cp ]$ & & \cmark\ & \xmark\ by \ref{prop:mode_maps_violating_AP}(\hyperref[item:violation_AP_b]{b}) & --- \\
	20 & $[ \forall \ns \forall \as \exists \cp ]$ & & \cmark\ & \xmark\ by \ref{prop:mode_maps_violating_AP}(\hyperref[item:violation_AP_a]{a}, \hyperref[item:violation_AP_b]{b}) & --- \\
	21 & $[ \forall \ns \exists \as \forall \cp ]$ & (\swgw) & \cmark\ & \xmark\ by \ref{prop:mode_maps_violating_CP}(\hyperref[item:violation_CP_c]{c}) & --- \\
	22 & $[ \forall \ns \exists \as \exists \cp ]$ & & \cmark\ & \xmark\ by \ref{prop:mode_maps_violating_AP}(\hyperref[item:violation_AP_a]{a}) & --- \\
	23 & $[ \forall \ns \forall \cp \forall \as ]$ & & 19 & --- & --- \\
	24 & $[ \forall \ns \forall \cp \exists \as ]$ & (\swwg) & \cmark\ & \xmark\ by \ref{prop:mode_maps_violating_CP}(\hyperref[item:violation_CP_c]{c}) & --- \\
	25 & $[ \forall \ns \forall \cp \forall \cpas ]$ & & \cmark\ & \xmark\ by \ref{prop:mode_maps_violating_CP}(\hyperref[item:violation_CP_a]{a}) & --- \\
	26 & $[ \forall \ns \forall \cp \exists \cpas ]$ & & \cmark\ & \xmark\ by \ref{prop:mode_maps_violating_AP}(\hyperref[item:violation_AP_e]{e}) & --- \\
	27 & $[ \forall \ns \exists \cp \forall \as ]$ & & \cmark\ & \xmark\ by \ref{prop:mode_maps_violating_AP}(\hyperref[item:violation_AP_a]{a}, \hyperref[item:violation_AP_b]{b}) & --- \\
	28 & $[ \forall \ns \exists \cp \exists \as ]$ & & 22 & --- & --- \\
	29 & $[ \forall \ns \exists \cp \forall \cpas ]$ & & \cmark\ & \xmark\ by \ref{prop:mode_maps_violating_AP}(\hyperref[item:violation_AP_a]{a}) & --- \\
	30 & $[ \forall \ns \exists \cp \exists \cpas ]$ & & \cmark\ & \xmark\ by \ref{prop:mode_maps_violating_AP}(\hyperref[item:violation_AP_a]{a}, \hyperref[item:violation_AP_e]{e}) & --- \\
	31 & $[ \exists \ns \forall \as \forall \cp ]$ & & \cmark\ & \xmark\ by \ref{prop:mode_maps_violating_AP}(\hyperref[item:violation_AP_d]{d}) & --- \\
	32 & $[ \exists \ns \forall \as \exists \cp ]$ & & \cmark\ & \xmark\ by \ref{prop:mode_maps_violating_AP}(\hyperref[item:violation_AP_a]{a}, \hyperref[item:violation_AP_d]{d}) & --- \\
	33 & $[ \exists \ns \exists \as \forall \cp ]$ & (\swpgw) & \cmark\ & \xmark\ by \ref{prop:mode_maps_violating_CP}(\hyperref[item:violation_CP_b]{b}) & --- \\
	34 & $[ \exists \ns \exists \as \exists \cp ]$ & & \cmark\ & \xmark\ by \ref{prop:mode_maps_violating_AP}(\hyperref[item:violation_AP_a]{a}) & --- \\
	35 & $[ \exists \ns \forall \cp \forall \as ]$ & & 31 & --- & --- \\
	36 & $[ \exists \ns \forall \cp \exists \as ]$ & (\swpwg) & \cmark\ & \xmark\ by \ref{prop:mode_maps_violating_CP}(\hyperref[item:violation_CP_b]{b}) & --- \\
	37 & $[ \exists \ns \forall \cp \forall \cpas ]$ & (\swpwa) & \cmark\ & \xmark\ by \ref{prop:mode_maps_violating_CP}(\hyperref[item:violation_CP_a]{a}) & --- \\
	38 & $[ \exists \ns \forall \cp \exists \cpas ]$ & & \cmark\ & \xmark\ by \ref{prop:mode_maps_violating_AP}(\hyperref[item:violation_AP_c]{c}) & --- \\
	39 & $[ \exists \ns \exists \cp \forall \as ]$ & & \cmark\ & \xmark\ by \ref{prop:mode_maps_violating_AP}(\hyperref[item:violation_AP_a]{a}, \hyperref[item:violation_AP_d]{d}) & --- \\
	40 & $[ \exists \ns \exists \cp \exists \as ]$ & & 34 & --- & --- \\
	41 & $[ \exists \ns \exists \cp \forall \cpas ]$ & & \cmark\ & \xmark\ by \ref{prop:mode_maps_violating_AP}(\hyperref[item:violation_AP_a]{a}) & --- \\
	42 & $[ \exists \ns \exists \cp \exists \cpas ]$ & & \cmark\ & \xmark\ by \ref{prop:mode_maps_violating_AP}(\hyperref[item:violation_AP_a]{a}, \hyperref[item:violation_AP_c]{c}) & --- \\
	43 & $[ \forall \as \forall \ns \forall \cp ]$ & & 19 & --- & --- \\
	44 & $[ \forall \as \forall \ns \exists \cp ]$ & & 20 & --- & --- \\
	45 & $[ \forall \as \exists \ns \forall \cp ]$ & & \cmark\ & \cmark\ by \ref{prop:mode_map_all_axioms} & 9 \\
	46 & $[ \forall \as \exists \ns \exists \cp ]$ & & \cmark\ & \xmark\ by \ref{prop:mode_maps_violating_AP}(\hyperref[item:violation_AP_a]{a}) & --- \\
	47 & $[ \forall \as \forall \cp \forall \ns ]$ & & 19 & --- & --- \\
	48 & $[ \forall \as \forall \cp \exists \ns ]$ & & \cmark\ & \cmark\ by \ref{prop:mode_map_all_axioms} & 16 \\
	49 & $[ \forall \as \exists \cp \forall \ns ]$ & & \cmark\ & \xmark\ by \ref{prop:mode_maps_violating_AP}(\hyperref[item:violation_AP_a]{a}, \hyperref[item:violation_AP_b]{b}) & --- \\
	50 & $[ \forall \as \exists \cp \exists \ns ]$ & & 46 & --- & --- \\
	51 & $[ \exists \as \forall \ns \forall \cp ]$ & & \cmark\ & \cmark\ by \ref{prop:mode_map_all_axioms} & 5 \\
	52 & $[ \exists \as \forall \ns \exists \cp ]$ & & \cmark\ & \xmark\ by \ref{prop:mode_maps_violating_AP}(\hyperref[item:violation_AP_a]{a}) & --- \\
	53 & $[ \exists \as \exists \ns \forall \cp ]$ & (\swgpw) & 33 & --- & --- \\
	54 & $[ \exists \as \exists \ns \exists \cp ]$ & & 34 & --- & --- \\
	55 & $[ \exists \as \forall \cp \forall \ns ]$ & & 51 & --- & --- \\
	56 & $[ \exists \as \forall \cp \exists \ns ]$ & (\swgwp) & \cmark\ & \xmark\ by \ref{prop:mode_maps_violating_CP}(\hyperref[item:violation_CP_b]{b}) & --- \\
	57 & $[ \exists \as \exists \cp \forall \ns ]$ & & \cmark\ & \xmark\ by \ref{prop:mode_maps_violating_AP}(\hyperref[item:violation_AP_a]{a}) & --- \\
	58 & $[ \exists \as \exists \cp \exists \ns ]$ & & 34 & --- & --- \\
	59 & $[ \forall \cp \forall \ns \forall \as ]$ & & 19 & --- & --- \\
	60 & $[ \forall \cp \forall \ns \exists \as ]$ & (\swwg) & 24 & --- & --- \\
	61 & $[ \forall \cp \forall \ns \forall \cpas ]$ & & 25 & --- & --- \\
	62 & $[ \forall \cp \forall \ns \exists \cpas ]$ & & 26 & --- & --- \\
	63 & $[ \forall \cp \exists \ns \forall \as ]$ & & \cmark\ & \xmark\ by \ref{prop:mode_maps_violating_AP}(\hyperref[item:violation_AP_d]{d}) & --- \\
	64 & $[ \forall \cp \exists \ns \exists \as ]$ & (\swwpg) & \cmark\ & \xmark\ by \ref{prop:mode_maps_violating_CP}(\hyperref[item:violation_CP_b]{b}) & --- \\
	65 & $[ \forall \cp \exists \ns \forall \cpas ]$ & (\swwpa) & \cmark\ & \xmark\ by \ref{prop:mode_maps_violating_CP}(\hyperref[item:violation_CP_a]{a}) & --- \\
	66 & $[ \forall \cp \exists \ns \exists \cpas ]$ & & \cmark\ & \xmark\ by \ref{prop:mode_maps_violating_AP}(\hyperref[item:violation_AP_c]{c}) & --- \\
	67 & $[ \forall \cp \forall \as \forall \ns ]$ & & 19 & --- & --- \\
	68 & $[ \forall \cp \forall \as \exists \ns ]$ & & 48 & --- & --- \\
	69 & $[ \forall \cp \exists \as \forall \ns ]$ & & \cmark\ & \cmark\ by \ref{prop:mode_map_all_axioms} & 5 \\
	70 & $[ \forall \cp \exists \as \exists \ns ]$ & (\swwgp) & 64 & --- & --- \\
	71 & $[ \forall \cp \forall \cpas \forall \ns ]$ & & 25 & --- & --- \\
	72 & $[ \forall \cp \forall \cpas \exists \ns ]$ & (\swwap) & \cmark\ & \cmark\ by \ref{prop:mode_map_all_axioms} & 16 \\
	73 & $[ \forall \cp \exists \cpas \forall \ns ]$ & & \cmark\ & \cmark\ by \ref{prop:mode_map_all_axioms} & 5 \\
	74 & $[ \forall \cp \exists \cpas \exists \ns ]$ & & 66 & --- & --- \\
	75 & $[ \exists \cp \forall \ns \forall \as ]$ & & \cmark\ & \xmark\ by \ref{prop:mode_maps_violating_AP}(\hyperref[item:violation_AP_a]{a}, \hyperref[item:violation_AP_b]{b}) & --- \\
	76 & $[ \exists \cp \forall \ns \exists \as ]$ & & \cmark\ & \xmark\ by \ref{prop:mode_maps_violating_AP}(\hyperref[item:violation_AP_a]{a}) & --- \\
	77 & $[ \exists \cp \forall \ns \forall \cpas ]$ & & \cmark\ & \xmark\ by \ref{prop:mode_maps_violating_AP}(\hyperref[item:violation_AP_a]{a}) & --- \\
	78 & $[ \exists \cp \forall \ns \exists \cpas ]$ & & \cmark\ & \xmark\ by \ref{prop:mode_maps_violating_AP}(\hyperref[item:violation_AP_a]{a}, \hyperref[item:violation_AP_e]{e}) & --- \\
	79 & $[ \exists \cp \exists \ns \forall \as ]$ & & 39 & --- & --- \\
	80 & $[ \exists \cp \exists \ns \exists \as ]$ & & 34 & --- & --- \\
	81 & $[ \exists \cp \exists \ns \forall \cpas ]$ & & 41 & --- & --- \\
	82 & $[ \exists \cp \exists \ns \exists \cpas ]$ & & 42 & --- & --- \\
	83 & $[ \exists \cp \forall \as \forall \ns ]$ & & 75 & --- & --- \\
	84 & $[ \exists \cp \forall \as \exists \ns ]$ & & \cmark\ & \xmark\ by \ref{prop:mode_maps_violating_AP}(\hyperref[item:violation_AP_a]{a}) & --- \\
	85 & $[ \exists \cp \exists \as \forall \ns ]$ & & 57 & --- & --- \\
	86 & $[ \exists \cp \exists \as \exists \ns ]$ & & 34 & --- & --- \\
	87 & $[ \exists \cp \forall \cpas \forall \ns ]$ & & 77 & --- & --- \\
	88 & $[ \exists \cp \forall \cpas \exists \ns ]$ & & \cmark\ & \xmark\ by \ref{prop:mode_maps_violating_AP}(\hyperref[item:violation_AP_a]{a}) & --- \\
	89 & $[ \exists \cp \exists \cpas \forall \ns ]$ & & \cmark\ & \xmark\ by \ref{prop:mode_maps_violating_AP}(\hyperref[item:violation_AP_a]{a}) & --- \\
	90 & $[ \exists \cp \exists \cpas \exists \ns ]$ & & 42 & --- & --- \\
	91 & $[ \forall \ns \forall \as \forall \cp \forall \cpas ]$ & & \cmark\ & \xmark\ by \ref{prop:mode_maps_violating_AP}(\hyperref[item:violation_AP_b]{b}), \ref{prop:mode_maps_violating_CP}(\hyperref[item:violation_CP_a]{a}) & --- \\
	92 & $[ \forall \ns \forall \as \forall \cp \exists \cpas ]$ & & \cmark\ & \xmark\ by \ref{prop:mode_maps_violating_AP}(\hyperref[item:violation_AP_b]{b}, \hyperref[item:violation_AP_e]{e}) & --- \\
	93 & $[ \forall \ns \forall \as \exists \cp \forall \cpas ]$ & & \cmark\ & \xmark\ by \ref{prop:mode_maps_violating_AP}(\hyperref[item:violation_AP_a]{a}, \hyperref[item:violation_AP_b]{b}) & --- \\
	94 & $[ \forall \ns \forall \as \exists \cp \exists \cpas ]$ & & \cmark\ & \xmark\ by \ref{prop:mode_maps_violating_AP}(\hyperref[item:violation_AP_a]{a}, \hyperref[item:violation_AP_b]{b}, \hyperref[item:violation_AP_e]{e}) & --- \\
	95 & $[ \forall \ns \exists \as \forall \cp \forall \cpas ]$ & (\swgwa) & \cmark\ & \xmark\ by \ref{prop:mode_maps_violating_CP}(\hyperref[item:violation_CP_d]{d}) & --- \\
	96 & $[ \forall \ns \exists \as \forall \cp \exists \cpas ]$ & & \cmark\ & \xmark\ by \ref{prop:mode_maps_violating_AP}(\hyperref[item:violation_AP_e]{e}), \ref{prop:mode_maps_violating_CP}(\hyperref[item:violation_CP_c]{c}) & --- \\
	97 & $[ \forall \ns \exists \as \exists \cp \forall \cpas ]$ & & \cmark\ & \xmark\ by \ref{prop:mode_maps_violating_AP}(\hyperref[item:violation_AP_a]{a}) & --- \\
	98 & $[ \forall \ns \exists \as \exists \cp \exists \cpas ]$ & & \cmark\ & \xmark\ by \ref{prop:mode_maps_violating_AP}(\hyperref[item:violation_AP_a]{a}, \hyperref[item:violation_AP_e]{e}) & --- \\
	99 & $[ \forall \ns \forall \cp \forall \as \forall \cpas ]$ & & 91 & --- & --- \\
	100 & $[ \forall \ns \forall \cp \forall \as \exists \cpas ]$ & & 92 & --- & --- \\
	101 & $[ \forall \ns \forall \cp \exists \as \forall \cpas ]$ & (\swwga) & \cmark\ & \xmark\ by \ref{prop:mode_maps_violating_CP}(\hyperref[item:violation_CP_d]{d}) & --- \\
	102 & $[ \forall \ns \forall \cp \exists \as \exists \cpas ]$ & & \cmark\ & \xmark\ by \ref{prop:mode_maps_violating_AP}(\hyperref[item:violation_AP_e]{e}), \ref{prop:mode_maps_violating_CP}(\hyperref[item:violation_CP_c]{c}) & --- \\
	103 & $[ \forall \ns \forall \cp \forall \cpas \forall \as ]$ & & 91 & --- & --- \\
	104 & $[ \forall \ns \forall \cp \forall \cpas \exists \as ]$ & (\swwag) & \cmark\ & \xmark\ by \ref{prop:mode_maps_violating_CP}(\hyperref[item:violation_CP_e]{e}) & --- \\
	105 & $[ \forall \ns \forall \cp \exists \cpas \forall \as ]$ & & \cmark\ & \xmark\ by \ref{prop:mode_maps_violating_AP}(\hyperref[item:violation_AP_b]{b}, \hyperref[item:violation_AP_e]{e}) & --- \\
	106 & $[ \forall \ns \forall \cp \exists \cpas \exists \as ]$ & & 102 & --- & --- \\
	107 & $[ \forall \ns \exists \cp \forall \as \forall \cpas ]$ & & \cmark\ & \xmark\ by \ref{prop:mode_maps_violating_AP}(\hyperref[item:violation_AP_a]{a}, \hyperref[item:violation_AP_b]{b}) & --- \\
	108 & $[ \forall \ns \exists \cp \forall \as \exists \cpas ]$ & & \cmark\ & \xmark\ by \ref{prop:mode_maps_violating_AP}(\hyperref[item:violation_AP_a]{a}, \hyperref[item:violation_AP_b]{b}, \hyperref[item:violation_AP_e]{e}) & --- \\
	109 & $[ \forall \ns \exists \cp \exists \as \forall \cpas ]$ & & 97 & --- & --- \\
	110 & $[ \forall \ns \exists \cp \exists \as \exists \cpas ]$ & & 98 & --- & --- \\
	111 & $[ \forall \ns \exists \cp \forall \cpas \forall \as ]$ & & 107 & --- & --- \\
	112 & $[ \forall \ns \exists \cp \forall \cpas \exists \as ]$ & & \cmark\ & \xmark\ by \ref{prop:mode_maps_violating_AP}(\hyperref[item:violation_AP_a]{a}) & --- \\
	113 & $[ \forall \ns \exists \cp \exists \cpas \forall \as ]$ & & \cmark\ & \xmark\ by \ref{prop:mode_maps_violating_AP}(\hyperref[item:violation_AP_a]{a}, \hyperref[item:violation_AP_b]{b}, \hyperref[item:violation_AP_e]{e}) & --- \\
	114 & $[ \forall \ns \exists \cp \exists \cpas \exists \as ]$ & & 98 & --- & --- \\
	115 & $[ \exists \ns \forall \as \forall \cp \forall \cpas ]$ & & \cmark\ & \xmark\ by \ref{prop:mode_maps_violating_AP}(\hyperref[item:violation_AP_d]{d}), \ref{prop:mode_maps_violating_CP}(\hyperref[item:violation_CP_a]{a}) & --- \\
	116 & $[ \exists \ns \forall \as \forall \cp \exists \cpas ]$ & & \cmark\ & \xmark\ by \ref{prop:mode_maps_violating_AP}(\hyperref[item:violation_AP_c]{c}, \hyperref[item:violation_AP_d]{d}) & --- \\
	117 & $[ \exists \ns \forall \as \exists \cp \forall \cpas ]$ & & \cmark\ & \xmark\ by \ref{prop:mode_maps_violating_AP}(\hyperref[item:violation_AP_a]{a}, \hyperref[item:violation_AP_d]{d}) & --- \\
	118 & $[ \exists \ns \forall \as \exists \cp \exists \cpas ]$ & & \cmark\ & \xmark\ by \ref{prop:mode_maps_violating_AP}(\hyperref[item:violation_AP_a]{a}, \hyperref[item:violation_AP_c]{c}, \hyperref[item:violation_AP_d]{d}) & --- \\
	119 & $[ \exists \ns \exists \as \forall \cp \forall \cpas ]$ & (\swpgwa) & \cmark\ & \xmark\ by \ref{prop:mode_maps_violating_CP}(\hyperref[item:violation_CP_d]{d}) & --- \\
	120 & $[ \exists \ns \exists \as \forall \cp \exists \cpas ]$ & & \cmark\ & \xmark\ by \ref{prop:mode_maps_violating_AP}(\hyperref[item:violation_AP_c]{c}), \ref{prop:mode_maps_violating_CP}(\hyperref[item:violation_CP_b]{b}) & --- \\
	121 & $[ \exists \ns \exists \as \exists \cp \forall \cpas ]$ & & \cmark\ & \xmark\ by \ref{prop:mode_maps_violating_AP}(\hyperref[item:violation_AP_a]{a}) & --- \\
	122 & $[ \exists \ns \exists \as \exists \cp \exists \cpas ]$ & & \cmark\ & \xmark\ by \ref{prop:mode_maps_violating_AP}(\hyperref[item:violation_AP_a]{a}, \hyperref[item:violation_AP_c]{c}) & --- \\
	123 & $[ \exists \ns \forall \cp \forall \as \forall \cpas ]$ & & 115 & --- & --- \\
	124 & $[ \exists \ns \forall \cp \forall \as \exists \cpas ]$ & & 116 & --- & --- \\
	125 & $[ \exists \ns \forall \cp \exists \as \forall \cpas ]$ & (\swpwga) & \cmark\ & \xmark\ by \ref{prop:mode_maps_violating_CP}(\hyperref[item:violation_CP_d]{d}) & --- \\
	126 & $[ \exists \ns \forall \cp \exists \as \exists \cpas ]$ & & \cmark\ & \xmark\ by \ref{prop:mode_maps_violating_AP}(\hyperref[item:violation_AP_c]{c}), \ref{prop:mode_maps_violating_CP}(\hyperref[item:violation_CP_b]{b}) & --- \\
	127 & $[ \exists \ns \forall \cp \forall \cpas \forall \as ]$ & & 115 & --- & --- \\
	128 & $[ \exists \ns \forall \cp \forall \cpas \exists \as ]$ & (\swpwag) & \cmark\ & \xmark\ by \ref{prop:mode_maps_violating_CP}(\hyperref[item:violation_CP_e]{e}) & --- \\
	129 & $[ \exists \ns \forall \cp \exists \cpas \forall \as ]$ & & \cmark\ & \xmark\ by \ref{prop:mode_maps_violating_AP}(\hyperref[item:violation_AP_c]{c}, \hyperref[item:violation_AP_d]{d}) & --- \\
	130 & $[ \exists \ns \forall \cp \exists \cpas \exists \as ]$ & & 126 & --- & --- \\
	131 & $[ \exists \ns \exists \cp \forall \as \forall \cpas ]$ & & \cmark\ & \xmark\ by \ref{prop:mode_maps_violating_AP}(\hyperref[item:violation_AP_a]{a}, \hyperref[item:violation_AP_d]{d}) & --- \\
	132 & $[ \exists \ns \exists \cp \forall \as \exists \cpas ]$ & & \cmark\ & \xmark\ by \ref{prop:mode_maps_violating_AP}(\hyperref[item:violation_AP_a]{a}, \hyperref[item:violation_AP_c]{c}, \hyperref[item:violation_AP_d]{d}) & --- \\
	133 & $[ \exists \ns \exists \cp \exists \as \forall \cpas ]$ & & 121 & --- & --- \\
	134 & $[ \exists \ns \exists \cp \exists \as \exists \cpas ]$ & & 122 & --- & --- \\
	135 & $[ \exists \ns \exists \cp \forall \cpas \forall \as ]$ & & 131 & --- & --- \\
	136 & $[ \exists \ns \exists \cp \forall \cpas \exists \as ]$ & & \cmark\ & \xmark\ by \ref{prop:mode_maps_violating_AP}(\hyperref[item:violation_AP_a]{a}) & --- \\
	137 & $[ \exists \ns \exists \cp \exists \cpas \forall \as ]$ & & \cmark\ & \xmark\ by \ref{prop:mode_maps_violating_AP}(\hyperref[item:violation_AP_a]{a}, \hyperref[item:violation_AP_c]{c}, \hyperref[item:violation_AP_d]{d}) & --- \\
	138 & $[ \exists \ns \exists \cp \exists \cpas \exists \as ]$ & & 122 & --- & --- \\
	139 & $[ \forall \as \forall \ns \forall \cp \forall \cpas ]$ & & 91 & --- & --- \\
	140 & $[ \forall \as \forall \ns \forall \cp \exists \cpas ]$ & & 92 & --- & --- \\
	141 & $[ \forall \as \forall \ns \exists \cp \forall \cpas ]$ & & 93 & --- & --- \\
	142 & $[ \forall \as \forall \ns \exists \cp \exists \cpas ]$ & & 94 & --- & --- \\
	143 & $[ \forall \as \exists \ns \forall \cp \forall \cpas ]$ & & \cmark\ & \xmark\ by \ref{prop:mode_maps_violating_CP}(\hyperref[item:violation_CP_a]{a}) & --- \\
	144 & $[ \forall \as \exists \ns \forall \cp \exists \cpas ]$ & & \cmark\ & \xmark\ by \ref{prop:mode_maps_violating_AP}(\hyperref[item:violation_AP_c]{c}) & --- \\
	145 & $[ \forall \as \exists \ns \exists \cp \forall \cpas ]$ & & \cmark\ & \xmark\ by \ref{prop:mode_maps_violating_AP}(\hyperref[item:violation_AP_a]{a}) & --- \\
	146 & $[ \forall \as \exists \ns \exists \cp \exists \cpas ]$ & & \cmark\ & \xmark\ by \ref{prop:mode_maps_violating_AP}(\hyperref[item:violation_AP_a]{a}, \hyperref[item:violation_AP_c]{c}) & --- \\
	147 & $[ \forall \as \forall \cp \forall \ns \forall \cpas ]$ & & 91 & --- & --- \\
	148 & $[ \forall \as \forall \cp \forall \ns \exists \cpas ]$ & & 92 & --- & --- \\
	149 & $[ \forall \as \forall \cp \exists \ns \forall \cpas ]$ & & \cmark\ & \xmark\ by \ref{prop:mode_maps_violating_CP}(\hyperref[item:violation_CP_a]{a}) & --- \\
	150 & $[ \forall \as \forall \cp \exists \ns \exists \cpas ]$ & & \cmark\ & \xmark\ by \ref{prop:mode_maps_violating_AP}(\hyperref[item:violation_AP_c]{c}) & --- \\
	151 & $[ \forall \as \forall \cp \forall \cpas \forall \ns ]$ & & 91 & --- & --- \\
	152 & $[ \forall \as \forall \cp \forall \cpas \exists \ns ]$ & & \cmark\ & \cmark\ by \ref{prop:mode_map_all_axioms} & 16 \\
	153 & $[ \forall \as \forall \cp \exists \cpas \forall \ns ]$ & & \cmark\ & \xmark\ by \ref{prop:mode_maps_violating_AP}(\hyperref[item:violation_AP_b]{b}) & --- \\
	154 & $[ \forall \as \forall \cp \exists \cpas \exists \ns ]$ & & 150 & --- & --- \\
	155 & $[ \forall \as \exists \cp \forall \ns \forall \cpas ]$ & & \cmark\ & \xmark\ by \ref{prop:mode_maps_violating_AP}(\hyperref[item:violation_AP_a]{a}, \hyperref[item:violation_AP_b]{b}) & --- \\
	156 & $[ \forall \as \exists \cp \forall \ns \exists \cpas ]$ & & \cmark\ & \xmark\ by \ref{prop:mode_maps_violating_AP}(\hyperref[item:violation_AP_a]{a}, \hyperref[item:violation_AP_b]{b}, \hyperref[item:violation_AP_e]{e}) & --- \\
	157 & $[ \forall \as \exists \cp \exists \ns \forall \cpas ]$ & & 145 & --- & --- \\
	158 & $[ \forall \as \exists \cp \exists \ns \exists \cpas ]$ & & 146 & --- & --- \\
	159 & $[ \forall \as \exists \cp \forall \cpas \forall \ns ]$ & & 155 & --- & --- \\
	160 & $[ \forall \as \exists \cp \forall \cpas \exists \ns ]$ & & \cmark\ & \xmark\ by \ref{prop:mode_maps_violating_AP}(\hyperref[item:violation_AP_a]{a}) & --- \\
	161 & $[ \forall \as \exists \cp \exists \cpas \forall \ns ]$ & & \cmark\ & \xmark\ by \ref{prop:mode_maps_violating_AP}(\hyperref[item:violation_AP_a]{a}, \hyperref[item:violation_AP_b]{b}) & --- \\
	162 & $[ \forall \as \exists \cp \exists \cpas \exists \ns ]$ & & 146 & --- & --- \\
	163 & $[ \exists \as \forall \ns \forall \cp \forall \cpas ]$ & & \cmark\ & \xmark\ by \ref{prop:mode_maps_violating_CP}(\hyperref[item:violation_CP_d]{d}) & --- \\
	164 & $[ \exists \as \forall \ns \forall \cp \exists \cpas ]$ & & \cmark\ & \xmark\ by \ref{prop:mode_maps_violating_AP}(\hyperref[item:violation_AP_e]{e}) & --- \\
	165 & $[ \exists \as \forall \ns \exists \cp \forall \cpas ]$ & & \cmark\ & \xmark\ by \ref{prop:mode_maps_violating_AP}(\hyperref[item:violation_AP_a]{a}) & --- \\
	166 & $[ \exists \as \forall \ns \exists \cp \exists \cpas ]$ & & \cmark\ & \xmark\ by \ref{prop:mode_maps_violating_AP}(\hyperref[item:violation_AP_a]{a}, \hyperref[item:violation_AP_e]{e}) & --- \\
	167 & $[ \exists \as \exists \ns \forall \cp \forall \cpas ]$ & (\swgpwa) & 119 & --- & --- \\
	168 & $[ \exists \as \exists \ns \forall \cp \exists \cpas ]$ & & 120 & --- & --- \\
	169 & $[ \exists \as \exists \ns \exists \cp \forall \cpas ]$ & & 121 & --- & --- \\
	170 & $[ \exists \as \exists \ns \exists \cp \exists \cpas ]$ & & 122 & --- & --- \\
	171 & $[ \exists \as \forall \cp \forall \ns \forall \cpas ]$ & & 163 & --- & --- \\
	172 & $[ \exists \as \forall \cp \forall \ns \exists \cpas ]$ & & 164 & --- & --- \\
	173 & $[ \exists \as \forall \cp \exists \ns \forall \cpas ]$ & (\swgwpa) & \cmark\ & \xmark\ by \ref{prop:mode_maps_violating_CP}(\hyperref[item:violation_CP_d]{d}) & --- \\
	174 & $[ \exists \as \forall \cp \exists \ns \exists \cpas ]$ & & \cmark\ & \xmark\ by \ref{prop:mode_maps_violating_AP}(\hyperref[item:violation_AP_c]{c}), \ref{prop:mode_maps_violating_CP}(\hyperref[item:violation_CP_b]{b}) & --- \\
	175 & $[ \exists \as \forall \cp \forall \cpas \forall \ns ]$ & & 163 & --- & --- \\
	176 & $[ \exists \as \forall \cp \forall \cpas \exists \ns ]$ & (\swgwap) & \cmark\ & \cmark\ by \ref{prop:mode_map_all_axioms} & \cmark \\
	177 & $[ \exists \as \forall \cp \exists \cpas \forall \ns ]$ & & \cmark\ & \cmark\ by \ref{prop:mode_map_all_axioms} & 5 \\
	178 & $[ \exists \as \forall \cp \exists \cpas \exists \ns ]$ & & 174 & --- & --- \\
	179 & $[ \exists \as \exists \cp \forall \ns \forall \cpas ]$ & & \cmark\ & \xmark\ by \ref{prop:mode_maps_violating_AP}(\hyperref[item:violation_AP_a]{a}) & --- \\
	180 & $[ \exists \as \exists \cp \forall \ns \exists \cpas ]$ & & \cmark\ & \xmark\ by \ref{prop:mode_maps_violating_AP}(\hyperref[item:violation_AP_a]{a}, \hyperref[item:violation_AP_e]{e}) & --- \\
	181 & $[ \exists \as \exists \cp \exists \ns \forall \cpas ]$ & & 121 & --- & --- \\
	182 & $[ \exists \as \exists \cp \exists \ns \exists \cpas ]$ & & 122 & --- & --- \\
	183 & $[ \exists \as \exists \cp \forall \cpas \forall \ns ]$ & & 179 & --- & --- \\
	184 & $[ \exists \as \exists \cp \forall \cpas \exists \ns ]$ & & \cmark\ & \xmark\ by \ref{prop:mode_maps_violating_AP}(\hyperref[item:violation_AP_a]{a}) & --- \\
	185 & $[ \exists \as \exists \cp \exists \cpas \forall \ns ]$ & & \cmark\ & \xmark\ by \ref{prop:mode_maps_violating_AP}(\hyperref[item:violation_AP_a]{a}) & --- \\
	186 & $[ \exists \as \exists \cp \exists \cpas \exists \ns ]$ & & 122 & --- & --- \\
	187 & $[ \forall \cp \forall \ns \forall \as \forall \cpas ]$ & & 91 & --- & --- \\
	188 & $[ \forall \cp \forall \ns \forall \as \exists \cpas ]$ & & 92 & --- & --- \\
	189 & $[ \forall \cp \forall \ns \exists \as \forall \cpas ]$ & (\swwga) & 101 & --- & --- \\
	190 & $[ \forall \cp \forall \ns \exists \as \exists \cpas ]$ & & 102 & --- & --- \\
	191 & $[ \forall \cp \forall \ns \forall \cpas \forall \as ]$ & & 91 & --- & --- \\
	192 & $[ \forall \cp \forall \ns \forall \cpas \exists \as ]$ & (\swwag) & 104 & --- & --- \\
	193 & $[ \forall \cp \forall \ns \exists \cpas \forall \as ]$ & & 105 & --- & --- \\
	194 & $[ \forall \cp \forall \ns \exists \cpas \exists \as ]$ & & 102 & --- & --- \\
	195 & $[ \forall \cp \exists \ns \forall \as \forall \cpas ]$ & & \cmark\ & \xmark\ by \ref{prop:mode_maps_violating_AP}(\hyperref[item:violation_AP_d]{d}), \ref{prop:mode_maps_violating_CP}(\hyperref[item:violation_CP_a]{a}) & --- \\
	196 & $[ \forall \cp \exists \ns \forall \as \exists \cpas ]$ & & \cmark\ & \xmark\ by \ref{prop:mode_maps_violating_AP}(\hyperref[item:violation_AP_c]{c}, \hyperref[item:violation_AP_d]{d}) & --- \\
	197 & $[ \forall \cp \exists \ns \exists \as \forall \cpas ]$ & (\swwpga) & \cmark\ & \xmark\ by \ref{prop:mode_maps_violating_CP}(\hyperref[item:violation_CP_d]{d}) & --- \\
	198 & $[ \forall \cp \exists \ns \exists \as \exists \cpas ]$ & & \cmark\ & \xmark\ by \ref{prop:mode_maps_violating_AP}(\hyperref[item:violation_AP_c]{c}), \ref{prop:mode_maps_violating_CP}(\hyperref[item:violation_CP_b]{b}) & --- \\
	199 & $[ \forall \cp \exists \ns \forall \cpas \forall \as ]$ & & 195 & --- & --- \\
	200 & $[ \forall \cp \exists \ns \forall \cpas \exists \as ]$ & (\swwpag) & \cmark\ & \xmark\ by \ref{prop:mode_maps_violating_CP}(\hyperref[item:violation_CP_e]{e}) & --- \\
	201 & $[ \forall \cp \exists \ns \exists \cpas \forall \as ]$ & & \cmark\ & \xmark\ by \ref{prop:mode_maps_violating_AP}(\hyperref[item:violation_AP_c]{c}, \hyperref[item:violation_AP_d]{d}) & --- \\
	202 & $[ \forall \cp \exists \ns \exists \cpas \exists \as ]$ & & 198 & --- & --- \\
	203 & $[ \forall \cp \forall \as \forall \ns \forall \cpas ]$ & & 91 & --- & --- \\
	204 & $[ \forall \cp \forall \as \forall \ns \exists \cpas ]$ & & 92 & --- & --- \\
	205 & $[ \forall \cp \forall \as \exists \ns \forall \cpas ]$ & & 149 & --- & --- \\
	206 & $[ \forall \cp \forall \as \exists \ns \exists \cpas ]$ & & 150 & --- & --- \\
	207 & $[ \forall \cp \forall \as \forall \cpas \forall \ns ]$ & & 91 & --- & --- \\
	208 & $[ \forall \cp \forall \as \forall \cpas \exists \ns ]$ & & 152 & --- & --- \\
	209 & $[ \forall \cp \forall \as \exists \cpas \forall \ns ]$ & & 153 & --- & --- \\
	210 & $[ \forall \cp \forall \as \exists \cpas \exists \ns ]$ & & 150 & --- & --- \\
	211 & $[ \forall \cp \exists \as \forall \ns \forall \cpas ]$ & & \cmark\ & \xmark\ by \ref{prop:mode_maps_violating_CP}(\hyperref[item:violation_CP_d]{d}) & --- \\
	212 & $[ \forall \cp \exists \as \forall \ns \exists \cpas ]$ & & \cmark\ & \xmark\ by \ref{prop:mode_maps_violating_AP}(\hyperref[item:violation_AP_e]{e}) & --- \\
	213 & $[ \forall \cp \exists \as \exists \ns \forall \cpas ]$ & (\swwgpa) & 197 & --- & --- \\
	214 & $[ \forall \cp \exists \as \exists \ns \exists \cpas ]$ & & 198 & --- & --- \\
	215 & $[ \forall \cp \exists \as \forall \cpas \forall \ns ]$ & & 211 & --- & --- \\
	216 & $[ \forall \cp \exists \as \forall \cpas \exists \ns ]$ & (\swwgap) & \cmark\ & \cmark\ by \ref{prop:mode_map_all_axioms} & \cmark \\
	217 & $[ \forall \cp \exists \as \exists \cpas \forall \ns ]$ & & \cmark\ & \cmark\ by \ref{prop:mode_map_all_axioms} & 5 \\
	218 & $[ \forall \cp \exists \as \exists \cpas \exists \ns ]$ & & 198 & --- & --- \\
	219 & $[ \forall \cp \forall \cpas \forall \ns \forall \as ]$ & & 91 & --- & --- \\
	220 & $[ \forall \cp \forall \cpas \forall \ns \exists \as ]$ & (\swwag) & 104 & --- & --- \\
	221 & $[ \forall \cp \forall \cpas \exists \ns \forall \as ]$ & & \cmark\ & \xmark\ by \ref{prop:mode_maps_violating_AP}(\hyperref[item:violation_AP_d]{d}) & --- \\
	222 & $[ \forall \cp \forall \cpas \exists \ns \exists \as ]$ & (\swwapg) & \cmark\ & \xmark\ by \ref{prop:mode_maps_violating_CP}(\hyperref[item:violation_CP_e]{e}) & --- \\
	223 & $[ \forall \cp \forall \cpas \forall \as \forall \ns ]$ & & 91 & --- & --- \\
	224 & $[ \forall \cp \forall \cpas \forall \as \exists \ns ]$ & & 152 & --- & --- \\
	225 & $[ \forall \cp \forall \cpas \exists \as \forall \ns ]$ & (\swe) & \cmark\ & \cmark\ by \ref{prop:mode_map_all_axioms} & \cmark \\
	226 & $[ \forall \cp \forall \cpas \exists \as \exists \ns ]$ & (\swwagp) & 222 & --- & --- \\
	227 & $[ \forall \cp \exists \cpas \forall \ns \forall \as ]$ & & \cmark\ & \xmark\ by \ref{prop:mode_maps_violating_AP}(\hyperref[item:violation_AP_b]{b}) & --- \\
	228 & $[ \forall \cp \exists \cpas \forall \ns \exists \as ]$ & & \cmark\ & \xmark\ by \ref{prop:mode_maps_violating_CP}(\hyperref[item:violation_CP_c]{c}) & --- \\
	229 & $[ \forall \cp \exists \cpas \exists \ns \forall \as ]$ & & 201 & --- & --- \\
	230 & $[ \forall \cp \exists \cpas \exists \ns \exists \as ]$ & & 198 & --- & --- \\
	231 & $[ \forall \cp \exists \cpas \forall \as \forall \ns ]$ & & 227 & --- & --- \\
	232 & $[ \forall \cp \exists \cpas \forall \as \exists \ns ]$ & & \cmark\ & \xmark\ by \ref{prop:mode_maps_violating_AP}(\hyperref[item:violation_AP_c]{c}) & --- \\
	233 & $[ \forall \cp \exists \cpas \exists \as \forall \ns ]$ & & 217 & --- & --- \\
	234 & $[ \forall \cp \exists \cpas \exists \as \exists \ns ]$ & & 198 & --- & --- \\
	235 & $[ \exists \cp \forall \ns \forall \as \forall \cpas ]$ & & \cmark\ & \xmark\ by \ref{prop:mode_maps_violating_AP}(\hyperref[item:violation_AP_a]{a}, \hyperref[item:violation_AP_b]{b}) & --- \\
	236 & $[ \exists \cp \forall \ns \forall \as \exists \cpas ]$ & & \cmark\ & \xmark\ by \ref{prop:mode_maps_violating_AP}(\hyperref[item:violation_AP_a]{a}, \hyperref[item:violation_AP_b]{b}, \hyperref[item:violation_AP_e]{e}) & --- \\
	237 & $[ \exists \cp \forall \ns \exists \as \forall \cpas ]$ & & \cmark\ & \xmark\ by \ref{prop:mode_maps_violating_AP}(\hyperref[item:violation_AP_a]{a}) & --- \\
	238 & $[ \exists \cp \forall \ns \exists \as \exists \cpas ]$ & & \cmark\ & \xmark\ by \ref{prop:mode_maps_violating_AP}(\hyperref[item:violation_AP_a]{a}, \hyperref[item:violation_AP_e]{e}) & --- \\
	239 & $[ \exists \cp \forall \ns \forall \cpas \forall \as ]$ & & 235 & --- & --- \\
	240 & $[ \exists \cp \forall \ns \forall \cpas \exists \as ]$ & & \cmark\ & \xmark\ by \ref{prop:mode_maps_violating_AP}(\hyperref[item:violation_AP_a]{a}) & --- \\
	241 & $[ \exists \cp \forall \ns \exists \cpas \forall \as ]$ & & \cmark\ & \xmark\ by \ref{prop:mode_maps_violating_AP}(\hyperref[item:violation_AP_a]{a}, \hyperref[item:violation_AP_b]{b}, \hyperref[item:violation_AP_e]{e}) & --- \\
	242 & $[ \exists \cp \forall \ns \exists \cpas \exists \as ]$ & & 238 & --- & --- \\
	243 & $[ \exists \cp \exists \ns \forall \as \forall \cpas ]$ & & 131 & --- & --- \\
	244 & $[ \exists \cp \exists \ns \forall \as \exists \cpas ]$ & & 132 & --- & --- \\
	245 & $[ \exists \cp \exists \ns \exists \as \forall \cpas ]$ & & 121 & --- & --- \\
	246 & $[ \exists \cp \exists \ns \exists \as \exists \cpas ]$ & & 122 & --- & --- \\
	247 & $[ \exists \cp \exists \ns \forall \cpas \forall \as ]$ & & 131 & --- & --- \\
	248 & $[ \exists \cp \exists \ns \forall \cpas \exists \as ]$ & & 136 & --- & --- \\
	249 & $[ \exists \cp \exists \ns \exists \cpas \forall \as ]$ & & 137 & --- & --- \\
	250 & $[ \exists \cp \exists \ns \exists \cpas \exists \as ]$ & & 122 & --- & --- \\
	251 & $[ \exists \cp \forall \as \forall \ns \forall \cpas ]$ & & 235 & --- & --- \\
	252 & $[ \exists \cp \forall \as \forall \ns \exists \cpas ]$ & & 236 & --- & --- \\
	253 & $[ \exists \cp \forall \as \exists \ns \forall \cpas ]$ & & \cmark\ & \xmark\ by \ref{prop:mode_maps_violating_AP}(\hyperref[item:violation_AP_a]{a}) & --- \\
	254 & $[ \exists \cp \forall \as \exists \ns \exists \cpas ]$ & & \cmark\ & \xmark\ by \ref{prop:mode_maps_violating_AP}(\hyperref[item:violation_AP_a]{a}, \hyperref[item:violation_AP_c]{c}) & --- \\
	255 & $[ \exists \cp \forall \as \forall \cpas \forall \ns ]$ & & 235 & --- & --- \\
	256 & $[ \exists \cp \forall \as \forall \cpas \exists \ns ]$ & & \cmark\ & \xmark\ by \ref{prop:mode_maps_violating_AP}(\hyperref[item:violation_AP_a]{a}) & --- \\
	257 & $[ \exists \cp \forall \as \exists \cpas \forall \ns ]$ & & \cmark\ & \xmark\ by \ref{prop:mode_maps_violating_AP}(\hyperref[item:violation_AP_a]{a}, \hyperref[item:violation_AP_b]{b}) & --- \\
	258 & $[ \exists \cp \forall \as \exists \cpas \exists \ns ]$ & & 254 & --- & --- \\
	259 & $[ \exists \cp \exists \as \forall \ns \forall \cpas ]$ & & 179 & --- & --- \\
	260 & $[ \exists \cp \exists \as \forall \ns \exists \cpas ]$ & & 180 & --- & --- \\
	261 & $[ \exists \cp \exists \as \exists \ns \forall \cpas ]$ & & 121 & --- & --- \\
	262 & $[ \exists \cp \exists \as \exists \ns \exists \cpas ]$ & & 122 & --- & --- \\
	263 & $[ \exists \cp \exists \as \forall \cpas \forall \ns ]$ & & 179 & --- & --- \\
	264 & $[ \exists \cp \exists \as \forall \cpas \exists \ns ]$ & & 184 & --- & --- \\
	265 & $[ \exists \cp \exists \as \exists \cpas \forall \ns ]$ & & 185 & --- & --- \\
	266 & $[ \exists \cp \exists \as \exists \cpas \exists \ns ]$ & & 122 & --- & --- \\
	267 & $[ \exists \cp \forall \cpas \forall \ns \forall \as ]$ & & 235 & --- & --- \\
	268 & $[ \exists \cp \forall \cpas \forall \ns \exists \as ]$ & & 240 & --- & --- \\
	269 & $[ \exists \cp \forall \cpas \exists \ns \forall \as ]$ & & \cmark\ & \xmark\ by \ref{prop:mode_maps_violating_AP}(\hyperref[item:violation_AP_a]{a}, \hyperref[item:violation_AP_d]{d}) & --- \\
	270 & $[ \exists \cp \forall \cpas \exists \ns \exists \as ]$ & & \cmark\ & \xmark\ by \ref{prop:mode_maps_violating_AP}(\hyperref[item:violation_AP_a]{a}) & --- \\
	271 & $[ \exists \cp \forall \cpas \forall \as \forall \ns ]$ & & 235 & --- & --- \\
	272 & $[ \exists \cp \forall \cpas \forall \as \exists \ns ]$ & & 256 & --- & --- \\
	273 & $[ \exists \cp \forall \cpas \exists \as \forall \ns ]$ & & \cmark\ & \xmark\ by \ref{prop:mode_maps_violating_AP}(\hyperref[item:violation_AP_a]{a}) & --- \\
	274 & $[ \exists \cp \forall \cpas \exists \as \exists \ns ]$ & & 270 & --- & --- \\
	275 & $[ \exists \cp \exists \cpas \forall \ns \forall \as ]$ & & \cmark\ & \xmark\ by \ref{prop:mode_maps_violating_AP}(\hyperref[item:violation_AP_a]{a}, \hyperref[item:violation_AP_b]{b}) & --- \\
	276 & $[ \exists \cp \exists \cpas \forall \ns \exists \as ]$ & & \cmark\ & \xmark\ by \ref{prop:mode_maps_violating_AP}(\hyperref[item:violation_AP_a]{a}) & --- \\
	277 & $[ \exists \cp \exists \cpas \exists \ns \forall \as ]$ & & 137 & --- & --- \\
	278 & $[ \exists \cp \exists \cpas \exists \ns \exists \as ]$ & & 122 & --- & --- \\
	279 & $[ \exists \cp \exists \cpas \forall \as \forall \ns ]$ & & 275 & --- & --- \\
	280 & $[ \exists \cp \exists \cpas \forall \as \exists \ns ]$ & & \cmark\ & \xmark\ by \ref{prop:mode_maps_violating_AP}(\hyperref[item:violation_AP_a]{a}, \hyperref[item:violation_AP_c]{c}) & --- \\
	281 & $[ \exists \cp \exists \cpas \exists \as \forall \ns ]$ & & 185 & --- & --- \\
	282 & $[ \exists \cp \exists \cpas \exists \as \exists \ns ]$ & & 122 & --- & --- \\
\end{longtable}

\end{center}

\section*{Acknowledgements}
\addcontentsline{toc}{section}{Acknowledgements}

This work has been partially supported by the Deut\-sche For\-schungs\-ge\-mein\-schaft (DFG) through project \href{https://gepris.dfg.de/gepris/projekt/415980428}{415980428}.
IK was funded by the DFG under Germany's Excellence Strategy (EXC-2046/1, project \href{https://gepris.dfg.de/gepris/projekt/390685689}{390685689}) through projects EF1-10 and EF1-19 of the Berlin Mathematics Research Center MATH+.
HL is supported by the Warwick Mathematics Institute Centre for Doctoral Training and gratefully acknowledges funding from the University of Warwick and the UK Engineering and Physical Sciences Research Council (Grant number: EP/W524645/1).
TJS conducted parts of this work under a Turing Fellowship at the Alan Turing Institute.

For the purpose of open access, the authors have applied a Creative Commons Attribution (CC BY) licence to any Author Accepted Manuscript version arising.

\addcontentsline{toc}{section}{References}
\bibliographystyle{abbrvnat}
\bibliography{myBibliography}

\end{document}